\documentclass[A4paper,oneside,11pt]{article}
\usepackage{amssymb,amsthm,amsmath,marvosym,amsfonts,fontenc}
\usepackage[usenames,dvipsnames]{color}
\usepackage{enumerate,marginnote}
\usepackage[dvips]{graphicx}
\usepackage{mdwlist}
\usepackage{pifont}
\usepackage{bm}
\usepackage{multind}\makeindex{general}\makeindex{mathsymbols}
\usepackage{caption}
\usepackage{fullpage} 
\usepackage{epsfig,psfrag}
\usepackage{fancyhdr} 
\numberwithin{equation}{section}
\numberwithin{table}{section} 
\numberwithin{figure}{section}
\newtheorem{theorem}{Theorem}[section]
\newtheorem{subtheorem}{Theorem}[theorem] 
\newtheorem{subsubtheorem}{Theorem}[subtheorem]
\newtheorem{lem}[theorem]{Lemma}
\newtheorem{lemma}[theorem]{Lemma}
\newtheorem{conjecture}[theorem]{Conjecture}

\newtheorem{proposition}[theorem]{Lemma}

\newtheorem{corollary}[theorem]{Corollary}
\newtheorem{fact}[theorem]{Fact}
\newtheorem{remark}[theorem]{Remark}
\newtheorem{definition}[theorem]{Definition}
\newtheorem{setting}[theorem]{Setting}

\theoremstyle{remark}
\newtheorem{claim}[subtheorem]{Claim}
\newtheorem{subremark}[subtheorem]{Remark}

\newtheorem{subclaim}[subsubtheorem]{Subclaim}

\newcommand{\By}[2]{\overset{\mbox{\tiny{#1}}}{#2}}
\newcommand{\ByRef}[2]{   \By{\eqref{#1}}{#2} }
\newcommand{\eqBy}[1]{    \By{#1}{=} }
\newcommand{\lBy}[1]{     \By{#1}{<} }

\newcommand{\leBy}[1]{    \By{#1}{\le} }
\newcommand{\geBy}[1]{    \By{#1}{\ge} }
\newcommand{\eqByRef}[1]{ \ByRef{#1}{=} }

\newcommand{\gByRef}[1]{  \ByRef{#1}{>} }
\newcommand{\leByRef}[1]{ \ByRef{#1}{\le} }
\newcommand{\geByRef}[1]{ \ByRef{#1}{\ge} }

\let\sm\setminus
\let\subset\subseteq 
\let\supset\supseteq 
\let\epsilon\varepsilon
\let\sharp\#
\def\dcup{\dot\cup} 
\renewcommand{\leq}{\leqslant}
\renewcommand{\le}{\leqslant}
\renewcommand{\geq}{\geqslant}
\renewcommand{\ge}{\geqslant}

\let\oldmarginpar\marginpar
\renewcommand\marginpar[1]{\-\oldmarginpar[\raggedleft\footnotesize #1]%
{\raggedright\footnotesize #1}}

\newcommand{\HIDDENPROOF}[1]{
}

\newcommand{\HIDDENTEXT}[1]{
}

\title{The Approximate Loebl--Koml\'os--S\'os
Conjecture} \author{Jan
Hladk\'y\thanks{\emph{Corresponding author.} Centre for Discrete Mathematics and its
Applications (DIMAP) and Mathematics Institute, University of
Warwick, Coventry, CV4~7AL, UK, and Computer Science Institute of Charles University, Malostransk\'e
n\'am.~25, 118~00, Praha~1. Email:
\texttt{honzahladky@gmail.com}}
\quad J\'anos Koml\'os\thanks{Department of Mathematics, Rutgers University, 110 Frelinghuysen Rd., Piscataway, NJ~08854-8019, USA} \quad Diana Piguet\thanks{School of Mathematics, University of Birmingham, Edgbaston, Birmingham, 
    B15~2TT,
    UK.
    Email: \texttt{D.Piguet@bham.ac.uk}}\\ Mikl\'os Simonovits\thanks{R\'enyi
    Institute, Budapest, Hungary. Email:
\texttt{miki@renyi.hu}} \quad Maya Stein\thanks{Centro de Modelamiento Matem\'atico,
Universidad de Chile, Blanco Encalada, 2120, Santiago, Chile,
Email: \texttt{mstein@dim.uchile.cl}} \quad Endre Szemer\'edi\thanks{Department
of Computer Science, Rutgers University, 110 Frelinghuysen Rd., Piscataway, NJ~08854-8019, USA}}

\newcommand{\HAPPY}[1]{}

\newcommand{\PARAMETERPASSING}[2]{{\triangleright\mathrm{#1}\ref{#2}}}
\newcommand{\PARAMETERPASSINGR}[3]{{\triangleright\mathrm{#1}\ref{#2}-\ref{#3}}}

\newcommand{\M}{\mathcal M}\newcommand{\C}{\mathcal C}\newcommand{\A}{\mathcal
A}\newcommand{\V}{\mathcal V}
\newcommand{\eps}{\epsilon}

\def\bigboxplus{\mbox{\Large $\boxplus$}}

\def\probability{\mathbf{P}}
\def\expectation{\mathbf{E}}

\def\mindeg{\mathrm{deg^{min}}}
\def\maxdeg{\mathrm{deg^{max}}}
\def\density{\mathrm{d}}
\def\neighbor{\mathrm{N}}
\def\neighbour{\mathrm{N}}
\def\Vodd{V_\mathrm{odd}}
\def\Veven{V_\mathrm{even}}

\def\dist{\mathrm{dist}}
\def\children{\mathrm{Ch}}
\def\parent{\mathrm{Par}}
\def\shadow{\mathbf{shadow}}
\newcommand{\treeclass}[1]{\mathbf{trees}({#1})}
\def\seed{\mathrm{Seed}}
\newcommand{\LKSgraphs}[3]{\mathbf{LKS}({#1},{#2},{#3})}
\newcommand{\LKSmingraphs}[3]{\mathbf{LKSmin}({#1},{#2},{#3})}
\newcommand{\LKSsmallgraphs}[3]{\mathbf{LKSsmall}({#1},{#2},{#3})}

\newcommand{\smallvertices}[3]{\mathbb{S}_{{#1},{#2}}({#3})}
\newcommand{\largevertices}[3]{\mathbb{L}_{{#1},{#2}}({#3})}

\newcommand{\JUSTIFY}[1]{\mbox{\tiny{(#1)}}\quad}

\def\Gcapt{G_\nabla}
\def\GD{G_{\mathcal{D}}}
\def\Gblack{G_{\mathrm{reg}}}
\def\Gexp{G_{\mathrm{exp}}}
\def\BGblack{\mathbf{G}_{\mathrm{reg}}}
\def\smallatoms{\mathfrak{A}}
\def\clusters{\mathbf{V}}
\def\class{\nabla}
\def\HugeVertices{\bm{\Psi}}
\def\DenseSpots{\mathcal{D}}

\def\YA{\mathbb{YA}}
\def\YB{\mathbb{YB}}
\def\WantiC{V_{\not\leadsto\HugeVertices}}
\def\ghost{\mathbf{ghost}}
 
\def\NUP{\mathrm{N}^{\uparrow}}
\def\NDOWN{\mathrm{N}^{\downarrow}}
\def\Vgood{V_\mathrm{good}}

\def\exceptVertSplit{\bar V}
\def\exceptSemSplit{\bar \V}
\def\exceptClustSplit{\bar \clusters}

\def\gC{\mathcal{C}}
\def\gP{\mathtt{P}}
\def\gPatoms{\mathtt{P}_{\smallatoms}}

\def\shrubA{\mathcal S_{A}}
\def\shrubB{\mathcal S_{B}}

\def\Duplicate{\mathsf{Duplicate}}

\def\XA{\mathbb{XA}}
\def\XB{\mathbb{XB}}
\def\XC{\mathbb{XC}}
\def\BS{\mathbf S} 
\def\BL{\mathbf L}
\def\BSN{\mathbf S^0} 
\def\BSI{\mathbf S^{\mathrm{I}}} 
\def\BSR{\mathbf S^\mathrm{R}}
\def\SEPARATOR{\mathbf Q} 
\def\SR{S^{\mathrm R}} 
\def\SN{S^{0}}
\def\Mgood{\M_{\mathrm{good}}}
\def\NAtom{{\mathcal N_{\smallatoms}}}

\newcommand{\colouringp}[1]{\mathfrak{P}_{#1}}
\newcommand{\colouringpI}[1]{^{\restriction{#1}}}
\def\colouringpartition{(\colouringp{0},\colouringp{1},\colouringp{2})}
\newcommand{\proporce}[1]{\mathfrak{p}_{#1}}
\def\shadowsplit{\mathbb{F}}

\def\largeintoatoms{V_{\leadsto\smallatoms}}
\def\clustersintoatoms{\clusters_{\leadsto\smallatoms}}

\def\clustersize{\mathfrak{c}}

\def\LargeTen{\mathcal L^*}

\def\epsilonD{\pi}
\def\alphaD{\widehat{\alpha}}

\renewcommand{\today}{}
\date{}

\begin{document}
\pagenumbering{roman}
\maketitle
\begin{abstract}
We prove the following version of the
Loebl--Koml\'os--S\'os Conjecture: For every~$\alpha>0$
there exists a number $k_0$ such that for every~$k>k_0$
 every $n$-vertex graph $G$ with at least~$(\frac12+\alpha)n$ vertices
of degree at least~$(1+\alpha)k$ contains each tree $T$ of order $k$ as a
subgraph. 


The method to prove our result follows a strategy common to
approaches which employ the Szemer\'edi Regularity Lemma: we
decompose the graph~$G$, find a suitable combinatorial structure inside the decomposition, and then embed the tree~$T$ into~$G$ using this structure. However, the decomposition given by the Regularity Lemma is not of help when~$G$ is sparse. To surmount this shortcoming we use a more general decomposition technique: each graph can be decomposed into vertices of huge degree, regular pairs (in the sense of the Regularity Lemma), and two other objects each exhibiting certain expansion properties.
\end{abstract}


\newpage

\rhead{\today}

\newpage

\tableofcontents
\newpage
\pagenumbering{arabic}
\setcounter{page}{1}

\section{Introduction}\label{sec:intro}
\subsection{Statement of the problem}\label{ssec:veryIntro}
We provide an approximate solution of the Loebl--Koml\'os--S\'os Conjecture. This is a problem in extremal graph theory which fits the classical form \emph{Does a certain density condition imposed on a graph guarantee a certain subgraph?} Classical results of this type include Dirac's Theorem which determines the minimum degree threshold for containment of a Hamilton cycle, or Mantel's Theorem which determines the average degree threshold for containment of a triangle. Indeed, most of these extremal problems are formulated in terms of the minimum or average degree of the host graph.

We investigate density conditions which guarantee that a host graph contains
\emph{each} tree of order~$k$. The greedy tree-embedding strategy shows that
minimum degree more of than $k-2$ is a sufficient condition. Further, this bound
is best possible as any $(k-2)$-regular graph avoids the $k$-vertex star. However, Erd\H os and S\'os conjectured that the minimum degree condition can be relaxed to an average degree one still giving the same conclusion.
\begin{conjecture}[Erd\H os--S\'os Conjecture 1963]\label{conj:ES}\index{general}{Erdos-Sos Conjecture@Erd\H os-S\'os Conjecture}
Let $G$ be a graph of average degree greater than  $k-2$. Then $G$ contains each
tree of order $k$ as a subgraph.
\end{conjecture}
A solution of the Erd\H os--S\'os Conjecture for all $k$ bigger than an absolute constant was announced by Ajtai, Koml\'os, Simonovits, and Szemer\'edi in the early 1990's.
In a similar spirit, Loebl, Koml\'os, and S\'os conjectured that a \emph{median degree} of  $k-1$ or more is sufficient for containment of any tree of order $k$. By median degree we mean the degree of a vertex in the middle of the ordered degree sequence. 
\begin{conjecture}[Loebl--Koml\'os--S\'os Conjecture 1995~\cite{EFLS95}]\label{conj:LKS}
Suppose that $G$ is an $n$-vertex graph with at least $n/2$ vertices of degree more than $k-2$. Then $G$ contains each tree of order $k$.
\end{conjecture}
We discuss in detail Conjectures~\ref{conj:ES} and~\ref{conj:LKS} in
Section~\ref{ssec:LKS}. Here, we just state the main result of the paper, an
approximate solution of the Loebl--Koml\'os--S\'os Conjecture.
\begin{theorem}[Main result]\label{thm:main}
For every $\alpha>0$ there exists  $k_0$ such that for any
$k>k_0$ we have the following. Each $n$-vertex graph $G$ with at least
$(\frac12+\alpha)n$ vertices of degree at least $(1+\alpha)k$ contains each tree $T$ of
order $k$.
\end{theorem}

\subsection{Regularity lemma and dense graph theory}\label{ssec:iRL}
The Szemer\'edi Regularity Lemma has been a major tool in extremal graph theory
for three decades. It provides an approximate representation of a graph with
a so-called \emph{cluster graph}. This cluster graph representation is the key for graph-containment problems. The usual strategy here is that instead of solving the original problem one focuses on a modified simpler problem in the cluster graph.

The applicability of the Szemer\'edi Regularity Lemma is, however, limited to \emph{dense graphs}, i.e., graphs that contain a substantial proportion of all possible edges. Luckily enough many graphs arising in extremal graph theory are dense, as for example those coming from Dirac's and Mantel's Theorem above. But, while the proofs of these two results are elementary many of their extensions rely on the Regularity Lemma.

While the theory of dense graphs is well understood due to the Szemer\'edi Regularity Lemma, no such tool is available for sparse graphs. A regularity type representation of general (possibly sparse) graphs is one of the most important goals of contemporary discrete mathematics. By such a representation we mean an approximation of the input graph by a structure of bounded complexity carrying enough of the important information about the graph.

A central tool in the proof of Theorem~\ref{thm:main} is a structural
decomposition of the graph $G_\PARAMETERPASSING{T}{thm:main}$. This
decomposition --- which we call \emph{sparse decomposition} --- applies to any
graph whose average degree is bigger than an absolute constant. The sparse decomposition provides a partition of any graph into vertices of huge degrees and into a bounded degree part. The bounded degree part is further decomposed into dense regular pairs, an edge set with certain expander-like properties, and a vertex set which is expanding in a different way (we shall give a more precise description in Section~\ref{ssec:OverviewOfProof}). 
This kind of decomposition was first used by Ajtai, Koml\'os, Simonovits, and Szemer\'edi in their yet unpublished work on the Erd\H os--S\'os Conjecture.

In the case of dense graphs the sparse decomposition produces a Szemer\'edi regularity partition, and thus the decomposition lemma (Lemma~\ref{lem:decompositionIntoBlackandExpanding}) extends the Szemer\'edi Regularity Lemma. 
But, the interesting setting for
 the  Decomposition Lemma are sparse graphs. Being sparse, these graphs may be expected to contain less interesting substructures than dense graphs, and so, it comes as no surprise that the output of Lemma~\ref{lem:decompositionIntoBlackandExpanding} in this setting is  less useful than a Szemer\'edi regularity partition for dense graphs. If we think of  graph containment problems, the applicability of Lemma~\ref{lem:decompositionIntoBlackandExpanding} seems to be limited to simple structures as trees.

\subsection{Loebl--Koml\'os--S\'os Conjecture and Erd\H os--S\'os Conjecture}\label{ssec:LKS}
Let us first introduce some notation. We say that $H$ \emph{embeds} in a graph $G$ and write $H\subset G$ if $H$ is a (not necessarily induced) subgraph of $G$. The associated map $\phi:V(H)\rightarrow V(G)$ is called an \index{general}{embedding}\emph{embedding of $H$ in $G$}.
More generally, for a graph class $\mathcal H$ we
write $\mathcal H\subset G$ if $H\subset G$ for
every $H\in\mathcal H$. 
Let \index{mathsymbols}{*Trees@$\treeclass{k}$}$\treeclass{k}$ be the
class of all trees of order $k$.

\medskip

Conjecture~\ref{conj:LKS} is dominated by two parameters: one quantifies the number of vertices of `large' degree, and the other tells us how large this degree should actually be. Strengthening either of these bounds sufficiently,  the conjecture becomes trivial. 
\footnote{
Indeed, if we replace $n/2$ with $n$, then any tree of order $k$ can be embedded greedily.
Also, if we replace $k-2$ with $4k-4$, then $G$, being a graph of average degree at least $2k-2$, has a subgraph $G'$ of minimum degree at least~$k-1$. Again we can greedily embed any tree of order $k$.
}

On the other hand, one may ask whether lower bounds would suffice.   For the bound
  $k-2$, this is not the case, since stars of order $k$ require a vertex of degree at least $k-1$ in the host graph. As for the bound $n/2$, the following example shows that this number cannot be decreased much.

First, assume that $n$ is even, and that $n=k$. Let $G^*$ be obtained from the complete graph on~$n$ vertices by deleting all edges inside a set of  $\frac
n2+1$ vertices. It is easy to
check that $G^*$ does not contain the
path\footnote{In general, $G^*$ does not contain any
tree $T\in\treeclass{k}$ with independence number less than $\frac k2+1$.}
 $P_{k}\in\treeclass{k}$. Now, taking the union of 
several disjoint copies  of $G^*$ we obtain examples for other values of $n$. (And adding a small complete component we can get to \emph{any} value of $n$.) See
Figure~\ref{fig:ExtremalGraph} for an illustration.
\begin{figure}[t]
\centering 
\includegraphics[scale=0.7]{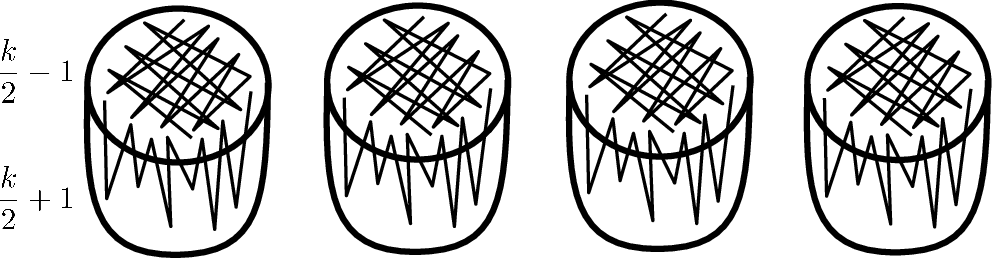}
\caption[Extremal graph for the Loebl--Koml\'os--S\'os Conjecture]{An extremal graph for the Loebl--Koml\'os--S\'os Conjecture.}
\label{fig:ExtremalGraph}
\end{figure}

However, we do not know of any example attaining the exact bound
 $ n/2$. Thus it might be possible to lower the bound $n/2$ from Conjecture~\ref{conj:LKS} to the one attained in our example above:
\begin{conjecture}\label{conj:LKSstronger}
Let $k\in \mathbb N$ and let $G$ be a graph on $n$ vertices, with more than $\frac n2-\lfloor \frac{n}{k}\rfloor - (n\mod k)$ vertices of degree at least $k-1$. Then $\treeclass{k}\subset G$.
\end{conjecture}
It might even be that
if $n/k$ is far from integrality, a slightly lower bound on the number of vertices of large degree still works (see~\cite{HladkyMSC,HlaPig:LKSdenseExact}).

\medskip

Several partial results concerning 
Conjecture~\ref{conj:LKS} have been obtained; let us briefly summarize  the major ones. Two main directions can be distinguished among those results that prove the conjecture for special classes of graphs: either one places restrictions on the host graph, or on the class of trees to be embedded. Of the latter type is the result by
Bazgan, Li, and Wo{\'z}niak~\cite{BLW00}, who  proved the
conjecture for paths. Also, Piguet and Stein~\cite{PS2} proved that 
Conjecture~\ref{conj:LKS} is true for trees of
diameter at most 5, which improved earlier results of Barr and Johansson~\cite{Barr} and Sun~\cite{Sun07}.

Restrictions on the host graph have led to the following results.
Soffer~\cite{Sof00} showed that Conjecture~\ref{conj:LKS} is true if the
host graph has girth at least 7. Dobson~\cite{Dob02} proved the
conjecture for host graphs whose complement does not contain a
$K_{2,3}$. This has been extended by Matsumoto and Sakamoto~\cite{MaSa} who replace the $K_{2,3}$ with a slightly larger graph.

\medskip

A different approach is to solve the conjecture for special values of $k$. One such case, known as the Loebl conjecture, or also as the ($n/2$--$n/2$--$n/2$)-Conjecture,  is the case $k=n/2$. Ajtai, Koml\'os, and Szemer\'edi~\cite{AKS95} solved an approximate version of this conjecture, and later
 Zhao~\cite{Z07+} used a refinement of this approach to prove the sharp version of the conjecture for large
graphs.

\medskip

An approximate version of Conjecture~\ref{conj:LKS} for dense graphs, that is, for $k$ linear in $n$, was proved by Piguet
and Stein~\cite{PS07+}. Let us take this opportunity to introduce a useful notation.
Write
\index{mathsymbols}{*LKSgraphs@$\LKSgraphs{n}{k}{\eta}$}$\LKSgraphs{n}{k}{\alpha}$
for the class of all $n$-vertex graphs with at least
$(\frac12+\alpha)n$ vertices of degrees at least
$(1+\alpha)k$. With this notation Conjecture~\ref{conj:LKS} states that every graph in $\LKSgraphs{n}{k}{0}$ contains every tree from $\treeclass{k+1}$.

\begin{theorem}[Piguet--Stein~\cite{PS07+}]\label{thm:PiguetStein}
For any $q>0$ and $\alpha>0$ there exists a number $n_0$ such that for any $n>n_0$ and
$k>qn$ the following holds. If  $G\in\LKSgraphs{n}{k}{\alpha}$ then
$\treeclass{k+1}\subset G$.
\end{theorem}

This result was proved using the regularity method. Adding stability arguments, 
 Hladk\'y and Piguet~\cite{HlaPig:LKSdenseExact}, and independently Cooley~\cite{Cooley08} proved Conjecture~\ref{conj:LKS} for large dense graphs.
 
\begin{theorem}[Hladk\'y--Piguet~\cite{HlaPig:LKSdenseExact}, Cooley~\cite{Cooley08}]\label{thm:denseLKS}
For any $q>0$ there exists a number $n_0=n_0(q)$ such that for any $n>n_0$ and $k>qn$ the following
holds. If $G\in\LKSgraphs{n}{k}{0}$ then $\treeclass{k+1}\subset G$.
\end{theorem}

\medskip

Let us now turn our attention to the  Erd\H os--S\'os Conjecture. It is particularly important to compare the structure of the respective extremal graph with the extremal graphs for the Loebl--K\'omlos--S\'os Conjecture.
The Erd\H os--S\'os Conjecture~\ref{conj:ES} is best possible whenever $n(k-2)$ is even. Indeed, in that case it suffices to consider a $(k-2)$-regular graph. This is a graph with average degree exactly $k-2$ which does not contain the star of order $k$. Even when the star (which in a sense is a pathological tree) is excluded from the considerations, we can --- at least when $k-1$ divides $n$ --- consider a disjoint union of $\frac{n}{k-1}$ cliques $K_{k-1}$. This graph contains \emph{no} tree from $\treeclass{k}$.

There is another important graph with many edges which does not contain for example the path $P_{k}$, depicted in Figure~\ref{fig:ExtremalGraphES2}. This graph has $\frac12(k-2)n-O(k^2)$ edges when $k$ is even and $\frac12(k-3)n-O(k^2)$ edges otherwise, and therefore gets close to the conjectured bound when $k\ll n$.
\begin{figure}[t]
\centering 
\includegraphics[scale=0.7]{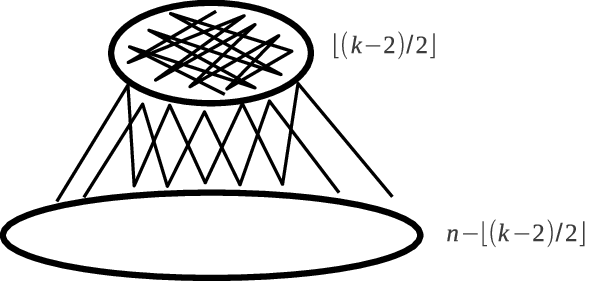}
\caption[Almost extremal graph for the Erd\H os--S\'os Conjecture]{An almost extremal graph for the Erd\H os--S\'os Conjecture.}
\label{fig:ExtremalGraphES2}
\end{figure}
Apart from the already mentioned announced breakthrough by Ajtai, Koml\'os, Simonovits, and Szemer\'edi, work on this conjecture includes~\cite{bradob,Haxell:TreeEmbeddings,MaSa,sacwoz,woz}.

\bigskip

Both Conjectures~\ref{conj:LKS} and Conjecture~\ref{conj:ES} have an important application in Ramsey theory. Each of them implies that the Ramsey number of two trees $T_{k+1}\in\treeclass{k+1}$, $T_{\ell+1}\in\treeclass{\ell+1}$ is bounded by $R(T_{k+1}, T_{\ell+1})\le k+\ell+1$. Actually more is implied: Any $2$-edge-colouring of $K_{k+\ell +1}$ contains either \emph{all} trees in $\treeclass{k+1}$ in red, or \emph{all} trees in $\treeclass{\ell+1}$ in blue.

The bound $R(T_{k+1}, T_{\ell+1})\le k+\ell+1$ is almost tight only for certain types of trees:
 Harary~\cite{Harary:RecentRamsey} showed $R(S_k,S_\ell)=k+\ell-2-\epsilon$ for stars $S_k\in\treeclass{k}$, $S_\ell\in\treeclass{\ell}$, where $\epsilon\in\{0,1\}$ depends on the
parity of $k$ and $\ell$. On the other hand,  Gerencs\'er and Gy\'arf\'as~\cite{GerencserGyarfas} showed $R(P_k,P_\ell)=\max\{k,\ell\}+\left\lfloor\frac{\min\{k,\ell\}}2\right\rfloor-1$ for paths
$P_k\in\treeclass{k}$, $P_\ell\in\treeclass{\ell}$.
 Haxell, \L uczak, and Tingley confirmed asymptotically~\cite{HLT02} that the discrepancy of the Ramsey bounds for trees depends on their balancedness, at least when the maximum degrees of the trees considered are moderately bounded.

\subsection{Related tree containment problems}

\paragraph{Trees in random graphs.}
To complete the picture of research involving tree containment problems we mention two rich and vivid (and also closely connected) areas: trees in random graphs, and trees in expanding graphs. The former area is centered around the following question: \emph{What is the probability threshold $p=p(n)$ for the Erd\H{o}s--R\'enyi random graph $G_{n,p}$ to contain asymptotically almost surely (a.a.s.) each tree/all trees from a given class $\mathcal F_n$ of trees?} Note that there is a difference between containing ``each tree'' and ``all trees'' as the error probabilities for missing individual trees might sum up.

Most research focused on containment of spanning trees, or almost spanning trees. The only well-understood case is when $\mathcal F_n=\{P_{k_n}\}$ is a path. The threshold $p=\frac{(1+o(1))\ln n}{n}$ for appearance of a spanning path (i.e., $k_n=n$) was determined by Koml\'os and Szemer\'edi~\cite{KomSze:HamiltonRandom}, and independently by Bollob\'as~\cite{Boll:HamiltonRandom}. Note that this threshold is the same as the threshold for a weaker property for connectedness. We should also mention a previous result of P\'osa~\cite{Posa:HamiltonRandom} which determined the order of magnitude of the threshold, $p=\Theta(\frac{\ln n}n)$. The heart of P\'osa's proof, the celebrated rotation-extension technique, is an argument about expanding graphs, and indeed many other results about trees in random graphs exploit the expansion properties of $G_{n,p}$ in the first place.

The threshold for the appearance of almost spanning paths in $G_{n,p}$ was determined by Fernandez de la Vega~\cite{Fern:LongRandom} and independently by Ajtai, Koml\'os, and Szemer\'edi~\cite{AKS:LongRandom}. Their results say that a path of length $(1-\epsilon)n$ appears a.a.s.\ in $G_{n,\frac{C}{n}}$ for $C=C(\epsilon)$ sufficiently large. This behavior extends to bounded degree trees. Indeed, Alon, Krivelevich, and Sudakov~\cite{AlKrSu:NearlySpanningTrees} proved that $G_{n,\frac{C}{n}}$ (for a suitable $C=C(\epsilon,\Delta)$) a.a.s.\ contains all trees of order $(1-\epsilon)n$ with maximum degree at most $\Delta$ (the constant $C$ was later improved in~\cite{BCPS:AlmostSpanningRandom}).

Let us now turn to spanning trees in random graphs. It is known~\cite{AlKrSu:NearlySpanningTrees}  that a.a.s.\  $G_{n,\frac{C\ln n}{n}}$ contains a single spanning tree $T$ with bounded maximum degree and linearly many leaves. This result can be reduced to the main result of~\cite{AlKrSu:NearlySpanningTrees} regarding almost spanning trees quite easily. The constant $C$ can be taken $C=1+o(1)$, as was shown recently by Hefetz, Krivelevich, and Szab\'o~\cite{HeKrSz:SpanningRandom}; obviously this is best possible. The same result also applies to trees that contain a path of linear length whose vertices all have degree two. A breakthrough in the area was achieved by Krivelevich~\cite{Kri:SpanningRandom} who gave an upper bound on the threshold $p=p(n,\Delta)$ for embedding a single spanning tree of a given maximum degree $\Delta$. This bound is essentially tight for $\Delta=n^c$, $c\in(0,1)$. Even though the argument in~\cite{Kri:SpanningRandom} is not difficult, it relies on a deep result of 
Johansson, Kahn and Vu~\cite{JoKaVu:Factors} about factors in random graphs.

\paragraph{Trees in expanders.}
By an expander graph we mean a graph with a large Cheeger constant, i.e., a graph which satisfies a certain isoperimetric property. As indicated above, random graphs are very good expanders, and this is the main motivation for studying tree containment problems in expanders. Another motivation comes from studying the universality phenomenon. Here the goal is to construct sparse graphs which contain all trees from a given class, and expanders are natural candidates for this. The study of sparse tree-universal graphs is a remarkable area by itself which brings challenges both in probabilistic and explicit constructions. For example, Bhatt, Chung, Leighton, and Rosenberg~\cite{BCLR:Universal} give an explicit construction of a graph with only $O_{\Delta}(n)$ edges which contains all $n$-vertex trees with maximum degree at most $\Delta$. More recently, 
Johannsen, Krivelevich, and 
Samotij~\cite{JoKriSa:Expanders} showed a number of universality results for spanning trees of maximum degree $\Delta=\Delta(n)$ both for random graphs, and for expanders.
 For example, they show universality for this class of each graph with a large Cheeger constant that satisfies a certain connectivity condition.

Friedman and Pippenger~\cite{FP87} extended P\'osa's rotation-extension technique  from paths to trees by  and found many applications (e.g.~\cite{HaKo:SizeRamsey,Haxell:TreeEmbeddings,BCPS:AlmostSpanningRandom}). Sudakov and Vondr\'ak~\cite{SudVoTree} use tree-indexed random walks to embed trees in $K_{s,t}$-free graphs (this property implies expansion); a similar approach is employed  by  Benjamini and Schramm~\cite{BeSch:TreeCheeger} in the setting of infinite graphs.

In our proof of Theorem~\ref{thm:main}, embedding trees in expanders play a crucial role, too. However, our notion of expansion is very different from those studied previously. (Actually, we introduce two, very different, notions  in Definitions~\ref{def:densespot} and~\ref{def:avoiding}.)

\paragraph{Minimum degree conditions for spanning trees.} Recall that the tight min-degree condition for containment of a general spanning tree $T$ in an $n$-vertex graph $G$ is the trivial one, $\mindeg(G)\ge n-1$. However, the only tree which requires this bound is the star. This indicates that this threshold can be lowered substantially if we have a control of $\maxdeg(T)$. Szemer\'edi and his collaborators~\cite{KSS:SpanningTrees,CsLeNaSz:BoundedDegree} showed that this is indeed the case, and obtained tight min-degree bounds for certain ranges of $\maxdeg(T)$. For example, if $\maxdeg(T)\le n^{o(1)}$, then $\mindeg(G)\ge (\frac12+o(1))n$ is a sufficient condition. (Note that $G$ may become disconnected close to this bound.)

\subsection{Overview of the proof of Theorem~\ref{thm:main}}\label{ssec:OverviewOfProof}
The structure of the proof of Theorem~\ref{thm:main} resembles the proof of the
dense case, Theorem~\ref{thm:PiguetStein}.
We obtain an approximate representation --- called \emph{sparse decomposition} --- of the
graph~$G_\PARAMETERPASSING{T}{thm:main}$.
Then we find a suitable combinatorial structure inside the sparse decomposition.
Finally, we embed the tree~$T_\PARAMETERPASSING{T}{thm:main}$ into~$G_\PARAMETERPASSING{T}{thm:main}$ using this structure. 

First, let us give a short outline of how the manuscript is structured. 
We use  
Sections~\ref{sec:preliminaries}--\ref{sec:embed} to introduce all tools necessary for the proof of Theorem~\ref{thm:main}, which is given  in a relatively short form in
Section~\ref{sec:proof}. Section~\ref{sec:conclremarks} discusses algorithmic aspects of our proof.

The preparation for the proof of the main theorem during Sections~\ref{sec:preliminaries}--\ref{sec:embed} starts with
introducing some  general preliminaries in Section~\ref{sec:preliminaries}. Then, in
Section~\ref{sec:cut}, the tree
$T_\PARAMETERPASSING{T}{thm:main}$ is pre-processed by being cut
into tiny subtrees, with few connecting vertices. 

Sections~\ref{sec:class}--\ref{sec:configurations} deal with the
graph $G_\PARAMETERPASSING{T}{thm:main}$. 
First, Section~\ref{sec:class} introduces the notion of the sparse
decomposition  which captures an approximate representation of $G_\PARAMETERPASSING{T}{thm:main}$. (Such a sparse decomposition exists for all graphs, and in a sense is comparable with the Szemer\'edi regularity partition.)
Then, in Sections~\ref{sec:augmenting} and~\ref{sec:LKSStructure} we gather more
structural information, specifically using the properties of graphs from
$\LKSgraphs{n}{k}{\alpha}$. This finally leads to several possible
``configurations'', as we call them, presented in
Section~\ref{sec:configurations}. These configurations give a quite precise
description of $G_\PARAMETERPASSING{T}{thm:main}$ that can be used for tree embedding.

Finally, in Section~\ref{sec:embed} we
introduce techniques for embedding small trees in a graph, based on the configurations we found in Section~\ref{sec:configurations}. In addition to the standard
filling-up-a-regular-pair technique usually employed in conjunction with the regularity
method, we employ several other techniques adapted to the diverse other parts of our sparse decomposition.

A scheme of the proof is given in Figure~\ref{fig:proofstructure}. 
\begin{figure}[ht!]
\centering 
\includegraphics[scale=0.85]{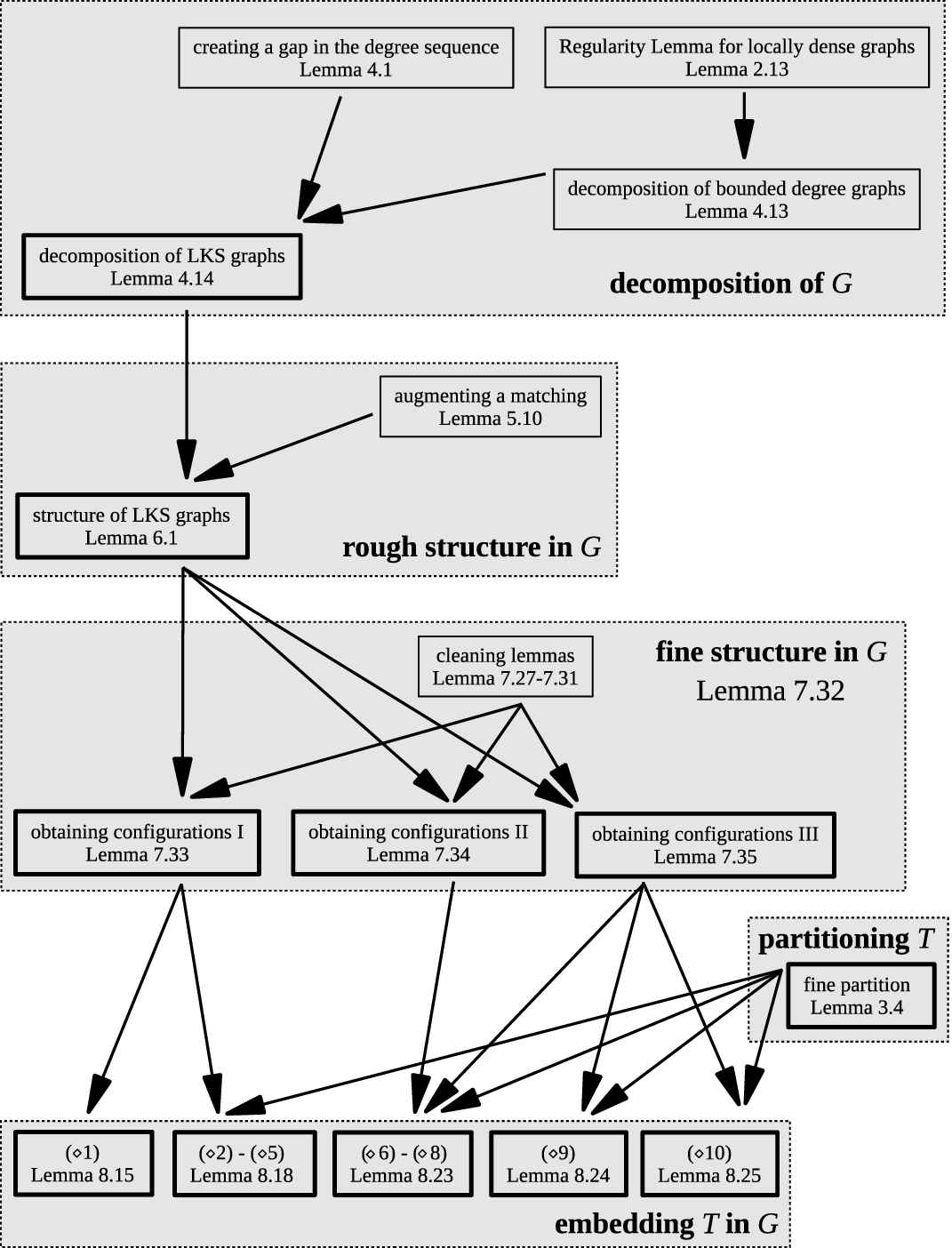}
\caption[Structure of proof of Theorem~\ref{thm:main}]{Structure of the proof of Theorem~\ref{thm:main}.}
\label{fig:proofstructure}
\end{figure}

\medskip

Let us describe now the key ingredients of the proof in more detail. The input graph $G_\PARAMETERPASSING{T}{thm:main}\in\LKSgraphs{n}{k}{\alpha}$ has $\Theta(kn)$ edges.\footnote{Indeed, an easy counting argument gives that $e(G_\PARAMETERPASSING{T}{thm:main})\ge kn/4$. On the other hand, we can assume that $e(G_\PARAMETERPASSING{T}{thm:main})< kn$, as otherwise $G_\PARAMETERPASSING{T}{thm:main}$ contains a subgraph with minimum degree at least $k$, and the assertion of Theorem~\ref{thm:main} follows.} Recall that the Szemer\'edi Regularity Lemma gives an approximation of dense graphs in which $o(n^2)$ edges are neglected. In analogy, the sparse decomposition captures all but at most $o(kn)$ edges. The vertices of $G_\PARAMETERPASSING{T}{thm:main}$ are partitioned into vertices of degrees $\gg k$ and vertices of degree $O(k)$. Further, the graph induced by the latter set is split into regular pairs (in the sense of the Szemer\'edi Regularity Lemma) with clusters of sizes $\Theta(k)$, and into two additional 
parts which exhibit certain expansion properties (the expansion properties of these two parts are different). The vertices of huge degrees, the regular pairs, and the two expanding parts form the sparse decomposition of $G_\PARAMETERPASSING{T}{thm:main}$. It is well-known that regular pairs are suitable for embedding small trees. In Section~\ref{sec:embed} we work out techniques for embedding small trees in each of the three remaining parts of the sparse decomposition.

 Tree-embedding results in the dense setting (e.g.\
 Theorem~\ref{thm:PiguetStein}) rely on finding a matching structure in the
 cluster graph. Indeed, this allows one to distribute different parts of the
 tree in the matching edges. In analogy, in Lemma~\ref{prop:LKSstruct} we find a
 structure which combines all four components of the sparse decomposition,
 and which we call the  \emph{rough structure}. Not only all parts of the sparse decomposition are contained in the rough structure, but also, on top of these, an additional
object, which we call a {\emph semiregular matching}.
 This is found with the help of Lemma~\ref{lem:Separate}, a step which 
we call ``augmenting a matching''.  The necessity of this step is discussed in detail in
Section~\ref{ssec:whyaugment}.

 However, the rough structure is not immediately suitable for embedding 
 $T_\PARAMETERPASSING{T}{thm:main}$, and we shall further refine it in
 Section~\ref{ssec:obtainingConf} to one of ten \emph{configurations},
 denoted by $\mathbf{(\diamond1)}$--$\mathbf{(\diamond10)}$.
Obtaining these configurations from the rough structure is based on pigeonhole-type arguments such as: if
there are many edges between two sets, and few ``kinds'' of edges, then many of
the edges are of the same kind. The different kinds of edges come from the sparse decomposition (and allow for different kinds of embedding techniques). Just ``homogenizing'' the situation by
restricting to one particular kind is not enough, we also need to
employ certain ``cleaning lemmas'' --- Lemmas~\ref{lem:envelope}--\ref{lem:clean-Match}.
A simplest such lemma would be that a graph with many edges contains a subgraph
with a large minimum degree; the latter property evidently being more directly
applicable for a sequential embedding of a tree. The actual cleaning lemmas we
use are complex extensions of this simple idea.

Finally, in Section~\ref{sec:embed}, we show how to embed the tree $T_\PARAMETERPASSING{T}{thm:main}$. This is done by first establishing some
elementary embedding lemmas for small subtrees in
Section~\ref{ssec:EmbeddingShrubs}, and then combine these in Section~\ref{sec:MainEmbedding}
for each of the
cases $\mathbf{(\diamond1)}$--$\mathbf{(\diamond10)}$ to yield an embedding of
the entire tree $T_\PARAMETERPASSING{T}{thm:main}$.

\section{Notation and preliminaries}\label{sec:preliminaries}
In this section we recall some standard terminology and introduce some further specific notation. We also state some basic results from graph theory.

\subsection{Notation}
The set $\{1,2,\ldots, n\}$ of the first $n$ positive integers is
denoted by \index{mathsymbols}{*@$[n]$}$[n]$. Suppose that we have a nonempty set $A$, and $\mathcal X$ and $\mathcal Y$ each partition $A$. Then~\index{mathsymbols}{*@$\boxplus$}$\boxplus$ denotes the coarsest common refinement of $\mathcal X$ and $\mathcal Y$, i.e., $$\mathcal X\boxplus\mathcal Y:=\{X\cap Y\::\: X\in \mathcal X, Y\in\mathcal Y\}\setminus \{\emptyset\}\;.$$

We frequently employ indexing by many indices. We write
superscript indices in parentheses (such as $a^{(3)}$), as
opposed to notation of powers (such as $a^3$).
We use sometimes subscript to refer to
parameters appearing in a fact/lemma/theorem. For example
$\alpha_\PARAMETERPASSING{T}{thm:main}$ refers to the parameter $\alpha$ from Theorem~\ref{thm:main}.
We omit rounding symbols when this does not affect the
correctness of the arguments.

We use lower case greek letters
to denote small positive constants. The exception is the letter~$\phi$ which is
reserved for embedding of a tree $T$ in a graph $G$, $\phi:V(T)\rightarrow
V(G)$. The capital greek letters are used for large constants.

\subsection{Basic graph theory notation}
All graphs considered in this paper are finite, undirected, without multiple edges, and without self-loops. We write \index{mathsymbols}{*VG@$V(G)$}$V(G)$ and \index{mathsymbols}{*EG@$E(G)$}$E(G)$ for the vertex set and edge set of a graph $G$, respectively. Further, \index{mathsymbols}{*VG@$v(G)$}$v(G)=|V(G)|$ is the order of $G$, and \index{mathsymbols}{*EG@$e(G)$}$e(G)=|E(G)|$ is its number of edges. If $X,Y\subset V(G)$ are two, not necessarily disjoint, sets of vertices we write \index{mathsymbols}{*EX@$e(X)$}$e(X)$ for the number of edges induced by $X$, and \index{mathsymbols}{*EXY@$e(X,Y)$}$e(X,Y)$ for the number of ordered pairs $(x,y)\in X\times Y$ such that $xy\in E(G)$. In particular, note that $2e(X)=e(X,X)$.

\index{mathsymbols}{*DEG@$\deg$}\index{mathsymbols}{*DEGmin@$\mindeg$}\index{mathsymbols}{*DEGmax@$\maxdeg$}
For a graph $G$, a vertex $v\in V(G)$ and a set $U\subset V(G)$, we write
$\deg(v)$ and $\deg(v,U)$ for the degree of $v$, and for the number of
neighbours of $v$ in $U$, respectively. We write $\mindeg(G)$ for the minimum
degree of $G$, $\mindeg(U):=\min\{\deg(u)\::\: u\in U\}$, and
$\mindeg(V_1,V_2)=\min\{\deg(u,V_2)\::\:u\in V_1\}$ for two sets $V_1,V_2\subset
V(G)$. Similar notation is used for the maximum degree, denoted by $\maxdeg(G)$.
The neighbourhood of a vertex $v$ is denoted by
\index{mathsymbols}{*N@$\neighbor(v)$}$\neighbor(v)$. We set $\neighbor(U):=\bigcup_{u\in
U}\neighbor(u)$. The symbol $-$ is used
for two graph operations: if $U\subset V(G)$ is a vertex
set then $G-U$ is the subgraph of $G$ induced by the set
$V(G)\setminus U$. If $H\subset G$ is a subgraph of $G$ then the graph
$G-H$ is defined on the vertex set $V(G)$ and corresponds
to deletion of edges of $H$ from $G$.

A subgraph $H\subset G$ of a graph $G$ is called \emph{spanning}\index{general}{spanning subgraph} if $V(H)=V(G)$.

The \index{general}{null graph}\emph{null graph} is the unique graph on zero vertices, while any graph with zero edges is called \index{general}{empty graph}\emph{empty}.

A family $\mathcal A$ of pairwise disjoint subsets of $V(G)$ is an \index{general}{ensemble}\index{mathsymbols}{*ENSEMBLE@$\ell$-ensemble}\emph{$\ell$-ensemble in $G$} if  $|A|\ge \ell$ for each $A\in\mathcal A$. We say that $\mathcal A$ is \emph{inside $X$} (or \emph{outside $Y$}) if $A\subset X$ (or $A\cap Y=\emptyset$) for each $A\in\mathcal A$.

If $T$ is a tree and $r\in V(T)$, then the pair $(T,r)$ is a \index{general}{rooted
tree}\emph{rooted tree} with root $r$. We then write 
\index{mathsymbols}{*Veven@$\Veven(T,r)$}\index{mathsymbols}{*Vodd@$\Vodd(T,r)$}$\Vodd(T,r)\subset V(T)$ for the set of
vertices of $T$ of odd distance from $r$. Analogously we define $\Veven(T,r)$.
Note that $r\in\Veven(T,r)\subset V(T)$. The distance between two vertices $v_1$ and $v_2$ in a tree is denoted by
\index{mathsymbols}{*DIST@$\dist(v_1,v_2)$}$\dist(v_1,v_2)$.

We next give two simple facts about the number of leaves in a tree. These have already appeared in~\cite{Z07+} and in~\cite{HlaPig:LKSdenseExact} (and most likely in some more classic texts as well). Nevertheless, for completeness we shall include their proofs here.

\begin{fact}\label{fact:treeshavemanyleaves}
Let $T$ be a tree with color-classes $X$ and $Y$, and
$v(T)\ge 2$. Then the set $X$ contains at least
$|X|-|Y|+1$ leaves of $T$. 
\end{fact}
\begin{proof} Root $T$
at an arbitrary vertex $r\in Y$.
Let $I$ be the set of internal vertices of $T$ that belong to $X$. Each $v\in I$ has at least one immediate successor in the tree order induced by $r$.
These successors are distinct for distinct $v\in I$ and all lie in $Y\setminus\{r\}$. Thus
$|I|\le |Y|-1$. The claim follows.
\end{proof}

\begin{fact}\label{fact:treeshavemanyleaves3vertices}
Let $T$ be a tree with $\ell$ vertices of degree at
least three. Then $T$ has at least $\ell+2$ leaves.
\end{fact}
\begin{proof}
Let $D_1$ be the set of leaves, $D_2$ the set of vertices
of degree two and $D_3$ be the set of vertices of degree
of at least three. Then
$$2(|D_1|+|D_2|+|D_3|)-2=2v(T)-2=2e(T)=\sum_{v\in
V(T)}\deg(v)\ge |D_1|+2|D_2|+3|D_3|\;,$$
and the statement follows.
\end{proof}

For the next lemma, note that for us, the minimum degree of the null graph is $\infty$.
\begin{lemma}\label{lem:subgraphswithlargeminimumdegree}
For all $\ell, n\in\mathbb N$, every $n$-vertex graph $G$ contains a (possibly empty) subgraph $G'$ such that
$\mindeg(G')\ge \ell$ and $e(G')\ge e(G)-(\ell-1) n$.
\end{lemma}
\begin{proof}
We construct the graph $G'$ by sequentially removing vertices of degree less
than $\ell$ from the graph $G$. In each step we remove at most $\ell-1$
edges. Thus the statement follows.
\end{proof}

\bigskip

We finish this section with stating the Gallai--Edmonds matching theorem.
A graph $H$ is called \index{general}{factor critical}{\em
factor-critical} if  $H-v$ has a
perfect matching for each $v\in V(H)$.
The following statement is a fundamental result in matching
theory. See~\cite{LP86}, for example.
\begin{theorem}[Gallai--Edmonds matching theorem]
\label{thm:GallaiEdmonds} Let $H$ be a graph. Then there exist a set
$Q\subset V(H)$ and a matching $M$ of size $|Q|$ in $H$ such that
\begin{enumerate}[1)]
\item every component of $H-Q$ is factor-critical, and
\item $M$
matches every vertex in $Q$ to a different component of $H-Q$.
\end{enumerate}
\end{theorem}
The set $Q$ in Theorem~\ref{thm:GallaiEdmonds} is often referred to as a \index{general}{separator}{\em separator}.

\subsection{LKS-minimal graphs}\index{mathsymbols}{*LKSmingraphs@$\LKSmingraphs{n}{k}{\eta}$}
Given a graph $G$, denote by
\index{mathsymbols}{*S@$\smallvertices{\eta}{k}{G}$}$\smallvertices{\eta}{k}{G}$ the set of those
vertices of $G$ that have degree less than $(1+\eta)k$ and by
\index{mathsymbols}{*L@$\largevertices{\eta}{k}{G}$}$\largevertices{\eta}{k}{G}$ the set of those
vertices of $G$ that have degree at least $(1+\eta)k$.\footnote{``$\mathbb S$'' stands for ``small'', and ``$\mathbb L$'' for ``large''.} Thus the sizes of the sets $\smallvertices{\eta}{k}{G}$ and $\largevertices{\eta}{k}{G}$ are what specifies the membership to $\LKSgraphs{n}{k}{\eta}$ (which we had defined as the class of all $n$-vertex graphs with at least
$(\frac12+\eta)n$ vertices of degrees at least
$(1+\eta)k$).

Define
 $\LKSmingraphs{n}{k}{\eta}$ as the set 
of all graphs $G\in \LKSgraphs{n}{k}{\eta}$ that are  edge-minimal with respect to the membership in $\LKSgraphs{n}{k}{\eta}$.
In order to prove Theorem~\ref{thm:main} it suffices to restrict our attention to graphs from $\LKSmingraphs{n}{k}{\eta}$, and this is why we introduce the class. 
Let us collect some properties of graphs in
$\LKSmingraphs{n}{k}{\eta}$ which follow directly from the definition.
\begin{fact}\label{fact:propertiesOfLKSminimalGraphs}
For any graph $G\in\LKSmingraphs{n}{k}{\eta}$ the following is true.
\begin{enumerate}
  \item $\smallvertices{\eta}{k}{G}$ is an independent set.\label{Sisindep}
  \item All the neighbours of every vertex $v\in V(G)$ with $\deg(v)>\lceil(1+\eta)k\rceil$ have degree exactly $\lceil(1+\eta)k\rceil$. \label{neighbinL}
  \item $|\largevertices{\eta}{k}{G}|\le\lceil
  (1/2+\eta)n\rceil+1$.\label{fewlargevs}
\end{enumerate}
\end{fact}
Observe that every edge in a graph $G\in\LKSmingraphs{n}{k}{\eta}$ is incident
to at least one vertex of degree exactly $\lceil(1+\eta)k\rceil$. This gives
the following inequality.
\begin{equation}\label{eq:LKSminimalNotManyEdges}
e(G)\le\lceil(1+\eta)k\rceil
\left|\largevertices{\eta}{k}{G}\right|\leBy{F\ref{fact:propertiesOfLKSminimalGraphs}(\ref{fewlargevs}.)}
\lceil(1+\eta)k\rceil\left(\left\lceil
\left(\frac12+\eta\right)n\right\rceil+1\right)<kn\;.
\end{equation}
(The last inequality is valid under the additional mild
assumption that, say, $\eta<\frac1{20}$ and $n>k>20$. This can be assumed throughout the paper.)

\begin{definition}\label{def:LKSsmall}
Let \index{mathsymbols}{*LKSsmallgraphs@$\LKSsmallgraphs{n}{k}{\eta}$}$\LKSsmallgraphs{n}{k}{\eta}$ be the class of those graphs $G\in\LKSgraphs{n}{k}{\eta}$ for which we have the following three properties:
\begin{enumerate}
   \item All the neighbours of every vertex $v\in V(G)$ with $\deg(v)>\lceil(1+2\eta)k\rceil$ have degrees at most $\lceil(1+2\eta)k\rceil$.\label{def:LKSsmallA}
   \item All the neighbours of every vertex of $\smallvertices{\eta}{k}{G}$
    have degree exactly $\lceil(1+\eta)k\rceil$. \label{def:LKSsmallB}
   \item We have $e(G)\le kn$.\label{def:LKSsmallC}
\end{enumerate}
\end{definition}
Observe that the graphs from $\LKSsmallgraphs{n}{k}{\eta}$ also satisfy~\ref{Sisindep}., and a quantitatively somewhat weaker version of~\ref{neighbinL}.\ of Fact~\ref{fact:propertiesOfLKSminimalGraphs}. This suggests that in some sense $\LKSsmallgraphs{n}{k}{\eta}$ is a good approximation of $\LKSmingraphs{n}{k}{\eta}$.

As said, we will prove Theorem~\ref{thm:main} only for graphs from $\LKSmingraphs{n}{k}{\eta}$. However, it turns out that the structure of $\LKSmingraphs{n}{k}{\eta}$ is too rigid. In particular, $\LKSmingraphs{n}{k}{\eta}$ is not closed under discarding a small amount of edges during our cleaning procedures. This is why the class $\LKSsmallgraphs{n}{k}{\eta}$ comes into play: starting with a graph in $\LKSmingraphs{n}{k}{\eta}$ we perform some initial cleaning and obtain a graph that lies in $\LKSsmallgraphs{n}{k}{\eta/2}$. We then heavily use its structural properties from Definition~\ref{def:LKSsmall} throughout the proof.

\subsection{Regular pairs}

In this section we introduce the notion of regular pairs which is central for
Szemer\'edi's Regularity Lemma and its extension which we discuss in Section~\ref{sec:RegL}. We also list some simple properties of regular pairs.

Given a graph $H$ and a pair $(U,W)$ of disjoint
sets $U,W\subset V(H)$ the
\index{general}{density}\index{mathsymbols}{*D@$\density(U,W)$}\emph{density of the pair $(U,W)$} is defined as
$$\density(U,W):=\frac{e(U,W)}{|U||W|}\;.$$
Similarly, for a bipartite graph $G$ with colour
classes $U$, $W$ we talk about its \index{general}{bipartite
density}\emph{bipartite density}\index{mathsymbols}{*D@$\density(G)$} $\density(G)=\frac{e(G)}{|U||W|}$.
For a given $\varepsilon>0$, a pair $(U,W)$ of disjoint
sets $U,W\subset V(H)$ 
is called an \index{general}{regular pair}\emph{$\epsilon$-regular
pair} if $|\density(U,W)-\density(U',W')|<\epsilon$ for every
$U'\subset U$, $W'\subset W$ with $|U'|\ge \epsilon |U|$, $|W'|\ge
\epsilon |W|$. If the pair $(U,W)$ is not $\epsilon$-regular,
then we call it \index{general}{irregular}\emph{$\epsilon$-irregular}. A stronger notion than
regularity is that of super-regularity which we recall now. A pair $(A,B)$ is
\index{general}{super-regular pair}\emph{$(\epsilon,\gamma)$-super-regular} if it is
$\epsilon$-regular, and we have $\mindeg(A,B)\ge\gamma |B|$, and
$\mindeg(B,A)\ge \gamma |A|$. Note that then $(A,B)$ has bipartite density at least $\gamma$.

We list two useful and well-known properties of
regular pairs.
\begin{fact}\label{fact:BigSubpairsInRegularPairs}
Suppose that $(U,W)$ is an $\varepsilon$-regular pair of density
$d$. Let $U'\subset W, W'\subset W$ be sets of vertices with $|U'|\ge
\alpha|U|$, $|W'|\ge \alpha|W|$, where $\alpha>\epsilon$.
Then the pair $(U',W')$ is a $2\varepsilon/\alpha$-regular pair of density at least
$d-\varepsilon$.
\end{fact}
\begin{fact}\label{fact:manyTypicalVertices}
Suppose that $(U,W)$ is an $\varepsilon$-regular pair of density
$d$. Then all but at most $\epsilon|U|$ vertices $v\in U$ satisfy
$\deg(v,W)\ge (d-\epsilon)|W|$.
\end{fact}

The following fact states a simple relation between the
density of a (not necessarily regular) pair and the densities of its
subpairs.

\begin{fact}\label{fact:CanADensePairConsistOnlyOfSparseSubpairs?}
Let $H=(U,W;E)$ be a bipartite graph of
$\density(U,W)\ge \alpha$. Suppose that the sets $U$ and
$W$ are partitioned into sets $\{U_i\}_{i\in I}$ and
$\{W_j\}_{j\in J}$, respectively. Then at most
$\beta e(H)/\alpha$ edges of $H$ belong to a
pair $(U_i,W_j)$ with $\density(U_i,W_j)\le\beta$.
\end{fact}
\begin{proof}
Trivially, we have
\begin{equation}\label{eq:volumE}
\sum_{i\in I, j\in J}\frac{|U_i||W_j|}{|U||W|}=1\;. 
\end{equation}

Consider a pair $(U_i,W_j)$ of
$\density(U_i,W_j)\le\beta$. Then $$e(U_i,W_j)\le \beta
|U_i||W_j|=\frac{\beta}{\alpha}\frac{|U_i||W_j|}{|U||W|}\alpha|U||W|\le
\frac{\beta}{\alpha}\frac{|U_i||W_j|}{|U||W|}e(U,W)\;.$$
Summing over all such pairs $(U_i,W_j)$ and using~\eqref{eq:volumE} yields the
statement.
\end{proof}

The next lemma asserts that if we have many  $\epsilon$-regular pairs $(R,Q_i)$, then most vertices in $R$ have approximately the total degree into the set
$\bigcup_i Q_i$ that we would expect.
\begin{lemma}\label{lem:degreeIntoManyPairs}
Let $Q_1,\ldots,Q_\ell$ and $R$ be disjoint vertex sets. Suppose
further that  for each
$i\in[\ell]$, the pair $(R,Q_i)$ is $\epsilon$-regular. Then we have
\begin{enumerate}[(a)]
\item \label{mindestensdengrad}
$ \deg (v,\bigcup_i Q_i)\geq \frac{e(R,\bigcup_i Q_i)}{|R|}-\epsilon\left|\bigcup_i
Q_i\right|$ for all but at most $\epsilon|R|$ vertices $v\in R$, and
\item \label{hoechstensdengrad} 
$ \deg (v,\bigcup_i Q_i)\le
\frac{e\left(R,\bigcup_i Q_i\right)}{|R|}+\epsilon\left|\bigcup_i
Q_i\right|$ for all but at most $\epsilon|R|$ vertices $v\in R$.
\end{enumerate}
\end{lemma}
\begin{proof}
We prove \eqref{mindestensdengrad}, the other item is analogous.
Suppose for contradiction that~\eqref{mindestensdengrad} does not hold.
Without loss of generality, assume that there is a set $X\subset
R$, $|X|>\epsilon |R|$ such that $\frac{e(R,\bigcup
Q_i)}{|R|}-\epsilon|\bigcup Q_i|> \deg(v,\bigcup Q_i)$ for
each $v\in X$. By averaging, there is an index $i\in [\ell]$
such that $\frac{|X|}{|R|}e(R,Q_i)-\epsilon|X||Q_i|>
e(X,Q_i)$, or equivalently, $$\density(R,Q_i)-\epsilon>
\density(X,Q_i)\;.$$ This is a contradiction to the $\epsilon$-regularity of the pair $(R,Q_i)$.
\end{proof}

We use Lemma~\ref{lem:degreeIntoManyPairs} to obtain the following.

\begin{corollary}\label{lem:degreeIntoManyPairs2}
Let $Q_1,\ldots,Q_\ell$ and $R$ be disjoint vertex sets, each of size at most $q$, such that  for each
$i\in[\ell]$, the pair $(R,Q_i)$ is $\epsilon$-regular. Assume that more than $\eps |R|$ vertices of $R$ have degree at least $x$ into $\bigcup Q_i$, but each $v\in R$ has neighbours in at most $z$ of the sets $Q_i$. Then $\deg (v, \bigcup_i Q_i)\geq x- 2\epsilon zq$ for all but at most $\eps |R|$ vertices of $R$.
\end{corollary}

\begin{proof}
For each $w\in R$, let $I_w\subseteq [\ell]$ be the set of those indices $i$ for which there is at least one edge from $w$ to $Q_i$.
Now, by  Lemma~\ref{lem:degreeIntoManyPairs}\eqref{hoechstensdengrad} there is a vertex
$v\in R$ whose degree into $\bigcup_{i\in [\ell]} Q_i$ is at least $x$ and whose degree into $\bigcup_{i\in I_v} Q_i$ is at most $\frac{e\left(R,\bigcup_{i\in I_v} Q_i\right)}{|R|}+\epsilon\left|\bigcup_{i\in I_v}
Q_i\right|$.
So, $$x\leq \deg (v, \bigcup_{i\in [\ell]} Q_i)= \deg (v, \bigcup_{i\in I_v}  Q_i) \leq  \frac{e\left(R,\bigcup_{i\in I_v} Q_i\right)}{|R|}+\epsilon|\bigcup_{i\in I_v}
Q_i| \leq
 \frac{e\left(R,\bigcup_{i\in I_v} Q_i\right)}{|R|}+\epsilon zq.$$
Thus by Lemma~\ref{lem:degreeIntoManyPairs}\eqref{mindestensdengrad} all but at most $\eps |R|$ vertices of $R$ have degree at least $x- 2\epsilon zq$ into $\bigcup_i Q_i$.
\end{proof}

\subsection{Regularizing locally dense graphs}\label{sec:RegL}

The Regularity Lemma~\cite{Sze78} has proved to be
a powerful tool for attacking graph embedding
problems; see~\cite{KuhnOsthusSurv} 
for a survey. We first state the
lemma in its original form.

\begin{lemma}[Regularity lemma]\label{lem:RL}
For all $\epsilon>0$ and $\ell\in\mathbb N$ there exist $n_0,M\in\mathbb N$
such that for every $n\ge n_0$ the following holds. Let $G$ be an $n$-vertex graph whose
vertex set is pre-partitioned into sets
$V_1,\ldots,V_{\ell'}$, $\ell'\le \ell$. Then there exists
a partition $U_0,U_1,\ldots,U_p$ of $V(G)$, $\ell<p<M$, with the following properties.
\begin{enumerate}[1)]
\item For every $i,j\in [p]$ we have $|U_i|=|U_j|$, and  $|U_0|<\epsilon n$.
\item For every
$i\in [p]$ and every $j\in [\ell']$ either $U_i\cap
V_j=\emptyset$ or $U_i\subset
V_j$. 
\item All but at most
$\epsilon p^2$ pairs $(U_i,U_j)$, $i,j\in [p]$, $i\neq j$, are $\epsilon$-regular.
\end{enumerate}
\end{lemma}

We shall use Lemma~\ref{lem:RL}  for auxiliary
purposes only as it is helpful only in the setting of dense
graphs (i.e., graphs which have $n$ vertices and
$\Omega(n^2)$ edges). This is not necessarily the
case in Theorem~\ref{thm:main}. For this reason, we  give a version of the 
Regularity Lemma --- Lemma~\ref{lem:sparseRL} below --- which 
allows us to regularize even sparse graphs.

More precisely, suppose that we have an $n$-vertex graph $H$ whose edges lie in bipartite graphs $H[W_i,W_j]$, where $\{W_1,\ldots,W_\ell\}$ is an ensemble of sets of size $\Theta(k)$. Although $\ell$ may be unbounded, for a fixed $i\in[\ell]$ there are only a bounded number, say $m$, of indices $j\in[\ell]$ such that  $H[W_i,W_j]$ is non-empty. See Figure~\ref{fig:locallydensegraph} for an example.
\begin{figure}[t]
\centering 
\includegraphics[scale=0.8]{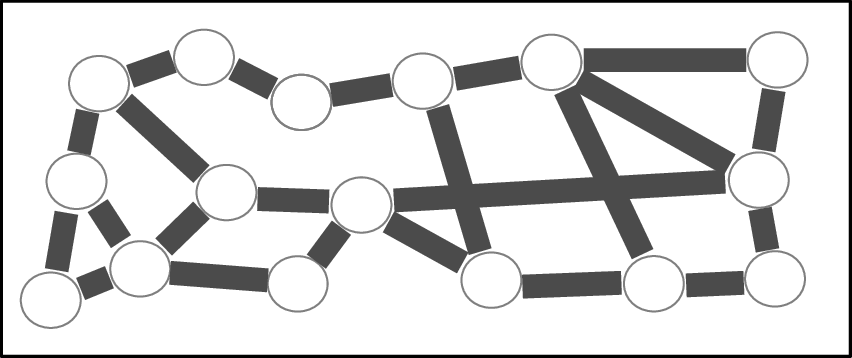}
\caption[Locally dense graph]{A locally dense graph as in Lemma~\ref{lem:sparseRL}. The sets $W_1,\ldots,W_\ell$ are depicted with grey circles. Even though there is a large number of them, each $W_i$ is linked to only boundedly many other $W_j$'s (at most four, in this example). Lemma~\ref{lem:sparseRL} allows us to regularize all the bipartite graphs using the same system of partitions of the sets $W_i$.}\label{fig:locallydensegraph}
\end{figure}
Lemma~\ref{lem:sparseRL} then allows us to regularize (in the sense of the Regularity Lemma~\ref{lem:RL}) all the bipartite graphs $G[W_i,W_j]$ using the same partition $\{W_i^{(0)}\dcup W_i^{(1)}\dcup\ldots\dcup W_i^{(p_i)}=W_i\}_{i=1}^\ell$. Note that when $|W_i|=\Theta(k)$ for all $i\in[\ell]$ then $H$ has at most
$$\Theta(k^2)\cdot m\cdot \ell\le \Theta(k^2)\cdot m\cdot \frac n{\Theta(k)}=\Theta(kn)$$
edges. Thus, when $k\ll n$, this is a regularization of a sparse graph. This ``sparse Regularity Lemma'' is very different to that of Kohayakawa and R\"odl (see e.g.~\cite{Kohayakawa97Szemeredi}). Indeed, the Kohayakawa--R\"odl Regularity Lemma only deals with graphs which have no local condensation of edges, such as subgraphs of random graphs.\footnote{There is a recent refinement of the Kohayakawa--R\"odl Regularity Lemma, due to Scott~\cite{Scott:RL}. Scott's Regularity Lemma gets around the no-condensation condition, which proves helpful in some situations, e.g.~\cite{AllKeeSudVer:TuranBipartite}; still the main features remain.} Consequently, the resulting regular pairs are of density $o(1)$. In contrast, Lemma~\ref{lem:sparseRL} provides us with regular pairs of density $\Theta(1)$, but, on the other hand, is useful only for graphs which are locally dense.

\begin{lemma}[Regularity Lemma for locally dense graphs]\label{lem:sparseRL}
For all $m,z\in \mathbb{N}$ and $\epsilon>0$  there
exists $q_\mathrm{MAXCL}\in\mathbb{N}$ such that the following is true. Suppose $H$ and $F$ are two graphs, $V(F)=[\ell]$ for some $\ell\in\mathbb N$, and 
$\maxdeg(F)\le m$. Suppose that $\mathcal Z=\{Z_1,\ldots,Z_z\}$ is a partition of $V(H)$. Let $\{W_1,\ldots, W_\ell\}$ 
be a $q_\mathrm{MAXCL}$-ensemble in $H$, such that for all
$i,j\in[\ell]$ we have
\begin{equation}
\label{lem:sparseRL(item)samesize} 
2|W_i|\ge |W_j|\;.
\end{equation}
Then for each $i\in [\ell]$ there exists a partition
$W_i^{(0)},W_i^{(1)},\ldots,W_i^{(p_i)}$ of the set
$W_i$  such that for all
$i,j\in[\ell]$ we have
\begin{enumerate}[(a)]
\item $1/\epsilon\le p_i\le q_\mathrm{MAXCL}$,
\item $|W_i^{(i')}|=|W_j^{(j')}|$ for each $i'\in [p_i]$,
$j'\in[p_j]$, 
\item for each $i'\in [p_i]$ there exists $x\in[z]$ such that $W_i^{(i')}\subset Z_x$,
\item $\sum_i|W_i^{(0)}|<\epsilon \sum_i|W_i|$, and\label{RLgarbage}
\item\label{item:RL:d} at most $\epsilon\left|\mathcal{Y}\right|$ pairs
$\left(W_i^{(i')},W_j^{(j')}\right)\in \mathcal{Y}$ form an
$\epsilon$-irregular pair in $H$, where
$$\mathcal{Y}:=\left\{\left(W_i^{(i')},W_j^{(j')}\right)\::\:ij\in
E(F), i'\in [p_i], j'\in [p_j]\right\}\;.$$
\end{enumerate}
\end{lemma}
We use Lemma~\ref{lem:sparseRL} in Lemma~\ref{lem:decompositionIntoBlackandExpanding}. Lemma~\ref{lem:decompositionIntoBlackandExpanding} is in turn the main tool in the proof of our main structural decomposition of the graph $G_\PARAMETERPASSING{T}{thm:main}$, Lemma~\ref{lem:LKSsparseClass}. In the proof of Lemma~\ref{lem:decompositionIntoBlackandExpanding} we decompose $G_\PARAMETERPASSING{T}{thm:main}$ into several parts with very different properties, and one of these parts is a locally dense graph which can be then regularized by Lemma~\ref{lem:decompositionIntoBlackandExpanding}.
A similar Regularity Lemma is used
in~\cite{AKSS07+}. 
\bigskip 

The proof of Lemma~\ref{lem:sparseRL} is similar to the proof of the standard
Regularity Lemma~\ref{lem:RL}, as given for example in~\cite{Sze78}. We assume the reader's
familiarity with the notion of the index (a.k.a.\ the mean square density), and of the Index-pumping Lemma from
there.


We give a proof of Lemma~\ref{lem:sparseRL} below, but before, let us describe how a more naive
approach fails. For each edge $ij\in E(F)$ consider a regularization of the
bipartite graph $H[W_i,W_j]$,  let
$\{U^{(i')}_{i,j}\}_{i'\in[q_{i,j}]}$ be the partition of $W_i$ into clusters, and let
$\{U^{(j')}_{j,i}\}_{j'\in[q_{j,i}]}$ be the partition of $W_j$ into
clusters such that almost all pairs $(U^{(i')}_{i,j},U^{(j')}_{j,i})\subseteq (W_i,W_j)$ form an
$\epsilon'$-regular pair (for some $\epsilon'$ of our taste). We would now be
done if the partition $\{U^{(i')}_{i,j}\}_{i'\in[q_{i,j}]}$ of $W_i$ was
independent of the choice of the edge $ij$. This however need not be the case.
The natural next step would therefore be to consider the common refinement
$$\underset{j:ij\in E(F)}{\bigboxplus} \big\{U^{(i')_{i,j}}\big\}_{i'\in
[q_{ij}]}$$
of all the obtained partitions of $W_i$. The pairs obtained in this way lack
however any regularity properties as they are too small. Indeed, it is a notorious drawback of the
Regularity Lemma that the number of clusters in the partition is enormous as a
function of the regularity parameter. In our setting, this means that
$q_{i,j}\gg\frac1{\epsilon'}$. Thus a typical cluster $U^{(i'_1)}_{i,j_1}$
occupies on average only a $\frac1{q_{i,j_1}}$-fraction of the cluster
$U^{(i'_2)}_{i,j_2}$, and thus already the set
$U^{(i'_1)}_{i,j_1}\cap U^{(i'_2)}_{i,j_2}\subset U^{(i'_2)}_{i,j_2}$ is not
substantial (in the sense of the regularity). The same issue arises when regularizing multicolored graphs
(cf.~\cite[Theorem~1.18]{KS96}). The solution is to impel
the regularizations to happen in a synchronized way.

\begin{proof}[Proof of Lemma~\ref{lem:sparseRL}] 
For the sake of brevity, and since this step is standard,
we omit respecting the prepartition~$\mathcal{Z}$ in this proof. 

We first recall the proof of the original Regularity Lemma~\ref{lem:RL} which we
then modify. Actually, it better suits our situation to illustrate this on a
procedure which regularizes a given bipartite graph $G=(A,B;E)$. We start with 
arbitrary bounded partitions $\mathcal W_A$ and $\mathcal W_B$ of $A$ and $B$.
Sequentially, we look whether there is a witness of irregularity of $\mathcal
W_A$ and $\mathcal W_B$. If there is, then the partition $\mathcal W_A$ and
$\mathcal W_B$ can be refined so that the index increases. The facts that one
can control the increase of the complexity of the partitions, and that the index
increases substantially are the keys for guaranteeing that  the iteration
terminates in a bounded number of steps. 
\medskip

Let us now see how we can adapt this proof to our setting.
By Vizing's Theorem we can cover the
edges of $F$ by disjoint matchings
$M_1,\ldots,M_{m+1}$. For each
$i\in[m+1]$ we shall introduce a variable $\mathrm{ind}_i$. The variable
$\mathrm{ind}_i$ is the average index of the bipartite graphs which
correspond to the edges of $M_i$ and the current partitions of the sets $W_x$. In each step
$i\in[m+1]$, we refine simultanously
partitions in all bipartite graphs $G[W_x,W_y]$ ($xy\in
M_i$) which possess witnesses of irregularity. More precisely, assume that in a certain step each set $W_z$ is partitioned into sets $\mathcal{W}_z$. We then define
\begin{align*}
\mathrm{ind}_i&=\frac1{|M_i|}\sum_{xy\in M_i}
\mathrm{ind}(\mathcal{W}_x,\mathcal{W}_y) \;,&\mbox{if $M_i\not=\emptyset$, and}\\
\mathrm{ind}_i&=1 \;,&\mbox{otherwise.}
\end{align*}
where $\mathrm{ind}$ is the usual index. The
Index-pumping Lemma asserts that when refining the partition of $G[W_x,W_y]$ the value
$\mathrm{ind}(\mathcal{W}_x,\mathcal{W}_y)$ increases
substantially. The fact that $M_i$ is a matching allows
us to perform these simultaneous refinements without
interference. It is well-known that none of
$\mathrm{ind}_j$ ($j<i$) did decrease during pumping $\mathrm{ind}_i$ up. Thus after a bounded number of steps there are no witnesses of irregularity in the graphs $G[W_x,W_y]$ ($xy\in E(H)$) with respect to the partitions $\mathcal{W}_x,\mathcal{W}_y$.
This suffices to give the statement.
\end{proof}

Usually after applying the Regularity Lemma to some graph $G$, one bounds the
number of  edges which correspond to irregular pairs, to regular, but sparse
pairs, or are incident with the exceptional sets $U_0$. We shall do the same
for the setting of Lemma~\ref{lem:sparseRL}.

\begin{lemma}\label{notmuchlost}
In the situation of Lemma~\ref{lem:sparseRL}, suppose that  
$\maxdeg(H)\leq \Omega k$ and $e(H)\le kn$, and that each edge $xy\in E(H)$ is
captured by some edge $ij\in E(F)$, i.e., $x\in W_i$, $y\in W_j$. Moreover
suppose that
\begin{equation}\label{Fdensi}
\text{$\density(W_i,W_j)\geq\gamma$ if $ij\in E(F)$.}
\end{equation}
Then all but at most $(\frac{4\epsilon}{\gamma}+\epsilon\Omega+\gamma)nk$  edges of $H$ belong to regular pairs $(W^{(i)}_{i'},W^{(j)}_{j'})$, $i,j\neq 0$,  of density at least
$\gamma^2$.
\end{lemma}

\begin{proof}Set $w:=\min\{|W_i|:i\in V(F)\}$.
By~\eqref{Fdensi}, each edge of $F$ represents at
least $\gamma w^2$ edges of $H$. Since $e(H)\le kn$ it follows that $e(F)\le
kn/(\gamma w^2)$. Thus,
by the assumption~\eqref{lem:sparseRL(item)samesize}, $\sum_{AB\in E(F)}|A||B|\le e(F)(2w)^2\le \frac{4kn}{\gamma}$.
Using~\eqref{item:RL:d} of Lemma~\ref{lem:sparseRL} we get that the number of
edges of $H$ contained in $\epsilon$-irregular pairs from $\mathcal
Y$ is at most
\begin{equation}\label{eq:boundTotalIrregII}
\frac{4\epsilon nk}{\gamma}\;.
\end{equation}

Write $E_1$ for the set of edges of $H$ which are
incident with a vertex in $\bigcup_{i\in[\ell]} W_i^{(0)}$. Then
by~\eqref{RLgarbage} of Lemma~\ref{lem:sparseRL}, and since
$\maxdeg(H)\leq\Omega k$,
\begin{equation}\label{eq_incidentToGarbageCluster}
|E_1|\le \epsilon\Omega nk\;.
 \end{equation}
 
Let $E_2$ be the set of those edges of $H$ which belong to $\epsilon$-regular pairs $(W^{(i')}_{i},W^{(j')}_{j})$ with
$ij\in E(F), i'\in[p_i],j'\in[p_j]$ of density at most
$\gamma^2$. We claim that
\begin{equation}\label{eq_edgesinsparseregular}
|E_2|\le\gamma kn\;.
 \end{equation}
Indeed, because of~\eqref{Fdensi} and by
Fact~\ref{fact:CanADensePairConsistOnlyOfSparseSubpairs?}
(with
$\alpha_\PARAMETERPASSING{F}{fact:CanADensePairConsistOnlyOfSparseSubpairs?}:=\gamma$
and
$\beta_\PARAMETERPASSING{F}{fact:CanADensePairConsistOnlyOfSparseSubpairs?}:=\gamma^2$),
for each $ij\in E(F)$ there are at most $\gamma e_H(W_i,W_j)$ edges contained in the
bipartite graphs $H[W^{(i')}_{i},W^{(j')}_{j}]$, 
$i'\in[p_i],j'\in[p_j]$, with $\density_H(W^{(i')}_{i},W^{(j')}_{j})\le \gamma^2$. Since $\sum_{ij\in E(F)}e_H(W_i,W_j)\le kn$, the validity of~\eqref{eq_edgesinsparseregular} follows. Combining \eqref{eq:boundTotalIrregII},~\eqref{eq_incidentToGarbageCluster}, and~\eqref{eq_edgesinsparseregular} we finish the proof.
\end{proof}

\section{Cutting trees: $\ell$-fine partitions}\label{sec:cut}
The purpose of this section is to introduce some notation
related to trees. The notion of an $\ell$-fine partition
of a tree shall be of particular interest. Roughly speaking, an $\ell$-fine
partition of a tree $T\in\treeclass{k}$ is a partition of the $T$ into a
small number of cut-vertices and subtrees of order at most $\ell$ with some
additional properties. This notion is essential for our proof of
Theorem~\ref{thm:main} as we use a certain sequential procedure to embed $T_\PARAMETERPASSING{T}{thm:main}$
into the host graph $G_\PARAMETERPASSING{T}{thm:main}$, embedding a subtree after subtree.

\smallskip

Let $T$ be a tree rooted at $r$, inducing the partial order $\preceq$\index{mathsymbols}{*$\preceq$@$\preceq$} on $V(T)$ (with $r$ as the minimal element).
 If $a\preceq b$ and $ab\in E(T)$ then we say $b$ is a \index{general}{child}{\em child of} $a$ and  $a$ is the \index{general}{parent}{\em parent of} $b$.
\index{mathsymbols}{*Ch@$\children(v)$}$\children(a)$ denotes the set of children of $a$,
and the parent of a vertex $b\not=r$ is denoted
\index{mathsymbols}{*Par@$\parent(v)$}$\parent(b)$. For a set $U\subset V(T)$ write
\index{mathsymbols}{*Par@$\parent(U)$}$\parent(U):=\bigcup_{u\in U\setminus
\{r\}}\parent(u)\setminus U$ and \index{mathsymbols}{*Ch@$\children(U)$}$\children(U):=\bigcup_{u\in
U}\children(u)\setminus U$.

We say that a
tree $T'\subset T$ is \index{general}{induced tree}{\em induced} by
a vertex $x\in V(T)$ if $V(T')$ is the up-closure of $x$ in $V(T)$, i.e., $V(T')=\{v\in V(T)\: :\: x \preceq
v\}$. We then write \index{mathsymbols}{*T@$T(r,\uparrow x)$}$T'=T(r,\uparrow x)$, or $T'=T(\uparrow x)$, if the root
is obvious from the context and call $T'$ an \index{general}{end subtree}{\em end
subtree}. Subtrees of $T$ that are not end subtrees are called  \index{general}{internal
subtree}{\em internal subtrees}.

Let $T$ be a tree
rooted at $r$ and let $T'\subset T$ be a subtree with $r\not \in V(T')$. The
\index{general}{seed}{\em seed of
$T'$} is the $\preceq$-maximal  vertex $x\in V(T)\setminus V(T')$ such that
$x\preceq v$ for all $v\in V(T')$. We write \index{mathsymbols}{*Seed@$\seed$} $\seed(T')=x$. A~\emph{fruit}\index{general}{fruit} in a rooted tree $(T,r)$ is any vertex $u\in V(T)$ whose distance from $r$ is even and at least four.

\smallskip

We can now state the most important definition of this section.
\begin{definition}[\bf $\ell$-fine partition]\label{ellfine}
Let $T\in \treeclass{k}$ be a tree rooted at $r$. An
\index{general}{fine partition}{\em $\ell$-fine partition of $T$} is a
quadruple $(W_A,W_B, \shrubA,
\shrubB)$, where $W_A,W_B\subseteq V(T)$ and $\shrubA$, $\shrubB$ are families of subtrees of $T$ such that
\begin{enumerate}[(a)]
\item  the three sets $W_A$, $W_B$ and $\{V(T^*)\}_{T^*\in\shrubA\cup
\shrubB}$ partition $V(T)$,\label{decompose}
\item $r\in W_A\cup W_B$,\label{root}
\item $\max\{|W_A|,|W_B|\}\leq 336k/{\ell}$,\label{few}
\item for $w_1,w_2\in W_A\cup W_B$ the distance $\dist(w_1,w_2)$ is odd if and only if one of them lies in $W_A$ and the other one in $W_B$,\label{parity}
\item $v(T^*)\leq \ell$ for every tree $T^*\in \shrubA\cup
\shrubB$,\label{small}
\item $V(T^*)\cap \neighbor(W_B)=\emptyset$ for every $T^*\in
\shrubA$ and $V(T^*)\cap \neighbor(W_A)=\emptyset$ for every $T^*\in
\shrubB$,\label{nice}
\item each tree of $\shrubA\cup\shrubB$ has its seed in
$W_A\cup W_B$,\label{cut:precede}
\item  $|V(T^*)\cap \neighbor(W_A\cup W_B)|\le 2$ for each $T^*\in\shrubA\cup \shrubB$,\label{2seeds}
\item if $V(T^*)\cap \neighbor(W_A\cup
W_B)$ contains two distinct vertices $y_1$, $y_2$ for some  $T^*\in\shrubA\cup
\shrubB$, then $\dist(y_1,y_2)\ge 4$,\label{short}
\item if $T_1,T_2\in\shrubA\cup\shrubB$ are two internal subtrees of $T$ such that $v_1\in T_1$ precedes $v_2\in T_2$ 
then $\dist_T(v_1,v_2)>2$,\label{ellfine:separatedinternalshrubs}
\item $\shrubB$ does not contain any internal tree of $T$, and\label{Bend}
\item $\sum_{\substack{T^*\in
\shrubA\\T^*\textrm{ end tree of $T$}}}v(T^*)\ge\sum_{T^*\in \shrubB}v(T^*)\;\mbox{.}$
\label{Bsmall}
\end{enumerate}
\end{definition}

\begin{remark}\label{rem:ellfineW}
It  is easy to see that any $\ell$-fine partition $(W_A,W_B, \shrubA,\shrubB)$  of a tree $(T,r)$ is determined once we know the set $W=W_A\cup W_B$, except possibly for being able to swap $W_A$ with $W_B$ and $\shrubA$ with $\shrubB$. Indeed, the division of $W$ into two sets $W'$ and $W''$ follows the bipartition of~$T$,  and conditions~\eqref{Bend} and~\eqref{Bsmall} determine which of $W'$, $W''$ is $W_A$ unless  $T-W$ contains no internal trees and~\eqref{Bsmall} would hold either way. 
During the proof of Lemma~\ref{lem:TreePartition} below we shall therefore sometimes just say one of the conditions~\eqref{decompose}--\eqref{Bsmall} holds for the set $W$, and not explicitly mention the tuple $(W_A,W_B, \shrubA,\shrubB)$.
\end{remark}

\begin{remark}\label{rem:fruits}
Suppose that $(W_A,W_B, \shrubA,\shrubB)$ is an $\ell$-fine partition of a tree $(T,r)$, and suppose that $T^*\in\shrubA\cup\shrubB$ is such that
$|V(T^*)\cap \neighbor(W_A\cup W_B)|=2$. Let us root $T^*$ at the neighbour $r_1$ of its seed, and
let $r_2$ be the other vertex of $V(T^*)\cap \neighbor(W_A\cup W_B)$. Then
\eqref{parity},~\eqref{nice}, and~\eqref{short} imply that $r_2$ is a fruit in
$(T^*,r_1)$.
\end{remark}

The following is the main lemma of this section. It asserts that each tree of order $k$ has $\ell$-fine
partitions for all values of $\ell\leq k$.

\begin{lemma}\label{lem:TreePartition}
Let $T\in \treeclass{k}$ be a tree rooted at $r$ and let $\ell\in
\mathbb{N}$ with $\ell\le k$. Then $T$ has an $\ell$-fine partition.
\end{lemma}
Similar but simpler tree-cutting procedures were used in other literature concerning the
Loebl--Koml\'os--S\'os Conjecture in the dense setting,
cf.~\cite{AKS95,HlaPig:LKSdenseExact,PS07+,Z07+}. There, using the notation of
Conjecture~\ref{conj:LKS}, the trees in $\shrubA\cup\shrubB$ of an $\ell$-fine
partition of a tree $T\in\treeclass{k}$ are embedded in regular pairs of a
Regularity Lemma decomposition of the host graph $G$. In the current paper
however, a more complex decomposition result (Lemma~\ref{lem:LKSsparseClass}) than the Regularity Lemma is used to capture the structure of $G$. To this end we had to further strengthen the features of the $\ell$-fine partition. In particular, features~\eqref{2seeds},~\eqref{short},~\eqref{ellfine:separatedinternalshrubs} of Definition~\ref{ellfine} were introduced to handle the more complex embedding procedures in our setting.

\begin{remark}\label{rem:internalVSend}
\begin{enumerate}[(i)]
\item 
In our proof of Theorem~\ref{thm:main}, we shall apply
Lemma~\ref{lem:TreePartition} to a tree $T_\PARAMETERPASSING{T}{thm:main}\in\treeclass{k}$. The number
$\ell_\PARAMETERPASSING{L}{lem:TreePartition}$ will  be linear in $k$, and
thus~\eqref{few} of Definition~\ref{ellfine} tells us that the size of the  sets $W_A$ and
$W_B$ is bounded by an absolute constant.
\item\label{it:fewinternaltrees}
Each internal tree in $\shrubA$ of an $\ell$-fine partition has a unique vertex from $W_A$ above it. Thus with $\ell_\PARAMETERPASSING{L}{lem:TreePartition}$ as above also the number of internal trees in $\shrubA$ is bounded by an absolute constant. This need not not be the case for the number of end trees. For instance, if $(T_\PARAMETERPASSING{T}{thm:main},r)$ is a star with $k-1$ leaves and rooted at its centre $r$ then $W_A=\{r\}$ while the $k-1$ leaves of $T_\PARAMETERPASSING{T}{thm:main}$ form the end shrubs in $\shrubA$.
\end{enumerate} 
\end{remark}

\begin{proof}[Proof of Lemma~\ref{lem:TreePartition}]
First we shall use an inductive construction to get candidates for $W_A$, $W_B$,
$\shrubA$ and $\shrubB$, which we shall modify later on, so that they satisfy
all the  conditions required by Definition~\ref{ellfine}.

Set $T_0:=T$. Now, inductively for $i\ge 1$ choose a $\preceq$-maximal vertex
$x_i\in V(T_{i-1})$ with the property that $v(T_{i-1}(\uparrow x_i))>\ell$. We set $T_i:=T_{i-1}-(V(T_{i-1}(\uparrow x_i))\setminus
\{x_i\})$. If, say at step $i=i_\textrm{end}$, no such
$x_{i}$ exists, then $v(T_{i-1})\le \ell$. In that case,
set $x_i:=r$, set $W_1:=\{x_i\}_{i=1}^{i_\textrm{end}}$ and terminate. 
The fact that $v(T_{i-1}- V(T_i))\geq \ell$ for each $i<i_\textrm{end}$ implies
that
\begin{equation}\label{X}
|W_1|-1=i_\textrm{end}-1\le k/\ell\;.
\end{equation}

 Let $\mathcal{C}$  be the set of all
components of the forest $T-W_1$.
Observe that by the choice of the $x_i$ each
$T^*\in\mathcal C$ has order at most $\ell$. 

Let $A$ and $B$ be the colour classes of $T$ such that $r\in A$. Now, choosing
$W_A$ as $W_1\cap A$ and $W_B$ as $W_1\cap B$ and dividing $\mathcal C$
adequately into sets $\shrubA$ and $\shrubB$ would yield a quadruple that
satisfies conditions~\eqref{decompose}, \eqref{root}, \eqref{few},
\eqref{parity}, \eqref{small} and~\eqref{cut:precede}. In order to find also the
remaining properties satisfied,  we shall refine our tree partition by adding
more vertices to $W_1$, thus making the trees in $\shrubA\cup\shrubB$ smaller. In doing so, we have to be careful not to end up
violating~\eqref{few}. We shall enlarge the set of cut vertices in several
steps, accomplishing sequentially, in this order, also properties~\eqref{2seeds}, \eqref{ellfine:separatedinternalshrubs},
\eqref{nice},
\eqref{short}, and in the last step at the same time~\eqref{Bend}
and~\eqref{Bsmall}. It will be easy to check that in each of the steps none of the
previously established properties is lost, so we will not explicitly check them, except for~\eqref{few}.

For condition~\eqref{2seeds}, first define $T'$ as the subtree of $T$ that contains all vertices of $W_1$ and all vertices that lie on paths in $T$ which have both endvertices in $W_1$. Now, if a subtree $T^*\in\mathcal C$ does not already satisfy~\eqref{2seeds} for $W_1$, then $V(T^*)\cap V(T')$ must contain some  vertices of degree at least three. We will add the set $Y(T^*)$ of all these vertices to $W_1$. Formally, let $Y$ be the union of the sets $Y(T^*)$ over all $T^*\in\mathcal C$, and set $W_2:=W_1\cup Y$. Then the components of $T-W_2$ satisfy~\eqref{2seeds}.

Let us upper-bound the size of the set $W_2$. For each $T^*\in\mathcal C$, note that by Fact~\ref{fact:treeshavemanyleaves3vertices} for $T^*\cap T'$, we know that $|Y(T^*)|$ is at most the number of leaves of $T^*\cap T'$ (minus two). On the other hand, each leaf of $T^*\cap T'$ has a child in $W_1$ (in $T$). As these children are distinct for different trees $T^*\in\mathcal C$, we find that $ |Y|\leq |W_1|$ and thus
\begin{equation}\label{Y}
 |W_2|\leq 2|W_1|\;.
\end{equation}


Next, for condition~\eqref{ellfine:separatedinternalshrubs}, observe that by setting $W_3:=W_2
\cup \parent_T(W_2)$ the components of $T-W_3$ fulfill~\eqref{ellfine:separatedinternalshrubs}. We have 
\begin{equation}\label{eq:YY}
|W_3|\le 2|W_2|\overset{\eqref{Y}}\le 4|W_1|\;.
\end{equation} 

In order to ensure condition~\eqref{nice}, let $R^*$ be the set of the roots ($\preceq$-minimal vertices) of those components $T^*$ of  $T-W_3$ which contain neighbours of both colour classes of $T$. Setting $W_4:=W_3\cup R^*$ we see that~\eqref{nice} is satisfied for $W_4$. Furthermore, as for each vertex in $R^*$ there is a distinct member of $W_3$ above it in the order on $T$, we obtain
\begin{equation}\label{W3}
|W_4|\leq 2|W_3|\overset{\eqref{eq:YY}}\leq 8|W_1|.
\end{equation}

Next, we shall aim for a stronger version of property~\eqref{short}, namely,
\begin{enumerate}
\item[(\ref{short}')] if $V(T^*)\cap \neighbor_{T}(W_A\cup
W_B)=\{y_1,y_2\}$ with $y_1\neq y_2$ for some  $T^* \in\shrubA\cup
\shrubB$, then $\dist(y_1,y_2)\ge 6$.\label{shortBetter}
\end{enumerate}
The reason for requiring this strengthening is that later we might introduce additional cut vertices which would
``shorten $T^*$ by two''.

Consider a component $T^*$ of $ T-W_4$ which is an
internal tree of $T$. If $T^*$ contains two distinct neighbours $y_1$, $y_2$ of $W_4$ such that
$\dist_{T^*}(y_1,y_2)< 6$, then we call $T^*$ \emph{short}. Observe that there are at most
$|W_4|$ short trees, because each of these trees has a unique vertex from $W_4$
above it. Let $Z(T^*)\subset V(T^*)$ be the
vertices on the path from $y_1$ to $y_2$. Then
$|Z(T^*)|\le 6$. Letting $Z$ be the union of the sets $Z(T^*)$ over all short trees in
$T-W_4$, and set $W_5:=W_4\cup Z$, we obtain
\begin{equation}\label{W4}
|W_5|\leq |W_4|+6|W_4|\overset{\eqref{W3}}\leq 56|W_1|\overset{\eqref{X}}\leq 112 k/\ell.
\end{equation}

We still need to ensure~\eqref{Bend} and~\eqref{Bsmall}. To this end, consider 
the set $\mathcal C'$ of all components of $T-W_5$. Set $\mathcal C'_A:=\{T^*\in \mathcal C':\seed(T^*)\in A\}$ and set $\mathcal C'_B:=\mathcal C'\setminus \mathcal C'_A$. We
assume that 
\begin{equation}\label{eq:AssumeEndTrees}
\sum_{T^*\in \mathcal C'_A\ : \ T^*\textrm{ end tree of $T$}} v(T^*)\ge \sum_{T^*\in \mathcal C'_B\ : \ T^*\textrm{ end tree of $T$}}
v(T^*)\;,
\end{equation}
 as otherwise we can simply swap $A$ and $B$.


Now, for each $T^*\in \mathcal C'_B$ that is not a end subtree of $T$, set $X(T^*):=V(T^*)\cap\neighbor_{T}(W_5)$. Let $X$ be the union of all such sets $X(T^*)$. Observe that 
\begin{equation}\label{XW5}
|X|\le 2|W_5\cap B|\le 2|W_5|.
\end{equation} 
For $W:=W_5 \cup X$,  all internal trees of $T-W$ have their seeds in $A$. This will guarantee~\eqref{Bend}, and, together with~\eqref{eq:AssumeEndTrees}, also~\eqref{Bsmall}.

 Finally, set $W_A:=W\cap A$ and $W_B:=W\cap B$, and
let $\shrubA$ and $\shrubB$ be the sets of those components of $T-W$ that have
their seeds in $W_A$ and $W_B$, respectively. By construction,
$(W_A,W_B,\shrubA,\shrubB)$ has all the properties of an $\ell$-fine partition. In particular, for~\eqref{few}, we find with~\eqref{W4} and~\eqref{XW5} that $|W|\leq |W_5|+2|W_5\cap B|\leq 336 k/\ell$.
\end{proof}

For an $\ell$-fine partition $(W_A,W_B,\shrubA,\shrubB)$ of a rooted tree $(T,r)$, the
trees $T^*\in \shrubA\cup\shrubB$ are called \index{general}{shrub}{\em shrubs}. An
\emph{end shrub} is a shrub which is an end subtree. An \emph{internal shrub} is a shrub which is an internal subtree.  A \emph{knag}\index{general}{knag} is a component of the forest $T[W_A\cup W_B]$. Suppose that $T^*\in \shrubA$ is an internal shrub, and $r^*$ its $\preceq_r$-minimal vertex. Then $T^*-r^*$ contains a unique component with a vertex from $\neighbor_T(W_A)$. We call this component \index{general}{subshrub}\index{general}{principal subshrub}\emph{principal subshrub}, and the other components \index{general}{peripheral subshrub}\emph{peripheral subshrubs}.
\begin{definition}[\bf ordered skeleton]\index{general}{ordered
skeleton} We say that the sequence $\big(X_0,X_1,\ldots, X_m\big)$
is an \emph{ordered skeleton} of the $\ell$-fine partition
$(W_A,W_B,\shrubA,\shrubB)$ of a rooted tree $(T,r)$ if 
\begin{itemize}
 \item $X_0$ is a  knag and contains~$r$, and all other $X_i$ are either knags or shrubs, 
 \item $V(\bigcup_{i\le m}X_i)=V(T)$, and
 \item for each $i=1,\ldots,m$, the subgraph formed by
  $X_0\cup X_1\cup\ldots\cup X_i$  is connected in $T$.
  \end{itemize}
\end{definition}
Directly from Definition~\ref{ellfine} we get:
\begin{lemma}\label{lem:orderedskeleton}
Any $\ell$-fine
partition of any rooted tree has  an ordered skeleton. 
\end{lemma}

\bigskip
Figure~\ref{fig:BinaryTreeCut} shows an $(\tau k)$-fine partition $(W_A,W_B, \shrubA, \shrubB)$ of a binary tree $T\in \treeclass{k}$, for a fixed $\tau>0$ and $k$ large. The vertices whose distance is $O(\log(\tau^{-1}))$ from the root comprise a sole knag of $T$ (with respect to $(W_A,W_B, \shrubA, \shrubB)$). This example will be important in Section~\ref{sssec:whyGexp}.
\begin{figure}[h]
 \centering
 \includegraphics[scale=0.7]{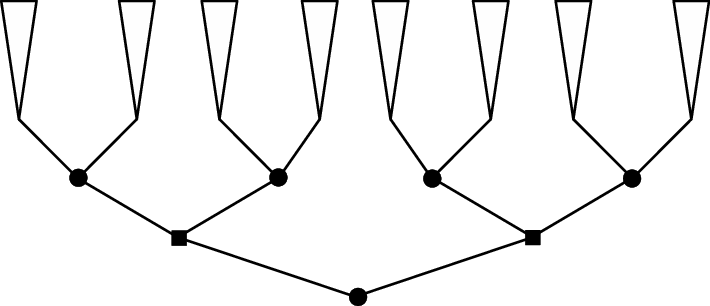}
 \caption[Fine partition of binary tree]{An $(\tau k)$-fine partition $(W_A,W_B, \shrubA,
\shrubB)$ of a binary tree $T\in \treeclass{k}$. The elements of the set $W_A$ are drawn as circles and those of $W_B$ as squares. The sole knag is of depth $O(\log(\tau^{-1}))$, two in this picture. Each schematic triangle represents one end shrub of $\shrubA\cup\shrubB$.}
 \label{fig:BinaryTreeCut}
\end{figure}

\section{Decomposing sparse graphs}\label{sec:class}
In this section, we work out a structural decomposition of a possibly
sparse graph which is suitable for embedding trees. Our motivation comes from
the success of the Regularity Method in the  setting of dense graphs
(see~\cite{KuhnOsthusSurv}). The main technical result of this section, the
``decomposition lemma'', Lemma~\ref{lem:decompositionIntoBlackandExpanding},
provides such a decomposition. Roughly speaking, each graph of a moderate
maximum degree can be decomposed into regular pairs, and two different expanding
parts.

We then combine Lemma~\ref{lem:decompositionIntoBlackandExpanding} with a lemma
on creating a gap in the degree sequence (Lemma~\ref{prop:gap}) to get a decomposition lemma
for graphs from~$\LKSgraphs{n}{k}{\eta}$, Lemma~\ref{lem:LKSsparseClass}.
Lemma~\ref{lem:LKSsparseClass} asserts that each graph
from~$\LKSgraphs{n}{k}{\eta}$ can be decomposed into vertices of degree much
larger than $k$, regular pairs, and expanding parts. 
Further we give a non-LKS-specific version of Lemma~\ref{lem:LKSsparseClass} in Lemma~\ref{lem:genericBD},
which asserts that \emph{each} graph with average degree bigger than an absolute constant has a sparse decomposition. Such a decomposition lemma was used by Ajtai, Koml\'os, Simonovits and Szemer\'edi in their work on the Erd\H os--S\'os conjecture and we expect that it will find applications in other tree embedding problems, and possibly elsewhere.


\subsection{Creating a gap in the degree sequence}\label{ssec:class-gap}
The goal of this section is to show that any graph
$G\in\LKSmingraphs{n}{k}{\eta}$ has a subgraph  $G'\in
\LKSsmallgraphs{n}{k}{\eta/2}$ which has a gap in its degree
sequence. Note that $G'$ then contains almost all the edges of $G$. This is formulated in the next lemma.

\begin{proposition}\label{prop:gap}
Let $G\in\LKSmingraphs{n}{k}{\eta}$ and let $(\Omega_i)_{i\in\mathbb N}$
be a sequence of positive numbers with
$\Omega_j/\Omega_{j+1}\leq \eta^2/100$ for all $j\in\mathbb N$. Then there is an index $i^{*}\leq 100\eta^{-2}$ and  a subgraph $G'\subseteq G$ such that 
\begin{enumerate}[(i)]
\item\label{item:(prop:gap)stillLKS}$G'\in \LKSsmallgraphs{n}{k}{\eta/2}$, and
\item\label{item:(prop:gap)gap} no vertex $v\in V(G')$ has degree
$\deg_{G'}(v)\in[\Omega_{i^*}k,\Omega_{i^{*}+1}k)$.
\end{enumerate}
\end{proposition}

\begin{proof}
Set $R:=\lfloor 100 \eta^{-2}\rfloor$. 
For $i\in [R]$ and any graph $H\subseteq G$ define the sets $X_i(H):=\{v\in
V(H)\::\: \deg_H(v)\in[\Omega_i k, \Omega_{i+1}k)\}$ and for $i=R+1$ set
$X_i(H):=\{v\in V(H)\::\: \deg_H(v)\in[\Omega_i k, \infty)\}$. As
\begin{equation*}
\sum_{i\in[R]}\quad\sum_{v\in X_{i}(G)\cup X_{i+1}(G)}\deg(v)\le 4e(G)\;,
\end{equation*}
by averaging  we find an index $i^*\in [R]$ such that
\begin{equation}\label{eq:sparselyconnectedpair}
\sum_{v\in X_{i^*}(G)\cup X_{i^*+1}(G)}\deg(v)\frac{4e(G)}{R}.
\end{equation}

Let $E_0$ be the set of all the edges incident with $X_{i^*}(G)\cup X_{i^*+1}(G)$. Now, starting with $G_0:=G-E_0$, successively define graphs $G_j\subsetneq G_{j-1}$ for $j\geq 1$ using any of the following two types of edge deletions:
\begin{enumerate}
\item[(T1)] If there is a vertex $v_j\in X_{i^*}(G_{j-1})$ then we choose an edge $e_j$ that is
incident with $v_j$, and set $G_j:=G_{j-1}-e_j$.
\item[(T2)] If there is an edge $e_j=u_jv_j$ of $G_{j-1}$ with $u_j\in \smallvertices{\eta/2}{k}{G_{j-1}}$ and $v_j\in \bigcup_{i=i^*+1}^{R+1}X_{i}(G_{j-1})$ then we set $G_j:=G_{j-1}-e_j$. 
\end{enumerate}
Since we keep deleting edges, the procedure stops at some point, say at step $j^*$, when neither of (T1), (T2) is applicable. Note that the resulting graph $G_{j^*}$ already has Property~\eqref{item:(prop:gap)gap}.

 Let $E_1 \subset E(G)$ be the set of those edges deleted by applying (T1). We shall estimate the size of $E_1$. First, observe that
\begin{equation*}
\left|\bigcup_{i=i^*+2}^{R+1} X_i(G)\right| \le \frac{2e(G)}{\Omega_{i^*+2}k}\;.
\end{equation*}
Moreover, each vertex of $\bigcup_{i=i^*+2}^{R+1} X_i(G)$ appears at most $ (\Omega_{i^*+1}-\Omega_{i^*})k < \Omega_{i^*+1}k$ times as the vertex $v_j$ in the deletions of type (T1). Consequently, 
\begin{equation}\label{kanteninE1}
 |E_1|\le \Omega_{i^*+1}\left|\bigcup_{i=i^*+2}^{R+1} X_i(G)\right|k\le \frac{2\Omega_{i^*+1}e(G)}{\Omega_{i^*+2}} \;.
\end{equation}

Now, observe that the vertices in $ \largevertices{\eta}{k}{G}\cap\smallvertices{\eta/2}{k}{G_{j^*}}$ have dropped their degree from  $(1+\eta)k$ to  $(1+\eta/2)k$ by operations other than~(T2). So each of these vertices  is incident with at least $\eta k/2$ edges from the set $E_0\cup E_1$. Therefore, by the definition of $E_0$, by~\eqref{eq:sparselyconnectedpair}, and by~\eqref{kanteninE1},
$$\left|\largevertices{\eta}{k}{G}\cap\smallvertices{\eta/2}{k}{G_{j^*}}\right|\le \frac{2\cdot|E_0\cup E_1|}{ \eta k/2}\le\left(\frac{4}{R}+\frac{2\Omega_{i^*+1}}{\Omega_{i^*+2}}\right)\cdot\frac{4e(G)}{\eta k}\leByRef{eq:LKSminimalNotManyEdges} \frac{\eta n}2\;.$$
Thus
$$|\largevertices{\eta/2}{k}{G_{j^*}}|\ge |\largevertices{\eta}{k}{G}|-|\largevertices{\eta}{k}{G}\cap\smallvertices{\eta/2}{k}{G_{j^*}}|\ge(1/2+\eta/2)n\;,$$ and consequently, $G_{j^*}\in\LKSgraphs{n}{k}{\eta/2}$. 

Last, we obtain the graph $G'$ by successively deleting any edge from $G_{j^*}$ which connects a vertex from $\smallvertices{\eta/2}{k}{G_{j^*}}$  with a vertex whose degree is not exactly $\lceil(1+\frac\eta2)k\rceil$. This does not affect the already obtained Property~\eqref{item:(prop:gap)gap},
since we could not apply~(T2) to $G_{j^*}$. We claim that for the resulting graph $G'$ we have $G'\in\LKSsmallgraphs{n}{k}{\eta/2}$. Indeed, $\largevertices{\eta/2}{k}{G'}=\largevertices{\eta/2}{k}{G_{j^*}}$, and thus $G'\in\LKSgraphs{n}{k}{\eta/2}$. Property~\ref{def:LKSsmallB} of Definition~\ref{def:LKSsmall} follows from the last step of the construction of $G'$. To see Property~\ref{def:LKSsmallA} of Definition~\ref{def:LKSsmall} we use  Fact~\ref{fact:propertiesOfLKSminimalGraphs}(2)  for $G$ (which by assumption is in $\LKSmingraphs{n}{k}{\eta}$).
\end{proof}


\subsection{Decomposition of graphs with moderate maximum
degree}\label{ssec:class-black} 

First we introduce some useful notions. We start with dense spots which indicate an accumulation of edges in a sparse graph.

\begin{definition}[\bf \index{general}{dense spot}$(m,\gamma)$-dense spot,
\index{general}{nowhere-dense}$(m,\gamma)$-nowhere-dense]\label{def:densespot} An \emph{$(m,\gamma)$-dense spot} in a graph $G$ is a non-empty bipartite sub\-graph  $D=(U,W;F)$ of  $G$ with
$\density(D)>\gamma$ and $\mindeg (D)>m$. We call $G$
\emph{$(m,\gamma)$-nowhere-dense} if it does not contain any $(m,\gamma)$-dense spot.
\end{definition}
We remark that dense spots as bipartite graphs do not have a 
specified orientation, that is, we view $(U,W;F)$ and $(W,U;F)$ as
the same object.

\begin{fact}\label{fact:sizedensespot}
Let $(U,W;F)$ be a $(\gamma k,\gamma)$-dense spot in a
graph $G$ of maximum degree at most $\Omega k$. Then
$\max\{|U|,|W|\}\le \frac{\Omega}{\gamma}k.$
\end{fact}
\begin{proof}
It suffices to observe that $$\gamma |U||W|\leq
e(U,W)\leq\maxdeg(G)\cdot\min\{|U|,|W|\}\leq\Omega k\cdot \min\{|U|,|W|\}.$$
\end{proof}

The next fact asserts that in a bounded degree graph there cannot be too many edge-disjoint dense spots containing a given vertex.
\begin{fact}\label{fact:boundedlymanyspots}
Let $H$ be a graph of maximum degree at most $\Omega k$, let $v\in V(H)$, and let $\DenseSpots$ be a family of edge-disjoint $(\gamma k,\gamma)$-dense spots. Then less than $\frac{\Omega}{\gamma}$ dense spots from $\DenseSpots$ contain $v$.
\end{fact}
\begin{proof}
This follows as $v$ sends more than $\gamma k$ edges to each dense spot from $\DenseSpots$ it is incident with, the dense spots $\DenseSpots$ are edge-disjoint, and $\deg(v)\le \Omega k$. 
\end{proof}

Last, we include a bound concerning the total size of dense spots intersecting substantially a given set.
\begin{fact}\label{fact:substantialintersectspots}
Let $H$ be a graph of maximum degree at most $\Omega k$. Let $Y\subset V(H)$ be a set of size at most $A k$, and $\DenseSpots$ a family of edge-disjoint $(\gamma k,\gamma)$-dense spots. Define $\DenseSpots':=\{D\in \DenseSpots\::\: |V(D)\cap Y|\ge \beta k\}$.  Then for the set $X:=\bigcup_{D\in\DenseSpots'} V(D)$ we have $|X|\le \frac{2A\Omega^2}{\beta\gamma^2}k$.
\end{fact}
\begin{proof}
Let us count the number of certain pairs $(y,D)$ in two different ways.
$$\beta k|\DenseSpots'|\le \big|\{(y,D)\::\:y\in Y, D\in\DenseSpots', y\in V(\DenseSpots')\}\big|\overset{\mathrm{F}\ref{fact:boundedlymanyspots}}\leq |Y|\frac{\Omega}{\gamma}\;.$$
Put together, $|\DenseSpots'|\le \frac{A\Omega}{\beta\gamma}$. The fact now follows from Fact~\ref{fact:sizedensespot}.
\end{proof}

\medskip
Our second definition of this section might seem less intuitive at first sight.
It describes a property for finding dense spots outside some ``forbidden'' set
$U$, which in later applications will be the set of vertices already used for
a partial embedding of a tree $T_\PARAMETERPASSING{T}{thm:main}\in\treeclass{k}$ in Theorem~\ref{thm:main} during our sequential embedding procedure.

\begin{definition}[\bf
\index{general}{avoiding}$(\Lambda,\epsilon,\gamma,k)$-avoiding set]\label{def:avoiding} Suppose that
$G$ is a graph and $\DenseSpots$ is a family of dense spots in $G$. A set
$\smallatoms\subset \bigcup_{D\in\DenseSpots} V(D)$ is \emph{$(\Lambda,\epsilon,\gamma,k)$-avoiding} with
respect to $\DenseSpots$ if for every $\bar U\subset V(G)$ with $|\bar U|\le \Lambda k$ the following holds that for all but at most $\epsilon k$ vertices $v\in\smallatoms$. There is a dense spot $D\in\DenseSpots$ with $|\bar U\cap V(D)|\le \gamma^2 k$ that contains $v$.
\end{definition}

Note that a subset of a $(\Lambda,\epsilon,\gamma,k)$-avoiding set is also $(\Lambda,\epsilon,\gamma,k)$-avoiding.

We now come to the main concepts of this section, the bounded and the sparse
decompositions. These notions in a way correspond to the partition structure
from the Regularity Lemma, although naturally more
complex since we deal with (possibly) sparse graphs here. Lemma~\ref{lem:decompositionIntoBlackandExpanding} is then a corresponding
regularization result. 

\begin{definition}[\index{general}{bounded decomposition}{\bf
$(k,\Lambda,\gamma,\epsilon,\nu,\rho)$-bounded decomposition}]\label{bclassdef}
Let $\mathcal V=\{V_1, V_2,\ldots, V_s\}$ be a partition of the vertex set of a graph $G$. We say that $( \clusters,\DenseSpots, \Gblack, \Gexp,
\smallatoms )$ is a {\em $(k,\Lambda,\gamma,\epsilon,\nu,\rho)$-bounded
decomposition} of $G$ with respect to $\mathcal V$ if the following properties
are satisfied:
\begin{enumerate}
\item\label{defBC:clusters} The elements of $\clusters$ are disjoint subsets of 
$ V(G)$.
\item\label{defBC:RL} $\Gblack$ is a subgraph of $G-\Gexp$ on the vertex set $\bigcup \clusters$. For each edge
 $xy\in E(\Gblack)$ there are distinct $C_x\ni x$ and $C_y\ni y$ from $\clusters$,
and  $G[C_x,C_y]=\Gblack[C_x,C_y]$. Furthermore, 
$G[C_x,C_y]$ forms an $\epsilon$-regular pair of  density at least $\gamma^2$.
\item We have $\nu k\le |C|=|C'|\le \epsilon k$ for all
$C,C'\in\clusters$.\label{Csize}
\item\label{defBC:densepairs}  $\DenseSpots$ is a family of edge-disjoint $(\gamma
k,\gamma)$-dense spots  in $G-\Gexp$.  For
each $D=(U,W;F)\in\DenseSpots$ all the edges of $G[U,W]$ are covered
by $\DenseSpots$ (but not necessarily by $D$).
\item\label{defBC:dveapul} If  $\Gblack$
contains at least one edge between $C_1,C_2\in\clusters$ then there exists a dense
spot $D=(U,W;F)\in\DenseSpots$ such that $C_1\subset U$ and $C_2\subset
W$.
\item\label{defBC:prepartition}
For
all $C\in\clusters$ there is $V\in\mathcal V$ so that either $C\subseteq V\cap V(\Gexp)$ or $C\subseteq V\setminus V(\Gexp)$.
For
all $C\in\clusters$ and $D=(U,W; F)\in\DenseSpots$ we have $C\cap U\in\{\emptyset, C\}$.
\item\label{defBC:nowheredense}
$\Gexp$ is  a $(\gamma k,\gamma)$-nowhere-dense subgraph of $G$ with $\mindeg(\Gexp)>\rho k$.
\item\label{defBC:avoiding}
$\smallatoms$ is a $(\Lambda,\epsilon,\gamma,k)$-avoiding subset  of
$V(G)\setminus \bigcup \clusters$ with respect to dense spots $\DenseSpots$.
\end{enumerate}

\smallskip
We say that the bounded decomposition $(\clusters,\DenseSpots, \Gblack, \Gexp,
\smallatoms )$ {\em respects the avoiding threshold~$b$}\index{general}{avoiding threshold} if for each $C\in \clusters$ we either have $\maxdeg_G(C,\smallatoms)\le b$, or $\mindeg_G(C,\smallatoms)> b$.
\end{definition}
Let us remark that ``exp'' in $\Gexp$ stands for ``expander'' and ``reg'' in $\Gblack$ stands for ``regular(ity)''.

The members of $\clusters$ are called \index{general}{cluster}{\it clusters}. Define the
{\it cluster graph} \index{mathsymbols}{*Gblack@$\BGblack$}  $\BGblack$ as the graph
on the vertex set $\clusters$ that has an edge $C_1C_2$
for each pair $(C_1,C_2)$ which has density at least $\gamma^2$ in the graph
$\Gblack$. 

Property~\ref{defBC:prepartition} tells us that the clusters may be prepartitioned, just as it is the case in the classic Regularity Lemma. When classifying the graph $G_\PARAMETERPASSING{T}{thm:main}$ in Lemma~\ref{lem:LKSsparseClass} below we shall use the prepartition into (roughly) $\smallvertices{\alpha_\PARAMETERPASSING{T}{thm:main}}{k}{G_\PARAMETERPASSING{T}{thm:main}}$ and $\largevertices{\alpha_\PARAMETERPASSING{T}{thm:main}}{k}{G_\PARAMETERPASSING{T}{thm:main}}$.

As said above, the notion of bounded decomposition is needed for our Regularity
Lemma type decomposition given in
Lemma~\ref{lem:decompositionIntoBlackandExpanding}. It turns out that such a
decomposition is possible only when the graph is of moderate maximum degree. On
the other hand, Lemma~\ref{prop:gap} tells us that the vertex set of any
graph\footnote{Lemma~\ref{prop:gap} is stated only for graphs from
$\LKSmingraphs{n}{k}{\eta}$, but a similar statement can be made about any
graph. See discussion in the outline of the proof of Lemma~\ref{lem:genericBD}.} can be decomposed into vertices of enormous degree and moderate degree.
The graph induced by the latter type of vertices then admits the decomposition from
Lemma~\ref{lem:decompositionIntoBlackandExpanding}. Thus, it makes sense to
enhance the structure of bounded decomposition by vertices of unbounded degree.
This is done in the next definition.

\begin{definition}[\bf \index{general}{sparse
decomposition}$(k,\Omega^{**},\Omega^*,\Lambda,\gamma,\epsilon,\nu,\rho)$-sparse decomposition]\label{sparseclassdef}
Let $\mathcal V=\{V_1, V_2,\ldots, V_s\}$ be a partition of the vertex set of a graph $G$. We say that 
$\class=(\HugeVertices, \clusters,\DenseSpots, \Gblack, \Gexp, \smallatoms )$
is a\\ 
{\em $(k,\Omega^{**},\Omega^*,\Lambda,\gamma,\epsilon,\nu,\rho)$-sparse decomposition} of $G$
with respect to $V_1, V_2,\ldots, V_s$ if the following holds.
\begin{enumerate}
\item\label{def:classgap} $\HugeVertices\subset V(G)$,
$\mindeg_G(\HugeVertices)\ge\Omega^{**}k$,
$\maxdeg_H(V(G)\setminus \HugeVertices)\le\Omega^{*}k$, where $H$ is spanned by the edges of $\bigcup\DenseSpots$, $\Gexp$, and
edges incident with $\HugeVertices$,
\item \label{def:spaclahastobeboucla} $( \clusters,\DenseSpots, \Gblack,\Gexp,\smallatoms)$ is a 
$(k,\Lambda,\gamma,\epsilon,\nu,\rho)$-bounded decomposition of
$G-\HugeVertices$ with respect to $V_1\setminus \HugeVertices, V_2\setminus \HugeVertices,\ldots, V_s\setminus \HugeVertices$.
\end{enumerate}
\end{definition}

If the parameters do not matter, we call $\class$ simply a {\em sparse
decomposition}, and similarly we speak about a {\em bounded decomposition}. 

 \begin{definition}[\bf \index{general}{captured edges}captured edges]\label{capturededgesdef}
In the situation of Definition~\ref{sparseclassdef}, we refer to the edges in
$ E(\Gblack)\cup E(\Gexp)\cup
E_G(\HugeVertices,V(G))\cup E_G(\smallatoms,\smallatoms\cup \bigcup \clusters)$
as \index{general}{captured edges}{\em captured} by the sparse decomposition. 
 We write
\index{mathsymbols}{*Gclass@$\Gcapt$}$\Gcapt$ for the subgraph of $G$ on the same vertex set which consists of the captured edges.
Likewise, the captured edges of a bounded decomposition
$(\clusters,\DenseSpots, \Gblack,\Gexp,\smallatoms )$ of a graph $G$ are those
in $E(\Gblack)\cup E(\Gexp)\cup E_G(\smallatoms,\smallatoms\cup\bigcup\clusters)$.
\end{definition}

Throughout the paper we write \index{mathsymbols}{*GD@$\GD$}$\GD$ for the subgraph of $G$
which consists of the edges contained in $\DenseSpots$. We now include an easy fact about the relation of $\GD$ and $\Gblack$.
\begin{fact}\label{fact:denseVSblack}
Let $\class=(\HugeVertices, \clusters,\DenseSpots, \Gblack, \Gexp,\smallatoms )$ be a sparse decomposition of
a graph $G$. Then each edge $xy\in E(\GD)$ with
$x,y\in\bigcup\clusters$ is either contained in $\Gblack$, or is not captured.
\end{fact}
\begin{proof}
Indeed, suppose that $xy\in E(\GD)$,
$x,y\in\bigcup\clusters$, and $xy\not\in E(\Gblack)$.  Property~\ref{def:spaclahastobeboucla} of
Definition~\ref{sparseclassdef} says that $x,y\notin \HugeVertices$. Further, by Property~\ref{defBC:avoiding} of Definition~\ref{bclassdef}, we
have $x,y\not\in \smallatoms$.  Last, Property~\ref{defBC:densepairs} of Definition~\ref{bclassdef} implies that $xy\not\in E(\Gexp)$. Hence $xy$ is not captured, as desired.
\end{proof}

We now give a bound on the number of clusters reachable through edges of
the dense spots from a fixed vertex outside $\HugeVertices$. 

\begin{fact}\label{fact:clustersSeenByAvertex}
Let  $\class=(\HugeVertices, \clusters,\DenseSpots, \Gblack, \Gexp,\smallatoms )$ be a 
$(k,\Omega^{**},\Omega^*,\Lambda,\gamma,\epsilon,\nu,\rho)$-sparse
decomposition of a graph  $G$. Let $x\in V(G)\setminus \HugeVertices$. Assume that $\clusters\not=\emptyset$, and let $\clustersize$ be the size of each of the members of~$\clusters$. Then there are less than
$$\frac{2(\Omega^*)^2k}{\gamma^2 \clustersize}\le\frac{2(\Omega^*)^2}{\gamma^2\nu}$$ clusters
$C\in\clusters$ with $\deg_{\GD}(x,C)>0$.
\end{fact}
\begin{proof}
Property~\ref{def:classgap} of Definition~\ref{sparseclassdef} says that
$\deg_{\GD}(x)\le \Omega^*k$. For each $D\in
\DenseSpots$ with $x\in V(D)$
we have that
$\deg_{D}(x)>
\gamma k$, since $D$ is a $(\gamma k,\gamma )$-dense spot. 
By Fact~\ref{fact:boundedlymanyspots}
\begin{equation}\label{labello}
 |\{D\in\DenseSpots:\deg_{D}(x)>0\}|< \frac {\Omega^*}{\gamma}.
\end{equation}

Furthermore, 
by Fact~\ref{fact:sizedensespot}, and using Property~\ref{Csize} of Definition~\ref{bclassdef}, we see that for a fixed~$D\in\DenseSpots$, we have
$$|\{C\in\clusters\::\:
C\subset V(D)\}|\le \frac{2\Omega^*
k}{\gamma}\cdot\frac1{\clustersize}\le\frac{2\Omega^*}{\gamma\nu}\;.$$
Together with~\eqref{labello} this gives that the number of clusters $C\in\clusters$ with
$\deg_{\GD}(x,C)>0$ is less than $$\frac{\Omega^*}{\gamma}\cdot
\frac{2\Omega^*k}{\gamma \clustersize}\le\frac{\Omega^*}{\gamma}\cdot
\frac{2\Omega^*}{\gamma\nu}\;,$$ as desired.
\end{proof}

As a last step before we state the main result of this section we show that the
cluster graph $\BGblack$ corresponding to a
$(k,\Omega^{**},\Omega^*,\Lambda,\gamma,\epsilon,\nu,\rho)$-sparse
decomposition $(\HugeVertices, \clusters,\DenseSpots, \Gblack, \Gexp,\smallatoms)$  has bounded degree.

\begin{fact}\label{fact:ClusterGraphBoundedDegree}
Let  $\class=(\HugeVertices, \clusters,\DenseSpots, \Gblack, \Gexp,\smallatoms )$ be a 
$(k,\Omega^{**},\Omega^*,\Lambda,\gamma,\epsilon,\nu,\rho)$-sparse
decomposition of a graph  $G$, and let $\BGblack$ be the
corresponding cluster graph. Let $\clustersize$ be the size of each cluster in $\clusters$. 
Then $\maxdeg(\BGblack)\le \frac{\Omega^*
k}{\gamma^2\clustersize}\le\frac{\Omega^*}{\gamma^2\nu}$.
\end{fact}
\begin{proof}
Let $C\in\clusters$. Then by the definition of $\BGblack$, and by the properties of
Definitions~\ref{bclassdef} and~\ref{sparseclassdef}, we get 
$$\deg_{\BGblack}(C)\leq \sum_{C'\in\neighbor_{\BGblack}(C)}\frac{e_{\Gblack}(C,C')}{\gamma^2|C||C'|} \leq\frac{\Omega^*k|C|}{ \gamma^2 |C|\clustersize} \le \frac{\Omega^*}{ \gamma^2 \nu },$$
 as desired.
\end{proof}

\medskip

We now state the most important lemma of this section. It says that any
 graph of bounded degree has a bounded decomposition which captures almost all its edges. This lemma can be considered as a sort
of Regularity Lemma for sparse graphs.

\begin{lemma}[Decomposition lemma]\label{lem:decompositionIntoBlackandExpanding}
For each $\Lambda,\Omega,s\in\mathbb N$  and each $\gamma,\epsilon,\rho>0$ there
exist $k_0\in\mathbb{N}$, $\nu>0$ such that for every $k\ge k_0$ and every
$n$-vertex graph $G$ with $e(G)\le kn$, $\maxdeg(G)\le \Omega k$, and with a
given partition $\mathcal V$ of its vertex set into at most $s$ sets, there
exists a $(k,\Lambda,\gamma,\epsilon,\nu,\rho)$-bounded decomposition
$(\clusters,\DenseSpots, \Gblack, \Gexp,\smallatoms )$ with respect to $\mathcal
V$, which captures all but at most
$(\frac{4\epsilon}{\gamma}+\epsilon\Omega+\gamma+\rho)kn$ edges of $G$. 
Furthermore, this bounded decomposition respects any given avoiding threshold $b$ and we have
\begin{equation}\label{eq:sEr}
 |E(\DenseSpots)\setminus (E(\Gblack)\cup
E_G[\smallatoms,\smallatoms\cup\bigcup\clusters])|\le
(\frac{4\epsilon}{\gamma}+\epsilon\Omega +\gamma) kn\;.
\end{equation}
\end{lemma}

A proof of Lemma~\ref{lem:decompositionIntoBlackandExpanding} is given in
Section~\ref{ssec:ProofOfDecomposition}. 

\subsection{Decomposition of LKS graphs}\label{ssec:DecompositionOfLKSGraphs}
Lemma~\ref{prop:gap} and
Lemma~\ref{lem:decompositionIntoBlackandExpanding} enable us to decompose graphs
in $\LKSgraphs{n}{k}{\eta}$ in a particular manner.

\begin{lemma}\label{lem:LKSsparseClass}
For every $\eta, \Lambda,\gamma,\epsilon,\rho>0$ there are $\nu>0$ and $k_0\in\mathbb N$ such
that for every $k>k_0$ and  for every number $b$ the following holds.
For every   sequence $(\Omega_j)_{j\in\mathbb N}$ of positive numbers with
$\Omega_j/\Omega_{j+1}\le \eta^2/100$ for all $j\in\mathbb N$
 and for every $G\in\LKSgraphs{n}{k}{\eta}$ there are an index $i$ and a subgraph $G' $ of $G$ with the following properties:
 \begin{enumerate}[(a)]
 \item $G'\in\LKSsmallgraphs{n}{k}{\eta/2}$,
 \item $i\leq 100\eta^{-2}$,
 \item\label{LKSclassif:prepart} $G'$ has a
 $(k,\Omega_{i+1},\Omega_{i},\Lambda,\gamma,\epsilon,\nu,\rho)$-sparse decomposition
$(\HugeVertices, \clusters,\DenseSpots,
\Gblack',\Gexp',\smallatoms)$ with respect to the partition
$\{V_1,V_2\}:=\{\smallvertices{\eta/2}{k}{G'},\largevertices{\eta/2}{k}{G'}\}$, and with respect to avoiding threshold $b$,
\item\label{it:LKSsparseGreyCapt} $(\HugeVertices, \clusters,\DenseSpots,
\Gblack',\Gexp',\smallatoms)$ captures all but at most
$(\frac{4\epsilon}{\gamma}+\epsilon\Omega_{\lfloor
100\eta^{-2}\rfloor}+\gamma+\rho )kn$ edges of $G'$, and
\item $|E(\DenseSpots)\setminus (E(\Gblack')\cup
E_{G'}[\smallatoms,\smallatoms\cup\bigcup\clusters])|\le
(\frac{4\epsilon}{\gamma}+\epsilon\Omega_{\lfloor
100\eta^{-2}\rfloor}+\gamma) kn$.
\end{enumerate}
\end{lemma}
\begin{proof}
Let $\nu$ and $k_0$
be given by Lemma~\ref{lem:decompositionIntoBlackandExpanding}
for input parameters
$\Omega_\PARAMETERPASSING{L}{lem:decompositionIntoBlackandExpanding}:=\Omega_{\lfloor 100\eta^{-2}\rfloor}$,
$\Lambda_\PARAMETERPASSING{L}{lem:decompositionIntoBlackandExpanding}:=\Lambda,
\gamma_\PARAMETERPASSING{L}{lem:decompositionIntoBlackandExpanding}:=\gamma,
\epsilon_\PARAMETERPASSING{L}{lem:decompositionIntoBlackandExpanding}:=\epsilon,
\rho_\PARAMETERPASSING{L}{lem:decompositionIntoBlackandExpanding}:=\rho,
b_\PARAMETERPASSING{L}{lem:decompositionIntoBlackandExpanding}:=b$, 
 and
$s_\PARAMETERPASSING{L}{lem:decompositionIntoBlackandExpanding}:=2$. Now, given
$G$, let us consider a subgraph $\tilde G$ of $G$ such that $\tilde G\in
\LKSmingraphs{n}{k}{\eta}$. Lemma~\ref{prop:gap} applied to the sequence
$(\Omega_j)_j$ and $\tilde G$ yields a graph $G'\in \LKSsmallgraphs{n}{k}{\eta/2}$
and an index $i\leq 100\eta^{-2}$. We set $\HugeVertices:=\{v\in V(G)\::\:\deg_{G'}(v)\ge
\Omega_{i+1}k\}$.

Observe that by~\eqref{eq:LKSminimalNotManyEdges}, $e(G')<kn$. Let
$(\HugeVertices,\DenseSpots, \Gblack',\Gexp',\smallatoms)$ be the
$(k,\Lambda,\gamma,\epsilon,\nu,\rho)$-bounded decomposition of the graph
$G'-\HugeVertices$ with respect to $\{\smallvertices{\eta/2}{k}{G'},\largevertices{\eta/2}{k}{G'}\setminus \HugeVertices\}$ that
is given by Lemma~\ref{lem:decompositionIntoBlackandExpanding}. Clearly,
$(\HugeVertices, \clusters,\DenseSpots,
\Gblack',\Gexp',\smallatoms)$ is  a
$(k,\Omega_{i+1},\Omega_{i},\Lambda,\gamma,\epsilon,\nu,\rho)$-sparse
decomposition of $G'$ capturing at least as many edges as promised in the
statement of the lemma.
\end{proof}

\bigskip
The process of embedding a given tree $T_\PARAMETERPASSING{T}{thm:main}\in\treeclass{k}$ into $G_\PARAMETERPASSING{T}{thm:main}$ is based on the sparse decomposition
$\class=(\HugeVertices, \clusters,\DenseSpots,\Gblack,\Gexp,\smallatoms)$ of
a graph $G$ from Lemma~\ref{lem:LKSsparseClass} and is much more complex than in
approaches based on the standard Regularity Lemma. The embedding ingredient in the classic (dense)
Regularity Method inheres in Blow-up Lemma type statements which roughly tell
that  regular pairs of positive density in some sense behave like complete bipartite graphs.
In our setting, in addition to regular pairs\footnote{Some of the
regular pairs we shall use are already present in $\Gblack$, and there are some
additional regular pairs hidden in $\DenseSpots$ which we shall extract and make
use of in a form of so-called semiregular matchings
(Definition~\ref{def:semiregular}) in Sections~\ref{sec:augmenting}
and~\ref{sec:LKSStructure}.}  we shall use three other components of $\class$: the vertices of
huge degree $\HugeVertices$, the nowhere-dense graph $\Gexp$, and the avoiding
set $\smallatoms$. Each of these components requires a different strategy for
embedding (parts of) $T_\PARAMETERPASSING{T}{thm:main}$. Let us mention that rather major technicalities arise when combining these strategies; for example, for traversing between $\HugeVertices$ and the rest of the graph we have to introduce a certain ``cleaned'' structure in Lemma~\ref{lem:ConfWhenCXAXB}.

These strategies are described precisely and in detail
in Section~\ref{sec:embed}. A lighter informal account on the role of
$\smallatoms$ is given in Section~\ref{sssec:whyavoiding}. We discuss the use of $\Gexp$ in
Section~\ref{sssec:whyGexp}. Only very little can be said about the
set $\HugeVertices$ at an intuitive level: these vertices have huge degrees but
are very unstructured otherwise. If  only $o(kn)$ edges  are incident
with $\HugeVertices$ then we can neglect them. If, on the other hand, there are  $\Omega(kn)$ edges incident with 
$\HugeVertices$, then we have no choice but to use them for our embedding. Very roughly speaking, in that case we
find sets $\HugeVertices'\subset\HugeVertices$ and $V'\subset V(G)\setminus
\HugeVertices$ such that still $\mindeg(\HugeVertices',V')\gg k$, and
$\mindeg(V',\HugeVertices')=\Omega(k)$, and then use $\HugeVertices'$ and $V'$ in our
embedding.

\medskip
Last, let us note that when $G_\PARAMETERPASSING{T}{thm:main}$ is close to the extremal graph (depicted in Figure~\ref{fig:ExtremalGraph}) then all the structure in $G_\PARAMETERPASSING{T}{thm:main}$ captured by Lemma~\ref{lem:LKSsparseClass} accumulates in the cluster graph $\Gblack'$, i.e., $\HugeVertices$, $\Gexp'$ and $\smallatoms$ are all almost empty. For that reason, when some of $\HugeVertices$, $\Gexp'$ or $\smallatoms$ is substantial we gain some extra aid. In comparison, one of the almost extremal graphs for the Erd\H{o}s-S\'os Conjecture~\ref{conj:ES} has a substantial $\HugeVertices$-component (see Figure~\ref{fig:ExtremalGraphES2}).

\subsection{Decomposition of general graphs}
A version of Lemma~\ref{lem:LKSsparseClass} can be formulated for general graphs. To illustrate this, we present below a generic lemma of this type, which will not be used in the present paper. 

\begin{lemma}\label{lem:genericBD}
For every $\eta, \Lambda,\gamma,\epsilon,\rho>0$ there are $\nu>0$ and $k_0\in\mathbb N$ such that
for every   sequence $(\Omega_j)_{j\in\mathbb N}$ of positive numbers with
$\Omega_j/\Omega_{j+1}\le \eta^2/100$ the following holds. Suppose that $G$ is a graph of order $n$ with average degree $k>k_0$. Then there is an index $i\leq 100\eta^{-2}$, such that 
$G$ has a
 $(k,\Omega_{i+1},\Omega_{i},\Lambda,\gamma,\epsilon,\nu,\rho)$-sparse decomposition
$(\HugeVertices, \clusters,\DenseSpots,
\Gblack,\Gexp,\smallatoms)$ 
that captures all but at most
$(\eta+\frac{4\epsilon}{\gamma}+\epsilon\Omega_{\lfloor
100\eta^{-2}\rfloor}+\gamma+\rho )kn$ edges.
\end{lemma}
The proof follows the same strategy as that of Lemma~\ref{lem:LKSsparseClass}.
\begin{proof}[Proof outline]
First we apply a non-LKS-specific version of Lemma~\ref{prop:gap}. Such a lemma says that for each $G$ with average degree $k$ there exists a spanning sugraph $G'$ of $G$ with $e(G)-e(G')<\eta kn$, and an index $i\leq 100\eta^{-2}$ such that the assertion of Lemma~\ref{prop:gap}\eqref{item:(prop:gap)gap} is fulfilled. The proof of such a lemma follows the same lines as that of Lemma~\ref{prop:gap}. Using the notation of that lemma, we partition $V(G)$ into sets $X_i(G)$, and find an index $i$ such that~\eqref{eq:sparselyconnectedpair} holds. We then keep erasing edges using the rule~(T1). We do not apply the LKS-specific rule~(T2). The bound on the total number of erased edges  holds in this version as well (actually, only the bound~\eqref{kanteninE1} is needed).

The bounded-degree part can then be decomposed using Lemma~\ref{lem:decompositionIntoBlackandExpanding}, yielding the desired sparse decomposition.
\end{proof}
This decomposition could be used to attack other problems; probably with a version of Lemma~\ref{lem:genericBD} tailored to a particular setting similarly as we did in Lemma~\ref{lem:LKSsparseClass}.\footnote{We are not sure whether the property of Lemma~\ref{lem:LKSsparseClass}\eqref{it:LKSsparseGreyCapt} --- which gives a fine bound on the number of some specific type of uncaptured edges --- is a general feature required, or a specific requirement in our approach.} However, our feeling is that such a decomposition lemma is limited in applications to tree-containment problems. The reason is that two of the features of the sparse decomposition, the nowhere-dense graph $\Gexp$ and the avoiding set $\smallatoms$, seem to be useful only for embedding trees. See Section~\ref{sssec:whyavoiding} and Section~\ref{sssec:whyGexp} for a discussion of the respective embedding strategies.

\subsection{The role of the avoiding set $\smallatoms$}\label{sssec:whyavoiding}
Let us explain the role of the avoiding set $\smallatoms$
in Lemma~\ref{lem:decompositionIntoBlackandExpanding}. As said above, our aim in
Lemma~\ref{lem:decompositionIntoBlackandExpanding} will be to locally regularize
parts of the input graph $G$. Of course, first we try to regularize as large a part of the
$G$ as possible. The avoiding set arises as a result of the
impossibility to regularize certain parts of the graph. Indeed, it is one of the most surprising steps in our proof of
Theorem~\ref{thm:main} that the set $\smallatoms$ is initially defined as --
very loosely speaking -- ``those vertices where the Regularity Lemma fails to
work properly'', and only then we prove\footnote{See the last step of the
proof of Lemma~\ref{lem:decompositionIntoBlackandExpanding}.} that $\smallatoms$
actually satisfies the useful conditions of Definition~\ref{def:avoiding}.

We now sketch how to utilize avoiding sets for the purpose of embedding trees.
 In our proof of Theorem~\ref{thm:main} we preprocess the tree
$T=T_\PARAMETERPASSING{T}{thm:main}\in\treeclass{k}$ by considering its $(\tau k)$-fine partition, and then
sequentially embed its shrubs (and knags). Thus embedding techniques for
embedding a single shrub are the building blocks of our embedding machinery; and $\smallatoms$ is one of the enviroments which
provides us with such a technique. Let us discuss here the simpler case of end
shrubs. More precisely, we show how to extend a partial embedding of a tree by one end-shrub. To this end, let us suppose that $\phi$ is a partial embedding of a tree $T$, and $v\in V(T)$ is its \emph{active vertex}\index{general}{active vertex}, i.e., a vertex which is embedded, but not all its children are. We write $ U\subset V(G)$ for the current image of $\phi$. Let
$T'\subset T$ be an end-shrub which is not embedded yet, and suppose $u\in V(T')$ is adjacent to $v$. We have $v(T')\le\tau k$. 

We now show how to
extend the partial embedding $\phi$ to $T'$, assuming that
$\deg_G\big(\phi(v),\smallatoms\setminus U\big)\ge \gamma k$ for some
$(1,\epsilon,\gamma,k)$-avoiding set $\smallatoms$ (where $\tau\ll
\epsilon\ll \gamma\ll 1$). Let $X$ be the set of at most $\epsilon k$
exceptional vertices from Definition~\ref{def:avoiding} corresponding to the set
$U$. We now embed $T'$ into $G$, starting by embedding $u$ in a vertex of
$\smallatoms\setminus (U\cup X)$ in the neighborhood of $\phi(v)$. By Definition~\ref{def:avoiding}, there is a dense spot
$D=(A_D,B_D; F)\in\DenseSpots$ such that $\phi(u)\in V(D)$
and $|U\cap V(D)|\le \gamma^2k$. 
As $D$ is a dense spot, we have $\deg_G(\phi(u),V(D))>\gamma k$.
It is now easy to embed $T'$ into $D$ using the minimum degree in $D$. See
Figure~\ref{fig:EmbeddingAvoiding} for an illustration, and Lemma~\ref{lem:embed:avoidingFOREST} for a
precise formulation.
\begin{figure}[ht] \centering
\includegraphics[scale=1.00]{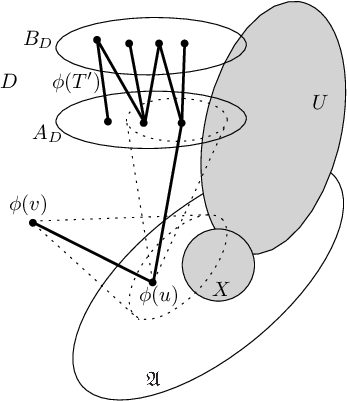}
\caption[Embedding using the set $\smallatoms$]{Embedding using the set $\smallatoms$.}
\label{fig:EmbeddingAvoiding}
\end{figure}

We indeed use the avoiding set for embedding shrubs of a fine partition of
$T$ as above. The major simplification we made in the exposition is that we
only discussed the case when $T'$ is an end shrub. To cover embedding of an internal
shrub $T'$ as well, one needs to have a more detailed control over the embedding, i.e., one must be able to extend the embedding from leaves of $T'$ to the neighboring cut-vertices of the fine partition, is such a way that one can then continue embedding of the shrubs below these cut-vertices. 

Last, let us remark, that unlike our baby-example above, we use an
$(\Lambda,\epsilon,\gamma,k)$-avoiding set with $\Lambda\gg 1$. This is because
in the actual proof one has to avoid more vertices than just the current image
of the embedding.

\subsection{The role of the nowhere-dense graph $\Gexp$ and using the $(\tau k)$-fine partition}\label{sssec:whyGexp}

In this section we shall give some intuition on how the $(\gamma k, \gamma)$-nowhere-dense graph $\Gexp$ from the
$(k,\Omega^{**},\Omega^*,\Lambda,\gamma,\epsilon',\nu,\rho)$-sparse
decomposition\footnote{We shall assume that $17\sqrt{\gamma}<\rho$; this will be
the setting of the sparse decomposition we shall work with in the proof of
Theorem~\ref{thm:main}.} $(\HugeVertices, \clusters,\DenseSpots, \Gblack,
\Gexp,\smallatoms)$ of a graph $G$ is useful for embedding a
given tree $T\in\treeclass{k}$. We start out with the rather simple case when $T$ is a path. We then  point out an issue
with this approach for trees with many branching vertices and show how to overcome this problem using the $(\tau k)$-fine partition from  Lemma~\ref{lem:TreePartition}. 

\paragraph{Embedding a path in $\Gexp$.} 
Assume we are given a path $T=u_1u_2\cdots u_k\in\treeclass{k}$ and we wish to embed it into $\Gexp$.
The naive idea is to apply a one-step look-ahead strategy. We first embed  $u_1$ in an arbitrary vertex $v\in V(\Gexp)$. Then, we extend our
embedding $\phi_\ell$ of the path $u_1\cdots u_\ell$ in $\Gexp$ in 
step $\ell$
 by em\-bedding $u_{\ell +1}$ in a (yet unused) neighbour $w$ of the image of the 
 \emph{active} vertex~$u_\ell$, requiring that
\begin{equation}\label{eq:inductiveembedding3}
\deg_{\Gexp}\big(w,\phi_\ell(u_1\cdots u_\ell)\big)<\sqrt\gamma k\;.
\end{equation}
Let us argue that such a vertex $w$ exists. First, observe that
Property~\ref{defBC:nowheredense} of Definition~\ref{bclassdef} implies that  $\phi_\ell(u_\ell)$ has  at least $\rho k$ neighbours.
By~\eqref{eq:inductiveembedding3} applied to $\ell -1$, at most $\sqrt\gamma k$ of these neighbours
lie inside $\phi_\ell(u_1\cdots u_{\ell-1})$; this property is also trivially satisfied when $\ell=1$. 
Further, an easy calculation shows that at most $16\sqrt\gamma k$ of them have degree more than $\sqrt\gamma k$ in $\Gexp$ into the set $\phi_\ell(u_1\cdots u_\ell)$, otherwise we would get a contradiction to $\Gexp$ being $(\gamma k, \gamma)$-nowhere-dense. Since we assumed $\rho> 17\sqrt\gamma$ we can find a vertex $w$ as desired and thus embed all of~$T$.


\paragraph{Embedding trees with many branching points and the role of fine partitions.}
We certainly cannot hope that a nonempty graph $\Gexp$ alone will provide us
with embeddings of all trees $T\in\treeclass{k}$ from Theorem~\ref{thm:main}. For instance, if $T$ is a star, then we need in $G$ a vertex of degree $k-1$, which $\Gexp$ might not have. In order to run into a problem with the method described above, we do not even need to have such a large degree in our tree $T$.

Consider a binary tree $T\in\treeclass{k}$, rooted
at its central vertex $r$. Now if we try to embed~$T$
 sequentially as above we will arrive at a moment
when there are many (as many as $k/2$) active vertices; regardless in which order we embed. Now, the
neighbourhoods of the images of the active vertices cannot be controlled much, i.e., they may be  intersecting considerably. Hence, embedding  children of active vertices we might block available space in the neighbourhoods of other active
vertices. 
 See
Figure~\ref{fig:gettingstuck} for an illustration.
\begin{figure}[t]
\centering 
\includegraphics{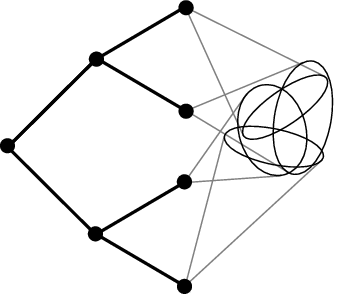}
\caption[Getting stuck while embedding binary tree]{Embedded part of the binary tree in bold. 
The neighbourhoods of active vertices may overlap.}\label{fig:gettingstuck}
\end{figure}

To rescue the situation we use the $(\tau k)$-fine partition $(W_A,W_B,\shrubA,\shrubB)$ of
$T$ (for some $0<\tau\ll \gamma$) given by Lemma~\ref{lem:TreePartition}.
Recall the structure of this partition, as shown in
Figure~\ref{fig:BinaryTreeCut}: the first $q$ levels of $T$ from the root $r$
comprise the sole knag. All other vertices
make up the end shrubs $T^*_1,\ldots,T^*_h$. 

We first embed the knag, which consists of
the cut vertices $W_A\cup W_B$, and so has size at most $O(\frac{1}{\tau})$. As $\rho k$ will be much larger than that, following a strategy similar to the one above we ensure that all of $W_A\cup W_B$ gets correctly embedded, we even have a (limited) choice for its images.
The next step is to make the transitions
 at the
 $q$-th level from embedding cut vertices $W_A\cup W_B$ to embedding shrubs $T^*_1,\ldots,T^*_h$. But since this step requires to exploit the structure of LKS graphs, we skip the details in the high-level overview here. We just remark that one needs to put the cut vertices $W_A\cup W_B$ 
 in the sets $\XA$ and $\XB$ from Lemma~\ref{prop:LKSstruct}; these vertices are powerful enough to allow such a transition. 
 
For the point we wish to make here, it is more relevant to see how to complete the last part of our embedding, that is, how to  embed a tree $T^*_i$ whose root $r_i$
 is already embedded in a vertex $\phi(r_i)\in  V(\Gexp)$.
Let $\text{im}_i:=\text{im}(\phi)$ be the current (partial) image of $\phi$ at this stage. 
We emphasize that at this moment we are working exclusively with the tree $T^*_i$, i.e., any other tree
$T^*_j$ is either completely embedded, or will be embedded only after we finish
the embedding of $T^*_i$. Suppose we are about to embed a vertex $v\in V(T^*_i)$ whose ancestor $v'\in
V(T^*_i)$ is already embedded in $V(\Gexp)$. We choose for the image of $v$ any (yet unused) vertex $w$ in the neighbourhood of $\varphi (v')$, requiring that
\begin{equation}\label{eq:wheretoembedinbinary}
\deg_{\Gexp}(w,\text{im}_i)<\rho k/100.
\end{equation}
This condition is very similar to our path-embedding procedure above, and can be proved in exactly the same way, using the fact that $\Gexp$ is $(\gamma k, \gamma)$-nowhere-dense. Note that during our embedding
$|\text{im}(\phi)\setminus \text{im}_i |$ will grow, but however is at most $v(T^*_i)\le \tau k$. Thus, for every vertex $v''\in
V(T^*_i)$, when its time comes to be embedded, we still have
$\deg_{\Gexp}\big(\phi(v'),\text{im}(\phi)\big)\le \rho k/100+\tau k<\rho
k/99$, and thus $v''$ can be embedded.

Note that the trick here was to keep on working on one subtree $T^*_i$, whose size is small enough to be negligible in comparison to the degree of a vertex in $\Gexp$ so that it does not matter that the set we wish to avoid having a considerable degree into ($\text{im}(\phi)$) is not the same as the one we can actually avoid having a considerable degree into ($\text{im}_i$). (Observe that since $\text{im}(\phi)$ keeps changing during the procedure, we cannot have direct control over it.)
Thus, breaking up the tree into tiny shrubs in the $(\tau k)$-fine partition was the key to successfully embedding it in this case.

\subsection{Proof of
Lemma~\ref{lem:decompositionIntoBlackandExpanding}}\label{ssec:ProofOfDecomposition}
This subsection is devoted to the proof of
Lemma~\ref{lem:decompositionIntoBlackandExpanding}. We give an overview of our
decomposition procedure. We start by extracting the edges of as many $(\gamma
k,k)$-dense spots from $G$ as possible; these together with the incident vertices will form the auxiliary graph $\GD$. Most of the
remaining edges will form the edge set of the graph $\Gexp$. Next, we consider
the intersections of the dense spots captured in $\GD$. To the subgraph of $\GD$
that is spanned by the large intersections we apply the Regularity Lemma for locally dense graphs
(Lemma~\ref{lem:sparseRL}), and thus obtain $\Gblack$. The other part of
$V(\GD)$ will be taken as the  $(\Lambda,\epsilon,\gamma,k)$-avoiding
set~$\smallatoms$.

\smallskip

\paragraph{Setting up the parameters.} We start by setting 
$$\tilde\nu:=\epsilon\cdot 3^{-\frac{\Omega\Lambda}{\gamma^3}}.$$
Let $q_\mathrm{MAXCL}$ be  given by Lemma~\ref{lem:sparseRL} for input parameters 
\begin{equation}\label{eq:in213}
m_\PARAMETERPASSING{L}{lem:sparseRL}:=\frac{\Omega}{\gamma\tilde\nu}\quad, \quad z_\PARAMETERPASSING{L}{lem:sparseRL}:=4s\quad \mbox{ and } \quad
\epsilon_\PARAMETERPASSING{L}{lem:sparseRL}:=\epsilon\;.
\end{equation}
 Define
an auxiliary parameter
$q:=\max\{q_\mathrm{MAXCL}, \epsilon^{-1}\}$ and choose the output parameters of
Lemma~\ref{lem:decompositionIntoBlackandExpanding} as $$k_0:=\left\lceil\frac{
q_\mathrm{MAXCL}}{\tilde\nu}\right\rceil \ \    \ \ \text{ and } \ \    \ \
\nu:=\frac{\tilde \nu}q .$$

\paragraph{Defining $\DenseSpots$ and $\Gexp$.}
Given a graph $G$, take a set $\DenseSpots$ of edge-disjoint
$(\gamma k, \gamma)$-dense spots such that the resulting
graph $\GD\subset G$ (which contains those vertices and edges that are
contained in~$\bigcup\DenseSpots$) has a maximal number  of edges. 

Then by
Lemma~\ref{lem:subgraphswithlargeminimumdegree} there exists a graph
$\Gexp\subset G- \GD$ with $\mindeg(\Gexp)> \rho k$ and such that
\begin{equation}\label{eq:almostalldashed}
|E(G)\setminus (E(\Gexp)\cup  E(\GD))|\le\rho kn\;.
\end{equation} 
This choice of $\DenseSpots$ and $\Gexp$ already satisfies Properties~\ref{defBC:densepairs} and~\ref{defBC:nowheredense}  of Definition~\ref{bclassdef}.


\paragraph{Preparing for an application of the Regularity Lemma.}
Let  
$$
 \mathcal X:= \bigboxplus_D \{ U, W, V(G)\setminus V(D)\}\;,
$$
where the partition refinement ranges over all $D=(U,W;F)\in\DenseSpots$.
Let $\mathcal B:=\{X\in\mathcal X\::\: X\subset V(\GD)\}$, $\tilde{\mathcal B}:=\{B\in\mathcal{B}\::\:|B|>2\tilde\nu k \}$, and 
$\tilde{\mathcal C}:=\mathcal B\setminus\tilde{\mathcal B}$. Furthermore let $\tilde B:=\bigcup_{B\in\tilde{\mathcal B}}B$ and  $\smallatoms:=\bigcup_{C\in\tilde{\mathcal C}}C$.  Let $V_{\leadsto \smallatoms}:=\{v\in V(G)\::\:\deg(v,\smallatoms)>b\}$.

Now,  partition each set $B\in\tilde{ \mathcal B}$  into
$c_B:=\lceil|B|/2\tilde\nu k\rceil$ sets $B_1,\ldots,B_{c_B}$ of
cardinalities differing by at most one, and let $\mathcal B'$ be the set
containing all the sets  $B_i$ (for all $B\in
\tilde{\mathcal B}$).
Then for each $B\in\mathcal B'$ we have that
\begin{equation}\label{eq:ClustersOfRightSize}
\tilde\nu k\le |B| \le 2 \tilde\nu k\le \epsilon k\;.
\end{equation}

Construct a graph $H$ on 
$\mathcal B'$ by making two vertices $A_1,A_2\in \mathcal
B'$ adjacent in $H$ if 
\begin{enumerate}[(A)]
\item\label{it:VV} there is a dense spot $D=(U,W; F)\in \DenseSpots$ such that  $A_1\subset U$ and $A_2\subset W$, and
\item\label{item:density} $\density_{G}(A_1,A_2)\ge \gamma$.
\end{enumerate}
Note that it follows from the way $\DenseSpots$ was chosen that if $A_1A_2\in E(H)$ then $G[A_1,A_2]=\GD[A_1,A_2]$. But on the other hand note that we do not necessarily have $G[A_1,A_2]=D[A_1,A_2]$ for the dense spot $D$ appearing in~\eqref{it:VV}; just because there may be several such dense spots $D$.

By assumption of Lemma~\ref{lem:decompositionIntoBlackandExpanding}, $\maxdeg(G)\le\Omega k$. So, for each
$B\in\mathcal B'$ we have
$e_G(B,\tilde B\setminus B)\le
\Omega k|B|$. On the other hand,
\eqref{eq:ClustersOfRightSize} and~\eqref{item:density}
imply that $\gamma\tilde\nu k|B|\deg_H(B)\le
e_G(B,\tilde B\setminus B)$.
We conclude that 
\begin{equation}\label{maxundmoritz}
\maxdeg(H)\le
\frac{\Omega}{\gamma\tilde\nu }=m_\PARAMETERPASSING{L}{lem:sparseRL}\;.
\end{equation} 

\paragraph{Regularising the dense spots in $\tilde B$.} 
We use Lemma~\ref{lem:sparseRL} with parameters $m_\PARAMETERPASSING{L}{lem:sparseRL},z_\PARAMETERPASSING{L}{lem:sparseRL}$ and $\epsilon_\PARAMETERPASSING{L}{lem:sparseRL}$ as defined by~\eqref{eq:in213}  on the graphs
$H_\PARAMETERPASSING{L}{lem:sparseRL}:=\GD$ and
$F_\PARAMETERPASSING{L}{lem:sparseRL}:=H$, together with the ensemble
$\mathcal{B}'$ in the role of the sets $W_i$, and partition of $V(\GD)$ induced by
 $$\mathcal Z_\PARAMETERPASSING{L}{lem:sparseRL}:=\mathcal V\boxplus \big\{V(\Gexp), V(G)\sm V(\Gexp)\big\}\boxplus\big\{V_{\leadsto \smallatoms},V(G)\setminus V_{\leadsto \smallatoms}\big\}\;.$$
Observe that $\mathcal B'$
is an $(\tilde\nu k)$-ensemble satisfying condition~\eqref{lem:sparseRL(item)samesize} of Lemma~\ref{lem:sparseRL},
by~\eqref{eq:ClustersOfRightSize}, by the choice of $k_0$, and by~\eqref{maxundmoritz}.
We thus  obtain integers $\{p_A\}_{A\in \mathcal{B}'}$ and a family
$\clusters=\{W^{(1)}_A,\ldots,W^{(p_A)}_A\}_{A\in\mathcal B'}$ and a set
$W_0:=\bigcup_{A\in\mathcal B'} W_A^{(0)}$ such that in particular we have the
following.
\begin{enumerate}[(I)]
\item We have $\epsilon^{-1}\le p_A\le q_\mathrm{MAXCL}$ for
all $A\in \mathcal{B}'$.\label{aaaaa}
\item We have $|W_{A}^{(a)}|=|W_{B}^{(b)}|$
for any $A,B\in \mathcal{B}'$ and for any $a\in [p_{A}]$, $b\in [p_B]$.\label{bbbbb}
\item \label{ccH}
For any $A\in\mathcal B'$ and any $a\in [p_A]$, there is $V\in\V$ such that $W_A^{(a)}\subset V$. We either have that $W_A^{(a)}\subset V(\Gexp)$, or $W_A^{(a)}\cap V(\Gexp)=\emptyset$ and $W_A^{(a)}\subset V_{\leadsto \smallatoms}$, or $W_A^{(a)}\cap V_{\leadsto \smallatoms}=\emptyset$.
\item \label{eq:boundTotalIrregularity}
$\sum_{e\in E(H)}|\mathrm{irreg}(e)|\le \epsilon\sum_{AB\in
E(H)}|A||B|$, where $\mathrm{irreg}(AB)$ is the set of all
edges of the graph $G$ contained in an $\epsilon$-irregular
pair $(W^{(a)}_A,W^{(b)}_B)$, with $a\in[p_A]$, $b\in[p_B]$, $AB\in E(H)$.
\end{enumerate}

Let $\Gblack$ be obtained from $\GD$ by erasing all vertices in $W_0$, and all edges that lie in pairs $(W^{(a)}_A,W^{(b)}_B)$ which are
irregular or of density at most $\gamma^2$. Then
Properties~\ref{defBC:clusters},~\ref{defBC:RL},~\ref{defBC:dveapul}
and~\ref{defBC:prepartition} of Definition~\ref{bclassdef} are satisfied.
Further, Lemma~\ref{notmuchlost} implies~\eqref{eq:sEr}.

Note that Properties~\eqref{aaaaa}, ~\eqref{bbbbb} and~\eqref{eq:ClustersOfRightSize}
imply that for all $A\in
\mathcal{B}'$ and for any $a\in [p_{A}]$ we have that
\[
\epsilon k\ge |A|\geq |W_{A}^{(a)}|\geq \frac{\tilde\nu k}{q_\mathrm{MAXCL}} \geq \frac{\tilde\nu k}{q}  =\nu k.
\]
Thus also Property~\ref{Csize} of Definition~\ref{bclassdef} holds.

 Furthermore,
by~\eqref{eq:almostalldashed} and~\eqref{eq:sEr},  the number of edges that are not captured
by $( \clusters,\DenseSpots, \Gblack, \Gexp,\smallatoms )$ is at most $(\frac{4\epsilon}{\gamma}+\epsilon\Omega+\gamma+\rho)kn$.

So, it only remains to see Property~\ref{defBC:avoiding} of Definition~\ref{bclassdef}.

\paragraph{The avoiding property of $\smallatoms$.}  
 In order to see
Property~\ref{defBC:avoiding} of Definition~\ref{bclassdef}, we have to show that $\smallatoms$ is $(\Lambda,\epsilon,\gamma,k)$-avoiding with respect to $\DenseSpots$. For this, let $\bar U\subset V(G)$ be such that  $|\bar U|\le \Lambda k$. Let $X$ be the set of those 
vertices $v\in\smallatoms$ that are not contained in any
dense spot $D\in\mathcal D$ for which $|\bar U\cap
V(D)|\le\gamma^2k$. Our aim is to see that $|X|\le \epsilon k$.

Let
$\mathcal{D}_X\subset\DenseSpots$ be the set of all dense
spots $D$ with $X\cap V(D)\neq\emptyset$. 
Setting $\mathcal{A}:=\{A\in \mathcal{\tilde C} :A\cap X\neq\emptyset \}$,
the definition of $\smallatoms$ trivially implies that
$\frac{|X|}{2\tilde\nu k}\le|\mathcal{A}|$. Now, by
the definition of $\mathcal B$, we know that there are at most
$3^{|\mathcal{D}_X|}$ sets $A\in\mathcal A$.
Indeed, for each $D=(U,W;F)\in\mathcal{D}_X$,
 either  $A$ is a subset of $U$, or of $W$, or of
$V(G)\setminus V(D)$. Thus,
\begin{equation}\label{eq_moleculesdetermineatom}
3^{|\mathcal{D}_X|}\geq |\mathcal A|\geq 
\frac{|X|}{\tilde\nu k}\;.
\end{equation}

By Fact~\ref{fact:boundedlymanyspots}, each vertex of $V(G)$ lies in at most $\Omega/\gamma$
of the $(\gamma k,\gamma)$-dense spots from $\DenseSpots$.
Hence
 $$\frac{\Omega}{\gamma}|\bar U|\ge\sum_{D\in\mathcal{D}_X}|V(D)\cap\bar U|\ge |\mathcal{D}_X|\gamma^2 k\overset{\eqref{eq_moleculesdetermineatom}}{\ge}\log_3\left(\frac{|X |}{\tilde\nu k}\right)\gamma^2 k\;,$$ where the second inequality holds by the definition of~$X$. Thus
$$|X|\leq 3^{\frac{\Omega\Lambda}{\gamma^3}}\cdot \tilde\nu k=\epsilon k\;,$$
 as desired.
 This finishes the proof of
Lemma~\ref{lem:decompositionIntoBlackandExpanding}.

\begin{remark}\label{rem:ambiguity}
The bounded decomposition given by
Lemma~\ref{lem:decompositionIntoBlackandExpanding} is not uniquely determined,
and can actually vary vastly. This is caused by the arbitrariness in the choice of
 the dense spots from which we obtain the cluster graph
$\Gblack$.

This situation is an acute contrast with the situation of decomposition of
dense graphs (which is given by the Szemer\'edi Regularity Lemma).  Indeed, in
the dense setting the structure of the cluster graph is essentially
unique, cf.~\cite{AlShSt:DistinctyRegularityPartitions}.\footnote{The setting needs to be somewhat strengthened as otherwise there are counterexamples to uniqueness; compare Theorem~1 and Theorem~2
in~\cite{AlShSt:DistinctyRegularityPartitions}. However morally this is true because of the uniqueness of graph limits~\cite{borgs-2008}.}

Of course, the ambiguity of the bounded
decomposition of $G$ propagates to Lemma~\ref{lem:LKSsparseClass}. We will have
to deal with implications of this ambiguity in Section~\ref{sec:LKSStructure}.
\end{remark}

\subsection{Lemma~\ref{lem:decompositionIntoBlackandExpanding} algorithmically}\label{sssec:DecomposeAlgorithmically}
Let us look back at the proof of
Lemma~\ref{lem:decompositionIntoBlackandExpanding} and see that we can get a
bounded decomposition of any bounded-degree graph
algorithmically in quasipolynomial time (in the order of the graph). Note that
this in turn provides efficiently a sparse classification of any graph since the
initial step of splitting the graph into huge degree vertices and bounded degree
(cf.~Lemma~\ref{prop:gap}) can be done in polynomial time.

There are only two steps in the proof of
Lemma~\ref{lem:decompositionIntoBlackandExpanding} which need to
be done algorithmically: the extraction of dense spots, and the simultaneous regularization of some dense pairs. 

It will be more convenient to work with a relaxation of the notion of dense spots. We call a graph $H$
\emph{$(d,\ell)$-thick}\index{general}{thick graph} if $v(H)\ge \ell$, and
$e(H)\ge d v(H)^2$. Thick graphs are a relaxation of dense spots, where the
minimum degree condition is replaced by imposing a lower bound on the order, and
the bipartiteness requirement is dropped. It can be verified that in our
proof it is not important  that the dense spots $\DenseSpots$ and the nowhere-dense graph
$\Gexp$ are parametrized by the same constants, i.e., the entire proof would go through even if the spots in $\DenseSpots$ were $(\gamma k,\gamma)$-dense, and $\Gexp$ was
$(\beta k,\beta)$-nowhere-dense for some $\beta\gg \gamma$. Each $(\beta
k,\beta)$-thick graph gives (algorithmically) a $(\beta k/4,\beta/4)$-dense
spot, and thus it is enough to extract thick graphs.

For the extraction of thick graphs we would need to efficiently answer the
following: Given a number $\beta>0$ find a number $\gamma>0$ such that for 
an input number $h$ and an $N$-vertex graph we can localize in $G$ a $(\gamma,h)$-thick graph if it contains a $(\beta,h)$-thick graph, or output NO otherwise.\footnote{We could additionally assume that $\maxdeg(G)\le O(h)$ due to the previous step of removing the set $\HugeVertices$ of huge degree vertices.}
Employing techniques from
a deep paper of Arora, Frieze and Kaplan~\cite{ArFrKa02}, one can solve this problem in quasipolynomial time $O(N^{c\cdot\log N})$. This was communicated to us by Maxim Sviridenko. On the negative side, a truly polynomial algorithm seems to be out of reach as Alon, Arora, Manokaran, Moshovitz, and Weinstein~\cite{Alonetal:Inapproximability} reduced the problem to the notorious hidden clique problem whose tractability has been open for twenty years.
\begin{theorem}[Alon et al.~\cite{Alonetal:Inapproximability}]\label{thm:hardness}
If there is no polynomial time algorithm for solving the clique problem for a planted clique of size $n^{1/3}$ then for any $\epsilon\in(0,1)$ and $\delta>0$ there is no polynomial time algorithm that distinguishes between a graph $G$ on $N$ vertices containing a clique of size $\kappa=N^{\epsilon}$ and a graph $G'$ on $N$ vertices in which the densest subgraph on $\kappa$ vertices has density at most $\delta$.\footnote{The result as stated in~\cite{Alonetal:Inapproximability} covers only the range $\epsilon\in(\frac13,1)$. However there is a simple reduction by taking many disjoint copies of the general range to the restricted one.}
\end{theorem}
Of course, Theorem~\ref{thm:hardness} leaves some hope for a polynomial time algorithm when $h=N^{o(1)}$ (which corresponds to $k_\PARAMETERPASSING{L}{lem:decompositionIntoBlackandExpanding}=n_\PARAMETERPASSING{L}{lem:decompositionIntoBlackandExpanding}^{o(1)}$).

\medskip

The regularity lemma can be made algorithmic~\cite{Alon94thealgorithmic}.
The algorithm from~\cite{Alon94thealgorithmic} is based on index pumping-up, and
thus applies even to the locally dense setting of Lemma~\ref{lem:sparseRL}.

\medskip

It will turn out that the extraction of dense spots is the only obstruction to a polynomial time algorithm for Theorem~\ref{thm:main}. In
Section~\ref{ssec:algorithmic} we sketch a truly polynomial time algorithm which avoids this step. It seems that the method sketched there is generally applicable for problems which employ sparse classifications.

\section{Augmenting a matching}\label{sec:augmenting}
In previous papers~\cite{AKS95,Z07+,PS07+,Cooley08,HlaPig:LKSdenseExact}
concerning the LKS~Conjecture in the dense setting the crucial turn was to find
a matching in the cluster graph of the host graph possessing certain properties.
We will prove a similar ``structural result'' in Section~\ref{sec:LKSStructure}.
In the present section, we prove the main tool for
Section~\ref{sec:LKSStructure}, namely Lemma~\ref{lem:Separate}. All  preceding
statements are only preparatory. The only exception is (the easy) Lemma~\ref{lem:edgesEmanatingFromDensePairsIII} which is recycled later, in Section~\ref{sec:configurations}.

\subsection{Dense spots and semiregular matchings}
We need two definitions  concerning graphs covered by dense spots.

\begin{definition}[\bf $(m,\gamma)$-dense
cover]\index{general}{dense cover}
 A \emph{$(m,\gamma)$-dense cover} of a
graph $G$ is a family $\DenseSpots$ of edge-disjoint
$(m,\gamma)$-dense
spots such that $E(G)=\bigcup_{D\in\DenseSpots}E(D)$.
\end{definition}

\begin{definition}[$\mathcal G(n,k,\Omega,\rho,\nu, \tau)$ and $\bar{\mathcal G}(n,k,\Omega,\rho,\nu)$]\label{tupelclass}
We define $\mathcal G(n,k,\Omega,\rho,\nu, \tau)$\index{mathsymbols}{*G@$\mathcal G(n,k,\Omega,\rho,\nu, \tau)$} to be the class of all tuples $(G,\DenseSpots,H,\mathcal A)$ with the following properties:
\begin{enumerate}[(i)]
\item  $G$ is a graph of order $n$ with $\maxdeg(G)\le \Omega k$,\label{maxroach}
\item $H$ is a bipartite subgraph of
$G$ with colour classes $A_H$ and $B_H$ and with $e(H)\ge \tau kn$,\label{duke}
\item  $\DenseSpots$ is a $(\rho k, \rho)$-dense cover of $G$,
\item $\mathcal A$ is a $(\nu k)$-ensemble in $G$,
and $A_H\subseteq \bigcup \mathcal A$,\label{bird}
\item  $A\cap U\in\{\emptyset,A\}$ for each $A\in\mathcal A$ and for each $D=(U,W;F)\in\DenseSpots$.\label{mingus}
\end{enumerate}
Those $G$, $\mathcal D$ and $\mathcal A$ for which all conditions but~\eqref{duke} and the last part of~\eqref{bird} hold will make up the triples $(G,\DenseSpots,\mathcal A)$ of the class  $\bar{\mathcal G}(n,k,\Omega,\rho,\nu)$\index{mathsymbols}{*G@$\bar{\mathcal G}(n,k,\Omega,\rho,\nu)$}.
\end{definition}

We now prove our first auxiliary lemma on our way towards
Lemma~\ref{lem:Separate}.

\begin{lemma}\label{lem:edgesEmanatingFromDensePairsII}
For every $\Omega\in\mathbb N$ and $\epsilon,\rho, \tau>0$ 
there is a number $\alpha>0$ such that for every $\nu\in(0,1)$
there exists a number $k_0\in\mathbb N$ such
that for each $k>k_0$ the following holds. 

For every $(G,\DenseSpots,H,\mathcal A)\in\mathcal
G(n,k,\Omega,\rho,\nu,\tau)$ there are
$(U,W;F)\in\DenseSpots$, $A\in\mathcal A$
and $X,Y\subseteq V(G)$ such that 
\begin{enumerate}[1)]
  \item $|X|=|Y|>\alpha\nu k$,
  \item  $X\subset A\cap U\cap A_H$ and $Y\subset W\cap B_H$, where $A_H$ and
  $B_H$ are the colour classes of $H$, and
  \item $(X,Y)$ is an $\epsilon$-regular pair in $G$ of density
  $\density(X,Y)\ge \frac{\tau\rho}{4\Omega}$.
\end{enumerate}
\end{lemma}

\begin{proof}
Let $\Omega$, $\epsilon$, $\rho$ and $\tau$ be given. Applying
Lemma~\ref{lem:RL} to
$\epsilon_\PARAMETERPASSING{L}{lem:RL}:=\min\{\epsilon
,\frac{\rho^2}{8\Omega}\}$
and $\ell_\PARAMETERPASSING{L}{lem:RL}:=2$, we obtain numbers $n_0$ and $M$.
We set 
\begin{equation}\label{allllpha}
\alpha:=\frac{\tau\rho}{\Omega^2M},
\end{equation}
 and given $\nu\in(0,1)$, we set
$$k_0:=\frac{2n_0}{\alpha\nu M}.$$
Now suppose we are given  $k>k_0$ and $(G,\DenseSpots,H,\mathcal A)\in\mathcal G(n,k,\Omega,\rho,\nu,\tau)$.

Property~\eqref{maxroach} of Definition~\ref{tupelclass} gives that $e(G)\le \Omega
k n/2$, and Property~\eqref{duke} says  that $e(H)\ge
\tau kn$. So $e(H)/e(G)\geq 2\tau/ \Omega$. Averaging, we find a dense spot $D=(U,W;F)\in\DenseSpots$ such that
\begin{align}\label{eq:proporcialniDensity}
e_{D}(A_H,B_H)= |F\cap E(H)|\geq \frac{e(H)}{e(G)}|F|\ge
\frac{2\tau |F|}{\Omega}\;.
 \end{align}

Without loss of generality, we assume that
\begin{equation}\label{eq:WLOGUW}
e_{D}(U\cap A_H,W\cap B_H)\geq \frac12\cdot e_{D}(A_H,B_H) \ge e_{D}(U\cap B_H,W\cap A_H)\;,
\end{equation}
as otherwise one can just interchange the roles of $U$ and $W$.
Then,
\begin{align}
e_G(U\cap A_H,W\cap B_H)&\geBy{\eqref{eq:WLOGUW}} \frac12\cdot e_{D}(A_H,B_H) \geBy{\eqref{eq:proporcialniDensity}}\frac{\tau}{\Omega}\cdot |F|.\label{colchon}
\end{align}

Let $\mathcal A'\subset\mathcal A$ denote the set of
those $A\in\mathcal A$ with
$0< e_G(A\cap U\cap
A_H,W\cap B_H)< \frac{\tau}{\Omega}\cdot |F|\cdot \frac{|A|}{|U|}$. 
Note that for each $A\in\A'$ we have $A\subset U$ by Definition~\ref{tupelclass}~\eqref{mingus}. Therefore,
\begin{equation*}
e_G\left(\bigcup \mathcal A'\cap U\cap A_H,W\cap
B_H\right)
< \ \frac \tau\Omega \cdot|F|\cdot \frac{|\A'|}{|U|}
\leq \ \frac \tau\Omega \cdot|F|
\overset{\eqref{colchon}}\leq \ e_G(U\cap A_H,W\cap B_H)\;.
\end{equation*}

As $\mathcal A$ covers $A_H$, $G$ has an edge $xy$ with $x\in U\cap A_H\cap A$ for some $A\in\mathcal A\setminus \mathcal A'$ and $y\in W\cap
B_H$.  Set $X':=A\cap U\cap A_H=A\cap A_H$ and $Y':=W\cap B_H$. Then directly
from the definition of $\mathcal A'$ and since $D$ is a $(\rho k,\rho)$-dense spot, 
we obtain that
\begin{equation}
\label{eq:denX'Y'}
\density_G(X',Y')= \frac{e_G(X',Y')}{|X'||Y'|}  \ge
\frac{\frac{\tau}{\Omega}\cdot |F|\cdot \frac{|A|}{|U|}}{|A||W|}  >
\ \frac{\tau\rho}{\Omega}.
\end{equation}

Also,
since $(U,W;F)\in\DenseSpots$, we have
\begin{equation}\label{eq:denseIndMany}
|F|\ge \rho k |U|\;.
\end{equation}
This enables us to bound the size of $X'$ as follows.
\begin{align}
\begin{split}
\label{eq:sizPreX'}
|X'| &\geq
\frac{e_G(X',Y')}{\maxdeg{(G)}}
 \\[6pt]
\JUSTIFY{as $A\not\in\mathcal A'$ and by D\ref{tupelclass}\eqref{maxroach}}&  \ge
 \frac{\frac{\tau}{\Omega}\cdot \frac{|F|}{|U|}\cdot |A|}{\Omega k} \\[6pt]
\JUSTIFY{by \eqref{eq:denseIndMany}}& 
\ge \frac{\tau\cdot \rho k\cdot |A|}{\Omega^2 k} \\
&\ge \frac{\tau\rho \nu k}{\Omega^2 }\\
&\eqByRef{allllpha} \alpha \nu kM\;.
\end{split}
\end{align}

In the same way we see  that
\begin{align}
\label{eq:sizPreY'}
|Y'|
 \geq \ \alpha\nu kM\;.
\end{align}

Applying Lemma~\ref{lem:RL} to $G[X',Y']$ with prepartition
$\{X',Y'\}$ we obtain a collection of sets $\mathcal
C=\{C_i\}_{i=0}^p$, with $p<M$. By~\eqref{eq:sizPreX'},
and~\eqref{eq:sizPreY'}, we have that $|C_i|\ge \alpha\nu k$
for every $i\in[p]$. It is easy to
deduce from~\eqref{eq:denX'Y'} that
there is at least one $\epsilon_\PARAMETERPASSING{L}{lem:RL}$-regular (and thus $\epsilon$-regular) pair $(X,Y)$, $X,Y\in\mathcal
C\setminus\{C_0\}$, $X\subset X'$, $Y\subset Y'$ with
$\density(X,Y)\ge \frac{\tau\rho}{4\Omega}$. Indeed, it suffices to count the number of edges incident with $C_0$, lying in $\epsilon_{\mathrm L\ref{lem:RL}}$-irregular pairs or
belonging to too sparse pairs. These are strictly less than
$$(\epsilon_\PARAMETERPASSING{L}{lem:RL} +
\epsilon_\PARAMETERPASSING{L}{lem:RL} + \frac{\rho^2}{4\Omega}
)|X||Y| \leq\frac{\rho^2}{2\Omega} |X||Y|
\overset{\eqref{eq:denX'Y'}}\leq e(X',Y')$$ many, and thus not
all edges between $X'$ and $Y'$. This finishes the
proof of Lemma~\ref{lem:edgesEmanatingFromDensePairsII}.
\end{proof}

Instead of just one pair $(X,Y)$, as it is given by 
Lemma~\ref{lem:edgesEmanatingFromDensePairsII}, we shall later need several disjoint pairs. This motivates the following definition.

\begin{definition}[\bf $(\epsilon,d,\ell)$-semiregular matching]\label{def:semiregular} A collection
$\mathcal N$ of pairs $(A,B)$ with $A,B\subset V(H)$ is called an
\index{general}{semiregular matching}\emph{$(\epsilon,d,\ell)$-semiregular matching} of a
graph $H$ if \begin{enumerate}[(i)] \item $|A|=|B|\ge \ell$ for each
$(A,B)\in\mathcal N$, \item $(A,B)$ induces in $H$ an $\epsilon$-regular pair of
density at least $d$, for each $(A,B)\in\mathcal N$, and \item all involved sets
$A$ and $B$ are  pairwise disjoint.
\end{enumerate}
Sometimes, when the parameters do not matter (as for instance in Definition~\ref{altPath} below) we write lazily \emph{semiregular matching}.
\end{definition}

For a semiregular matching $\mathcal N$, we shall write 
\index{mathsymbols}{*V1@$\V_1(\M)$, $\V_2(\M)$, $\V(\M)$}$\V_1(\mathcal
N):=\{A\::\:(A,B)\in\mathcal N\}$, $\V_2(\mathcal N):=\{B\::\:(A,B)\in\mathcal N\}$
and $\V(\mathcal N):=\mathcal V_1(\mathcal N)\cup \mathcal V_2(\mathcal N)$. 
Furthermore, we set 
\index{mathsymbols}{*V1@$V_1(\mathcal M)$,
$V_2(\mathcal M)$, $V(\mathcal M)$} 
$V_1(\mathcal N):=\bigcup \mathcal V_1(\mathcal N)$,
$V_2(\mathcal N):=\bigcup \mathcal V_2(\mathcal N)$ and $V(\mathcal N):=V_1(\mathcal
N)\cup V_2(\mathcal N)= \bigcup \mathcal V(\mathcal N)$. 
 As these definitions suggest, the orientations of
the pairs $(A,B)\in\mathcal N$ are important. The sets $A$ and $B$ are called \index{mathsymbols}{*VERTEX@$\M$-vertex}\index{general}{vertex@$\M$-vertex}\emph{$\mathcal N$-vertices} and the pair $(A,B)$ is a \index{mathsymbols}{*EDGE@$\M$-edge}\index{general}{edge@$\M$-edge}\emph{$\mathcal N$-edge}.

We say that a semiregular matching $\mathcal N$
\index{general}{absorb}\emph{absorbes} a semiregular matching $\mathcal M$ if for every $(S,T)\in\mathcal M$ there exists $(X,Y)\in\mathcal N$ such that $S\subset X$
and $T\subset Y$. In the same way, we say that a family of dense spots $\mathcal D$
\index{general}{absorb}\emph{absorbes} a semiregular matching $\mathcal M$ if for every $(S,T)\in\mathcal M$ there exists $(U,W;F)\in\mathcal D$ such that $S\subset U$
and $T\subset W$.

We  later need
the following easy bound on the size of the elements  of
$\mathcal V(\mathcal M)$.

\begin{fact}\label{fact:boundMatchingClusters}
Suppose that $\mathcal M$ is an
$(\epsilon,d,\ell)$-semiregular
matching in a graph $H$. Then $|C|\le
\frac{\maxdeg{(H)}}{d}$ for each $C\in
\mathcal V(\mathcal M)$.
\end{fact}
\begin{proof}
Let for example $(C,D)\in\mathcal M$. The maximum degree of $H$ is at least as large as the average
degree of the vertices in $D$, which is at least
$d |C|$.
\end{proof}

The next lemma,
Lemma~\ref{lem:edgesEmanatingFromDensePairsIII}, is a second step towards Lemma~\ref{lem:Separate}. Whereas
Lemma~\ref{lem:edgesEmanatingFromDensePairsII} gives
one dense regular pair, in the same setting 
Lemma~\ref{lem:edgesEmanatingFromDensePairsIII} provides us with
a dense semiregular
matching.

\begin{lemma}\label{lem:edgesEmanatingFromDensePairsIII}
For every $\Omega\in\mathbb N$ and $\rho,\epsilon,\tau \in(0,1)$ there exists
$\alpha>0$ such that for every $\nu\in (0,1)$ there is a number $k_0\in\mathbb N$
such that the following holds for every $k>k_0$. 

 For each $(G,\DenseSpots,H,\mathcal A)\in\mathcal G(n,k,\Omega,\rho,\nu,\tau)$ there exists an
$(\epsilon ,\frac{\tau\rho}{8\Omega},\alpha\nu k)$-semiregular matching $\mathcal M$ of $G$ such that
\begin{enumerate}[(1)]
  \item for each $(X,Y)\in\mathcal M$ there are $A\in\mathcal A$, and
  $D=(U,W;F)\in \DenseSpots$ such that  $X\subset U\cap A\cap A_H$ and $Y\subset
  W\cap B_H$,\label{bedingung1}
  and
  \item $|V(\mathcal M)|\ge\frac{\tau}{2\Omega} n$.\label{bedingung3}
\end{enumerate}
\end{lemma}

\begin{proof}
Let
$\alpha:=\alpha_\PARAMETERPASSING{L}{lem:edgesEmanatingFromDensePairsII}>0$ be given by Lemma~\ref{lem:edgesEmanatingFromDensePairsII}
  for the  input parameters
  $\Omega_\PARAMETERPASSING{L}{lem:edgesEmanatingFromDensePairsII}:=\Omega$,
  $\epsilon_\PARAMETERPASSING{L}{lem:edgesEmanatingFromDensePairsII}:=\epsilon$,
  $\tau_\PARAMETERPASSING{L}{lem:edgesEmanatingFromDensePairsII}:=\tau/2$
  and $\rho_\PARAMETERPASSING{L}{lem:edgesEmanatingFromDensePairsII}:=\rho$.
    Now, for $\nu_\PARAMETERPASSING{L}{lem:edgesEmanatingFromDensePairsII}:=\nu$,
  Lemma~\ref{lem:edgesEmanatingFromDensePairsII}  yields a number $k_0\in\mathbb
  N$.

Now let  $(G,\DenseSpots,H,\mathcal A)\in\mathcal G(n,k,\Omega,\rho,\nu,\tau)$. Let $\mathcal M$ be an inclusion-maximal
$(\epsilon \rho,\frac{\tau\rho}{8\Omega},\alpha\nu k)$-semiregular matching with property~\eqref{bedingung1}. 
We claim that
\begin{equation}\label{bedingung2}
  e_G(A_H\setminus V_1(\mathcal M),B_H\setminus V_2(\mathcal M))<\frac\tau 2 kn.
\end{equation}
Indeed, suppose otherwise. Then 
the bipartite subgraph $H'$ of $G$ induced by the sets $A_H\setminus
V_1(\mathcal M)=A_H\setminus V(\mathcal M)$ and $B_H\setminus V_2(\mathcal M)=B_H\setminus V(\mathcal M)$
satisfies Property~\eqref{duke} of Definition~\ref{tupelclass}, with
$\tau_\PARAMETERPASSING{D}{tupelclass}:=\tau/2$. So, we have that $(G, \DenseSpots, H', \mathcal A)\in \mathcal G(n,k,\Omega,\rho,\nu,\tau/2)$.

Thus
 Lemma~\ref{lem:edgesEmanatingFromDensePairsII} for $(G, \DenseSpots, H',
 \mathcal A)$ yields a dense spot $D=(U,W;F)\in\DenseSpots$ and a set
 $A\in\mathcal A$, together with two sets $X\subset U\cap A\cap (A_H\setminus 
 V(\mathcal M))$, $Y\subset W\cap (B_H\setminus  V(\mathcal M))$ such that
 $|X|=|Y|>\alpha_\PARAMETERPASSING{L}{lem:edgesEmanatingFromDensePairsII}\nu
 k=\alpha\nu k$, and such that $(X,Y)$ is $\epsilon_\PARAMETERPASSING{L}{lem:edgesEmanatingFromDensePairsII}$-regular and has density at least $$\frac{\tau_\PARAMETERPASSING{L}{lem:edgesEmanatingFromDensePairsII}\rho_\PARAMETERPASSING{L}{lem:edgesEmanatingFromDensePairsII}}{4\Omega_\PARAMETERPASSING{L}{lem:edgesEmanatingFromDensePairsII}}=\frac{\tau\rho}{8\Omega}.$$ As this contradicts the maximality of $\mathcal M$, we have shown~\eqref{bedingung2}.

In order to see~\eqref{bedingung3}, it suffices to observe that by~\eqref{bedingung2} and by Property~\eqref{duke} of Definition~\ref{tupelclass}, the set $V(\mathcal M)$ is incident with at least $\tau  kn-\frac\tau 2 kn=\frac\tau 2 kn$ edges. By Definition~\ref{tupelclass}~\eqref{maxroach}, it follows that $|V(\mathcal M)|\geq \frac\tau 2 kn\cdot \frac 1{\Omega k}\geq \frac{\tau}{2\Omega}n$, as desired.
\end{proof}

\subsection{Augmenting paths for matchings}

We now prove the main lemma of Section~\ref{sec:augmenting}, namely Lemma~\ref{lem:Separate}. We will use an augmenting path technique for our semiregular matchings, similar to the augmenting paths commonly used for traditional matching theorems. For this, we need the following definitions.

\begin{definition}[\bf Alternating path, augmenting
path\index{general}{alternating path}\index{general}{augmenting path}]\label{altPath} Given an $n$-vertex graph $G$,
and a semi\-regular matching $\M$, we  call a sequence $\mathfrak S=(Y_0,\A_1,Y_1, \A_2, Y_2,\ldots ,\A_h ,Y_h)$ ($h\ge 0$) an {\em $(\delta,s)$-alternating path for $\M$ from $Y_0$} if for all $i\in[h]$ we have 
 \begin{enumerate}[(i)]
\item $\A_i\subseteq \mathcal V_1(\mathcal M)$ and the sets $\mathcal A_i$ are pairwise disjoint, \label{alt1e}
\item  $Y_0\subseteq V(G) \setminus V(\mathcal M)$ and 
$Y_i=\bigcup_{(A,B)\in\M, A\in\A_i}B$,\label{alte2}
\item  $|Y_{i-1}|\geq\delta n$, and\label{alte3}
\item $e(A, Y_{i-1})\geq s\cdot |A|$,  for each $A\in\mathcal A_i$.\label{alte4}
\suspend{enumerate}
If in addition there is a set $\mathcal C$ of disjoint subsets of
$V(G)\setminus (Y_0\cup V(\M))$ such that 
\resume{enumerate}[{[(i)]}]
\item  $e(\bigcup \mathcal C,
Y_h)\geq t\cdot n$,\label{augm5}
\end{enumerate}
then we say that $\mathfrak S'=(Y_0,\A_1,Y_1, \A_2, Y_2,\ldots ,\A_h ,Y_h,\mathcal C)$ is an \emph{$(\delta,s,t)$-augmenting path for $\M$ from $Y_0$ to $\mathcal C$}.

The number $h$ is called the \index{general}{length of alternating path} the \emph{length} of $\mathfrak{S}$ (or of $\mathfrak{S'}$).
\end{definition}

Next, we show that a semiregular matching either has an augmenting path or admits a partition into two parts so that there are only few edges which cross these parts in a certain way.

\begin{lemma}\label{lem:augmentingPath}
 Given an $n$-vertex graph $G$ with $\maxdeg(G)\le \Omega
k$, a number $\tau\in (0,1)$, a semiregular matching $\M$, a set $Y_0\subseteq V(G)\setminus V(\M)$, and  a set $\C$ of disjoint subsets of $V(G)\setminus (V(\M)\cup Y_0)$, one of the following holds:
  \begin{itemize}
  \item[{\bf(M1)}] 
    There is a semiregular matching $\mathcal
  M''\subseteq \mathcal M$ with $e\left(\bigcup \mathcal C\cup
  V_1(\mathcal M\setminus\mathcal M''),Y_0\cup  V_2(\mathcal M'')\right)<\tau nk$,
    \item[{\bf(M2)}] 
$\M$ has an $(\frac\tau{2\Omega}, \frac{\tau^2}{8\Omega}k,\frac{\tau^2}{16\Omega}
k)$-augmenting path of length at most $2\Omega/\tau$ from $Y_0$ to $\mathcal C$.
    \end{itemize}
\end{lemma}
\begin{proof}
If $|Y_0|\le \frac{\tau}{2\Omega} n$ then {\bf(M1)} is satisfied for $\M'':=\emptyset$. Let us therefore assume otherwise.
 
Choose a $(\frac\tau{2\Omega},\frac{\tau^2}{8\Omega} k)$-alternating path  $\mathfrak
S=(Y_0,\A_1,Y_1, \A_2,$ $Y_2,\ldots,\A_h, Y_h)$ for $\mathcal M$
 with $|\bigcup_{\ell=1}^h\mathcal A_\ell|$ maximal.

 Now, let $\ell^*\in \{0,1,\ldots,h\}$ be maximal with $|Y_{\ell^*}|\geq \frac\tau{2\Omega} n$. Then $\ell^*\in\{h,h-1\}$. Moreover,
 as $|Y_\ell|\geq\frac\tau{2\Omega} n$ for all $\ell\leq\ell^*$, we have that $(\ell^*+1)\cdot \frac\tau{2\Omega} n\leq |\bigcup_{\ell\leq\ell^*} Y_\ell| \leq n$ and thus
\begin{equation}\label{ellstarbounded}
 \ell^*+1\leq \frac{2\Omega}\tau.
\end{equation}

Let $\mathcal M''\subseteq \mathcal M$ consist of all $\M$-edges $(A,B)\in \M$ with  $A\in\bigcup_{\ell\in[h]}\A_\ell$. Then, by the choice of~$\mathfrak S$,
\begin{align} \nonumber
 e\left(V_1(\M\setminus \M''), \bigcup_{\ell=0}^{\ell^*}Y_{\ell}\right)&=\sum_{\ell=0}^{\ell^*}e\left(V_1(\M\setminus \M''), Y_{\ell}\right)\\ \label{endofthepath}
&  < (\ell^*+1)\cdot \frac{\tau^2}{8\Omega} k \cdot |V_1(\M\setminus \M'')|  \leByRef{ellstarbounded} \frac\tau 4 kn.
\end{align}
Furthermore, if $\ell^*=h -1$ (that is, if $|Y_{h}|<\frac\tau{2\Omega} n$) then  
\begin{equation}\label{deltanisnothing}
e\left(V_1(\M\setminus \M'')\cup\bigcup \mathcal C,Y_h\right)\ <\  \frac\tau{2\Omega} n\cdot\maxdeg(G)\ \leq \ \frac\tau{2\Omega}\Omega kn \ =\ \frac\tau 2kn.
\end{equation}

So, regardless whether $h=\ell^*$ or $h=\ell^*+1$, we get from~\eqref{endofthepath} and~\eqref{deltanisnothing} that
\[
e\left(V_1(\M\setminus \M'')\cup\bigcup \mathcal C, Y_0\cup V_2(\M'')\right)<
\frac 34\tau kn + e\left(\bigcup \mathcal C, \bigcup_{\ell=0}^{\ell^*}Y_{\ell}\right).
\]

Thus, if $e(\bigcup \mathcal C,\bigcup_{\ell=0}^{\ell^*}Y_\ell)\leq \frac\tau 4kn$,  we see that~$\mathbf{(M1)}$  satisfied for $\M''$. So, assume otherwise. 
Then, by~\eqref{ellstarbounded},
there is an index $j\in\{0,1,\ldots,\ell^*\}$ so that
\begin{equation*}\label{CYedges}
 e\left(\bigcup \mathcal C,Y_j\right)
 \ > \ \frac{\tau^2}{16\Omega}kn,
\end{equation*}
and thus, $(Y_0,\A_1,Y_1, \A_2,$ $Y_2,\ldots,\A_h, Y_h,\C)$ is an $(\frac\tau{2\Omega}, \frac{\tau^2}{8\Omega}
k, \frac{\tau^2}{16\Omega} k)$-augmenting path for $\mathcal M$.
This shows {\bf(M2)}.
\end{proof}

Building on Lemma~\ref{lem:edgesEmanatingFromDensePairsIII} and Lemma~\ref{lem:augmentingPath}
we prove the following.

\begin{lemma}\label{lem:AugmentORSeparate}
For every
$\Omega\in \mathbb N$ and  $\tau\in (0,\frac{1}{2\Omega})$ there is a number $\tau'\in (0,\tau)$ such that for every $\rho\in(0,1)$ there is a number
$\alpha\in (0,\tau'/2)$ such that for every $\epsilon\in (0,\alpha)$ there is a number
$\pi>0$ 
 such that for every $\gamma>0$
there is $k_0\in\mathbb N$ such
that the following holds for every $k>k_0$ and every $h\in (\gamma k,k/2)$.

Let $G$ be a graph  of order $n$ with
$\maxdeg(G)\le \Omega k$, with an $(\eps^3,\rho,h)$-semiregular matching~$\M$
 and with a $(\rho k, \rho)$-dense cover $\DenseSpots$ that  absorbs $\M$. Let $Y\subset
V(G)\setminus  V(\mathcal M)$, and let
$\mathcal C$ be an $h$-ensemble in $G$ outside $V(\mathcal M)\cup Y$.
Assume that $U\cap C\in\{\emptyset,C\}$ for each $D=(U,W;F)\in\DenseSpots$ and each $C\in\mathcal C\cup \mathcal V_1(\mathcal M)$. 

Then one of the following holds.
\begin{enumerate}[(I)]
  \item\label{it:AugmAss1} There is a semiregular matching $\mathcal
  M''\subseteq \mathcal M$ such that $$e\left(\bigcup \mathcal C\cup
  V_1(\mathcal M\setminus\mathcal M''),Y\cup  V_2(\mathcal M'')\right)<\tau nk.$$
\item\label{it:AugmAss2} There is an $(\epsilon,\alpha ,\pi h)$-semiregular
matching $\mathcal M'$ such that
  \begin{itemize}
    \item[{\bf(C1)}] $|V(\mathcal M)\setminus V(\mathcal M')|\le
    \epsilon n$, and $
    |V(\mathcal M')|\geq  |V(\mathcal M)| + \frac{\tau'}2 n$, and
    \item[{\bf(C2)}] for each $(T,Q)\in\mathcal M'$ there are sets $C_1\in\mathcal V_1(\mathcal
 M)\cup \mathcal C$, $C_2\in \mathcal V_2(\mathcal  M)\cup\{Y\}$ and a dense
 spot $D=(U,W;F)\in\DenseSpots$ such that $T\subset C_1\cap U$ and $Q\subset
C_2\cap W$.
  \end{itemize}
\end{enumerate}
\end{lemma}
\begin{proof}
We divide the proof into five
steps.

\paragraph{Step 1: Setting up the parameters.}
Suppose that $\Omega$ and $\tau$ are given. 
For
$\ell=0,1,\ldots ,\lceil2\Omega/\tau\rceil$, we define the auxiliary para\-meters

\begin{equation}\label{auroraenpekin}
\tau^{(\ell)}:= \left(\frac{\tau^2}{32\Omega}\right)^{\lceil \frac{2\Omega}\tau\rceil -\ell +2}\;,
\end{equation}
and set $$\tau':=\frac{\tau^{(0)}}{2\Omega}.$$

Given $\rho$, we define 
$$ \alpha:=\frac{\tau'\rho}{16\Omega}.$$
Then, given $\epsilon$, 
for $\ell=0,1,\ldots ,\lceil2\Omega/\tau\rceil$, we define the further auxiliary parameters
$$\mu^{(\ell)}:= \alpha_\PARAMETERPASSING{L}{lem:edgesEmanatingFromDensePairsIII}\big(\Omega,\rho,\eps^3, \tau^{(\ell )}\big)$$
 which are given by
Lemma~\ref{lem:edgesEmanatingFromDensePairsIII} for
input parameters $\Omega_\PARAMETERPASSING{L}{lem:edgesEmanatingFromDensePairsIII}:=\Omega$, $\rho_\PARAMETERPASSING{L}{lem:edgesEmanatingFromDensePairsIII}:=\rho$, 
$\epsilon_\PARAMETERPASSING{L}{lem:edgesEmanatingFromDensePairsIII}:=\epsilon ^3$, and $\tau_\PARAMETERPASSING{L}{lem:edgesEmanatingFromDensePairsIII}:=\tau^{(\ell )}$.
Set
$$\pi:=\frac{\epsilon}{2}  \cdot \min\left\{\mu^{(\ell)}\::\:\ell=0,\ldots,\lceil 2\Omega/\tau\rceil\right\}\;,$$

Given the next input parameter $\gamma$, 
Lemma~\ref{lem:edgesEmanatingFromDensePairsIII} for parameters as above and the final input $\nu_\PARAMETERPASSING{L}{lem:edgesEmanatingFromDensePairsIII}:=\gamma$ yields $k_{0_\PARAMETERPASSING{L}{lem:edgesEmanatingFromDensePairsIII}}=:k_0^{(\ell)}$.
Set
$$k_0:=\max\left\{k_0^{(\ell)}\::\:\ell=0,\ldots,\lceil 2\Omega/\tau\rceil \right\}.$$

\paragraph{Step 2: Finding an augmenting path.}
We apply Lemma~\ref{lem:augmentingPath} to $G$, $\tau$, $\mathcal M$, $Y$ and $\mathcal C$. Since~$\mathbf{(M1)}$ corresponds to~\eqref{it:AugmAss1}, let us assume that the outcome of the lemma is $\mathbf{(M2)}$. Then there is a $(\frac\tau{2\Omega}, \frac{\tau^2}{8\Omega}k,\frac{\tau^2}{16\Omega} k)$-augmenting path $\mathfrak S'=(Y_0,\A_1,Y_1, \A_2,$ $Y_2,\ldots,$
$\A_{j^*}, Y_{j^*}, \C)$ for $\mathcal M$ starting from $Y_0:=Y$
such that $j^* \leq 2\Omega/\tau$.
%
%

Our aim is now to show that~\eqref{it:AugmAss2} holds. 

\paragraph{Step 3: Creating parallel
matchings.}

Inductively, for $\ell=j^*,j^*-1,\dots ,0$ we shall define
auxiliary bipartite induced subgraphs
$H^{(\ell)}\subset G$  with colour classes $P^{(\ell)}$ and $Y_{\ell}$ that satisfy
\begin{enumerate}[(a)]
\item \label{Hell}
$e(H^{(\ell)})\ge\tau^{(\ell)} kn,$
\end{enumerate}
and $(\eps^3,2\alpha,\mu^{(\ell)}h)$-semiregular matchings $\mathcal M^{(\ell)}$ that satisfy
 \begin{enumerate}[(a)]\setcounter{enumi}{1}
 \item $V_1(\mathcal M^{(\ell)})\subseteq P^{(\ell)}$,\label{liegtinP}
  \item for each $(A',B')\in\mathcal M^{(\ell)}$ there are a dense spot $(U,W;F)\in
  \DenseSpots$ and a set $A\in \V_1(\mathcal M)$ (or a set $A\in\mathcal C$ if $\ell=j^*$)  such that  $A'\subset U\cap A$  and $B'\subset W\cap Y_{\ell}$,\label{ribot}
  \item $|V(\mathcal M^{(\ell)})|\ge \frac{\tau^{(\ell)} }{2\Omega}n$, and\label{dasaltealpha}
  \item\label{dasneuef} $|B\cap V_2(\mathcal M^{(\ell)})|=|A\cap P^{(\ell -1)}|$ for each edge $(A,B)\in \mathcal M$, if $\ell>0$.
\end{enumerate}

We take $H^{(j^*)}$ as the induced
bipartite subgraph of $G$ with colour classes  $P^{(j^*)}:=\bigcup \mathcal C$
and $Y_{j^*}$. Definition~\ref{altPath}~\eqref{augm5} together with~\eqref{auroraenpekin} ensures~\eqref{Hell} for $\ell=j^*$.
 Now, for $\ell \leq j^*$, suppose $H^{(\ell)}$ is defined already. Further, if $\ell< j^*$ suppose also that $\mathcal M^{(\ell +1)}$ is defined already. We shall define $\mathcal M^{(\ell)}$, and, if $\ell >0$, we shall also define $H^{(\ell-1)}$.

Observe that  $(G,\DenseSpots,H^{(\ell)},\mathcal A_\ell)\in\mathcal G(n,k,\Omega,\rho,\frac{h}k,\tau^{(\ell)})$, because of~\eqref{Hell} and the assumptions of the lemma. So, 
applying Lemma~\ref{lem:edgesEmanatingFromDensePairsIII} to
  $(G,\DenseSpots,H^{(\ell)},\mathcal A_\ell)$  and noting that $\frac{\tau^{(\ell)}\rho}{8\Omega}\geq 2\alpha$ we obtain
  an
$(\eps^3,2\alpha,\mu^{(\ell)}h)$-semiregular matching $\mathcal
M^{(\ell)}$ that satisfies conditions~\eqref{liegtinP}--\eqref{dasaltealpha}.

If $\ell>0$, we define $H^{(\ell-1)}$ as follows. 
For each $(A,B)\in \mathcal M$ take a set $\tilde A\subset A$ of
cardinality $|\tilde A|=|B\cap V(\mathcal M^{(\ell)})|$ so that 
\begin{equation}\label{schoengross}
e(\tilde A,Y_{\ell-1})\ \ge\ \frac{\tau^2}{8\Omega}k\cdot |\tilde A| \;.
\end{equation}
This is possible by Definition~\ref{altPath}~\eqref{alte4}: just choose those vertices from $A$ for $\tilde A$ that send most edges to $Y_{\ell -1}$.
Let $P^{(\ell-1)}$ be the union of all the sets $\tilde A$.
Then, \eqref{dasneuef} is satisfied. Furthermore,
$$|P^{(\ell-1)}|\ =\ |V_2(\mathcal M^{(\ell)})|
\ \overset{\eqref{dasaltealpha}}\ge \ \frac{\tau^{(\ell)}}{4\Omega}n.$$ 
So, by~\eqref{schoengross},
\begin{equation}\label{KantOrGoethe}
e(P^{(\ell-1)},Y_{\ell-1})\ \ge\ \frac{\tau^2}{8\Omega} k \cdot |P^{(\ell-1)}|\  \geq \ \frac{\tau^2\cdot \tau^{(\ell)}}{32\Omega^2}kn \ \overset{\eqref{auroraenpekin}}= \ \tau^{(\ell-1)}kn\;.
\end{equation}

We let $H^{(\ell -1)}$ be the
bipartite subgraph of $G$ induced by the colour classes $P^{(\ell-1)}$ and $Y_{\ell-1}$. 
 Then~\eqref{KantOrGoethe}
establishes~\eqref{Hell} for $H^{(\ell -1)}$. This finishes step $\ell$.\footnote{Recall that  the matching $\mathcal M^{(\ell -1)}$ is only to be defined in step $\ell-1$.}

\paragraph{Step 4: Harmonising the matchings.}
Our semiregular matchings $\mathcal
M^{(0)},\ldots,\mathcal M^{(j^*)}$ will be a good base for constructing the semiregular matching $\mathcal M'$ we are after. However, we do not know anything about  $|B\cap V_2(\mathcal M^{(\ell)})| - |A\cap V_1(\mathcal M^{(\ell -1)})|$ for the $\M$-edges $(A,B)\in \mathcal M$.
But this term will be crucial in determining how much of $V(\M)$ gets lost when we replace some of its $\M$-edges with $\bigcup \M^{(\ell)}$-edges. For this reason, we refine $\mathcal M^{(\ell)}$ in a way that its $\mathcal M^{(\ell)}$-edges become almost equal-sized.

Formally, we shall inductively construct
semiregular matchings $\mathcal N^{(0)},\ldots, \mathcal
N^{(j^*)}$ such that  for $\ell=0,\ldots ,j^*$ we have
\begin{enumerate}[(A)]
\item \label{eq:NellSemiregular} $\mathcal N^{(\ell)}$ is an
$(\epsilon,\alpha,\pi h)$-semiregular matching,
\item\label{absorbi}$\mathcal M^{(\ell)}$ absorbes $\mathcal N^{(\ell)}$,
\item \label{dasneueF} if $\ell>0$ and $(A,B)\in\mathcal M$ with $A\in\mathcal A_\ell$ then $|A\cap  V(\mathcal N^{(\ell-1)})|\ge |B \cap V(\mathcal N^{(\ell)})| $, and
\item \label{eq:SizeNInduc} $|V_2(\mathcal N^{(\ell)})|\ge |V_1(\mathcal
N^{(\ell-1)})| - \frac{\epsilon}{2}\cdot |V_2(\mathcal M^{(\ell)})|$  if
$\ell>0$ \ and $|V_2(\mathcal N^{(0)})|\ge \frac{\tau^{(0)}}{2\Omega} n = \tau'
n$.
\end{enumerate}

Set $\mathcal N^{(0)}:=\mathcal M^{(0)}$. Clearly~\eqref{absorbi} holds for
$\ell=0$,~\eqref{eq:NellSemiregular} is easy to check, and~\eqref{dasneueF} is void. Finally, Property~\eqref{eq:SizeNInduc} holds because of~\eqref{dasaltealpha}.
Suppose now $\ell>0$ and that we already constructed matchings
$\mathcal N^{(0)},\ldots,\mathcal N^{(\ell-1)}$ satisfying Conditions~\eqref{eq:NellSemiregular}--\eqref{eq:SizeNInduc}.

 Observe that for any $(A,B)\in\mathcal M$ we have that
\begin{equation}\label{thisisenough}
|B\cap  V_2(\mathcal M^{(\ell)})| \
\overset{\eqref{liegtinP}, \eqref{dasneuef}}\geq \ |A\cap V_1(\mathcal M^{(\ell-1)})| \
\ge\ |A\cap V_1(\mathcal N^{(\ell-1)})|,
\end{equation} 
where the last inequality holds because of~\eqref{absorbi} for $\ell -1$.

So, we can choose a subset $X^{(\ell)}\subseteq V_2(\mathcal M^{(\ell)})$ such that 
$|B\cap X^{(\ell)}|=|A\cap V(\mathcal N^{(\ell-1)})|$ for each $(A,B)\in\mathcal M$.
Now, for each $(S,T)\in \mathcal M^{(\ell)}$ write $\widehat{T}:=T\cap X^{(\ell)}$, and choose a subset $\widehat{S}$ of $S$ of size $|\widehat{T}|$. 
 Set
\[
\mathcal{N}^{(\ell)}:=\left\{(\widehat{S},\widehat{T})\::\:(S,T)\in \mathcal
M^{(\ell)}, |\widehat{T}|\geq \frac{\epsilon}{2}\cdot |T| \right\}. 
\]
Then~\eqref{absorbi}  and~\eqref{dasneueF} hold for $\ell$.

For~\eqref{eq:NellSemiregular}, note that Fact~\ref{fact:BigSubpairsInRegularPairs} implies that $\mathcal N^{(\ell)}$ is an $\left(\eps,2\alpha-\eps^3,\frac{\epsilon}{2}\mu^{(\ell)}h\right)$-semiregular matching.

In order to see~\eqref{eq:SizeNInduc}, it suffices to observe that
\begin{align*}
|V_2(\mathcal N^{(\ell)})|
&=\sum_{(\widehat S,\widehat T)\in\mathcal N^{(\ell)}} |\widehat{T}|\\[8pt]
&\geq   | X^{(\ell)}|\  -\ \sum_{(S,T)\in\mathcal M^{(\ell)} } \frac{\epsilon}{2} \cdot |T|\\[8pt]
&\geq  \sum_{(A,B)\in\mathcal M} |A\cap V_1(\mathcal N^{(\ell-1)})|\ -\ \frac{\epsilon}{2} \cdot |V_2(\mathcal M^{(\ell)} )|\\[8pt]
&= |V_1(\mathcal N^{(\ell-1)})|\ -\ \frac{\epsilon }{2} \cdot |V_2(\mathcal M^{(\ell)} )|.
\end{align*}

\paragraph{Step 5: The final matching.}
For each $\ell=1,2,\ldots,j^*$ let $\mathcal L$ denote the set of all $\M$-edges $(A,B)\in\mathcal M$ with  $|A'|>\frac{\epsilon}{2} \cdot |A|$, where $A':=A\setminus
V_1(\mathcal N^{(\ell-1)})$. Further,  for each $(A,B)\in\mathcal M$, choose a set $B'\subset B\setminus V_2(\mathcal N^{(\ell)})$ of cardinality $|A'|$. This is possible by~\eqref{dasneueF}.
Set 
$$\mathcal K:=\{(A', B'):(A,B)\in\mathcal L\}.$$ 
By the assumption of the lemma, for every $(A',B')\in\mathcal K$
there are an edge $(A,B)\in\mathcal M$ and a dense spot $D=(U,W;F)\in\DenseSpots$ such that 
\begin{equation}\label{lacolegiala}
\text{$A'\subset A\subset U$ and
$B'\subset B\subset W$.}
\end{equation}
 Since $\mathcal M$ is $(\eps^3,\rho,h)$-semiregular we have by Fact~\ref{fact:BigSubpairsInRegularPairs} that $\mathcal K$ is a $(\eps,\rho-\eps^3,\frac{\epsilon}{2} h)$-semiregular matching. Set $$\mathcal M':=\mathcal
K\cup\bigcup_{\ell=0}^{j^*}\mathcal N^{(\ell)},$$ now it is easy to check that $\mathcal M'$ is an
$(\epsilon,\alpha,\pi h)$-semiregular matching. Using~\eqref{lacolegiala} together with~\eqref{absorbi} and \eqref{ribot}, we see that  {\bf(C2)} holds for $\mathcal M'$. 

In order to see {\bf(C1)}, we calculate
\begin{align}
\notag
|V(\mathcal M) \setminus V(\mathcal M')|
&\le
\sum_{(A,B)\in\mathcal M\setminus \mathcal L}|A'\cup B'|\ 
+\ 
\sum_{(A,B)\in\mathcal L}\ \sum_{\ell=1}^{j^*} \Big( |A\cap V_1(\mathcal N^{(\ell -1)})| - |B\cap V_2(\mathcal N^{(\ell)})| \Big)\\[10pt]
\notag &\le 
\frac{\epsilon}{2} \cdot  \sum_{(A,B)\in\mathcal M\setminus \mathcal L}|A\cup B| \ 
+ \ 
\sum_{\ell=1}^{j^*} \Big( |V_1(\mathcal N^{(\ell -1)})| - |V_2(\mathcal N^{(\ell)})| \Big) \\[10pt]
\notag &\overset{\eqref{eq:SizeNInduc}}\le
\frac{\epsilon}{2}   n\ 
+\ 
\sum_{\ell=1}^{j^*}\frac{\epsilon }{2} \cdot |V_2(\mathcal M^{(\ell)} )|\\[10pt]
 &\le
\epsilon  n\;.
\end{align}

Using the fact that $V_2(\mathcal N^{(0)})\subseteq V(\mathcal M')\setminus V(\mathcal M)$ the last calculation also implies that
 \begin{align*} 
|V(\mathcal M')|-|V(\mathcal M)|
&\ge |V_2(\mathcal N^{(0)})|-|V(\mathcal M) \setminus V(\mathcal M')|\\[8pt]
&\overset{\eqref{eq:SizeNInduc}}\ge
\tau' n -\epsilon n\\
&>\frac{\tau'}2 n\;,
\end{align*}
since $\eps<\alpha\leq \tau'/2$ by assumption.
\end{proof}

\medskip

Iterating Lemma~\ref{lem:AugmentORSeparate} we prove the main result of the section.

\begin{lemma}\label{lem:Separate}
For every
$\Omega\in \mathbb N$, $\rho\in (0,1/\Omega)$ there exists a number $\beta>0$ such that for
every $\epsilon\in (0,\beta)$, 
there are
$\epsilon',\pi>0$ such that for each $\gamma>0$ there exists $k_0\in\mathbb N$ such that  the following holds for
every $k>k_0$ and $c\in(\gamma k,k/2)$.

Let $G$ be a graph  of order $n$, with $\maxdeg(G)\le \Omega k$. Let 
$\DenseSpots$ be a $(\rho k, \rho)$-dense cover of $G$, and let $\mathcal M$ be
an $(\epsilon',\rho,c)$-semiregular matching that is absorbed by $\DenseSpots$. Let
$\mathcal C$ be a $c$-ensemble in $G$ outside $V(\mathcal M)$. 
Let $Y\subset V(G)\setminus (V(\mathcal M)\cup
 \bigcup\mathcal C)$.
Assume
that for each $(U,W;F)\in\DenseSpots$, and for each $C\in\mathcal V_1(\mathcal M)\cup \mathcal C$ we have that 
 \begin{equation}\label{eq:510ass1}
 U\cap C\in\{\emptyset,C\}\;.
 \end{equation} 

Then there exists an
$(\epsilon,\beta,\pi c)$-semiregular matching $\mathcal
M'$ such that
  \begin{enumerate}[(i)]
    \item\label{it:Sep1} $|V(\mathcal M)\setminus V(\mathcal M')|\le
    \epsilon n$,
    \item\label{it:Sep2} 
 for each $(T,Q)\in\mathcal M'$ there are sets $C_1\in\mathcal V_1(\mathcal
 M)\cup \mathcal C$, $C_2\in \mathcal V_2(\mathcal  M)\cup\{Y\}$ and a dense
 spot $D=(U,W;F)\in\DenseSpots$ such that $T\subset C_1\cap U$ and $Q\subset
C_2\cap W$, and
    \item\label{it:Sep3} $\mathcal M'$ can be partitioned into $\mathcal M_1$
    and $\mathcal M_2$ so that
    $$e\left( (\bigcup \mathcal C\cup V_1(\mathcal M )) \setminus   V_1(\mathcal M_1) \  , \ (Y\cup V_2(\mathcal M)) \setminus V_2(\mathcal M_2) \right)\  < \ \rho kn\;.$$
  \end{enumerate}
\end{lemma}

\begin{proof}
Let $\Omega$ and $\rho$ be given. Let $\tau':=\tau'_\PARAMETERPASSING{L}{lem:AugmentORSeparate} $ be the output given by Lemma~\ref{lem:AugmentORSeparate} for input
parameters $\Omega_\PARAMETERPASSING{L}{lem:AugmentORSeparate} :=\Omega$ and  $\tau_\PARAMETERPASSING{L}{lem:AugmentORSeparate} :=\rho/2$.

Set $\rho^{(0)}:=\rho$, 
set $L:=\lceil 2/\tau'\rceil +1$, and for $\ell\in [L]$, 
 inductively define
$\rho^{(\ell)}$ to be the output $\alpha_\PARAMETERPASSING{L}{lem:AugmentORSeparate} $ given by Lemma~\ref{lem:AugmentORSeparate} for the further input
parameter $\rho_\PARAMETERPASSING{L}{lem:AugmentORSeparate} :=\rho^{(\ell-1)}$ (keeping $\Omega_\PARAMETERPASSING{L}{lem:AugmentORSeparate}=\Omega $ and  $\tau_\PARAMETERPASSING{L}{lem:AugmentORSeparate}=\rho/2$ fixed).
Then $\rho^{(\ell+1)}\leq\rho^{(\ell)}$ for all $\ell$.
 Set
$\beta:=\rho^{(L)}$.

Given $\epsilon<\beta$
 we set
$\epsilon^{(\ell)}:=(\epsilon/2)^{3^{L-\ell}}$ for $\ell\in[L]\cup\{0\}$, and 
set 
$\epsilon':=\epsilon^{(0)}$.
Clearly,
\begin{equation}\label{summederepsilons}
\sum_{\ell=0}^L\epsilon^{(\ell)} \ \leq \ 
 \epsilon.
\end{equation}

Now, for $\ell+1\in[L]$, 
let $\pi^{(\ell)}:=\pi_\PARAMETERPASSING{L}{lem:AugmentORSeparate}$ be given by Lemma~\ref{lem:AugmentORSeparate} for input parameters $\Omega_\PARAMETERPASSING{L}{lem:AugmentORSeparate} :=\Omega$, $\tau_\PARAMETERPASSING{L}{lem:AugmentORSeparate} :=\rho/2$, $\rho_\PARAMETERPASSING{L}{lem:AugmentORSeparate} :=\rho^{(\ell)}$ and $\eps_\PARAMETERPASSING{L}{lem:AugmentORSeparate} :=\epsilon^{(\ell+1)}$.  For $\ell\in[L]\cup\{0\}$, set $\Pi^{(\ell)}:=\frac{\rho}{2\Omega}\prod_{j=0}^{\ell-1}\pi^{(j)}$. Let $\pi:=\Pi^{(L)}$.

Given $\gamma$, let $k_0$ be the maximum of the lower bounds $k_{0_\PARAMETERPASSING{L}{lem:AugmentORSeparate}}$ given by Lemma~\ref{lem:AugmentORSeparate} for input parameters $\Omega_\PARAMETERPASSING{L}{lem:AugmentORSeparate} :=\Omega$, $\tau_\PARAMETERPASSING{L}{lem:AugmentORSeparate} :=\rho/2$, $\rho_\PARAMETERPASSING{L}{lem:AugmentORSeparate} :=\rho^{(\ell -1)}$, $\eps_\PARAMETERPASSING{L}{lem:AugmentORSeparate} :=\epsilon^{(\ell)}$, $\gamma_\PARAMETERPASSING{L}{lem:AugmentORSeparate}:=\gamma\Pi^{(\ell)}$,  for $\ell\in[L]$.

\medskip

Suppose now we are given $G$, $\mathcal D$, $\mathcal C$, $Y$ and $\mathcal M$. Suppose further that $c>\gamma k>\gamma k_0$. 
Let $\ell\in\{0,1,\ldots,L\}$ be maximal such that there is a matching $\mathcal M^{(\ell)}$ with the following properties:

\begin{enumerate}[(a)]
 \item\label{MellowYellow}
 $\mathcal M^{(\ell)}$ is an $( \epsilon^{(\ell)}, \rho^{(\ell)},\Pi^{(\ell)} c)$-semiregular matching,
\item 
 \label{eq:NatImpAl}
$|V(\mathcal M^{(\ell)})|\ge \ell\cdot \frac{\tau'}2 n$,
\item
\label{epsipepsi}
$|V(\mathcal M)\setminus V(\mathcal M^{(\ell)})|\le \sum_{i=0}^\ell\epsilon^{(\ell)} n$, and
    \item\label{dontworryaboutMell} 
 for each $(T,Q)\in\mathcal M^{(\ell)}$ there are sets $C_1\in\mathcal V_1(\mathcal
 M)\cup \mathcal C$, $C_2\in \mathcal V_2(\mathcal  M)\cup\{Y\}$ and a dense
 spot $D=(U,W;F)\in\DenseSpots$ such that $T\subset C_1\cap U$ and $Q\subset
C_2\cap W$.
\end{enumerate}

Observe that such a number $\ell$ exists, as for $\ell=0$ we may take $\mathcal M^{(0)}=\mathcal M$. Also note that $\ell\leq 2/\tau' <L$ because of~\eqref{eq:NatImpAl}.

We now  apply Lemma~\ref{lem:AugmentORSeparate} with input parameters
$\Omega_\PARAMETERPASSING{L}{lem:AugmentORSeparate}:=\Omega$, $\ \tau_\PARAMETERPASSING{L}{lem:AugmentORSeparate}:=\rho/2$, $\ \rho_\PARAMETERPASSING{L}{lem:AugmentORSeparate}:=\rho^{(\ell)}$, $\ \epsilon_\PARAMETERPASSING{L}{lem:AugmentORSeparate}:=\epsilon^{(\ell+1)}<\beta\leq \rho^{(\ell +1)}=\alpha_\PARAMETERPASSING{L}{lem:AugmentORSeparate} $, 
$\ \gamma_\PARAMETERPASSING{L}{lem:AugmentORSeparate}:=\gamma\Pi^{(\ell)}$ to the  graph
$G$ with the $(\rho^{(\ell)} k,\rho^{(\ell)})$-dense cover $\DenseSpots$, the 
$( \epsilon^{(\ell)},  \rho^{(\ell)},\Pi^{(\ell)}c)$-semiregular matching $\mathcal M^{(\ell)}$, the set 
$$\tilde Y:=(Y\cup V_2(\mathcal M))\setminus V_2(\mathcal M^{(\ell)}),$$
and the $(\Pi^{(\ell)}c)$-ensemble
$$\tilde{\mathcal C}:=  \left\{ C\setminus
V(\mathcal M^{(\ell)}) \ : \ C\in \V_1(\mathcal M)\cup \mathcal C,\  |C\setminus
V_1(\mathcal M^{(\ell)})|\ge
\Pi^{(\ell)}c\right\}.$$ 

Lemma~\ref{lem:AugmentORSeparate} yields a semiregular matching  which either corresponds to $\mathcal M''$ as in Assertion~\eqref{it:AugmAss1} or to $\mathcal M'$ as in Assertion~\eqref{it:AugmAss2}.
Note that in the latter case, the matching $\mathcal M'$ actually constitutes an $( \epsilon^{(\ell+1)},  \rho^{(\ell+1)},\Pi^{(\ell+1)}c)$-semiregular matching $\mathcal M^{(\ell+1)}$ fulfilling all the above properties for $\ell+1\leq L$. In fact,~\eqref{eq:NatImpAl} and~\eqref{epsipepsi} hold for $\M^{(\ell+1)}$ because of $\mathbf{(C1)}$, and it is not difficult to deduce~\eqref{dontworryaboutMell} from $\mathbf{(C2)}$ and from~\eqref{dontworryaboutMell} for $\ell$. But this contradicts the choice of $\ell$. We conclude that we obtained a semiregular matching $\mathcal M''\subseteq \mathcal M^{(\ell)}$ as in Assertion~\eqref{it:AugmAss1} of Lemma~\ref{lem:AugmentORSeparate}.

Thus, in other words, $\mathcal M^{(\ell)}$ can be partitioned into $\mathcal M_1$ and $\mathcal M_2$
so that
\begin{equation}\label{eq:LifeIsNice}
e\Big(\bigcup \tilde\C\cup V_1(\mathcal M_2)\ , \ \tilde Y\cup
V_2(\mathcal M_1) \Big)\ <\ \tau_\PARAMETERPASSING{L}{lem:AugmentORSeparate} kn\ =\ \rho kn/2.
\end{equation}
 Set $\mathcal M':=\mathcal M^{(\ell)}$. Then $\mathcal M'$ is $(\epsilon,\beta,\pi c)$-semiregular by~\eqref{MellowYellow}.
Note that Assertion~\eqref{it:Sep1} of the lemma holds by~\eqref{summederepsilons} and by~\eqref{epsipepsi}.  Assertion~\eqref{it:Sep2}  holds because of~\eqref{dontworryaboutMell}. 

Since
$$(Y\cup V_2(\mathcal M)) \setminus V_2(\mathcal M_2) \ \subseteq \ \tilde Y\cup V_2(\mathcal M_1),$$
and because of~\eqref{eq:LifeIsNice} we know that in order to prove Assertion~\eqref{it:Sep3} it suffices to show that
 the set
\begin{align*}
X\ := \ & 
\big((\bigcup \mathcal C\cup V_1(\mathcal M )) \setminus   V_1(\mathcal M_1) \big) \setminus\big(\bigcup\tilde\C\cup V_1(\mathcal M_2)\big)
\\
\ = \ &
\big(\bigcup \mathcal C\cup V_1(\mathcal M)\big) \setminus\big(\bigcup\tilde\C\cup V_1(\mathcal M^{(\ell)})\big)
\end{align*}
 sends at most $ \rho kn/2$ edges to the rest of the graph. 
For this, it would be enough to see that $|X| \leq \frac\rho{2\Omega}n$, as by assumption, $G$ has maximum degree $\Omega k$.

To this end, note that
by assumption, $| \V_1(\mathcal M)\cup \mathcal C|\le \frac{n}{c}$. 
Further,
the definition of $\tilde\C$ implies
that for each $A\in \mathcal C\cup \V_1(\mathcal M)$ we have 
that 
$|A\setminus \big(\bigcup\tilde \C\cup V_1(\mathcal
M^{(\ell)}\big)|\le \Pi^{(\ell)}c$.
Combining these two observations, we obtain that
$$|X| 
< \Pi^{(\ell)}n
\le \frac \rho{2\Omega}n\;,$$ 
as desired.

\end{proof}

\section{Rough structure of LKS graphs}\label{sec:LKSStructure}
In this section we give a structural result for graphs
$G\in\LKSsmallgraphs{n}{k}{\eta}$, stated in Lemma~\ref{prop:LKSstruct}. Similar structural results were essential
also for proving Conjecture~\ref{conj:LKS} in the dense setting
in~\cite{AKS95,PS07+}. There, a certain matching structure was proved to exist
in the cluster graph of the host graph. This matching structure then allowed to
embed a given tree into the host graph.

Naturally, in our possibly sparse setting the sparse decomposition $\class$ of $G$ will enter the picture (instead of just the cluster graph of $G$). There is an important subtlety though: we need to ``re-regularize'' the cluster graph $\BGblack$ of $\class$. The necessity of this step arises from the ambiguity of the sparse decomposition $\class$ given by Lemma~\ref{lem:LKSsparseClass}, see Remark~\ref{rem:ambiguity}. Consequently, the cluster graph $\BGblack$ given by a sparse decomposition $(\HugeVertices, \clusters,\DenseSpots, \Gblack, \Gexp,\smallatoms)$ of $G$ might not be suitable for locating a matching structure in
analogue to the dense setting. In this case, we have to find another
regularization of parts of $G$, partially based on $\Gblack$. Lemma~\ref{lem:Separate} is the main tool to this end. The re-regularization is captured by the semiregular matchings $\M_A$ and $\M_B$. 

Let us note that this step is one
of the biggest differences between our approach and the announced solution of
the Erd\H os--S\'os Conjecture by Ajtai, Koml\'os, Simonovits and Szemer\'edi. In
other words, the nature of the graphs arising in the Erd\H os--S\'os Conjecture
allows a less careful approach with respect to regularization, still yielding a
structure suitable for embedding trees. We discuss the necessity of this step in further detail
in Section~\ref{ssec:whyaugment}, after proving the main result of this section, Lemma~\ref{prop:LKSstruct}, in Section~\ref{sec6first}.

\subsection{Finding the structure}\label{sec6first}

We now introduce some notation we need in order to state Lemma~\ref{prop:LKSstruct}. Suppose that $G$ is a graph with a $(k,\Omega^{**},\Omega^*,\Lambda,\gamma,\epsilon,\nu,\rho)$-sparse decomposition 
$$\class=(\HugeVertices, \clusters,\DenseSpots, \Gblack, \Gexp,\smallatoms
)\;$$ with respect to $\largevertices{\eta}{k}{G}$ and
$\smallvertices{\eta}{k}{G}$.
 Suppose further that  $\mathcal M_A,\mathcal M_B$ are
$(\epsilon',d,\gamma k)$-semiregular matchings  in $\GD$. We then define the triple \index{mathsymbols}{*XA@$\XA(\eta,\class, \mathcal M_A,\mathcal M_B)$} \index{mathsymbols}{*XB@$\XB(\eta,\class, \mathcal M_A,\mathcal M_B)$}
\index{mathsymbols}{*XC@$\XC(\eta,\class, \mathcal M_A,\mathcal
M_B)$}
$(\XA,\XB,\XC)=(\XA,\XB,\XC)(\eta,\class, \mathcal M_A,\mathcal M_B)$
by setting
\begin{align*}
\XA&:=\largevertices{\eta}{k}{G}\setminus V(\mathcal M_B)\;,\\
\XB&:=\left\{v\in V(\mathcal M_B)\cap \largevertices{\eta}{k}{G}\::\:\widehat{\deg}(v)<(1+\eta)\frac
k2\right\}\;,\\
\XC&:=\largevertices{\eta}{k}{G}\setminus(\XA\cup\XB)\;,
\end{align*}
where 
$\widehat{\deg}(v)$ on the second
line is defined by
\begin{equation}\label{eq:defhatdeg}
\widehat{\deg}(v):=\deg_G\big(v,\smallvertices{\eta}{k}{G}\setminus
(V(\Gexp)\cup\smallatoms\cup V(\mathcal M_A\cup\mathcal M_B)\big)\;.
\end{equation}
Clearly, $\{\XA,\XB,\XC\}$ is a partition of $\largevertices{\eta}{k}{G}$.

We now give the main and only lemma of this section, a structural result for graphs from $\LKSsmallgraphs{n}{k}{\eta}$.
\begin{proposition}\label{prop:LKSstruct}
For every $\eta>0, \Omega>0, \gamma\in (0,\eta/3)$ there is
$\beta>0$ so that for every $\epsilon\in(0,\frac{\gamma^2\eta}{12})$ there exist $\epsilon',\pi>0$ such that for every $\nu>0$ there exists $k_0\in \mathbb N$ such that for every $\Omega^*$ with  $\Omega^*< \Omega$ and every $k$ with $k> k_0$  the following holds.

Suppose $\class=(\HugeVertices, \clusters,\DenseSpots, \Gblack, \Gexp,\smallatoms )$
is a $(k,\Omega^{**},\Omega^*,\Lambda,\gamma,\epsilon',\nu,\rho)$-sparse
decomposition of a graph $G\in\LKSsmallgraphs{n}{k}{\eta}$ with respect to  $S:=\smallvertices{\eta}{k}{G}$ and
$L:=\largevertices{\eta}{k}{G}$ which captures all but at most $\eta kn/6$ edges
of $G$. Let $\clustersize$ be the size of the clusters $\clusters$.\footnote{The number $\clustersize$ is irrelevant when $\clusters=\emptyset$. In particular, note that in that case we necessarily have $\M_A=\M_B=\emptyset$ for the semiregular matchings given by the lemma.} Write 
\begin{equation}\label{eq:S0alt}
S^0:=S\setminus \left(V(\Gexp)\cup\smallatoms\right)\;.
\end{equation}

Then $\GD$ contains  two
$(\epsilon,\beta,\pi \clustersize)$-semiregular matchings
$\mathcal M_A$ and $\mathcal M_B$ such
that for the triple $(\XA,\XB,\XC):=(\XA,\XB,\XC)(\eta,\class, \mathcal M_A,\mathcal M_B)$
we have
\begin{enumerate}[(a)]
  \item \label{prop6.1a}
$V(\mathcal M_A)\cap V(\mathcal
M_B)=\emptyset$,
\item \label{eq:lastminute} $V_1(\M_B)\subset S^0$,
\item \label{eq:Mspots} for each $(T,Q)\in\M_A\cup\M_B$, there is a dense spot  $(A_D,B_D; E_D)\in \DenseSpots$ with $T\subset
A_D$, $Q\subset B_D$, and furthermore, either $T\subset S$ or $T\subset L$, and
$Q\subset S$ or $Q\subset L$,
\item \label{eq:M1} for each $X_1\in\V_1(\M_A\cup\M_B)$ there exists a cluster $C_1\in \clusters$ such that $X_1\subset C_1$, and for each $X_2\in\V_2(\M_A\cup\M_B)$ there exists $C_2\in \clusters\cup\{L\cap \smallatoms\}$ such that $X_2\subset C_2$,
\item \label{fewfewfew} $e_{\Gcapt}\big(\XA,S^0\setminus V(\mathcal M_A)\big)\le
\gamma kn$,
\item\label{newpropertyS6} $e_{\Gblack}(V(G)\setminus
V(\M_A\cup \M_B))\le \epsilon \Omega^* kn$,
\item\label{nicDoNAtom} for the semiregular matching $\NAtom:=\{(X,Y)\in\M_A\cup\M_B\::\: (X\cup Y)\cap\smallatoms\not=\emptyset\}$ we have $e_{\Gblack}\big(V(G)\setminus
V(\M_A\cup \M_B),V(\NAtom)\big)\le \epsilon \Omega^* kn$,
\item\label{Mgoodisblack} for
$\Mgood:=\{(A,B)\in \mathcal M_A\::\: A\cup B\subset \XA \}$ we have that each $\Mgood$-edge is an edge of
$\BGblack$,
and  at least one of the following conditions holds 
\begin{itemize}
\item[{\bf(K1)}]$2e_G(\XA)+e_G(\XA, \XB)\ge \eta kn/3$,
\item[{\bf(K2)}]  $|V(\Mgood)|\ge \eta n/3$.
\end{itemize}
\end{enumerate}
\end{proposition}

\begin{remark}\label{rem:K1vsK2}
In some sense, property~\eqref{Mgoodisblack} is the most important part of Lemma~\ref{prop:LKSstruct}.
Note that the assertion {\bf(K2)} implies a
quantitatively weaker version of {\bf(K1)}. Indeed, consider $(C,D)\in \M_A$. An
average vertex $v\in C$ sends at least $\beta\cdot \pi\clustersize\ge \beta\cdot \pi\nu k$ 
edges to
$D$. Thus, if $|V(\Mgood)|\ge \eta n/3$ then
$\Mgood$ induces at least $(\eta n/6)\cdot\beta\cdot
\pi\nu k=\Theta(kn)$ edges in $\XA$. Such a bound, however, would be insufficient
for our purposes as later $\eta\gg\pi,\nu$.
\end{remark}

\def\LPJJ{_\PARAMETERPASSING{L}{lem:Separate}}
\begin{proof}[Proof of Lemma~\ref{prop:LKSstruct}]
\addtocounter{theorem}{-1}
The idea of the proof is to first obtain some information about the structure of the graph $\BGblack$ with the help of the Gallai--Edmonds Matching Theorem (Theorem~\ref{thm:GallaiEdmonds}). Then this rough structure is refined by Lemma~\ref{lem:Separate} to yield the assertions of the lemma.

Let us begin with setting the parameters. 
Let $\beta:=\beta\LPJJ$ be given by Lemma~\ref{lem:Separate} for input parameters $\Omega\LPJJ:=\Omega$, $\rho\LPJJ:=\gamma^2$,
and let $\epsilon'$ and $\pi$ be given by Lemma~\ref{lem:Separate} 
for further input parameter $\epsilon\LPJJ:=\epsilon$. Last, let $k_0$ be given by Lemma~\ref{lem:Separate} with the above parameters and $\gamma\LPJJ:=\nu$.

Without loss of generality we assume
that $\epsilon'\le \epsilon$ and $\beta<\gamma^2$.
 We write $\BS :=\{C\in\clusters\::\: C\subset S\}$ and
$\BL:=\{C\in\clusters\::\: C\subset L\}$. Further, let $\BSN:=\{C\in\BS\::\: C\subset S^0\}$. 
 
Let $\SEPARATOR$ be a separator and $N_0$ a matching given by
Theorem~\ref{thm:GallaiEdmonds} applied to  the graph $\BGblack$. We will presume that the pair $(\SEPARATOR,N_0)$ is chosen among all
the possible choices so that the number of vertices of $\BSN$ that are isolated
in $\BGblack-\SEPARATOR$ and are not covered by
$N_0$ is minimized. Let $\BSI$ denote the set of
vertices in $\BSN$ that are isolated in $\BGblack-\SEPARATOR$. Recall that the components of $\BGblack-\SEPARATOR$ are factor critical.
 
Define
 $\BSR\subset V(\BGblack)$ as a minimal set such that
\begin{itemize}
\item $\BSI\sm V(N_0) \subseteq \BSR$, and
\item if $C\in \BS$  and there is an edge $DZ\in
E(\BGblack)$ with $Z\in \BSR$, $D\in
\SEPARATOR$, $CD\in N_0$ then $C\in \BSR$.
\end{itemize}

Then each vertex from $\BSR$ is
reachable from $\BSI\sm V(N_0)$ by 
 a path  in $\BGblack$ that alternates between $\BSR$ and $\SEPARATOR$, and has
 every second edge in $N_0$. 
Also note that for all $CD\in N_0$ with $C\in\SEPARATOR$ and $D\in\BSN\setminus\BSR$ we have 
\begin{equation}\label{anti-skub}
\deg_{\BGblack}(C,\BSR)=0\;.
\end{equation}

Let us show another property of $\BSR$.
\begin{claim}\label{cl:SRS0}
$\BSR\subseteq \BSI \subseteq \BSR\cup V(N_0)$. In particular, $\BSR\subset \BSN$.
\end{claim}
\begin{proof}[Proof of Claim~\ref{cl:SRS0}]
By the definition of $\BSR$, we only need to show that $\BSR\subset\BSI$.
So suppose there is  a vertex $C\in \BSR\sm\BSI$. 
By the definition of $\BSR$ there is a non-trivial path $R$ going from $C$ to $\BSI\sm
V(N_0)$, that alternates between $\BSR$ and $\SEPARATOR$, and has every second edge in
$N_0$. Then, the matching $N_0':=N_0\triangle E(R)$ covers more vertices of
$\BSI$ than $N_0$ does. Further, it is straightforward to check that the separator $\SEPARATOR$
together with the matching $N_0'$ satisfies the assertions of
Theorem~\ref{thm:GallaiEdmonds}. This is a contradiction, as desired.
\end{proof}

Using a very similar alternating path argument we see the
following.
\begin{claim}\label{augmiL}
If $CD\in N_0$ with $C\in \SEPARATOR$ and $D\notin \BSI$ then
$\deg_{\BGblack}(C,\BSR)=0$.
\end{claim}
\HIDDENPROOF{ Indeed, suppose
there is $CD\in N_0$ with $C\in \SEPARATOR$ and $D\notin \BSI$ but $CC'\in E(\BGblack)$ for some $C'\in\BSR$. Then there is a
non-trivial path $R$ with its first
three vertices $D,C$ and $C'$ and its last vertex in $\BSI\sm V(N_0)$ that alternates between $\BSR$ and $\SEPARATOR$, and has every second edge
in $N_0$. Then, the matching $N_0':=N_0\triangle E(R)$ covers
more vertices of $\BSI$ than $N_0$ does. Further, the
separator $\SEPARATOR$ together with the matching $N_0'$ satisfies the
assertions of Theorem~\ref{thm:GallaiEdmonds}. This
contradiction proves Claim~\ref{augmiL}.}

Using the factor-criticality of the components of $\BGblack-\SEPARATOR$ we extend $N_0$ to a matching $N_1$ as follows. For each component $K$ of
$\BGblack-\SEPARATOR$ which meets $V(N_0)$, we add a perfect matching of
$K-V(N_0)$. Furthermore, for each non-singleton component
$K$ of $\BGblack-\SEPARATOR$ which does not meet
$V(N_0)$, we add a matching which meets all but exactly one
vertex of $\BL\cap V(K)$. 
This is possible as by the
definition of the class $\LKSsmallgraphs{n}{k}{\eta}$ we have
that $\BGblack-\BL$ is independent, and so
$\BL\cap V(K)\neq\emptyset$.
This choice of $N_1$ guarantees that
\begin{equation}\label{eq:bir1}
e_{\BGblack}(\clusters\setminus V(N_1))=0\;.
\end{equation}

We set
\begin{align*}
M&:=\left\{ AB\in N_0\::\: A\in \BSR, B\in \SEPARATOR\right\}\;.
\end{align*}
We have that
\begin{equation}\label{eq:645}
e_{\BGblack}\big(\clusters\setminus V(N_1),V(M)\cap\BSR\big)=0\;.
\end{equation}

As $\BS$ is an independent set in $\BGblack$, we have that
\begin{equation}\label{eq:BinL}
\SEPARATOR_M:=V(M)\cap \SEPARATOR\subset \BL\;.
\end{equation}

The matching $M$ in $\BGblack$ corresponds to an $(\epsilon',\gamma^2,\clustersize)$-semiregular matching $\mathcal M$ in the underlying graph $\Gblack$, with $V_2(\mathcal M)\subset \bigcup \SEPARATOR$ (recall that semiregular matchings have orientations on their edges). Likewise, we define $\mathcal N_1$ as the $(\epsilon',\gamma^2,\clustersize)$-regular matching corresponding to $N_1$. The $\mathcal N_1$-edges are oriented so that $V_1(\mathcal N_1)\cap \bigcup \SEPARATOR=\emptyset$; this condition does not specify orientations of all the $\mathcal N_1$-edges and we orient the remaining ones in an arbitrary fashion. We write $\SR:=\bigcup\BSR$.

\begin{claim}\label{cl:CrossEdges1}
$e_{\Gcapt}\big(L\setminus (\smallatoms\cup V(\mathcal M)),\SR\big)=0$.
\end{claim}

\begin{proof}[Proof of Claim~\ref{cl:CrossEdges1}] We start by
showing that for every cluster $C\in \BL\setminus V(M)$ we have
\begin{equation}\label{eq:degCSR}
\deg_{\BGblack}(C,\BSR)=0\;.
\end{equation} 
First, if $C\not
\in \SEPARATOR$, then~\eqref{eq:degCSR} is true since $\BSR\subseteq \BSI$ by
Claim~\ref{cl:SRS0}. So suppose that $C\in \SEPARATOR$, and let $D\in V(\BGblack)$ be such that $DC\in N_0$. Now if  $D\notin\BSI$
then~\eqref{eq:degCSR} follows from Claim~\ref{augmiL}. 
On the other hand, suppose $D\in\BSI\subseteq
\BS^0$. As $C\notin V(M)$, we know that $D\notin\BSR$, and
thus,~\eqref{eq:degCSR}
follows from~\eqref{anti-skub}. 

Now, by~\eqref{eq:degCSR}, $\Gblack$  has no edges between $L\setminus (\smallatoms\cup
V(\mathcal M))$ and $\SR$. Also, no such edges can be in
$\Gexp$  or incident with $\smallatoms$, since
$\BSR\subseteq \BSN$ by Claim~\ref{cl:SRS0}. Finally, since $G\in\LKSsmallgraphs{n}{k}{\eta}$, there
are no edges between $\HugeVertices$ and $S$. This proves the claim.
\end{proof}

We prepare ourselves for an application of Lemma~\ref{lem:Separate}. The
numerical parameters of the lemma are $\Omega\LPJJ,\rho\LPJJ,\epsilon\LPJJ$ and $\gamma\LPJJ$ as above. The input objects for the lemma are the graph $\GD$ of
order $n'\le n$, the collection of $(\gamma k,\gamma)$-dense spots
$\DenseSpots$, the matching $\mathcal M$, the $(\nu k)$-ensemble $\mathcal
C_\PARAMETERPASSING{L}{lem:Separate}:=\BSR\sm V(N_1)$, and the set $Y_\PARAMETERPASSING{L}{lem:Separate}:=L\cap
\smallatoms$.
Note that Definition~\ref{bclassdef}, item~\ref{defBC:dveapul}, implies that $\DenseSpots$ absorbes $\mathcal M$. Further,~\eqref{eq:510ass1} is satisfied by Definition~\ref{bclassdef}, item~\ref{defBC:prepartition}.

The output of Lemma~\ref{lem:Separate} is an $(\epsilon,\beta,\pi \clustersize)$-semiregular matching $\mathcal M'$ with the following properties. 
\begin{enumerate}[(I)]
\item\label{eq:MsubMApp}
$|V(\mathcal M)\setminus V(\mathcal M')|<\epsilon n'\le \epsilon n$.
\item\label{theCandDwelike}
For each $(T,U)\in\mathcal M'$ there are sets  $C\in \BSR$ 
and $D=(A_D,B_D;E_D)\in\DenseSpots$
such that $T\subseteq C\cap A_D$ and $U\subseteq ((L\cap \smallatoms)\cup\bigcup
\SEPARATOR_M)\cap B_D$. 

(Indeed, to see this we use that $\V_1(\mathcal
M)\subset \BSR$ and that $V_2(\mathcal M)\subseteq \bigcup \SEPARATOR_M$ by the
definition of~$\mathcal M$.)
\item\label{eq:SeparatedM} There is a partition of $\mathcal M'$ into $\mathcal M_1$ and $ \mathcal M_B$ such that
$$e_{\GD}\left(
\:
\left(\left(\SR\setminus V(\mathcal N_1)\right)\cup V_1(\mathcal M)\right)\setminus V_1(\mathcal M_1)
\:,\:
\left((L\cap \smallatoms)\cup V_2(\mathcal M)\right)\setminus V_2(\mathcal M_B)
\:\right)<\gamma k n' \;.$$
\end{enumerate}
We claim that also
\begin{enumerate}[(I)]\setcounter{enumi}{3}
\item\label{eq:M'Next}
$V(\mathcal{M'})\cap V(\mathcal N_1\setminus \mathcal M)=\emptyset$.
\end{enumerate}
 Indeed, let $(T,U)\in\mathcal M'$
be arbitrary. Then by~\eqref{theCandDwelike} there is $C\in
\BSR$ such that $T\subset C$. By Claim~\ref{cl:SRS0},
$C$ is a singleton component of $\BGblack-\SEPARATOR$. In particular, if $C$ is covered by $N_1$
then $C\in V(M)$. It follows that $T\cap V(\mathcal N_1\setminus \mathcal M)=\emptyset$.
In a similar spirit, the easy fact that $(Y\cup\bigcup \SEPARATOR_M)\cap V(\mathcal N_1\setminus \mathcal M)=\emptyset$ together with~\eqref{theCandDwelike} gives $U\cap V(\mathcal N_1\setminus \mathcal M)=\emptyset$. This establishes~\eqref{eq:M'Next}.

\medskip

Observe that~\eqref{theCandDwelike} implies that $V_1(\mathcal M')\subset \SR$, and so,  by Claim~\ref{cl:SRS0} we know that 
\begin{equation}\label{eq:Asub}
V_1(\mathcal M_B)\subseteq \SR\subseteq \bigcup\BSI\subseteq\SN.
\end{equation}

Set
\begin{equation}\label{eq:defMA} 
\mathcal M_A:=(\mathcal N_1\setminus \mathcal M)\cup \mathcal M_1\;.
\end{equation}

Then $\mathcal M_A$ is an
$(\epsilon,\beta,\pi\clustersize)$-semiregular
matching. Note that from now on, the sets $\XA, \XB$ and $\XC$ are defined. The
situtation is illustrated in Figure~\ref{fig:struktura1}.
\begin{figure}[t] \centering
\includegraphics[scale=1.0]{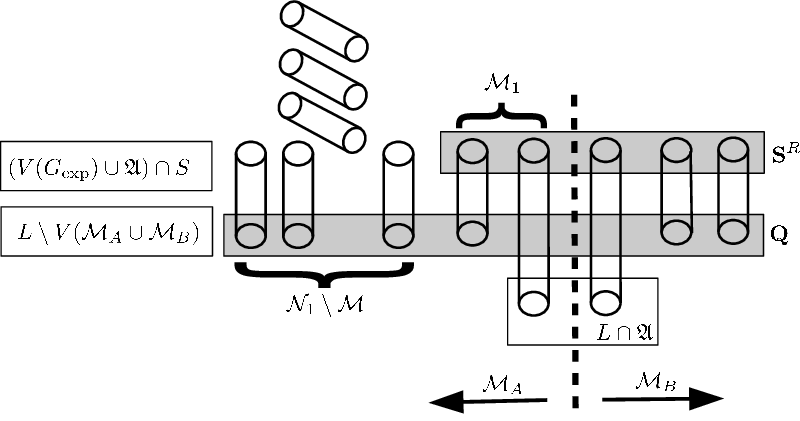}
\caption[Situation in $G$ in Lemma~\ref{prop:LKSstruct} after applying Lemma~\ref{lem:Separate}]{The situation in $G$ after applying Lemma~\ref{lem:Separate}. The dotted line illustrates the separation as in~(III).} 
\label{fig:struktura1}
\end{figure}
By~\eqref{eq:M'Next}, we
have $V(\mathcal M_A)\cap V(\mathcal M_B)=\emptyset$, as required for
Lemma~\ref{prop:LKSstruct}\eqref{prop6.1a}. Lemma~\ref{prop:LKSstruct}\eqref{eq:lastminute} follows from~\eqref{eq:Asub}. Observe that
by~\eqref{theCandDwelike}, also
Lemma~\ref{prop:LKSstruct}\eqref{eq:Mspots} and
Lemma~\ref{prop:LKSstruct}\eqref{eq:M1} are satisfied.

We now turn to Lemma~\ref{prop:LKSstruct}\eqref{fewfewfew}. First we prove
some auxiliary statements. 

\begin{claim}\label{cl:S0mN1}
We have 
$\BSN\setminus V(N_1\setminus M)\subset \BSR$.
\end{claim}
\begin{proof}[Proof of Claim~\ref{cl:S0mN1}]
Let $C\in\BSN\setminus V(N_1\setminus M)$. 
Note that if $C\notin \BSI$, then $C\in V(N_1)$. On the other hand, if $C\in \BSI$, then we use  Claim~\ref{cl:SRS0} to see that $C\in\BSR\cup V(N_1)$.
We deduce that in either case $C\in \BSR\cup V(N_1)$. The choice of $C$ implies that thus $C\in \BSR\cup V(M)$.
Now, if $C\in V(M)$, then $C\in\BSR$ by~\eqref{eq:BinL} and by the definition of $M$. Thus $C\in\BSR$ as desired.
\end{proof}

It will be convenient to work with a set $\bar S^0\subset S^0$, $\bar S^0:=\left(S\cap\bigcup \clusters\right)\setminus V(\Gexp)=\bigcup \BSN$. 
Note that $\bar S^0$ is essentially the same as $S^0$; the vertices in $S^0\setminus \bar S^0$ are isolated in $\Gcapt$ and thus have very little effect on our considerations. 

By Claim~\ref{cl:S0mN1}, we have
\begin{equation}\label{guanacos}
\bar S^0\setminus V(\M_A) \subset
\big(\bigcup \BSN\setminus V(N_1\setminus M)\big) \setminus V(\M_A) 
  \subset
\SR \setminus V(\M_A).
\end{equation}

 As every edge incident to
$S^0\setminus \bar S^0$ is uncaptured, we see that
\begin{align}
\label{fewfewfew:i}
E_{\Gcapt}\big(\XA\cap \smallatoms,S^0\setminus V(\mathcal
M_A)\big)
& \subset 
E_{\GD}\big((L\cap \smallatoms)\setminus V(\mathcal M_B),\bar S^0\setminus V(\mathcal
M_A)\big)\notag \\
\JUSTIFY{by \eqref{guanacos}}&\subset E_{\GD}\big(\:(L\cap \smallatoms)\setminus V(\mathcal
M_B)\:,\:\SR \setminus V(\M_A)\:\big).
\end{align}

We claim that furthermore
\begin{align}
\label{fewfewfew:ii}
E_{\Gblack}\big(\XA\cap\bigcup\clusters,S^0\setminus V(\mathcal
M_A)\big)\subset  E_{\GD}\left(
\left((L\cap \smallatoms)\cup V_2(\mathcal M)\right)\setminus V_2(\mathcal M_B),
\SR \setminus V(\M_A)
\:\right).
\end{align} 

Before proving~\eqref{fewfewfew:ii}, let us see that it implies Lemma~\ref{prop:LKSstruct}\eqref{fewfewfew}.
As
$G\in\LKSsmallgraphs{n}{k}{\eta}$, there are no edges between $\HugeVertices$
and $S$. That means that any captured edge from $\XA$ to $S^0\setminus V(\mathcal M_A)$ must start in 
$\smallatoms$ or in $\bigcup\clusters$.
Thus Lemma~\ref{prop:LKSstruct}\eqref{fewfewfew} follows by
plugging~\eqref{eq:SeparatedM} into~\eqref{fewfewfew:i}
and~\eqref{fewfewfew:ii}.

Let us now prove~\eqref{fewfewfew:ii}. First, observe that by the definition of $\XA$ and by the definition of $\mathcal M$ (and $M$) we have
\begin{equation}\label{eq:Hsplit}
 \XA\cap\bigcup\clusters\subset (V_2(\mathcal M)\setminus V_2(\mathcal M_B))\cup (L\setminus (\smallatoms\cup V(\mathcal M)))\;.
\end{equation}
Further, by applying~\eqref{guanacos} and Claim~\ref{cl:CrossEdges1} we get
\begin{equation}\label{eq:H0}
E_{\Gblack}\big(L\setminus (\smallatoms\cup V(\mathcal M)),\bar S^0\setminus V(\mathcal M_A)\big)=\emptyset\;.
\end{equation}

Therefore, we obtain
\begin{align*}
E_{\Gblack}\big(\XA\cap\bigcup\clusters, S^0\setminus V(\mathcal M_A)\big)
&\subset
E_{\Gblack}\big(\XA\cap\bigcup\clusters, \bar S^0\setminus V(\mathcal M_A)\big)\\
\JUSTIFY{by~\eqref{eq:Hsplit}}&\subset
E_{\Gblack}\big(V_2(\mathcal M)\setminus V_2(\mathcal M_B),\bar S^0\setminus
V(\mathcal M_A)\big)\\
&~~~~\cup E_{\Gblack}\big(L\setminus (\smallatoms\cup V(\mathcal
M)),\bar S^0\setminus V(\mathcal M_A)\big)
\\
\JUSTIFY{by~\eqref{guanacos}, \eqref{eq:H0}}&\subset
E_{\Gblack}\big(V_2(\mathcal M)\setminus V_2(\mathcal M_B),\SR \setminus V(\M_A)\big)\;,
\end{align*}
as needed for~\eqref{fewfewfew:ii}.

In order to
prove~\eqref{newpropertyS6} we first observe that
\begin{align}
V(\mathcal N_1)\setminus V(\M_A\cup \M_B)&\eqByRef{eq:defMA} V(\mathcal
N_1)\setminus V\big( (\mathcal N_1\setminus \M)\cup \M_1\cup \M_B\big)\notag\\
&= (V(\mathcal N_1)\cap V(\mathcal \M))\setminus V(\M_B\cup \M_1)\notag\\
&\eqByRef{eq:SeparatedM} (V(\mathcal N_1)\cap V(\mathcal \M))\setminus V(\M')\notag\\
&=  V(\M)\setminus V(\M')\;.\label{cl:matchINCL}
\end{align}

Now, we have
\begin{align*}
e_{\Gblack}(V(G)\setminus
V(\M_A\cup \M_B))&\le e_{\Gblack}(V(G)\setminus
V(\mathcal N_1))+\sum_{v\in V(\mathcal N_1)\setminus V(\M_A\cup
\M_B)}\deg_{\Gcapt}(v)\\
\JUSTIFY{by~\eqref{eq:bir1} and \eqref{cl:matchINCL}}&\le\sum_{v\in
V(\M)\setminus V(\M')}\deg_{\Gcapt}(v)\\ 
&\le |V(\M)\setminus V(\M')|\Omega^*k\\
\JUSTIFY{by~\eqref{eq:MsubMApp}} &<\epsilon
\Omega^* kn\;,
\end{align*}
which shows~\eqref{newpropertyS6}.
\smallskip

Let us turn to proving~\eqref{nicDoNAtom}. First, recall that we have
$V(\NAtom)\subset V(\M')\cup V(\mathcal N_1)$ (cf.~\ref{eq:defMA}). Since
$V(\mathcal N_1)\cap\smallatoms=\emptyset$ we actually have
\begin{equation}\label{eq:patchA}
V(\NAtom)=V(\NAtom)\cap V(\M')\;.
\end{equation}
Using~\eqref{eq:patchA} and~\eqref{theCandDwelike} we get
\begin{align}\nonumber
e_{\Gblack}\left(V(G)\setminus V(\mathcal
N_1),V(\NAtom)\right)&\le
e_{\Gblack}\left(V(G)\setminus V(\mathcal
N_1),V(\M')\cap \SR\right)\\
\nonumber
\JUSTIFY{by~\eqref{eq:645}}&\le 
e_{\Gblack}\left(V(G)\setminus V(\mathcal
N_1),(V(\M')\setminus V(\M))\cap \SR\right)\\
\nonumber
\JUSTIFY{by~\eqref{eq:M'Next}}&\le
e_{\Gblack}\left(V(G)\setminus V(\mathcal
N_1),(V(\M')\setminus V(\mathcal N_1))\cap \SR\right)\\
\label{eq:patchB}
&\le 2 e_{\Gblack}\left(V(G)\setminus V(\mathcal
N_1)\right)\eqByRef{eq:bir1}0\;. 
\end{align}
We have
\begin{align*}
e_{\Gblack}\left(V(G)\setminus V(\M_A\cup\M_B),V(\NAtom)\right) &\le e_{\Gblack}\left(V(G)\setminus V(\mathcal N_1),V(\NAtom)\right)\\ & \ \ +
e_{\Gblack}\left(V(\mathcal N_1)\setminus V(\M_A\cup\M_B),V(G)\right)\\
\JUSTIFY{by~\eqref{eq:patchB}}&\le 0+|V(\mathcal N_1)\setminus
V(\M_A\cup\M_B)|\Omega^* k\\
\JUSTIFY{by~\eqref{cl:matchINCL}, \eqref{eq:MsubMApp}}&\le \epsilon\Omega^* kn\;,
\end{align*}
as needed.

\smallskip

We have thus shown
Lemma~\ref{prop:LKSstruct}\eqref{prop6.1a}--\eqref{nicDoNAtom}. It only remains
to prove Lemma~\ref{prop:LKSstruct}\eqref{Mgoodisblack}, which we will do in the
remainder of this section.

\smallskip

We first collect several properties of $\XA$ and $\XC$.
The definitions of $\XC$ and $S^0$ give
\begin{equation}\label{eq:XCS0}
|\XC|(1+\eta)\frac{k}2 \le e_G\big(\XC,\SN \setminus V(\mathcal{M}_A\cup \mathcal{M}_B)\big) \le |\SN\setminus V(\mathcal{M}_A\cup \mathcal{M}_B)|(1+\eta)k\;.
\end{equation}

Each $v\in\XC$ has neighbours in $S$. Thus, by~\ref{def:LKSsmallB}.\ of Definition~\ref{def:LKSsmall} we have
\begin{equation}\label{tomze}
\deg_G(v)=\lceil(1+\eta)k\rceil
\end{equation}
for each $v\in\XC$. Further, each
vertex of $\XC$ has degree at least
$(1+\eta)\frac{k}2$ into $S$, and so,
\begin{equation}\label{eq:withrounding}
e_G(S,\XC)\ge |\XC|\left\lceil(1+\eta)\frac{k}2\right\rceil\;.
\end{equation}
Consequently (using the elementary inequality $\lceil a\rceil -\lceil\frac a2\rceil\le \frac a2$),
\begin{align}
e_G\big(\XA,\XC\big)
& \overset{\eqref{tomze}}\le  |\XC|\lceil(1+\eta)k\rceil - e_G(S,\XC)\notag \\
&\overset{\eqref{eq:withrounding}}\le  |\XC|(1+\eta)\frac{k}2\label{eq:PrI} \\
&\overset{\eqref{eq:XCS0}}\leq |\SN\setminus V(\mathcal{M}_A\cup
\mathcal{M}_B)|(1+\eta)k\;.\label{eq:eXAXC}
\end{align}

Let $\Mgood$ be defined as in Lemma~\ref{prop:LKSstruct}\eqref{Mgoodisblack},
that is, $\Mgood:=\{(A,B)\in \mathcal M_A\::\: A\cup B\subset \XA \}$. Note that~\eqref{eq:Asub} implies that $A\subseteq S$
for every $(A,B)\in \mathcal M_B$. Thus by the definition of $\XA$, 
\begin{equation}\label{cl:som}
\text{if $(A,B)\in \mathcal M_A\cup\mathcal M_B$ with $A\cup B\subset L$ then 
$(A,B)\in \Mgood$.}
\end{equation}

We will now show the first part of Lemma~\ref{prop:LKSstruct}\eqref{Mgoodisblack}, that is, we show that each $\Mgood$-edge is an edge
of $\BGblack$. Indeed,
by~\eqref{theCandDwelike}, we have that $V_1(\mathcal
M_1)\subset S$, so as $\XA\cap S=\emptyset$, it
follows that $\mathcal M_1\cap \Mgood=\emptyset$. Thus  $\Mgood\subset \mathcal N_1$. As $\mathcal N_1$ corresponds to a matching in $\BGblack$, all is as desired.

Finally, let us assume that
neither~{\bf(K1)} nor~{\bf(K2)} are fulfilled. After five preliminary
observations (Claim~\ref{claim:DZ4}--Claim~\ref{claim:DZ5}), we will derive a
contradiction from this assumption.

\begin{claim}\label{claim:DZ4}
We have
$|S\cap V(\mathcal M_A)|\le|\XA\cap V(\mathcal M_A)|$.
\end{claim}
\begin{proof}[Proof of Claim~\ref{claim:DZ4}]
To see this, recall that each $\mathcal M_A$-vertex $U\in\V(\mathcal M_A)$ is
either contained in $S$, or in $L$. Further, if $U\subset S$ then
its partner in $\mathcal M_A$ must be in $L$, as $S$ is
independent. Now, the claim follows after noticing that $L\cap
V(\mathcal M_A)=\XA\cap V(\mathcal M_A)$.
\end{proof}

\begin{claim}\label{claim:DZ1}
We have $|S\setminus V(\mathcal M_A\cup \mathcal M_B)|+2\eta
n< |\XA\setminus V(\mathcal M_A)|+\eta n/3$.
\end{claim}
\begin{proof}[Proof of Claim~\ref{claim:DZ1}]
As $G\in\LKSgraphs{n}{k}{\eta}$, we have $|S|+2\eta n\le
|L|$. Therefore,
\begin{align*}
|S\setminus V(\mathcal M_A\cup \mathcal M_B)|+2\eta
n\le & \ |L\setminus V(\mathcal M_A\cup \mathcal
M_B)|+\sum_{\substack{(A,B)\in\mathcal
M_A\cup\mathcal M_B\\A\cup B\subset L}}|A\cup B|\\
\eqBy{\eqref{cl:som}} & |\XA\setminus V(\mathcal M_A)|+|V(\Mgood)|\\
\lBy{$\neg${\bf(K2)}} & |\XA\setminus V(\mathcal M_A)|+\eta n/3\;.
\end{align*}
\end{proof}

\begin{claim}\label{claim:DZ2}
We have $e_{\Gcapt}\left(\XA\cap(\smallatoms\cup V(\mathcal M))
,
\SR\setminus V(\mathcal M_A)
\right)<\eta kn/2$.
\end{claim}
\begin{proof}[Proof of Claim~\ref{claim:DZ2}]
As 
\begin{align*}
\XA\cap (\smallatoms\cup V(\mathcal M))&\subset \left((L\cap \smallatoms)\cup
V_2(\mathcal M)\right)\setminus V_2(\mathcal M_B)\quad \mbox{and}\\
\SR\setminus V(\mathcal M_A)
&\subset
\left(\left(\SR\setminus V(\mathcal N_1)\right)\cup V_1(\mathcal
M)\right)\setminus V_1(\mathcal M_1)\;,
 \end{align*}
 we get from~\eqref{eq:SeparatedM} that
\begin{equation}\label{eq:iiModif1}
e_{\GD}\left(\XA\cap(\smallatoms\cup V(\mathcal M))
,
\SR\setminus V(\mathcal M_A)
\right)\le \gamma kn\;.
\end{equation}
Observe now that both sets $\XA\cap (\smallatoms\cup V(\mathcal M))$ and
$\SR\setminus V(\mathcal M_A)$ avoid $\HugeVertices$. Further, no edges between them belong to $\Gexp$, because Claim~\ref{cl:SRS0} implies that $\SR\setminus V(\mathcal
M_A)\subset S^0\subseteq V(G)\setminus V(\Gexp)$. Therefore, we can pass from $\GD$ to $\Gcapt$ in~\eqref{eq:iiModif1} to get
\begin{equation*}
e_{\Gcapt}\left(\XA\cap(\smallatoms\cup V(\mathcal M))
,
\SR\setminus V(\mathcal M_A)
\right)\le \gamma kn<\eta kn/2\;.
\end{equation*}
\end{proof}

\begin{claim}\label{claim:DZ3}
We have
$S\setminus (\SR\cup V(\mathcal M_A) )
\subset
S\setminus (\bar S^0\cup V(\mathcal M_A\cup\mathcal M_B))$.
\end{claim}
\begin{proof}[Proof of Claim~\ref{claim:DZ3}]
The claim follows directly from the following two inclusions.
\begin{align}
\label{eq:FC1}
\SR\cup V(\mathcal M_A)&\supset S\cap V(\mathcal M_A\cup\mathcal
M_B)\;\mbox{, and}\\
\label{eq:FC2}
\SR\cup V(\mathcal M_A)&\supset \bar S^0\;.
\end{align}
Now,~\eqref{eq:FC1} is trivial, as by~\eqref{theCandDwelike} we have that
$\SR\supset S\cap V(\mathcal M_B)$.
 To see~\eqref{eq:FC2}, it suffices
by~\eqref{eq:defMA} to prove that $V(N_1\setminus M)\cup\BSR\supset \BSN$.
This
is however the subject of Claim~\ref{cl:S0mN1}.
\end{proof}

Next, we bound $e_{\Gcapt}\big(\XA,S\big)$. 
\begin{claim}\label{claim:DZ5}
We have $$e_{\Gcapt}\big(\XA,S\big)\le |S\cap V(\mathcal M_A)|(1+\eta)k+|S\setminus(\SN\cup V(\mathcal M_A\cup \mathcal M_B))|(1+\eta)k 
+\frac12\eta kn\;.$$
\end{claim}
\begin{proof}[Proof of Claim~\ref{claim:DZ5}]
We have 
\begin{align*}
e_{\Gcapt}\big(\XA,S \big)
=& \ e_{\Gcapt}\big(\XA,S\cap V(\mathcal M_A)\big)
\\ 
&\mbox{~~~~}
+  e_{\Gcapt}\big(\XA,S\setminus(\SR\cup V(\mathcal M_A))\big) 
\\ 
&\mbox{~~~~}
+e_{\Gcapt}\big(\XA\setminus (\smallatoms\cup V(\mathcal M)),\SR\setminus V(\mathcal M_A)\big)
\\ 
&\mbox{~~~~}
+e_{\Gcapt}\left(\XA\cap(\smallatoms\cup V(\mathcal M)),\SR\setminus V(\mathcal
M_A)\right)\;.
\end{align*}
To bound the first term we use that the vertices in $S\cap V(\mathcal M_A)$ each have degree at most $(1+\eta) k$, and thus obtain $e_{\Gcapt}(\XA,S\cap V(\mathcal M_A))\le |S\cap V(\mathcal M_A)|(1+\eta)k$. To bound the second term, we again use a bound on degree of vertices of $S\setminus\big((\SR\cup V(\mathcal
M_A))\cup(S^0\setminus \bar S^0))$, together with Claim~\ref{claim:DZ3}. The third term is zero by Claim~\ref{cl:CrossEdges1}. The fourth term can be bounded by Claim~\ref{claim:DZ2}.
\addtocounter{theorem}{1}
\end{proof}

A relatively short double counting below will lead to the final contradiction. The idea behind this computation is given in Figure~\ref{fig:strukturaContradiction}.
\begin{figure}[t] \centering
\includegraphics[scale=0.9]{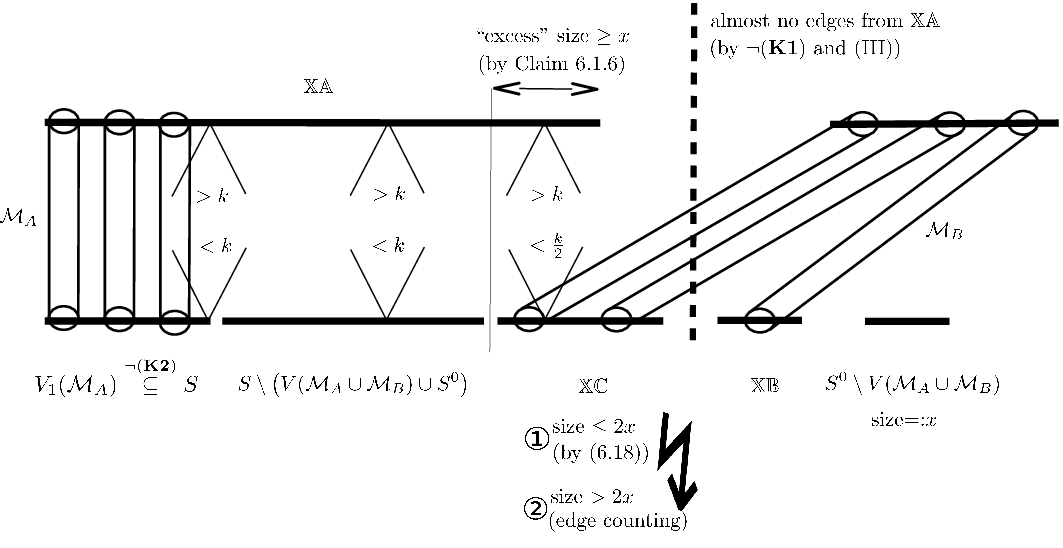}
\caption[Contradiction in Lemma~\ref{prop:LKSstruct}]{A simplified computation showing that $\neg {\bf(K1)}$, $\neg {\bf(K2)}$ leads to a contradiction. Denoting by $x$ the size of $S^0\setminus V(\M_A\cup\M_B)$ we get \ding{172} $|\XC|\le 2x$. On the other hand, each vertex of $\XA$ emanates $\gtrsim k$ edges which are absorbed by the sets $V_1(\M_A)$, $S\setminus (V(\M_A\cup\M_B)\cup S^0)$, and $\XC$. The vertices of $V_1(\M_A)$ and $S\setminus (V(\M_A\cup\M_B)\cup S^0)$ can absorb $\lesssim k$ edges. The vertices of $\XC$ receive $\lesssim\frac k2$ edges of $\XA$ by~\eqref{eq:PrI}. This leads to \ding{173} $|\XC|> 2x$, doubling the size of the ``excess'' vertices of $\XA$.} 
\label{fig:strukturaContradiction}
\end{figure}
\begin{align}
\begin{split}\label{eq:MCS}
|\XA|(1+\eta)k&\le \sum_{v\in \XA}\deg_G(v)\\
& \le \sum_{v\in \XA}\deg_{\Gcapt}(v)+2\big(e(G)-e(\Gcapt)\big)\\
& \le
2e_{\Gcapt}(\XA)+e_{\Gcapt}(\XA,\XB)+e_{\Gcapt}\big(\XA,\XC\big)\\
&\mbox{~~~~~~}+ e_{\Gcapt}\big(\XA,S\big)+\frac{\eta
kn}3\\[6pt]
\JUSTIFY{by $\neg${\bf(K1)}, \eqref{eq:eXAXC}, C\ref{claim:DZ5}}
&\le \frac 76\eta kn+ \big|\SN\setminus V(\mathcal{M}_A\cup
\mathcal{M}_B)\big|(1+\eta)k\\
&\mbox{~~~~~~}+ |S\cap V(\mathcal M_A)|(1+\eta)k\\
&\mbox{~~~~~~}+    |S\setminus(\SN\cup V(\mathcal M_A\cup \mathcal M_B))|(1+\eta)k
\\[6pt]
\JUSTIFY{by C\ref{claim:DZ4}}
&\le \frac 76\eta kn+|S\setminus V(\mathcal M_A\cup\mathcal
M_B)|(1+\eta)k\\
&\mbox{~~~~~~}+|\XA\cap V(\mathcal M_A)|(1+\eta)k\\[6pt]
\JUSTIFY{by C\ref{claim:DZ1}}
&\le \frac 76 \eta kn+\big(|\XA\setminus V(\mathcal M_A)|-\frac53\eta
n\big)(1+\eta)k\\
 &\mbox{~~~~~~}+|\XA\cap V(\mathcal M_A)|(1+\eta)k
\\[6pt]
&< |\XA|(1+\eta)k-\frac12\eta kn\;,
\end{split}
\end{align}
a contradiction. This finishes the proof of Lemma~\ref{prop:LKSstruct}.
\end{proof}

\subsection{The role of Lemma~\ref{lem:Separate} in the proof of
Lemma~\ref{prop:LKSstruct}}
\label{ssec:whyaugment}
Let us explain the role of Lemma~\ref{lem:Separate} in our proof of
Lemma~\ref{prop:LKSstruct}. First, let us attempt to use just the sparse
decomposition $\class$ to embed a tree $T\in \treeclass{k}$ in
$G\in \LKSgraphs{n}{k}{\eta}$. We will eventually see that this is
impossible and that we need to enhance $\class$ by a semiregular matching
(provided by Lemma~\ref{lem:Separate}).

We wish to find two sets $\mathbb{VA}$ and $\mathbb{VB}$ which are suitable for
embedding the cut vertices $W_A$ and $W_B$ of a $(\tau k)$-fine partition
$(W_A,W_B,\shrubA,\shrubB)$ of $T$, respectively.
In this sketch we just focus on finding $\mathbb{VA}$; the ideas behind finding
a set suitable set $\mathbb{VB}$ are similar.

To accommodate all the shrubs from $\shrubA$ --- which might contain up to $k$
vertices in total --- we need $\mathbb{VA}$ to have degree at least
$\sum_{T^*\in\shrubA}v(T^*)$ into a suitable set of vertices we reserve for
these shrubs. (The neighbourhood of a possible image of a vertex from $W_A$ has
to allow space for its children and for everything blocked by shrubs from
$\shrubA$ embedded earlier.)

Our methods of embedding in Section~\ref{sec:embed} determine which sets we find
`suitable' for $\shrubA$: these are the
large vertices $\largevertices{\eta}{k}{G}$, the
vertices of the nowhere-dense graph $\Gexp$, the avoiding set~$\smallatoms$, and
any matching consisting of regular pairs. This motivates us to look for a
semiregular matching~$\M$ which covers as much as possible of the set
$S^0:=\smallvertices{\eta}{k}{G}\setminus
\left(V(\Gexp)\cup\smallatoms\right)$ which consists of those vertices not
utilizable by any other of the methods above. As a next step one would prove
that there is a set $\mathbb{VA}$ with $$\mindeg\left(\mathbb{VA},V(G)\setminus (S^0\setminus
V(\M))\right)\gtrsim k\;.$$ In the dense setting~\cite{PS07+}, where the
structure of $G$ is determined by $\BGblack$, and
where $S^0=\smallvertices{\eta}{k}{G}$, such a matching $\M$ can be found inside
$\BGblack$ using the Gallai--Edmonds Matching Theorem. But here, just working
with $\BGblack$ is not enough for finding a suitable semiregular matching as the
following example shows.

\begin{figure}[ht] \centering \includegraphics[scale=0.8]{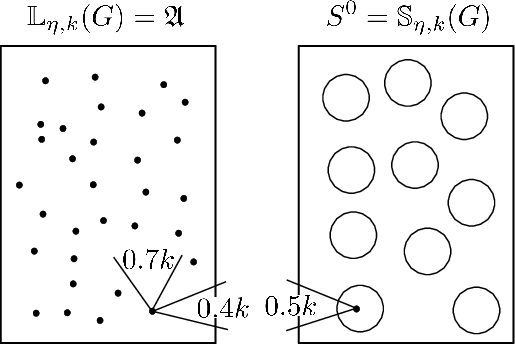}
\caption[Example of graph with $\BGblack$ empty]{An example of a graph $G\in\LKSgraphs{n}{k}{\eta:=\frac{1}{10}}$ in which $\BGblack$ is empty, and yet there is no candidate set for $\mathbb{VA}$ of vertices which have degrees at least $k$ outside the set~$S^0$.}
\label{fig:whyEnhancing}
\end{figure}
Figure~\ref{fig:whyEnhancing} shows a graph $G$ with
$\largevertices{\eta}{k}{G}\subseteq \smallatoms$, and where the vertices in
$S^0=\smallvertices{\eta}{k}{G}$ form clusters which do not induce any dense
regular pairs. Each $\largevertices{\eta}{k}{G}$-vertex sends $0.7k$ edges to
$\largevertices{\eta}{k}{G}$ and $0.4k$ edges to $\smallvertices{\eta}{k}{G}$,
and each $\smallvertices{\eta}{k}{G}$-vertex receives $0.5k$ edges
from $\largevertices{\eta}{k}{G}$. The edges between
$\largevertices{\eta}{k}{G}$ and $\smallvertices{\eta}{k}{G}$ are contained in
$\DenseSpots$. No vertex has degree $\gtrsim k$ outside $S^0$, and the
cluster graph $\BGblack$ contains no matching.

However in this situation we can still find a large semiregular matching $\M$ between $\largevertices{\eta}{k}{G}$ and $\smallvertices{\eta}{k}{G}$, by regularizing the crossing dense spots $\DenseSpots$.
(In general, obtaining a semiregular matching is of course more complicated.)

The example relates to Lemma~\ref{prop:LKSstruct} by setting $\XA:=\mathbb{VA}$, and $\M_A:=\M$. Indeed,~\eqref{fewfewfew} of Lemma~\ref{prop:LKSstruct} says that $\XA$-vertices send almost no edges to $S^0\setminus V(\M_A)$, and thus (since $\XA\subset \largevertices{\eta}{k}{G}$), they have degree $\gtrsim k$ outside $S^0\setminus V(\M_A)$.

\section{Configurations}\label{sec:configurations}
In this section we introduce ten configurations
--- called $\mathbf{(\diamond1)}$--$\mathbf{(\diamond10)}$ --- which may be found in a
graph $G\in \LKSgraphs{n}{k}{\eta}$. We will be able to infer from the main
results of this section
(Lemmas~\ref{lem:ConfWhenCXAXB}--\ref{lem:ConfWhenMatching}) and from other
structural results of this paper that each graph $G\in \LKSgraphs{n}{k}{\eta}$
contains at least one of these configurations.
Lemmas~\ref{lem:ConfWhenCXAXB}--\ref{lem:ConfWhenMatching} are based on the
structure provided by Lemma~\ref{prop:LKSstruct} which itself is in a sense the
most descriptive result of the structure of graphs from
$\LKSgraphs{n}{k}{\eta}$. However, the structure given by
Lemma~\ref{prop:LKSstruct} needs some burnishing. It will turn out in
Section~\ref{sec:embed} that each of the configurations  $\mathbf{(\diamond1)}$--$\mathbf{(\diamond10)}$ 
 is suitable for the
embedding of any tree from $\treeclass{k}$ as required for
Theorem~\ref{thm:main}.

This section is organized as follows. In Section~\ref{ssec:shadows} we introduce an auxiliary notion of shadows and prove some simple properties of them. Section~\ref{ssec:RandomSplittins}
introduces randomized splitting of the vertex set of an input
graph. In Section~\ref{ssec:ExceptionalVertices} we define certain cleaned
versions of the sets $\XA$ and $\XB$, and introduce other building blocks for
the configurations $\mathbf{(\diamond1)}$--$\mathbf{(\diamond10)}$. In
Section~\ref{ssec:TypesConf} we state some preliminary definitions and introduce the configurations $\mathbf{(\diamond1)}$--$\mathbf{(\diamond10)}$. In Section~\ref{ssec:cleaning} we prove certain ``cleaning lemmas''. The main results are then stated and proved  in Section~\ref{ssec:obtainingConf}. The results of
Section~\ref{ssec:obtainingConf} rely on the auxiliary lemmas of
Section~\ref{ssec:RandomSplittins} and~\ref{ssec:cleaning}.

\subsection{Shadows}\label{ssec:shadows}
We will find it convenient to work with the notion of a shadow. Given a graph $H$, a set
$U\subset V(H)$, and a number $\ell$ we define inductively
\index{mathsymbols}{*SHADOW@$\shadow$}
\begin{align*}
& \shadow^{(0)}_H(U,\ell):=U\text{, and }\\
& \shadow^{(i)}_H(U,\ell):=\{ v\in V(H)\::\:
\deg_H(v,\shadow^{(i-1)}_H(U,\ell))>\ell \} \text{ for }i\geq 1.
\end{align*}
We abbreviate $\shadow^{(1)}_H(U,\ell)$ as $\shadow_H(U,\ell)$. Further, the
graph $H$ is omitted from the subscript if it is clear from the context.
Note that the shadow of a set $U$ might intersect $U$.

Below, we state two facts which bound the size of a shadow of a given set.
Fact~\ref{fact:shadowbound} gives a bound in general graphs of bounded maximum
degree and Fact~\ref{fact:shadowboundEXPANDER} gives a stronger bound for
nowhere-dense graphs.
\begin{fact}\label{fact:shadowbound}
Suppose $H$ is a graph with $\maxdeg(H)\le \Omega k$. Then for each
$\alpha>0, i\in\{0,1,\ldots\}$, and each set $U\subset V(H)$, we have
$$|\shadow^{(i)}(U,\alpha k)|\le \left(\frac{\Omega}{\alpha}\right)^i|U|\;.$$
\end{fact}
\begin{proof}
Proceeding by induction on $i$ it suffices to show that
$|\shadow^{(1)}(U,\alpha k)|\le \Omega|U|/\alpha$. To this end, observe that
$U$ sends out at most $\Omega k |U|$ edges while each vertex of
$\shadow(U,\alpha k)$ receives at least $\alpha k$ edges from $U$.
\end{proof}
\begin{fact}\label{fact:shadowboundEXPANDER}
Let $\alpha,\gamma,Q>0$ be three numbers such that $Q\ge1$ and $16Q\le
\frac{\alpha}{\gamma}$. Suppose that $H$ is a $(\gamma k,\gamma)$-nowhere-dense graph, and let $U\subset V(H)$ with $|U|\le Qk$. Then we have $$|\shadow(U,\alpha
k)|\le \frac{16Q^2\gamma}{\alpha}k.$$
\end{fact}

\begin{proof}
Suppose otherwise and let $W\subset \shadow(U,\alpha k)$ be of size
$|W|=\frac{16Q^2\gamma}{\alpha}k\le Qk$. Then $e_H(U\cup W)\ge \frac12 \sum_{v\in W}\deg_H(v,U)\ge 8\gamma Q^2 k^2$. Thus $H[U\cup W]$ has average degree at least $$\frac{2 e_H(U\cup W)}{|U|+|W|}\ge 8\gamma Qk\;,$$ and therefore, by a well-known fact, contains a subgraph $H'$ of minimum degree at least $4\gamma Qk$. Taking a maximal cut $(A,B)$ in $H'$, it is easy to see that $H'[A,B]$ has minimum degree at least $2\gamma Qk\ge \gamma k$. Further, $H'[A,B]$ has density at least $\frac{|A|\cdot 2\gamma Qk}{|A||B|}\geq\gamma$,  contradicting the fact that $H$ is $(\gamma k,
\gamma)$-nowhere-dense.
\end{proof}

\subsection{Random splitting}\label{ssec:RandomSplittins}
\def\AXA{\mathfrak{Q}}
\def\aXa{\mathfrak{q}}
\def\VXV{\mathfrak{U}}
Suppose a graph $G$ (together with its bounded decomposition\footnote{Note that
in general we apply a \emph{sparse} decomposition (as opposed to a
\emph{bounded} decomposition) on the graph $G=G_\PARAMETERPASSING{T}{thm:main}$, cf.\
Lemma~\ref{lem:LKSsparseClass}. However, it turns out that when the vertices
$\HugeVertices$ of huge degrees form a substantial part of $G$ (which is when
the need of transition from bounded to sparse decomposition arises), the result
of this section is not needed.}) is given. In this section we split its vertex
set in several classes in a given ratio. It is important that most vertices will
have their degrees split obeying approximately this ratio. The
corresponding statement is given in Lemma~\ref{lem:randomSplit}. It will be used to split the vertices of the host graph $G=G_\PARAMETERPASSING{T}{thm:main}$
according to which part of the tree
$T=T_\PARAMETERPASSING{T}{thm:main}\in\treeclass{k}$ they will host. More
precisely, suppose that $(W_A,W_B,\shrubA,\shrubB)$ is an $\ell$-fine partition
of $T$ (for a suitable number $\ell$). Let $t_\mathrm{int}$ and $t_\mathrm{end}$
be the total sizes of the internal and end shrubs, respectively. We then
want to partition $V(G)$ into three sets
$\colouringp{0},\colouringp{1},\colouringp{2}$ (which correspond to $\VXV_1, \VXV_2,
\VXV_3$ in Lemma~\ref{lem:randomSplit}) in the ratio (approximately)
$$(|W_A|+|W_B|)\::\:t_\mathrm{int}\::\:t_\mathrm{end}$$ so that degrees of the
vertices of $V(G)$ are split proportionally. This will allow us to embed the vertices of
$W_A\cup W_B$ in $\colouringp{0}$, the internal shrubs in $\colouringp{1}$, and
end shrubs in $\colouringp{2}$. Actually, as our embedding procedure is
more complex, we not only require the degrees to be split proportionally, but
also to partition proportionally the objects from the bounded decomposition.
In Section~\ref{ssec:whyrandomsplitting} we give some reasons why such a random
splitting needs to be used.

Lemma~\ref{lem:randomSplit} below is formulated in an abstract setting, without
any reference to the tree $T$, and with a general number of classes in the partition.

\begin{lemma}\label{lem:randomSplit} For each $p\in\mathbb N$ and $a>0$ there exists $k_0>0$ such that
for each $k>k_0$ we have the following.

Suppose $G$ is a graph of order $n\ge k_0$ and $\maxdeg(G)\le\Omega^*k$ with its 
$(k,\Lambda,\gamma,\epsilon,k^{-0.05},\rho)$-bounded
decomposition $( \clusters,\DenseSpots, \Gblack, \Gexp, \smallatoms )$. As
usual, we write $\Gcapt$ for the subgraph captured by $( \clusters,\DenseSpots,
\Gblack, \Gexp, \smallatoms )$, and $\GD$ for the spanning subgraph of $G$
consisting of the edges in $\DenseSpots$. Let $\M$ be an
$(\epsilon,d,k^{0.95})$-semiregular matching in $G$, and $\VXV_1,\ldots, \VXV_p$
be subsets of $V(G)$. Suppose that $\Omega^*\ge 1$ and $\Omega^*/\gamma<k^{0.1}$.

Suppose that $\aXa_1,\ldots,\aXa_p\in\{0\}\cup[a,1]$ are reals with $\sum \aXa_i\le 1$.
Then there exists a partition $\AXA_1\cup \ldots\cup \AXA_p=V(G)$, and sets $\bar
V\subset V(G)$, $\bar\V\subset \V(\M)$, $\bar\clusters\subset\clusters$ with the following properties.
\begin{enumerate}[(1)]
  \item\label{It:H1} $|\bar V|\le \exp(-k^{0.1})n$,
  $|\bigcup\bar\V|\le \exp(-k^{0.1})n$,
  $|\bigcup\bar\clusters|<\exp(-k^{0.1})n$.
  \item\label{It:H2} For each $i\in [p]$ and each $C\in
  \clusters\setminus\bar\clusters$ we have $|C\cap \AXA_i|\ge
  \aXa_i|\AXA_i|-k^{0.9}$.
  \item\label{It:H3} For each $i\in [p]$ and each $C\in \V(\M)\setminus\bar\V$
  we have $|C\cap \AXA_i|\ge \aXa_i|\AXA_i|-k^{0.9}$.
  \item\label{It:H4} For each $i\in [p]$, $D=(U,W; F)\in\DenseSpots$
   and 
   $\mindeg_D (U\setminus \bar V,W\cap \AXA_i)\ge \aXa_i\gamma k-k^{0.9}$.
  \item\label{It:HproportionalSizes} For each $i,j\in[p]$ we have $|\AXA_i\cap
  \VXV_j|\ge \aXa_i|\VXV_j|-n^{0.9}$.
  \item\label{It:H5} For each $i\in [p]$ each $J\subset[p]$ and each $v\in V(G)\setminus \bar V$ we have $$\deg_H(v,\AXA_i\cap \VXV_J)\ge \aXa_i\deg_H(v,\VXV_J)-2^{-p}k^{0.9}\;,$$ for each of the graphs $H\in\{G,\Gcapt,\Gexp,\GD,\Gcapt\cup\GD\}$,
  where 
$\VXV_J:=\big(\bigcap_{j\in J}\VXV_j\big)\sm \big(\bigcup_{j\in [p]\setminus J}
\VXV_j\big)$.
  \item\label{It:H6} For each $i,i',j,j'\in [p]$ ($j\neq j'$), we have
\begin{align*}
e_H(\AXA_i\cap \VXV_{j},\AXA_{i'}\cap \VXV_{j'})&\ge
  \aXa_i\aXa_{i'} e_H(\VXV_j,\VXV_{j'})-k^{0.6}n^{0.6}\;,\\
e_H(\AXA_i\cap \VXV_{j},\AXA_{i'}\cap \VXV_{j})&\ge
  \aXa_i\aXa_{i'} e(H[\VXV_j])-k^{0.6}n^{0.6}\qquad\mbox{if $i\neq i'$, and}\\
  e(H[\AXA_i\cap \VXV_{j}])&\ge
  \aXa_i^2 e(H[\VXV_j])-k^{0.6}n^{0.6}\;.   
\end{align*}  
for each of the graphs
  $H\in\{G,\Gcapt,\Gexp,\GD,\Gcapt\cup\GD\}$.
  \item\label{It:H7} For each $i\in[p]$ if $\aXa_i=0$ then
  $\AXA_i=\emptyset$.
\end{enumerate}
\end{lemma}
\begin{proof}
We can assume that $\sum \aXa_i=1$ as all bounds in \eqref{It:H2}--\eqref{It:H6} are lower bounds. Assume that $k$ is
large enough. We assign each vertex $v\in V(G)$ to one of the sets $\AXA_1$,
\ldots, $\AXA_p$ at random with respective probabilities $\aXa_1,\ldots,\aXa_p$. Let $\bar V_1$ and $\bar V_2$ be the
vertices which do not satisfy~\eqref{It:H4} and~\eqref{It:H5},
respectively. Let $\bar\V$ be the sets of $\V(\M)$ which do
not satisfy~\eqref{It:H3}, and let $\bar\clusters$ be the
clusters of $\clusters$ which do not satisfy~\eqref{It:H2}.
Setting~$\bar V:=\bar V_1\cup\bar V_2$, we need to show
that~\eqref{It:H1}, \eqref{It:HproportionalSizes} and~\eqref{It:H6} are
fulfilled simultaneously with positive probability. Using the union bound,
it suffices to show that each of the properties~\eqref{It:H1},
\eqref{It:HproportionalSizes} and \eqref{It:H6} is violated with probability
at most $0.2$. The probability of each of these three
properties can be controlled in a
straightforward way by the Chernoff bound. We only give such a bound (with error probability at most $0.1$) on the size of the set $\bar V_1$ (appearing
in~\eqref{It:H1}), which is the most difficult one to control.

For $i\in [p]$, let $\bar V_{1,i}$ be the set of vertices $v$ for which there
exists $D=(U,W; F)\in\DenseSpots$, $U\ni v$, such that
$\deg_D(v,W\cap \AXA_i)<\aXa_i\gamma k-k^{0.9}$. We aim to show
that for each $i\in [p]$ the probability that $|\bar
V_{1,i}|>\exp(-k^{0.2})n$ is at most $\frac{1}{10p}$. Indeed, summing such an error bound together with similar bounds for other properties will allow us to conclude the statement.
This will in turn follow from the Markov Inequality provided that we show that
\begin{equation}\label{itsraining}
\expectation[|\bar V_{1,i}|]\le
\frac{1}{10p}\cdot\exp(-k^{0.2})n\;.
\end{equation}
Indeed, let us consider an arbitrary vertex $v\in V(G)$. By Fact~\ref{fact:sizedensespot},
$v$ is contained in at most $\Omega^*/\gamma$ dense spots of $\DenseSpots$. For a 
fixed dense spot $D=(U,W;F)\in\DenseSpots$ with $v\in
U$ let us bound the probability of the event $\mathcal
E_{v,i,D}$ that $\deg_D(v,W\cap \AXA_i)<\aXa_i\gamma k-k^{0.9}$.
To this end, fix a set $N\subset W\cap \neighbour_D(v)$ of size exactly
$\gamma k$ before the random assignment is performed. Now, elements of
$V(G)$ are distributed randomly into the sets $\AXA_1,\ldots, \AXA_p$. In particular, the number
$|\AXA_i\cap N|$ has binomial distribution with parameters
$\gamma k$ and $\aXa_i$. Using the Chernoff bound, we get
$$\probability[\mathcal E_{v,i,D}]\le
\probability\left[|\AXA_i\cap N|<\aXa_i\gamma k-k^{0.9}\right]\le
\exp(-k^{0.3}) \;.$$
Thus, it follows by summing
the tail over at most $\Omega^*/\gamma\le k^{0.1}$ dense
spots containing $v$, that
\begin{equation}\label{itsunny}
\probability[v\in \bar V_{1,i}]\le k^{0.1}\cdot
\exp(-k^{0.3})\;.
\end{equation}
Now,~\eqref{itsraining} follows by linearity of expectation.
\end{proof}

Lemma~\ref{lem:randomSplit} is utilized for the purpose of our proof of
Theorem~\ref{thm:main} using the notion of proportional partition introduced in
Definition~\ref{def:proportionalsplitting} below.

\subsection{Common settings}\label{ssec:ExceptionalVertices}
Throughout Section~\ref{sec:configurations} and Section~\ref{sec:embed} we shall
be working with the setting that comes from Lemma~\ref{prop:LKSstruct}. In order
to keep statements of the subsequent lemmas reasonably short we introduce the
following setting.

\begin{setting}\label{commonsetting}
We assume that the constants $\Lambda,\Omega^*,\Omega^{**},k_0$ and $\alphaD,\gamma,\epsilon,\epsilon',\eta,\pi,\rho, \tau, d$ satisfy
\begin{align}
\label{eq:KONST}
\begin{split}
 \eta\gg\frac1{\Omega^*}\gg \frac1{\Omega^{**}}\gg\rho\gg\gamma\gg
d \ge\frac1{\Lambda}\ge\epsilon\ge
\pi\ge  \alphaD
\ge \epsilon'\ge
\nu\gg \tau \gg \frac{1}{k_0}>0\;,
\end{split}
\end{align}
and that $k\ge k_0$.
Here, by writing $c>a_1\gg
a_2\gg \ldots \gg a_\ell>0$ we mean that there exist non-decreasing functions
$f_i:(0,c)^i\rightarrow (0,c)$ ($i=1,\ldots,\ell-1$) such that for each $i\in
[\ell-1]$ we have $a_{i+1}<f_{i}(a_1,\ldots,a_i)$.  \footnote{The precise relation between the parameters can be
found on page~\pageref{pageref:PAR}, with $\Omega^{**}:= \Omega_{j+1}$ and
$\Omega^*:= \Omega_j$ for a certain index $j\in [g]$ to be specified in the course of the proof there.}

\medskip

Suppose that $G\in\LKSsmallgraphs{n}{k}{\eta}$ is given together with its
$(k,\Omega^{**},\Omega^*,\Lambda,\gamma,\epsilon',\nu,\rho)$-sparse
decomposition $$\class=(\HugeVertices, \clusters,\DenseSpots, \Gblack,
\Gexp,\smallatoms )\;, $$
with respect to the partition
$\{\smallvertices{\eta}{k}{G},\largevertices{\eta}{k}{G}\}$, and with respect to the avoiding threshold $\frac{\rho k}{100\Omega^*}$. We write 
\begin{equation}
 \index{mathsymbols}{*VA@$\largeintoatoms$}\index{mathsymbols}{*VA@$\clustersintoatoms$}
\largeintoatoms:=\shadow_{\Gcapt-\HugeVertices}(\smallatoms,\frac{\rho k}{100\Omega^*})\quad\mbox{and}\quad\clustersintoatoms:=\{C\in\clusters\::\:C\subset \largeintoatoms\}\;.
\label{eq:deflargeintoatoms}
\end{equation}
The graph \index{mathsymbols}{*Gblack@$\BGblack$}$\BGblack$ is the corresponding cluster graph. Let $\clustersize$
\index{mathsymbols}{*C@$\clustersize$}
be the size of an arbitrary cluster in $\clusters$.\footnote{The number $\clustersize$ is not defined when $\clusters=\emptyset$. However in that case $\clustersize$ is never actually used.}
 Let \index{mathsymbols}{*G@$\Gcapt$}$\Gcapt$ be the spanning subgraph of $G$ formed by the edges captured by $\class$. There are two $(\epsilon,d,\pi \clustersize)$-semiregular matchings $\mathcal M_A$
and $\mathcal M_B$ in $\GD$, with the following
properties (we abbreviate
$\XA:=\XA(\eta,\class, \mathcal M_A,\mathcal M_B)$,
$\XB:=\XB(\eta,\class, \mathcal M_A,\mathcal M_B)$, and
$\XC:=\XC(\eta,\class, \mathcal M_A,\mathcal M_B)$):
\begin{enumerate}
\item\label{commonsetting1}
$V(\mathcal M_A)\cap V(\mathcal
M_B)=\emptyset$,
\item\label{commonsetting1apul}
$V_1(\mathcal \M_B)\subset S^0$, where
\begin{equation}\label{eq:defS0}
S^0:=\smallvertices{\eta}{k}{G}\setminus
(V(\Gexp)\cup\smallatoms)\;,
\end{equation}
\item\label{commonsetting2} for each $(X,Y)\in \M_A\cup\M_B$, there is a dense
spot $(U,W; F)\in \DenseSpots$ with $X\subset U$ and $Y\subset W$, and further,
either $X\subset \smallvertices{\eta}{k}{G}$ or $X\subset
\largevertices{\eta}{k}{G}$, and $Y\subset \smallvertices{\eta}{k}{G}$ or
$Y\subset \largevertices{\eta}{k}{G}$,
\item\label{commonsetting3}
for each $X_1\in\V_1(\M_A\cup\M_B)$ there exists a cluster $C_1\in \clusters$ such that $X_1\subset C_1$, and for each $X_2\in\V_2(\M_A\cup\M_B)$ there exists $C_2\in \clusters\cup\{\largevertices{\eta}{k}{G}\cap \smallatoms\}$ such that $X_2\subset C_2$,
\item\label{commonsettingMgood} each pair of the semiregular matching
$\Mgood:=\{(X_1,X_2)\in\M_A\::\: X_1\cup X_2\subset \XA\}$ corresponds to an
edge in $\BGblack$,
\item\label{commonsettingXAS0}$e_{\Gcapt}\big(\XA,S^0\setminus V(\mathcal
M_A)\big)\le \gamma kn$,
\item\label{commonsetting4} $e_{\Gblack}(V(G)\setminus V(\M_A\cup \M_B))\le
\gamma^2kn$,
\item\label{commonsettingNicDoNAtom} for the semiregular matching \index{mathsymbols}{*Natom@$\NAtom$}
$\NAtom:=\{(X,Y)\in\M_A\cup\M_B\::\: (X\cup Y)\cap\smallatoms\not=\emptyset\}$ we have $e_{\Gblack}\big(V(G)\setminus
V(\M_A\cup \M_B),V(\NAtom)\big)\le \gamma^2 kn$,
\item \label{commonsetting:numbercaptured}$|E(G)\setminus E(\Gcapt)|\le 2\rho
kn$, 
\item\label{commonsetting:DenseCaptured}$|E(\GD)\setminus (E(\Gblack)\cup
E_G[\smallatoms,\smallatoms\cup\bigcup\clusters])|\le \frac 54\gamma kn$.
\end{enumerate}

We write
\begin{align}
\label{eq:defV+}\index{mathsymbols}{*V_+@$V_+$}
V_+&:=V(G)\setminus 
(S^0\setminus V(\mathcal M_A\cup\mathcal
M_B)) \\  \label{defV+eq}
& = \largevertices{\eta}{k}{G}\cup
V(\Gexp)\cup\smallatoms\cup V(\mathcal M_A\cup\mathcal
M_B)\;,\\
\label{eq:defLsharp}\index{mathsymbols}{*L@$L_\#$}
L_\#&:=\largevertices{\eta}{k}{G}\setminus\largevertices{\frac9{10}\eta}{k}{\Gcapt}\;\mbox{,
and}\\ 
\label{eq:defVgood}\index{mathsymbols}{*Vgood@$\Vgood$} 
\Vgood&:=V_+\setminus (\HugeVertices\cup L_\#)\;,\\
\label{eq:defYA}
\YA&:=
 \shadow_{\Gcapt}\left(V_+\setminus L_\#, (1+\frac\eta{10})k\right) \setminus \shadow_{G-\Gcapt}\left(V(G),\frac\eta{100} k\right)\;,
 \index{mathsymbols}{*YA@$\YA$}\\
\label{eq:defYB} 
\YB&:=
 \shadow_{\Gcapt}\left(V_+\setminus L_\#, (1+\frac\eta{10})\frac k2\right) \setminus
 \shadow_{G-\Gcapt}\left(V(G),\frac\eta{100} k\right)\;,\index{mathsymbols}{*YB@$\YB$}\\
\WantiC&:=(\XA\cup\XB)\cap \shadow_G\left(\HugeVertices, \frac{\eta}{100}
k \right) \index{mathsymbols}{*V@$\WantiC$}\label{eq:defWantiC}\;,
\\ \nonumber
\gPatoms&:=\shadow_{\Gblack}(V(\NAtom),\gamma k)\setminus V(\M_A\cup\M_B)\;,
\index{mathsymbols}{*Pa@$\gPatoms$}\\
\nonumber
\gP_1&:=\shadow_{\Gblack}(V(G)\setminus V(\M_A\cup\M_B),\gamma k)\setminus V(\M_A\cup\M_B)\;,
\index{mathsymbols}{*P1@$\gP_1$}
\\ \nonumber
\gP&:=(\XA\setminus \YA)\cup((\XA\cup \XB)\setminus \YB)\cup\WantiC
\cup
L_\sharp \cup \gP_1\\
\nonumber
&~~~~~\cup\shadow_{\GD\cup\Gcapt}(\WantiC\cup
L_\sharp\cup\gPatoms\cup\gP_1,\frac{\eta^2 k}{10^5})\;,
\index{mathsymbols}{*P@$\gP$}
\\
\nonumber
\gP_2&:=\XA\cap \shadow_{\Gcapt}(S^0\setminus
V(\M_A),\sqrt\gamma k)\;,
\index{mathsymbols}{*P2@$\gP_2$}\\
\nonumber
\gP_3&:=\XA\cap\shadow_{\Gcapt}(\XA, \eta^3k/10^3)\;,\\
\label{def:Fcover}
 \mathcal F&:=\{C\in \V(\M_A):C\subset \XA\}\cup\V_1(\M_B)\;.
\end{align}
\end{setting}

The vertex set $\YA$ in Setting~\ref{commonsetting} should be regarded as 
$\XA$  cleaned from rare irregularities. Indeed, as it
turns out most of the vertices from $\XA$ are contained in $\YA$. Likewise,
$\YB$ should be regarded as a cleaned version of $\XA\cup\XB$. These properties are stated in
Lemma~\ref{lem:YAYB} below. 

On the interface between Lemma~\ref{outerlemma} and Lemma~\ref{lem:ConfWhenMatching} we shall need to work with a semiregular matching which is formed of only those edges $E(\DenseSpots)$ which are either incident with $\smallatoms$, or included in $\Gblack$. The following lemma provides us with an appropriate ``cleaned version of $\DenseSpots$''. The notion of being absorbed adapts in a straightforward way to two families of dense spots: a family of dense spots ${\DenseSpots}_1$ \emph{is absorbed} by another family ${\DenseSpots}_2$ if for every $D_1\in{\DenseSpots}_2$ there exists $D_2\in{\DenseSpots}_2$ such that $D_1$ is contained in $D_2$ as a subgraph.
\begin{lemma}\label{lem:clean-spots}
Assume Setting~\ref{commonsetting}. Then there exists a family
\index{mathsymbols}{*D@$\DenseSpots_\class$}
$\DenseSpots_\class$ of edge-disjoint
$(\gamma^3 k/4,\gamma/2)$-dense spots absorbed by $\DenseSpots$ such that
\begin{enumerate}
 \item $|E(\DenseSpots)\setminus E(\DenseSpots_\class)|\le \rho kn$, and
 \item $E(\DenseSpots_\class)\subset E(\Gblack)\cup E(G[\smallatoms, \smallatoms\cup\bigcup\clusters])$.
\end{enumerate}
\end{lemma}
The proof of Lemma~\ref{lem:clean-spots} is a warm-up for proofs in Section~\ref{ssec:cleaning}.
\begin{proof}[Proof of Lemma~\ref{lem:clean-spots}]
We discard those dense spots $D\in\DenseSpots$ for which 
\begin{align}\label{eq:dontviolate}
 \big|E(D)\setminus (E(\Gblack)\cup E(G[\smallatoms,
 \smallatoms\cup\bigcup\clusters])\big|\ge \sqrt{\gamma}e(D)\;.
\end{align}
 For each remaining dense spot $D\in\DenseSpots$ we show below how to extract a $(\gamma^3
 k/4,\gamma/2)$-dense spot $D'\subset D$ with $e(D')\ge
 (1-2\sqrt{\gamma})e(D)$  and $E(D)\subset E(\Gblack)\cup E(G[\smallatoms, \smallatoms\cup\bigcup\clusters])$. Let $\DenseSpots_\class$ be the set of all thus obtained $D'$. This way we ensure Property~2, and we also have Property~1, since
\begin{align*}
|E(\DenseSpots)\setminus E(\DenseSpots_\class)|&\le
\frac1{\sqrt{\gamma}}\left|E(\DenseSpots)\setminus\big(E(\Gblack)\cup
E(G[\smallatoms,\smallatoms\cup\bigcup\clusters])\big)\right|+2\sqrt{\gamma}\cdot
e(\DenseSpots)\\ \JUSTIFY{by
S\ref{commonsetting}\eqref{commonsetting:DenseCaptured}, and as
$e(\DenseSpots)\le e(G)\le kn$}&\le 3\rho kn\;.
\end{align*}

We now show how to extract a $(\gamma^3 k/4,\gamma/2)$-dense spot $D'\subset D$
with $e(D')\ge (1-2\sqrt{\gamma})e(D)$ and $E(D)\subset E(\Gblack)\cup E(G[\smallatoms, \smallatoms\cup\bigcup\clusters])$ from any spot $D\in\DenseSpots$ which does
not satisfy~\eqref{eq:dontviolate}. Let $D=(A,B; F)$, and $a:=|A|$, $b:=|B|$. As
$D$ is $(\gamma k, \gamma)$-dense, we have $a,b\ge \gamma k$. First, we 
discard from $D$ all edges not contained in $E(\Gblack)\cup
E(G[\smallatoms,\smallatoms\cup\bigcup\clusters])$ to obtain a dense spot
$D^*\subset D$ with $e(D^*)\ge (1-\sqrt{\gamma})e(D)$. Next, we perform a sequential cleaning procedure in $D^*$. As long as there are such vertices, discard from $A$ any vertex whose current degree is less
than $\gamma^2b/4$, and discard from $B$ any vertex whose current degree is less
than $\gamma^2a/4$. When this procedure terminates,  the
resulting graph $D'=(A',B'; F')$ has $\mindeg_{D'}(A')\ge \gamma^2 b/4\ge
\gamma^3 k/4$ and  $\mindeg_{D'}(B')\ge \gamma^3 k/4$. Note that we
deleted at most $a\cdot \gamma^2 b/4+b\cdot\gamma^2 a/4$ edges out of the at
least $(1-\sqrt{\gamma})e(D)$ edges of $D^*$. This means that $e(D')\ge (1-\sqrt{\gamma})e(D)-\gamma^2ab/2\ge (1-2\sqrt{\gamma})e(D)$, as desired. Thus we also have  the required density of $D'$, namely 
$\density_{D'}(A',B')\ge (1-2\sqrt{\gamma})\gamma\ge \gamma/2$. 
\end{proof}

\bigskip
In some cases, we shall in addition partition the set $V(G)$ into three sets as in
Lemma~\ref{lem:randomSplit}. This motivates the following
definition.
\begin{definition}[\bf Proportional
splitting]\label{def:proportionalsplitting}\index{general}{proportional splitting}
Let $\proporce{0},\proporce{1},\proporce{2}>0$ be three positive reals with
$\sum_i\proporce{i}\le 1$. Under Setting~\ref{commonsetting}, suppose that
$\colouringpartition$ is a partition of $V(G)\setminus\HugeVertices$ which satisfies assertions of Lemma~\ref{lem:randomSplit} with parameter
$p_\PARAMETERPASSING{L}{lem:randomSplit}:=10$ for graph
$G^*_\PARAMETERPASSING{L}{lem:randomSplit}:=(\Gcapt-\HugeVertices)\cup\GD$ (here,
by the union, we mean union of the edges), bounded decomposition $( \clusters,\DenseSpots,
\Gblack, \Gexp, \smallatoms )$, matching
$\M_\PARAMETERPASSING{L}{lem:randomSplit}:=\M_A\cup\M_B$, sets $\VXV_1:=\Vgood,
\VXV_2:=\XA\setminus(\HugeVertices\cup \gP)$, $\VXV_3:=\XB\setminus \gP$,
$\VXV_4:=V(\Gexp)$, $\VXV_5:=\smallatoms$, $\VXV_6:=\largeintoatoms$, $\VXV_7:=\gPatoms$, $\VXV_8:=\largevertices{\eta}{k}{G}$, $\VXV_9:=L_\sharp$, $\VXV_{10}:=\WantiC$ and reals
$\aXa_1:=\proporce{0},\aXa_2:=\proporce{1}$, $\aXa_3:=\proporce{2}$,
$\aXa_4:=\ldots\aXa_{10}=0$. Note that by
Lemma~\ref{lem:randomSplit}\eqref{It:H7} we have that $\colouringpartition$ is
a partition of $V(G)\setminus \HugeVertices$. We call $\colouringpartition$
\emph{proportional $(\proporce{0}:\proporce{1}:\proporce{2})$ splitting}.

We refer to properties of the proportional
$(\proporce{0}:\proporce{1}:\proporce{2})$  splitting $\colouringpartition$
using the numbering of Lemma~\ref{lem:randomSplit}; for example,
``Definition~\ref{def:proportionalsplitting}\eqref{It:HproportionalSizes}''
tells us among other things that $|(\XA\setminus\gP)\cap
\colouringp{0}|\ge\proporce{0}|\XA\setminus(\gP\cup\HugeVertices)|-n^{0.9}$.
\end{definition}
\begin{setting}\label{settingsplitting}
Under Setting~\ref{commonsetting}, suppose that we are given a
proportional 
\index{mathsymbols}{*Pa@$\proporce{i}$}
$(\proporce{0}:\proporce{1}:\proporce{2})$
splitting \index{mathsymbols}{*P@$\colouringp{i}$}
$\colouringpartition$ of $V(G)\setminus\HugeVertices$. We assume
that 
\begin{equation}\label{eq:proporcevelke}
\proporce{0},\proporce{1},\proporce{2}\ge\frac{\eta}{100}\;.
\end{equation}
Let\index{mathsymbols}{*V@$\exceptVertSplit$}\index{mathsymbols}{*V@$\exceptSemSplit$}\index{mathsymbols}{*V@$\exceptClustSplit$}
$\exceptVertSplit,\exceptSemSplit,\exceptClustSplit$ be the exceptional sets as in Definition~\ref{def:proportionalsplitting}\eqref{It:H1}.

We write \index{mathsymbols}{*F@$\shadowsplit$}
\begin{equation}\label{def:shadowsplit}
\shadowsplit:=\shadow_{\GD}\left(\bigcup 
\exceptSemSplit\cup\bigcup \exceptSemSplit^*\cup \bigcup \exceptClustSplit,\frac{\eta^2k}{10^{10}}\right)\;,
\end{equation} where \index{mathsymbols}{*V@$\exceptSemSplit^*$}$\exceptSemSplit^*$ are the partners of $\exceptSemSplit$ in $\M_A\cup\M_B$.

We have 
\begin{equation}\label{eq:boundShadowsplit}
|\shadowsplit|\le \epsilon n\;.
\end{equation}

For an arbitrary
set $U\subset V(G)$ and for $i\in\{0,1,2\}$ we write \index{mathsymbols}{*U@$U\colouringpI{i}$}$U\colouringpI{i}$ for
the set $U\cap\colouringp{i}$. 

For each $(X,Y)\in\M_A\cup\M_B$ such that $X,Y\notin\exceptSemSplit$
we write $(X,Y)\colouringpI{i}$ for an arbitrary fixed pair $(X'\subset
X,Y'\subset Y)$ with the property that
$|X'|=|Y'|=\min\{|X\colouringpI{i}|,|Y\colouringpI{i}|\}$. We extend this notion
of restriction to an arbitrary semiregular matching $\mathcal N\subset
\M_A\cup\M_B$ as follows. We set\index{mathsymbols}{*N@$\mathcal N\colouringpI{i}$}
$$\mathcal
N\colouringpI{i}:=\big\{(X,Y)\colouringpI{i}\::\:(X,Y)\in\mathcal
N \text{ with } X,Y\notin\exceptSemSplit\big\}\;.$$
\end{setting}

The next lemma provides some simple properties of a restriction of a
semiregular matching.

\begin{lemma}\label{lem:RestrictionSemiregularMatching}
Assume Setting~\ref{settingsplitting}. Then for each $i\in\{0,1,2\}$, and for
each $\mathcal N\subset\M_A\cup\M_B$ we have that $\mathcal N\colouringpI{i}$ is
a $(\frac{400\epsilon}{\eta},\frac
d2,\frac{\eta\pi}{200}\clustersize)$-semiregular matching satisfying
\begin{equation}\label{eq:restrictedmatchlarge}
|V(\mathcal N\colouringpI{i})|\ge
\proporce{i}|V(\mathcal N)|-2k^{-0.05}n\;.
\end{equation}
Moreover for all $v\not\in \shadowsplit$ and for all $i=0,1,2$ we have 
$\deg_{\GD}(v, V(\mathcal N)\colouringpI{i}\setminus V(\mathcal
N\colouringpI{i}))\le \frac {\eta^2k}{10^5}$. 
\end{lemma}
\begin{proof}
Let us consider an arbitrary pair $(X,Y)\in\mathcal N$. By
Definition~\ref{def:proportionalsplitting}\eqref{It:H3} we have
\begin{equation}\label{eq:lossesOnePair}
|X\colouringpI{i}|\ge\proporce{i}|X|-k^{0.9}\geByRef{eq:proporcevelke}
\frac{\eta}{200}|X|\quad\mbox{and}\quad|Y\colouringpI{i}|\ge\proporce{i}|Y|-k^{0.9}\geByRef{eq:proporcevelke}
\frac{\eta}{200}|Y|\;.
\end{equation}
In particular, Fact~\ref{fact:BigSubpairsInRegularPairs} gives that
$(X,Y)\colouringpI{i}$ is a $400\epsilon/\eta$-regular pair of density at least
$d/2$.

We now turn to~\eqref{eq:restrictedmatchlarge}. The total order of pairs
$(X,Y)\in\mathcal N$ excluded entirely from $\mathcal N\colouringpI{1}$ is
at most $2\exp(-k^{0.1})n<k^{-0.05}n$ by
Definition~\ref{def:proportionalsplitting}\eqref{It:H1}. Further, for each $(X,Y)\in\mathcal N$ whose part is included to $\mathcal
N\colouringpI{1}$ we have by that
$|V((X,Y)\colouringpI{i})|\ge\proporce{i}(|X|+|Y|)-2k^{0.9}$
by~\eqref{eq:lossesOnePair}. As $|\mathcal N|\le \frac{n}{2k^{0.95}}$, and~\eqref{eq:restrictedmatchlarge} follows.

For the moreover part, note that by Fact~\ref{fact:sizedensespot} and Fact~\ref{fact:boundedlymanyspots}
\begin{equation*}
\deg_{\GD}(v, V(\mathcal N)\colouringpI{i}\setminus V(\mathcal
N\colouringpI{i}))\le \frac
{\eta^2k}{10^{10}}+\frac {(\Omega^*)^2}{\pi\nu\gamma^2}\cdot 3k^{0.9}\le \frac {\eta^2k}{10^5}\;.
\end{equation*}
\end{proof}

\smallskip
\HIDDENTEXT{there were lemmas lem:clSpadlychA, lem:clSpadlychB, now under TRIVIALCLEANING}

The following lemma gives a useful bound on some of the sets defined on
page~\pageref{eq:defLsharp}.

\begin{lemma}\label{lem:YAYB}
Suppose we are in Setting~\ref{commonsetting}. Suppose that all but at most
$\beta kn$ edges are captured by $\class$. 
Then,\begin{align}
\label{eq:Lsharp-small}
|L_\#|&\le \frac{20\beta}{\eta}n\\
\label{eq:XAYA}
|\XA\setminus\YA|&\le \frac{600\beta}{\eta^2}n\;\mbox{,
and}\\\label{eq:XBYB} 
|(\XA\cup \XB)\setminus\YB|&\le
\frac{600\beta}{\eta^2}n \;.
\end{align}

Further, if $e_G(\HugeVertices,\XA\cup\XB)\le\tilde \beta kn$ then
\begin{align}\label{eq:WantiCbound}
|\WantiC|\le \frac{100\tilde\beta n}{\eta}\;. 
\end{align}
\end{lemma}
\begin{proof}
Let $W_1:=\{v\in V(G)\::\:\deg_G(v)-\deg_{\Gcapt}(v)\ge \eta k/100\}$. We have $|W_1|\le \frac{200\beta}{\eta}n$. 

Observe that $L_\#$ sends out at most
$(1+\frac9{10}\eta)k|L_\#|<\frac{40\beta}{\eta}kn$ edges in $\Gcapt$. Let $W_2:=\{v\in V(G)\::\:\deg_{\Gcapt}(v,L_\#)\ge \eta k/10\}$. We have $|W_2|\le \frac{400\beta}{\eta^2}n$.

Let $W_3:=\{v\in \XA\::\: \deg_{\Gcapt}(v,S^0\setminus V(\M_A))\ge \sqrt{\gamma} k\}$. By Property~\ref{commonsettingXAS0} we have $|W_3|\le \sqrt\gamma n$.

Now, observe that $\XA\setminus\YA\subset W_1\cup W_2\cup W_3$, and $\XB\setminus \YB\subset W_1\cup W_2$. 

The bound~\eqref{eq:WantiCbound} follows in a straightforward way.
\HIDDENTEXT{A lengthy proof removed and put into HIDDENTEXT.TEX under LENGTHY.}
\end{proof}

We finish this section with an auxiliary result which will only be used later in the proofs
of Lemmas~\ref{lem:ConfWhenNOTCXAXB} and~\ref{lem:ConfWhenMatching}.

\begin{lemma}\label{lem:propertyYA12}
Assume Settings~\ref{commonsetting} and~\ref{settingsplitting}.
We have that for $i=1,2$
\begin{align}\label{eq:trivka}
\XA\colouringpI{0}\setminus(\gP\cup \shadowsplit)\subseteq
\colouringp{0}\setminus \left(\shadowsplit\cup \shadow_{\GD}(\WantiC, \frac
{\eta^2 k}{10^5})\right)\;,\\
\label{eq:propertyYA12cA} 
\mindeg_{\Gcapt}\left(\XA\setminus(\gP\cup\exceptVertSplit),\Vgood\colouringpI{i} \right)\ge
\proporce{i}(1+\frac{\eta}{20})k\;,\\
\label{eq:propertyYA12cB2}
\mindeg_{\Gcapt}\left(\XB\setminus(\gP\cup\exceptVertSplit),\Vgood\colouringpI{i} \right)\ge
\proporce{i}(1+\frac{\eta}{20})\frac k2\;\mbox{, and}\\
\label{eq:propertyYA12cB3}
\maxdeg_{\Gcapt}\left(\XA\setminus(\gP_2\cup \gP_3),\bigcup
\mathcal F\right)\le \frac{3\eta^3}{2\cdot 10^3} k\;.
\end{align}
Moreover, $\mathcal F$ defined in~\eqref{def:Fcover} is an $(\M_A\cup\M_B)$-cover.
\end{lemma}

\begin{proof}
The definition of $\gP$ gives~\eqref{eq:trivka}.

For~\eqref{eq:propertyYA12cA} and~\eqref{eq:propertyYA12cB2}, assume that $i=2$ (the other case is analogous).
Observe that 
\begin{align*}
\mindeg_{\Gcapt}&\left(\YA\setminus
(\WantiC\cup\exceptVertSplit),\Vgood\colouringpI{2}\right)\\
\JUSTIFY{by Def~\ref{def:proportionalsplitting}\eqref{It:H5}}
&\ge\proporce{2}\cdot \mindeg_{\Gcapt}(\YA\setminus
\WantiC,\Vgood)-k^{0.9}\\ 
\JUSTIFY{by
\eqref{eq:defVgood}}
&\ge\proporce{2}\cdot \big(\mindeg_{\Gcapt}(\YA,V_+\setminus
L_\sharp)-\maxdeg_{\Gcapt}(\YA\setminus\WantiC,\HugeVertices)\big)-k^{0.9}\\ 
\JUSTIFY{by~\eqref{eq:defYA}, \eqref{eq:defWantiC}}
&\ge \proporce{2}\cdot \left((1+\frac{\eta}{10})k-\frac{\eta k}{100}\right)-k^{0.9}\\ 
\JUSTIFY{by~\eqref{eq:KONST}, \eqref{eq:proporcevelke}}
&\ge \proporce{2}\cdot (1+\frac{\eta}{20})k\;,
\end{align*}
which proves~\eqref{eq:propertyYA12cA}, as $\XA\setminus (\gP\cup
 \exceptVertSplit)\subseteq \YA\setminus (\WantiC\cup \exceptVertSplit)$.
Similarly, we obtain that $$\mindeg_{\Gcapt}\left(\YB\setminus
(\WantiC\cup\exceptVertSplit),\Vgood\colouringpI{2}\right)\ge
\proporce{2}(1+\frac{\eta}{20})\frac k2\;,$$ which proves~\eqref{eq:propertyYA12cB2}.

We have $\maxdeg_{\Gcapt}(\XA\setminus \gP_3,\XA)< \frac{\eta^3}{10^3} k$, and 
$\maxdeg_{\Gcapt}(\XA\setminus \gP_2,S^0\setminus V(\M_A))<\sqrt\gamma k$.
Thus~\eqref{eq:propertyYA12cB3} follows from
Setting~\ref{commonsetting}\eqref{commonsetting1apul}
and by~\eqref{eq:KONST}.

For the ``moreover'' part, it
suffices to prove that $\{C\in \V(\M_A):C\subset \XA\}=\mathcal F\sm\V_1(\M_B)$ is an $\M_A$-cover. Let  $(T_1,T_2)\subset \M_A$. As $G\in\LKSsmallgraphs{n}{k}{\eta}$, we have
by Setting~\ref{commonsetting}\eqref{commonsetting2} that for some
$i\in\{1,2\}$, $T_i$ is contained in $\largevertices{\eta}{k}{G}$. Then
by Setting~\ref{commonsetting}\eqref{commonsetting1}, $T_i\subset \XA$, as desired.
\end{proof}

\subsection{Types of configurations}\label{ssec:TypesConf}
We can now define the following preconfigurations~$\mathbf{(\clubsuit)}$,
$\mathbf{(\heartsuit1)}$, $\mathbf{(\heartsuit2)}$, $\mathbf{(exp)}$, and
$\mathbf{(reg)}$, and the configurations\footnote{The
word ``configuration'' is used for a final structure in a graph which is suitable for embedding purposes while ``preconfigurations'' are building blocks for configurations.} $\mathbf{(\diamond1)}$--$\mathbf{(\diamond10)}$. It will follow from results from other sections that at least one of the configurations $\mathbf{(\diamond1)}$--$\mathbf{(\diamond10)}$ appears in each graph
$\LKSgraphs{n}{k}{\eta}$. More precisely, after getting the ``rough structure'' in Lemma~\ref{prop:LKSstruct} we get one of the configurations $\mathbf{(\diamond1)}$--$\mathbf{(\diamond10)}$ from Lemma~\ref{outerlemma}. The latter lemma reduces the situation to one of three cases which are then dealt with in Lemmas~\ref{lem:ConfWhenCXAXB}--\ref{lem:ConfWhenMatching} separately. Then, in Section~\ref{sec:embed}, we provide with an
embedding for a given tree $T_\PARAMETERPASSING{T}{thm:main}\in\treeclass{k}$.

We now give a brief overview of these configurations. Configuration~$\mathbf{(\diamond1)}$ covers the easy and lucky case when $G$ is contains a subgraph with high minimum degree. A very simple tree-embedding strategy similar to the greedy strategy turns out to work in this case.

The purpose of
Preconfiguration~$\mathbf{(\clubsuit)}$ is to utilize vertices
of~$\HugeVertices$. On one hand these vertices seem very powerful because of
their large degree, on the other hand the edges incident with them are very
unstructured. Therefore Preconfiguration~$\mathbf{(\clubsuit)}$  distills some
structure in~$\HugeVertices$. This preconfiguration is then a part of
configurations~$\mathbf{(\diamond2)}$--$\mathbf{(\diamond5)}$ which deal with the
case when $\HugeVertices$ is substantial. Indeed,
Lemma~\ref{lem:ConfWhenCXAXB} asserts that whenever $\HugeVertices$ is incident with many edges in the setting provided by Lemma~\ref{prop:LKSstruct}, at least one of
configurations~$\mathbf{(\diamond1)}$--$\mathbf{(\diamond5)}$ must occur. 

The cases when the number of edges incident with $\HugeVertices$ is negligible
are covered by configurations $\mathbf{(\diamond6)}$--$\mathbf{(\diamond10)}$. More precisely, in this setting Lemma~\ref{outerlemma} transforms the output structure of Lemma~\ref{prop:LKSstruct} into an input structure for either Lemma~\ref{lem:ConfWhenNOTCXAXB} or Lemma~\ref{lem:ConfWhenMatching}. These lemmas then assert that indeed one of the Configurations $\mathbf{(\diamond6)}$--$\mathbf{(\diamond10)}$ must occur. 
The configurations~$\mathbf{(\diamond6)}$--$\mathbf{(\diamond8)}$ involve
combinations of one of the
two preconfigurations $\mathbf{(\heartsuit1)}$ and $\mathbf{(\heartsuit2)}$ and
one of the two preconfigurations $\mathbf{(exp)}$ and
$\mathbf{(reg)}$. The idea here is that the knags are embedded
using the structure of $\mathbf{(exp)}$ or $\mathbf{(reg)}$ (whichever
applicable), the internal shrubs are embedded using the structure which is
specific to each of the
configurations~$\mathbf{(\diamond6)}$--$\mathbf{(\diamond8)}$, and the end
shrubs are embedded using the structure of $\mathbf{(\heartsuit1)}$ or
$\mathbf{(\heartsuit2)}$. The configuration $\mathbf{(\diamond10)}$ is very similar to the structures obtained in the dense setting
in~\cite{PS07+,HlaPig:LKSdenseExact} (see Section~\ref{ssec:EmbedOverview10} for a discussion), and $\mathbf{(\diamond9)}$ should be
considered as half-way towards it.

The reader may find it helpful to compare the definitions of the configurations
with Section~\ref{ssec:embeddingOverview} where an overview is given how these
configurations are used to embed the tree $T_\PARAMETERPASSING{T}{thm:main}$.

\smallskip
Some of the  configurations below are accompanied with parameters in the parentheses; note that we do not make explicit those numerical parameters which are inherited from Setting~\ref{commonsetting}.

\bigskip We start off by giving definitions of Configuration~$\mathbf{(\diamond1)}$. This is a very easy configuration in which a modification of the greedy tree-embedding strategy works.
\begin{definition}[\bf Configuration~$\mathbf{(\diamond1)}$]\index{mathsymbols}{**1@$\mathbf{(\diamond1)}$}
We say that a graph $G$ is in \emph{Configuration~$\mathbf{(\diamond1)}$} if
there exists a non-empty bipartite graph $H\subset G$ with $\mindeg_G(V(H))\ge k$ and $\mindeg(H)\ge k/2$.
\end{definition}

\bigskip We now introduce the configurations $\mathbf{(\diamond2)}$--$\mathbf{(\diamond5)}$ which make use of the set $\HugeVertices$. These configurations build on Preconfiguration $\mathbf{(\clubsuit)}$. Figure~\ref{fig:DIAMOND25overview} shows common features of the configurations $\mathbf{(\diamond2)}$--$\mathbf{(\diamond5)}$.

\begin{definition}[\bf Preconfiguration 
$\mathbf{(\clubsuit)}$]\index{mathsymbols}{***@$\mathbf{(\clubsuit)}$}
\label{def:PreClub}
Suppose that we are in
Setting~\ref{commonsetting}. We say that the graph $G$ is in
\emph{Preconfiguration~$\mathbf{(\clubsuit)}(\Omega^\star)$} if the following
conditions are met.
$G$ contains non-empty sets $L''\subset L'\subset
\largevertices{\frac9{10}\eta}{k}{\Gcapt}\setminus\HugeVertices$, and a non-empty set $\HugeVertices'\subset
\HugeVertices$ such that
\begin{align}\label{eq:clubsuitCOND1}
\maxdeg_{\Gcapt} (L',\HugeVertices\setminus \HugeVertices')&<\frac{\eta
k}{100} \;\mbox{,}\\ 
\label{eq:clubsuitCOND2}
\mindeg_{\Gcapt}(\HugeVertices',L'')&\ge \Omega^\star k\;\mbox{, and}\\
\label{eq:clubsuitCOND3}
\maxdeg_{\Gcapt}(L'',\largevertices{\frac9{10}\eta}{k}{\Gcapt}\setminus(\HugeVertices\cup
L'))&\le\frac{\eta k}{100}\;.
\end{align}
\end{definition}

\begin{definition}[\bf Configuration
$\mathbf{(\diamond2)}$]\index{mathsymbols}{**2@$\mathbf{(\diamond2)}$}Suppose that we are in Setting~\ref{commonsetting}. We say that the graph $G$ is in
\emph{Configuration $\mathbf{(\diamond2)}(\Omega^\star,
\tilde\Omega,\beta)$} if the following
conditions are met.

 The triple $L'',L',\HugeVertices'$ witnesses preconfiguration
$\mathbf{(\clubsuit)}(\Omega^\star)$ in $G$. There exist a
non-empty set $\HugeVertices''\subset \HugeVertices'$, a set $V_1\subset V(\Gexp)\cap\YB\cap L''$, and a set $V_2\subset V(\Gexp)$ with the following properties.
\begin{align*}
\mindeg_{\Gcapt}(\HugeVertices'',V_1)&\ge\tilde\Omega k\;\\
\mindeg_{\Gcapt}(V_1,\HugeVertices'')&\ge \beta k\;,\\
\mindeg_{\Gexp}(V_1,V_2)&\ge \beta k\;,\\
\mindeg_{\Gexp}(V_2,V_1)&\ge \beta k\;.
\end{align*}
\end{definition}

\begin{definition}[\bf Configuration
$\mathbf{(\diamond3)}$]\index{mathsymbols}{**3@$\mathbf{(\diamond3)}$}
\label{def:CONF3}
Suppose that we are in
Setting~\ref{commonsetting}. We say that the graph $G$ is in
\emph{Configuration $\mathbf{(\diamond3)}(\Omega^\star,
\tilde\Omega,\zeta,\delta)$} if the following
conditions are met.

 The triple $L'',L',\HugeVertices'$ witnesses preconfiguration
$\mathbf{(\clubsuit)}(\Omega^\star)$ in $G$. There exist a
non-empty set $\HugeVertices''\subset \HugeVertices'$, a set $V_1\subset \smallatoms\cap \YB\cap L''$, and a set $V_2\subset V(G)\setminus \HugeVertices$ such that the
following properties are satisfied.
\begin{align}
\nonumber
\mindeg_{\Gcapt}(\HugeVertices'',V_1)&\ge \tilde \Omega k\;,\\
\nonumber
\mindeg_{\Gcapt}(V_1,\HugeVertices'')&\ge \delta k\;,\\
\label{eq:WHtc}
\maxdeg_{\GD}(V_1, V(G)\setminus
(V_2\cup \HugeVertices))&\le \zeta k\;,\\ 
\label{confi3theothercondi}
\mindeg_{\GD}(V_2,V_1)&\ge \delta k\;.
\end{align}
\end{definition}

\begin{definition}[\bf Configuration
$\mathbf{(\diamond4)}$]\index{mathsymbols}{**4@$\mathbf{(\diamond4)}$}
\label{def:CONF4}
Suppose that we are in
Setting~\ref{commonsetting}. We say that the graph $G$ is in
\emph{Configuration
$\mathbf{(\diamond4)}(\Omega^\star, \tilde\Omega,\zeta,\delta)$} if the
following conditions are met.

 The triple $L'',L',\HugeVertices'$ witnesses preconfiguration
$\mathbf{(\clubsuit)}(\Omega^\star)$ in $G$. There exists a
non-empty set $\HugeVertices''\subset \HugeVertices'$, sets $V_1\subset \YB\cap L''$,
$\smallatoms'\subset \smallatoms$, and $V_2\subset V(G)\setminus \HugeVertices$ with the following properties
\begin{align}
\mindeg_{\Gcapt}(\HugeVertices'',V_1)&\ge \tilde\Omega k\;,\notag \\
\mindeg_{\Gcapt}(V_1,\HugeVertices'')&\ge \delta k\;,\notag \\
\mindeg_{\Gcapt\cup\GD}(V_1,\smallatoms')&\ge \delta k\;,\label{confi4:3} \\
\mindeg_{\Gcapt\cup\GD}(\smallatoms',V_1)&\ge \delta k\;,\label{confi4:4} \\
\mindeg_{\Gcapt\cup\GD}(V_2,\smallatoms')&\ge \delta k\;,\label{confi4othercondi} \\
\maxdeg_{\Gcapt\cup\GD}(\smallatoms',V(G)\setminus
(\HugeVertices\cup V_2))&\le \zeta k\;. \label{confi4lastcondi}
\end{align}
\end{definition}

\begin{definition}[\bf Configuration
$\mathbf{(\diamond5)}$]\index{mathsymbols}{**5@$\mathbf{(\diamond5)}$}
\label{def:CONF5}
Suppose that we are in
Setting~\ref{commonsetting}. We say that the graph $G$ is in
\emph{Configuration
$\mathbf{(\diamond5)}(\Omega^\star,\tilde\Omega,\delta,\zeta,\tilde\pi)$} if the
following conditions are met.

 The triple $L'',L',\HugeVertices'$ witnesses preconfiguration
 $\mathbf{(\clubsuit)}(\Omega^\star)$ in $G$. There exists a non-empty set $\HugeVertices''\subset \HugeVertices'$, and a set $V_1\subset (\YB\cap
L''\cap \bigcup\clusters)\setminus V(\Gexp)$ such that the following conditions are fulfilled.
\begin{align}
\mindeg_{\Gcapt}(\HugeVertices'',V_1)&\ge \tilde\Omega k\;, \\
\mindeg_{\Gcapt}(V_1,\HugeVertices'')&\ge \delta k\;,
\label{eq:diamond5P2}\\
\mindeg_{\Gblack}(V_1)&\ge \zeta k\;.\label{confi5last}
\end{align}
Further, we have 
\begin{equation}\label{eq:diamond5P4}
\mbox{$C\cap V_1=\emptyset$ or $|C\cap V_1|\ge \tilde\pi|C|$}
\end{equation} for every
$C\in\clusters$.
\end{definition}

In remains to introduce
configurations~$\mathbf{(\diamond6)}$--$\mathbf{(\diamond10)}$. In these
configurations the set $\HugeVertices$ is not utilized. All these
configurations make use of Setting~\ref{settingsplitting}, i.e., the set
$V(G)\setminus \HugeVertices$ is partitioned into three sets
$\colouringp{0},\colouringp{1}$ and $\colouringp{2}$. The purpose of $\colouringp{0},\colouringp{1}$ and $\colouringp{2}$
is to make possible to embed the knags, the internal shrubs, and the
end shrubs of $T_\PARAMETERPASSING{T}{thm:main}$, respectively. Thus the
parameters $\proporce{0}, \proporce{1}$ and $\proporce{2}$ are chosen proportionally to
the sizes of these respective parts of $T_\PARAMETERPASSING{T}{thm:main}$.
A summary picture for Configurations $\mathbf{(\diamond6)}$--$\mathbf{(\diamond7)}$, $\mathbf{(\diamond8)}$, and $\mathbf{(\diamond9)}$ is given in Figures~\ref{fig:DIAMOND67overview}, \ref{fig:DIAMOND8} and~\ref{fig:DIAMOND9}, respectively.

We first introduce four preconfigurations $\mathbf{(\heartsuit 1)}$,
$\mathbf{(\heartsuit 2)}$, $\mathbf{(exp)}$ and $\mathbf{(reg)}$ which are
building bricks for
configurations~$\mathbf{(\diamond6)}$--$\mathbf{(\diamond9)}$.
The preconfigurations $\mathbf{(\heartsuit 1)}$ and $\mathbf{(\heartsuit 2)}$ will
be used for embedding end shrubs of a fine partition of the tree
$T_\PARAMETERPASSING{T}{thm:main}$, and preconfigurations $\mathbf{(exp)}$ and
$\mathbf{(reg)}$ will be used for embedding its knags.

 An \emph{$\M$-cover}\index{general}{cover}\index{mathsymbols}{*COVER@$\M$-cover} of a
semiregular matching $\M$ is a family $\mathcal F\subset \V(\M)$ with the property that at least one of
the elements $S_1$ and $S_2$ is a member of $\mathcal F$, for each $(S_1,S_2)\in
\M$.

\begin{definition}[\bf Preconfiguration
$\mathbf{(\heartsuit 1)}$]\index{mathsymbols}{**1@$\mathbf{(\heartsuit1)}$}
\label{def:heart1}
 Suppose that we are in
Setting~\ref{commonsetting} and Setting~\ref{settingsplitting}. We say that the
graph $G$ is in \emph{Preconfiguration
$\mathbf{(\heartsuit 1)}(\gamma',h)$} of $V(G)$ if there are two
non-empty sets $V_0,V_1\subset
\colouringp{0}\setminus \left(\shadowsplit\cup\shadow_{\GD}(\WantiC, \frac{\eta^2 k}{10^5})\right)$ with the following
properties.
\begin{align}
\label{COND:P1:3}
\mindeg_{\Gcapt}\left(V_0,\Vgood\colouringpI{2}\right)&\ge h/2 \;\mbox{, and}\\
\label{COND:P1:4}
\mindeg_{\Gcapt}\left(V_1,\Vgood\colouringpI{2}\right)&\ge h \;.
\end{align} Further, there is an
$(\M_A\cup\M_B)$-cover $\mathcal F$ such that
\begin{equation}\label{COND:P1:5}
\maxdeg_{\Gcapt}\left(V_1,\bigcup\mathcal F\right)\le \gamma' k\;.
\end{equation}
\end{definition}

\begin{definition}[\bf Preconfiguration
$\mathbf{(\heartsuit 2)}$]\index{mathsymbols}{**2@$\mathbf{(\heartsuit2)}$} Suppose that we
are in Setting~\ref{commonsetting} and Setting~\ref{settingsplitting}. We say that the
graph $G$ is in \emph{Preconfiguration
$\mathbf{(\heartsuit 2)}(h)$} of $V(G)$ if there are two
non-empty sets $V_0,V_1\subset \colouringp{0}\setminus \left(\shadowsplit\cup \shadow_{\GD}(\WantiC, \frac {\eta^2 k}{10^5})\right)$ with the
following properties.
\begin{align}
\begin{split}\label{COND:P2:4}
\mindeg_{\Gcapt}\left(V_0\cup V_1,\Vgood\colouringpI{2}\right)&\ge h.
\end{split}
\end{align}
\end{definition}

\begin{definition}[\bf Preconfiguration
$\mathbf{(exp)}$]\index{mathsymbols}{**exp@$\mathbf{(exp)}$} 
\label{def:exp8}
Suppose that we
are in Setting~\ref{commonsetting} and Setting~\ref{settingsplitting}. We say that the
graph $G$ is in \emph{Preconfiguration
$\mathbf{\mathbf{(exp)}}(\beta)$} if there are two
non-empty sets $V_0,V_1\subset \colouringp{0}$ with the following properties.
\begin{align}
\label{COND:exp:1}
\mindeg_{\Gexp}(V_0,V_1)&\ge \beta k\;,\\
\label{COND:exp:2}
\mindeg_{\Gexp}(V_1,V_0)&\ge \beta k\;.
\end{align}
\end{definition}

\begin{definition}[\bf Preconfiguration
$\mathbf{(reg)}$]\index{mathsymbols}{**reg@$\mathbf{(reg)}$}
\label{def:reg}
 Suppose that we
are in Setting~\ref{commonsetting} and Setting~\ref{settingsplitting}. We say that the graph $G$ is in \emph{Preconfiguration
$\mathbf{\mathbf{(reg)}}(\tilde \epsilon, d', \mu)$}  if there are two  non-empty sets $V_0,V_1\subset \colouringp{0}$
and a non-empty family of vertex-disjoint $(\tilde\epsilon,d')$-super-regular pairs $\{(Q_0^{(j)},Q_1^{(j)}\}_{j\in\mathcal Y}$ (with respect to the edge set $E(G)$) with $V_0:=\bigcup Q_0^{(j)}$ and $V_1:=\bigcup Q_1^{(j)}$ such that
\begin{align}
\label{COND:reg:0}
\min\left\{|Q_0^{(j)}|,|Q_1^{(j)}|\right\}&\ge\mu k\;.
\end{align}
\end{definition}

\begin{definition}[\bf Configuration
$\mathbf{(\diamond6)}$]\index{mathsymbols}{**6@$\mathbf{(\diamond6)}$}
\label{def:CONF6}
Suppose that we are in
Settings~\ref{commonsetting} and~\ref{settingsplitting}. We say that the graph $G$
is in \emph{Configuration
$\mathbf{(\diamond6)}(\delta, \tilde \epsilon,d',\mu, \gamma', h_2)$} if the
following conditions are met.

The vertex sets $V_0,V_1$ 
witness Preconfiguration
 $\mathbf{(reg)}(\tilde \epsilon,d',\mu)$ or
 Preconfiguration~$\mathbf{(exp)}(\delta)$ and either Preconfiguration~$\mathbf{(\heartsuit1)}(\gamma',h_2)$ or
Preconfiguration~$\mathbf{(\heartsuit2)}(h_2)$. There exist non-empty sets $V_2,V_3\subset \colouringp{1}$ such that
 \begin{align}\label{COND:D6:1}
\mindeg_{G}(V_1,V_2)&\ge \delta k\;,\\ 
\label{COND:D6:2}
\mindeg_{G}(V_2,V_1)&\ge \delta k\;,\\ 
\label{COND:D6:3}
\mindeg_{\Gexp}(V_2,V_3)&\ge \delta k \;,\mbox{and}\\
\label{COND:D6:4}
\mindeg_{\Gexp}(V_3,V_2)&\ge \delta k\;.
\end{align}
\end{definition}

\begin{definition}[\bf Configuration
$\mathbf{(\diamond7)}$]\index{mathsymbols}{**7@$\mathbf{(\diamond7)}$}
\label{def:CONF7}
Suppose that we are in
Settings~\ref{commonsetting} and~\ref{settingsplitting}. We say that the graph $G$
is in \emph{Configuration
$\mathbf{(\diamond7)}(\delta, \rho', \tilde \epsilon, d',\mu, \gamma', h_2)$} if
the following conditions are met.

The sets $V_0,V_1$  witness Preconfiguration
 $\mathbf{(reg)}( \tilde \epsilon, d',\mu)$ and either Preconfiguration~$\mathbf{(\heartsuit
 1)}(\gamma', h_2)$ or Preconfiguration~$\mathbf{(\heartsuit 2)}(h_2)$. There
 exist non-empty sets $V_2\subset \smallatoms\colouringpI{1}\setminus \exceptVertSplit$ and $V_3\subset \colouringp{1}$ such that
 \begin{align}
\label{COND:D7:1}
\mindeg_{G}(V_1,V_2)&\ge \delta k\;,\\
\label{COND:D7:2}
\mindeg_{G}(V_2,V_1)&\ge \delta k\;,\\
\label{COND:D7:3}
\maxdeg_{\GD}(V_2,\colouringp{1}\setminus V_3)&< \rho' k \;\mbox{and}\\
\label{COND:D7:4}
\mindeg_{\GD}(V_3,V_2)&\ge \delta k\;.
\end{align}
\end{definition}

\begin{definition}[\bf Configuration
$\mathbf{(\diamond8)}$]\index{mathsymbols}{**8@$\mathbf{(\diamond8)}$}
\label{def:CONF8}
Suppose that we are in
Settings~\ref{commonsetting} and~\ref{settingsplitting}. We say that the graph $G$
is in \emph{Configuration
$\mathbf{(\diamond8)}(\delta,\rho',\epsilon_1,\epsilon_2, d_1,d_2,\mu_1,\mu_2, h_1,h_2)$}
 if the following conditions are met.

The vertex sets $V_0,V_1$  witness Preconfiguration
 $\mathbf{(reg)}(\epsilon_2, d_2,\mu_2)$ and Preconfiguration~$\mathbf{(\heartsuit 2)}(h_2)$.
 There exist non-empty sets $V_2\subset \colouringp{0}$, $V_3,V_4\subset \colouringp{1}$, $V_3\subset\smallatoms\setminus \exceptVertSplit$, and an $(\epsilon_1, d_1, \mu_1 k)$-semiregular matching $\mathcal N$ absorbed by $(\M_A\cup\M_B)\setminus \NAtom$, $V(\mathcal N)\subset \colouringp{1}\setminus V_3$ such that
\begin{align}
\label{COND:D8:1}
\mindeg_{G}(V_1,V_2)&\ge \delta k\;,\\
\label{COND:D8:2}
\mindeg_{G}(V_2,V_1)&\ge \delta k\;,\\
\label{COND:D8:3}
\mindeg_{\Gcapt}(V_2,V_3)&\ge \delta k\;,\\
\label{COND:D8:4}
\mindeg_{\Gcapt}(V_3,V_2)&\ge \delta k\;,\\
\label{COND:D8:5}
\maxdeg_{\GD}(V_3,\colouringp{1}\setminus V_4)&< \rho' k \;,\\
\label{COND:D8:6}
\mindeg_{\GD}(V_4,V_3)&\ge \delta k\;\mbox{, and}\\
\label{COND:D8:7}
\deg_{\GD}(v,V_3)+\deg_{\Gblack}(v,V(\mathcal N))&\ge h_1\;\mbox{for each $v\in V_2$.}
\end{align} 
\end{definition}
\begin{definition}[\bf Configuration
$\mathbf{(\diamond9)}$]\index{mathsymbols}{**9@$\mathbf{(\diamond9)}$}
\label{def:CONF9}
Suppose that we are in
Settings~\ref{commonsetting}, and~\ref{settingsplitting}. We say that the graph $G$
is in \emph{Configuration
$\mathbf{(\diamond9)}(\delta, \gamma', h_1, h_2, \epsilon_1, d_1,
\mu_1,\epsilon_2, d_2,\mu_2)$} if the following conditions are met.

The sets $V_0,V_1$ together with the $(\M_A\cup\M_B)$-cover $\mathcal F'$
witness Preconfiguration~$\mathbf{(\heartsuit1)}(\gamma',h_2)$. 
 There exists an $(\epsilon_1, d_1, \mu_1 k)$-semiregular matching $\mathcal N$ absorbed by $\M_A\cup\M_B$, $V(\mathcal N)\subset \colouringp{1}$.
Further, there is a family $\{(Q_0^{(j)},Q_1^{(j)})\}_{j\in\mathcal Y}$ as in Preconfiguration~$\mathbf{(reg)}(\epsilon_2,d_2,\mu_2)$. There is a set $V_2\subseteq
 V(\mathcal N)\setminus \bigcup \mathcal F'\subset \bigcup\clusters$ with the
 following properties:
\begin{align}
\label{conf:D9-XtoV}
\mindeg_{\GD}\left(V_1, V_2\right)\ge h_1\;,&\\
\label{conf:D9-VtoX}\mindeg_{\GD}\left(V_2,V_1\right)\ge
\delta k\;.
\end{align}
\end{definition}

Our last configuration, Configuration~$\mathbf{(\diamond10)}$, will lead to an embedding very similar to the one
in the dense case (treated in~\cite{PS07+}; see Section~\ref{ssec:EmbedOverview10}). In order to be able to formalize the configuration we need a preliminary definition. We shall
generalize the standard concept of a regularity graph (in the context of regular
partitions and Szemer\'edi's Regularity Lemma) to graphs with clusters whose sizes are only bounded from below.

\begin{definition}[\bf{$( \epsilon,d,\ell_1,\ell_2)$-regularized
graph}]\index{general}{regularized graph}\label{def:regularizedGraph} Let $G$ be
a graph, and let $\mathcal V$ be an $\ell_1$-ensemble that partitions $V(G)$.
Suppose that $G[X]$ is empty for each $X\in \mathcal V$ and suppose $G[X,Y]$ is
$\epsilon$-regular and of density either $0$ or at least $d$ for each $X,Y\in
\mathcal V$. Further suppose that for all $X\in \V$ it holds that 
$|\bigcup\neighbor_G(X)|\le \ell_2$.
Then we say that $(G,\mathcal V)$ is an \emph{$(\eps, d,\ell_1,
\ell_2)$-regularized graph}.

A semiregular matching $\M$ of $G$ is \index{general}{consistent matching}\emph{consistent} with $(G,\mathcal V)$ if $\V(\M)\subset \V$.
\end{definition}

\begin{definition}[\bf Configuration
$\mathbf{(\diamond10)}(\tilde\eps,d',\ell_1, \ell_2,
\eta')$]\index{mathsymbols}{**10@$\mathbf{(\diamond10)}$}
\label{def:CONF10}
Assume Setting~\ref{commonsetting}. The graph $G$ contains an $(
\tilde\epsilon, d', \ell_1, \ell_2)$-regularized graph $(\tilde G,\V)$  and there
is a $( \tilde\epsilon,  d',\ell_1 )$-semiregular matching $\M$ consistent with
$(\tilde G,\V)$.
There are a family $\LargeTen\subset \V$ and distinct clusters $A, B\in\V$  with
\begin{enumerate}[(a)]
\item\label{diamond10cond1} $E(\tilde G[A,B])\neq \emptyset$, 
\item\label{diamond10cond2} $\deg_{\tilde G}\big(v,V(\M)\cup \bigcup \LargeTen\big)\ge (1+\eta')k$ for all but at most $\tilde\epsilon |A|$
vertices $v\in A$ and for all but at most $\tilde\epsilon|B|$ vertices $v\in B$, and
\item\label{diamond10cond3}
 for each $X\in\LargeTen$ we have $\deg_{\tilde G}(v)\ge (1+\eta')k$ for all but at most $\tilde\epsilon|X|$ vertices $v\in X$.
 \end{enumerate}
\end{definition}

\subsection{The role of random splitting}\label{ssec:whyrandomsplitting}
The random splitting as introduced in Setting~\ref{settingsplitting} is used in Configurations $\mathbf{(\diamond6)}$--$\mathbf{(\diamond9)}$; the set $\colouringp{0}$ will host the cut-vertices $W_A\cup W_B$, the set $\colouringp{1}$ will host the internal shrubs, and the set $\colouringp{2}$ will (essentially) host the end shrubs of a $(\tau k)$-fine partition of $T_\PARAMETERPASSING{T}{thm:main}$.

The need for introducing the random splitting is dictated by Configurations $\mathbf{(\diamond6)}$--$\mathbf{(\diamond9)}$. To see this, let us try to follow the embedding plan from, for example, Section~\ref{ssec:EmbedOverview67} without the random splitting, i.e., dropping the conditions $\subset \colouringp{0}$, $\subset \colouringp{1}$, $\subset \colouringp{2}$ from Definitions~\ref{def:heart1}--\ref{def:CONF7}. Then the sets $V_2$ and $V_3$ in Figure~\ref{fig:DIAMOND67overview}, which will host the internal shrubs, may interfere with $V_0$ and $V_1$ primarily designated for $W_A$ and $W_B$. In particular, the conditions on degrees between $V_0$ and $V_1$ given by \eqref{COND:exp:1}--\eqref{COND:exp:2} in Definition~\ref{def:exp8}, or given by the super-regularity in Definition~\ref{def:reg} (in which $\beta_\PARAMETERPASSING{D}{def:exp8}>0$, or $d'_\PARAMETERPASSING{D}{def:reg}\mu_\PARAMETERPASSING{D}{def:reg}>0$ are tiny) need not be sufficient for embedding greedily all 
the cut-vertices and all the internal shrubs of $T_\PARAMETERPASSING{T}{thm:main}$. It should be noted that this problem occurs even in 
Preconfiguration~$\mathbf{(exp)}$, i.e., the expanding property does not add
enough strength to the minimum degree conditions due to the same peculiarity as
in Figure~\ref{fig:gettingstuck}. Restricting $V_0$ and $V_1$ to host only the
cut-vertices (only $O(1/\tau)=o(k)$ of them in total, cf.
Definition~\ref{ellfine}\eqref{few}), resolves the problem.

The above justifies the distinction between the space $\colouringp{0}$ for embedding the cut-vertices and the space $\colouringp{1}\cup\colouringp{2}$ for embedding the shrubs. There are some other approaches which do not need to further split $\colouringp{1}\cup\colouringp{2}$ but doing so seems to be the most convenient.

\subsection{Cleaning}\label{ssec:cleaning}
This section contains five ``cleaning lemmas''
(Lemma~\ref{lem:envelope}--Lemma~\ref{lem:clean-Match}). The basic setting of
all these lemmas is the same. There is a system of vertex sets and some density assumptions on edges between certain sets of this
system. The assertion the is that a small number of vertices can be discarded
from the sets so that some conditions on the minimum degree are fullfilled.
While the cleaning strategy is simply discarding the vertices which violate these
minimum degree conditions the analysis of the outcome is non-trivial and
employs amortized analysis. The simplest application of such an approach was the proof of Lemma~\ref{lem:clean-spots} above.

Lemmas~\ref{lem:envelope}--\ref{lem:clean-Match} are used to get the structures
required by (pre-)configurations introduced in Section~\ref{ssec:TypesConf},
based on rough structures found in Lemma~\ref{prop:LKSstruct}.

\bigskip
The first lemma will be used to obtain preconfiguration $\mathbf{(\clubsuit)}$
in certain situations.
\begin{lemma}\label{lem:envelope}
Let $\psi\in (0,1)$, and $\Gamma,\Omega\ge 1$ be arbitrary. Let $P$ and $Q$ be two
disjoint vertex sets in a graph $G$. Assume that $Y\subset V(G)$ is given. We
assume that \begin{equation}\label{eq:envelopeAss1}\mindeg(P,Q)\ge \Omega
k\;,\end{equation} and $\maxdeg(Q)\le \Gamma k$. Then there exist sets
$P'\subset P$, $Q'\subset Q\setminus Y$ and $Q''\subset Q'$ such that the
following holds.
\begin{enumerate}[(a)]
\item $\mindeg(P',Q'')\ge \frac{\psi^3\Omega}{4\Gamma^2}k$,
\item $\maxdeg(Q',P\setminus P')< \psi k$,
\item $\maxdeg(Q'',Q\setminus Q')<\psi k$, and
\item $e(P',Q'')\ge (1-\psi)e(P,Q)-\frac{2|Y\cap Q|\Gamma^2 }{\psi}k$.
\end{enumerate}
\end{lemma}
\begin{proof}
Initially, set $P':=P$, $Q'=:Q\setminus Y$ and $Q'':=Q'\setminus Y$. We shall
sequentially discard from the sets $P'$, $Q'$ and $Q''$ those vertices which
violate any of the properties (a)--(c). Further, if a vertex $v\in Q$ is removed
from $Q'$ then we remove it from the set $Q''$ as well. This way, we have
$Q''\subset Q'$ in each step. After this sequential cleaning procedure
finishes it only remains to establish~(d).

First, observe that the way we constructed $P'$ ensures
that
\begin{equation}\label{DianaAndMatej}
e(P\setminus P',Q'')\le \frac{\psi^3}{4\Gamma^2}e(P,Q)\;.
\end{equation}

Let $Q^b\subset Q\setminus Q'$ be the set of the vertices removed because of
condition~(b). For a vertex $u\in P\setminus P'$, we write $Q''_u$ for the set
$Q''$ just before the moment when $u$ was removed from $P'$. Likewise, we define the sets $P'_v,Q'_v,Q''_v$ for each $v\in Q\setminus Q''$. For $u\in P\setminus P'$
let $f(u):=\deg(u,Q''_u)$, for $v\in Q\setminus (Q'\cup Y)$ let
$g(v):=\deg(v,P\setminus P'_v)$, and for $w\in Q'\setminus Q''$ let
$h(w):=\deg(w,Q\setminus Q'_w)$. Observe that $\sum_{u\in P\setminus P'}f(u)\ge
\sum_{v\in Q^b}g(v)$. Indeed, at the moment when $v\in Q$ is removed
from~$Q'$, the $g(v)$ edges that $v$ sends to the set $P\setminus P'_v$ are 
counted in $\sum_{w\in \neighbor(v)\cap (P\setminus P')}f(w)$. We therefore have
\begin{equation*}
\frac{\psi^3}{4\Gamma^2}e(P,Q)\ge \frac{\psi^3}{4\Gamma^2}\sum_{u\in
P\setminus P'}\deg(u,Q)\ge\sum_{u\in P\setminus P'}f(u)\ge \sum_{v\in
Q^b}g(v)\ge |Q^b|\psi k\;,
\end{equation*}
and consequently,
\begin{equation}\label{eq:BoundB}
|Q^b|\le \frac{\psi^2}{4\Gamma^2 k}e(P,Q)\;.
\end{equation}
We also have
\begin{equation}\label{eq:BoundC}
|Q'\setminus Q''|\psi k\le \sum_{w\in Q'\setminus Q''}h(w)\le |Q^b\cup (Y\cap
Q)|\Gamma k\leByRef{eq:BoundB} \frac{\psi^2}{4\Gamma}e(P,Q)+|Y\cap Q|\Gamma k\;.
\end{equation}
Finally, we can lower-bound $e(P',Q'')$ as follows.
\begin{align*}e(P',Q'')&\ge
e(P,Q)-e(P\setminus P',Q'')-|Y\cap Q|\Gamma k-|Q^b|\Gamma k-|Q'\setminus
Q''|\Gamma k\\
\JUSTIFY{by~\eqref{DianaAndMatej},~\eqref{eq:BoundB},~\eqref{eq:BoundC}}
&\ge
e(P,Q)\Big(1-\frac{\psi^3}{4\Gamma^2}-\frac{\psi^2}{4\Gamma}-\frac{\psi}{4}\Big)-|Y\cap
Q|\Big(\frac{\Gamma^2k}{\psi}+\Gamma k\Big)\\
&\ge(1-\psi)e(P,Q)-\frac2\psi|Y\cap Q|\Gamma^2 k \;.\end{align*}
\end{proof}

\HIDDENTEXT{There was another cleaning lemma, removed by V.6.53, now in
Hiddentext, under CLEANING}

\bigskip
The purpose of the lemmas below
(Lemmas~\ref{lem:clean-C+yellow}--\ref{lem:clean-Match}) is to distill
vertex-sets for configurations $\mathbf{(\diamond2)}$-$\mathbf{(\diamond10)}$.
They will be applied in Lemmas~\ref{lem:ConfWhenCXAXB},~\ref{lem:ConfWhenNOTCXAXB},~\ref{lem:ConfWhenMatching}.
This is the final ``cleaning step'' on our way to the proof of Theorem~\ref{thm:main} --- the outputs of these lemmas can by used for a vertex-by-vertex embedding of any tree $T\in\treeclass{k}$ (although the corresponding embedding procedures given in Section~\ref{sec:embed} are quite complex).

The first two of these cleaning lemmas (Lemmas~\ref{lem:clean-C+yellow}
and~\ref{lem:clean-C+black}) are suited when the set $\HugeVertices$ of
vertices of huge degrees (cf.\ Setting~\ref{commonsetting}) needs to be considered.

For the following lemma, recall that we defined $[r]$ as the set of the first $r$ natural numbers, {\em not} including $0$.

\begin{lemma}\label{lem:clean-C+yellow}For all
$r,\Omega^*,\Omega^{**}\in \mathbb N$, and
 $\delta,\gamma,\eta\in (0,1)$,   with
$\left(\frac{3\Omega^*}\gamma\right)^r\delta<\eta/10$,
and $\Omega^{**}>1000$ the
following holds. Suppose  there are vertex sets $X_0, X_1,\ldots,X_r$ and $Y$ of
an $n$-vertex graph $G$ such that
\begin{enumerate}
  \item\label{hyp:Ysmall} $|Y|<\eta n/(4\Omega^*)$,
  \item \label{hyp:C+y-edges}$e(X_0,X_1)\geq \eta kn$,
\item \label{hyp:C+y-large} $\mindeg(X_0,X_1)\ge \Omega^{**}k$,
  \item \label{hyp:C+y-deg}
  $\mindeg(X_i,X_{i+1})\ge \gamma k$ for all $i\in [r-1]$, and 
  \item \label{hyp:C+y-bounded} $\maxdeg\left(Y\cup\bigcup_{i\in [r]}X_i\right)\le
  \Omega^* k$.
\end{enumerate}
Then there are sets $X_i'\subseteq X_i$ for $i=0,1,\ldots,r$ such that
\begin{enumerate}[(a)]
  \item \label{conc:X1Ydisj} $X_1'\cap Y=\emptyset$,
  \item \label{conc:C+y-deg}  
  $\mindeg(X_i',X_{i-1}')\ge \delta k$ for all $i\in [r]$,
  \item \label{conc:C+y-avoid}  
  $\maxdeg(X_i',X_{i+1}\setminus X_{i+1}')<\gamma k/2$ for all $i\in [r-1]$,
  \item \label{conc:C+y-large} $\mindeg(X_0',X_1')\ge
  \sqrt{\Omega^{**}}k$, and
  \item \label{conc:C+u-edges} $e(X_0',X_1')\ge \eta kn/2$, in particular
  $X_0'\neq\emptyset$.
\end{enumerate}
\end{lemma}

\begin{proof}
In the formulae below we refer to hypotheses of the lemma
as~``\ref{hyp:Ysmall}.''--``\ref{hyp:C+y-bounded}.''.

Set $X_1':=X_1\setminus Y$. For $i=0,2,3,4,\ldots,r$, set
$X_i':=X_i$. Discard sequentially from $X_i'$ any vertex that
violates any of the
Properties~\eqref{conc:C+y-deg}--\eqref{conc:C+y-large}.
Properties~\eqref{conc:X1Ydisj}--\eqref{conc:C+y-large} are
trivially satisfied  when the procedure terminates. To show that
Property~\eqref{conc:C+u-edges} holds at this point, we bound the
number of edges from $e(X_0,X_1)$ that are incident with $X_0\setminus X_0'$ or
with $X_1\setminus X_1'$ in an amortized way.

For $i\in\{0,\ldots,r\}$ and for $v\in X_i\setminus X_i'$ we
write
\begin{align*}
f_i(v)&:=\deg\big(v,X_{i+1}(v)\setminus X_{i+1}'(v)\big)\;,\\
g_i(v)&:=\deg\big(v,X_{i-1}'(v)\big)\;\mbox{, and}\\
h_i(v)&:=\deg\big(v,X_{i+1}'(v)\big)\;.
\end{align*}
where the sets $X_{i-1}'(v),X_{i}'(v),X_{i+1}'(v)$
above refer to the moment when $v$ is removed from $X_i'$ (we do
not define $f_i(v)$ and $h_i(v)$ for $i=r$ and $g_i(v)$ for $i=0$).
 
For $i\in[r]$ let
$X_i^{\ref{conc:C+y-deg}}$ denote the vertices in 
 $X_i\setminus X_i'$ that were removed from $X_i'$ because of violating 
 Property~\eqref{conc:C+y-deg}. Then for a given $i\in [r]$ we have that 
 \begin{equation}\label{eq:X_i^a}
 \sum_{v\in X_i^{\ref{conc:C+y-deg}}}
 g_i(v)<\delta kn.
 \end{equation}
For $i=1,\ldots,r-1$
 let  $X_i^{\ref{conc:C+y-avoid}}$  denote
the vertices in $X_i\setminus X_i'$ that violated
Property~\eqref{conc:C+y-avoid}. Set $X_r^{\ref{conc:C+y-avoid}}:=\emptyset$. For
a given $i\in [r-1]$ we have
\begin{equation}\label{eq:C+y-induction}
|X_i^{\ref{conc:C+y-avoid}}|\cdot \gamma k/2\le \sum_{v\in
X_i^{\ref{conc:C+y-avoid}}}f_i(v)\le \sum_{v\in
X_{i+1}\setminus
X_{i+1}'}g_{i+1}(v)\;\overset{\ref{hyp:C+y-bounded}., \eqref{eq:X_i^a}}{<}\;\delta
kn+|X_{i+1}^{\ref{conc:C+y-avoid}}|\cdot \Omega^*k\;,
\end{equation}
as $X_i\setminus X_i'=X_i^{\ref{conc:C+y-deg}}\cup X_i^{\ref{conc:C+y-avoid}}$,
for $i=2,\ldots,r$. Using~\eqref{eq:C+y-induction}
for $j=0,\ldots,r-1$, we inductively deduce that
\begin{equation}\label{eq:Indk}
|X_{r-j}^{\ref{conc:C+y-avoid}}|\frac{\gamma
}2\le \sum_{i=0}^{j-1}\left(\frac{2\Omega^*}\gamma\right)^i\delta
n\;.
\end{equation}
(The left-hand side is zero for $j=0$.) The
bound~\eqref{eq:Indk} for $j=r-1$ gives
\begin{equation}\label{eq:B73}
|X_{1}^{\ref{conc:C+y-avoid}}|\le
\frac2\gamma\cdot\sum_{i=0}^{r-2}\left(\frac{2\Omega^*}\gamma\right)^i\delta
n\le\frac{2(2\Omega^*)^{r-1}}{\gamma^r}\delta n\;.
\end{equation}
Therefore,
 \begin{equation}\label{eq:e(YX_1^c,X_0)}
 e(X_0,Y\cup X_1^{\ref{conc:C+y-avoid}})\leq |Y\cup
 X_1^{\ref{conc:C+y-avoid}}|\cdot
 \Omega^*k \overset{\eqref{eq:B73},  
 \ref{hyp:Ysmall}.}{\le}\frac {\eta kn}{4}+\left(\frac
 {2\Omega^*}{\gamma}\right)^r\delta kn\;.
\end{equation}
For any vertex  $v\in X_0\setminus X_0'$ we
have $h_0(v)<\sqrt{\Omega^{**}}k$,
 and at the same time by Hypothesis~\ref{hyp:C+y-large}.\ we have
$\deg(v,X_1)\geq \Omega^{**}k$. So,
\begin{equation}\label{eq:h(X_0-X_0')}
\sum_{v\in
X_0\setminus X_0'}h_0(v)\le \frac{e(X_0,X_1)}{\sqrt{\Omega^{**}}}\;.
\end{equation} We have
\begin{align*}
e(X_0',X_1')&\ge e(X_0,X_1)-e(X_0,Y\cup
X_1^{\ref{conc:C+y-avoid}})-\sum_{v\in X_0\setminus
X_0'}h_0(v)-\sum_{v\in X_1^{\ref{conc:C+y-deg}}}g_1(v)\;.
\end{align*}
(It requires a minute of meditation to see that edges between $X_0\setminus
X'_0$ and $X_1^{\ref{conc:C+y-deg}}$ are indeed not counted on the right-hand
side.) Therefore,
\begin{align*}
e(X_0',X_1')&\ge e(X_0,X_1)-e(X_0,Y\cup
X_1^{\ref{conc:C+y-avoid}})-\sum_{v\in X_0\setminus
X_0'}h_0(v)-\sum_{v\in X_1^{\ref{conc:C+y-deg}}}g_1(v)\\
\JUSTIFY{by \eqref{eq:X_i^a},
\eqref{eq:e(YX_1^c,X_0)}, \eqref{eq:h(X_0-X_0')}}
&\ge
e(X_0,X_1)-\frac {\eta kn}{4}-\left(\frac {2\Omega^*}\gamma\right)^r \delta kn-
\frac{e(X_0,X_1)}{\sqrt{\Omega^{**}}}-\delta kn\\
\JUSTIFY{by \ref{hyp:C+y-edges}.}&\ge  \eta k/2\;,
\end{align*} 
proving Property~\eqref{conc:C+u-edges}.
\end{proof}

\begin{lemma}\label{lem:clean-C+black}
Let $\delta,\eta,\Omega^*,\Omega^{**},h>0$, let $G$ be  an $n$-vertex graph,
let $X_0, X_1, Y\subset V(G)$, and let $\mathcal C$ be a system of subsets of
$V(G)$  such that
\begin{enumerate}
\item\label{hypfburg}
$20(\delta+\frac2{\sqrt{\Omega^{**}}})<\eta$,
  \item \label{hyp:C+b-edges} $2 kn \geq e(X_0,X_1)\geq \eta kn$,
  \item \label{hyp:C+b-large} $\mindeg(X_0,X_1)\ge \Omega^{**}k$, 
  \item   \label{hyp:C+b-bounded} $\maxdeg(X_1)\le  \Omega^* k$,
  \item $|Y|<\eta n/(4\Omega^*)$, and\label{hypfburg5}
  \item $10h|\mathcal C|\Omega^*<\eta n$.\label{hypfburg6}
\end{enumerate} 
Then  there are sets $X_0'\subseteq X_0$ and $X_1'\subseteq X_1\setminus Y$ such that \begin{enumerate}[a)]
   \item \label{conc:C+b-large} $\mindeg(X_0',X_1')\ge
   \sqrt{\Omega^{**}}k$,
  \item  \label{conc:C+b-deg} $\mindeg(X_1',X_0')\ge\delta k$, 
  \item \label{conc:C+b-cluster}for all $C\in \mathcal C$, either $X_1'\cap
  C=\emptyset$, or $|X_1'\cap C|\ge h$, and 
  \item \label{conc:C+b-edges}$e(X_0',X_1')\ge \eta kn/2$.
\end{enumerate}
\end{lemma}
\begin{proof}
Set $X_0':=X_0$ and $X_1':=X_1\setminus Y$ and discard sequentially from $X_0'$,
any vertex violating
Property~\ref{conc:C+b-large}). Further, we discard from $X_1'$ any vertex
violating Property~\ref{conc:C+b-deg}), or any  $C\in
\mathcal C$ violating~\ref{conc:C+b-cluster}).
When the process ends, we verify Property~\ref{conc:C+b-edges}) by bounding the
number of edges in $e(X_0,X_1)$ incident with $X_0\setminus X_0'$ or with
$X_1\setminus X_1'$. Given Assumption~\ref{hyp:C+b-edges}, and since
by Assumption~\ref{hypfburg5} there are
at most $\frac14\eta kn$ edges incident with $Y\cap X_1$
it suffices to prove that
\begin{equation}\label{eq:ItSuff}
e(X_0, X_1)- e(X_0', X_1')- e(Y\cap X_1, X_0)<\frac{\eta
kn}4\;.
\end{equation}

Denote by $X_1^{\ref{conc:C+b-deg}}$ the set of vertices in $X_1\setminus
(Y\cup X_1')$ that violated Property~\ref{conc:C+b-deg}), and  by
$X_1^{\ref{conc:C+b-cluster}}$ the set of vertices in $X_1\setminus (Y\cup X_1')$ that violated
Property~\ref{conc:C+b-cluster}). For a vertex $v\in X_1\setminus (Y\cup X_1')$, let $g(v)$
denote the number $\deg(v,X_0')$ at the very time when $v$ is removed from
$X_1'$. Analogously we define $f(v)$, for $v\in X_0\setminus X_0'$, as
$\deg(v,X_1')$ where the set $X_1'$ is considered at the point of removal of
$v$. We have $\sum_{v\in X_1^{\ref{conc:C+b-deg}}}g(v)<\delta kn$, $\sum_{v\in X_1^{\ref{conc:C+b-cluster}}}g(v)\leq |X_1^{\ref{conc:C+b-cluster}}| \Omega^* k<
h|\mathcal C|\cdot  \Omega^* k$, and $$\sum_{v\in X_0\setminus X_0'}f(v)\le
\frac{e(X_0,X_1)}{\sqrt{\Omega^{**}}}\leBy{\ref{hyp:C+b-edges}.}
\frac2{\sqrt{\Omega^{**}}}kn\;.$$ Thus,
\begin{align*}
e(X_0, X_1)&- e(X_0', X_1')- e(Y\cap X_1, X_0)\\&=\sum_{v\in
X_1^{\ref{conc:C+b-deg}}}g(v)+\sum_{v\in X_1^{\ref{conc:C+b-cluster}}}g(v)+\sum_{v\in X_0\setminus
X_0'}f(v)\\
&<\big(\delta+\frac2{\sqrt{\Omega^{**}}}\big)kn+h|\mathcal
C|\Omega^*k\\
\JUSTIFY{by~\ref{hypfburg}.\ and \ref{hypfburg6}.}&<\frac{\eta kn}4\;.
\end{align*}
establishing~\eqref{eq:ItSuff}.
\end{proof}

The next two lemmas (Lemmas~\ref{lem:clean-yellow} and~\ref{lem:clean-Match})
deal with cleaning outside the set of huge degree vertices $\HugeVertices$.

\begin{lemma}\label{lem:clean-yellow} For all $r, \Omega\in \mathbb N$,
$r\ge 2$ and all $\gamma,\delta,\eta>0$ such that
\begin{equation}\label{eq:condCY}
 \left(\frac{8\Omega}\gamma\right)^r\delta\le\frac\eta{10}
\end{equation}
 the following holds. Suppose there are
vertex sets $Y, X_0, X_1,\ldots,X_r\subset V$, where $V$ is a set of $n$
vertices. Suppose that edge sets $E_1,\ldots,E_r$ are given on $V$.
The expressions $\deg_i$, $\maxdeg_i$,
$\mindeg_i$, and $e_i$ below refer to the edge set $E_i$. Suppose that the following properties are fulfilled
\begin{enumerate}
  \item \label{hyp:yel-Ysmall}$|Y|<\delta n$,
  \item \label{hyp:yel-edges}$e_1(X_0,X_1)\geq \eta kn$,
  \item \label{hyp:yel-deg}for all $i\in [r-1]$ we have $\mindeg_{i+1}( X_i\setminus
  Y,X_{i+1})\ge \gamma k$, 
  \item \label{hyp:yel-bounded} for all $i\in\{0,\ldots,r-1\}$, we have
  $\maxdeg_{i+1}(X_{i})\le \Omega k$, and
  $\maxdeg_{i+1}(X_{i+1})\le \Omega k$.
\end{enumerate}
Then there are sets $X_i'\subseteq X_i\setminus Y$ ($i=0,\ldots,r$) satisfying
the following.
\begin{enumerate}[a)]
  \item \label{conc:yel-deg}For all $i\in [r]$ and we have 
  $\mindeg_{i}(X_{i}',X_{i-1}')\ge \delta k$,
  \item \label{conc:yel-avoid}for all $i\in [r-1]$ we have
  $\maxdeg_{i+1}(X_i',X_{i+1}\setminus X_{i+1}')<\gamma k/2$,
    \item \label{conc:yel-X0X1}$\mindeg_1(X_0',X'_1)\ge \delta k$, and
    \item \label{conc:yel-edges} $e_1(X'_0,X'_1)\ge \eta kn/2$
\end{enumerate}
\end{lemma}
\begin{proof}
We proceed similarly as in the proof of Lemma~\ref{lem:clean-C+yellow}.
Set $X_i':=X_i\setminus Y$ for each $i=0,\ldots,r$.
Discard sequentially from $X_i'$ any vertex that violates
Property~\ref{conc:yel-deg}) or~\ref{conc:yel-avoid}), or~\ref{conc:yel-X0X1}).
When the procedure terminates, we certainly have
that~\ref{conc:yel-deg})--\ref{conc:yel-X0X1}) hold. We then show that
Property~\ref{conc:yel-edges}) holds by bounding the number of edges from
$e_1(X_0,X_1)$ that are incident with $X_0\setminus X_0'$ or with $X_1\setminus
X_1'$. For $i\in\{0,\ldots,r\}$ and
for $v\in X_i\setminus X_i'$ we write
\begin{align*}
f_{i+1}(v)&:=\deg_{i+1}(v,X_{i+1}\setminus X_{i+1}')\;,\\
g_{i}(v)&:=\deg_i(v,X_{i-1}')\;\mbox{, and}\\
h(v)&:=\deg_1(v,X_{1}')\;,
\end{align*}
where the sets $X_1', X_{i-1}'$ and $X_{i+1}'$
above refer to the moment\footnote{if $v\in Y$ then this
moment is the zero-th step} when $v$ is removed from $X_i'$ or from
$X_1'$ (we do not define $f_{i+1}(v)$ for
$i=r$ and $g_i(v)$ for $i=0$).

Let $X^{\ref{conc:yel-deg}}_i\subset X_i$, $X^{\ref{conc:yel-avoid}}_i\subset
X_i$ for $i\in [r-1]$ be the sets of vertices
removed from $X'_i$ because of Property~\ref{conc:yel-deg})
and~\ref{conc:yel-avoid}), respectively. Set
$X^{\ref{conc:yel-deg}}_r:=X_r\setminus X_r'$ and
$X^{\ref{conc:yel-X0X1}}_0:=X_0\setminus X'_0$.
We have for each $i\in[r]$,
\begin{align}\label{eq:HezkejVodotrysk}
\sum_{v\in X_i^{\ref{conc:yel-deg}}}g_i(v)&< \delta kn\;.
\end{align}
Also, note that we have
\begin{equation}\label{eq:X0X1}
\sum_{v\in X_0^{\ref{conc:yel-X0X1}}}h(v)\le \delta kn\;.
\end{equation}

We set $X_r^{\ref{conc:yel-avoid}}:=\emptyset$. For a given $i\in [r-1]$ we have
\begin{align}\nonumber
|X_i^{\ref{conc:yel-avoid}}|\cdot \frac{\gamma k}2&\le \sum_{v\in
X_i^{\ref{conc:yel-avoid}}}f_{i+1}(v)\\
\nonumber
&\le \sum_{v\in
X_{i+1}\setminus
X_{i+1}'}g_{i+1}(v)\\ 
\JUSTIFY{by~\ref{hyp:yel-bounded}.,
\eqref{eq:HezkejVodotrysk}}&\leq \delta
kn+|X_{i+1}^{\ref{conc:yel-avoid}}| \Omega k\;,\label{eq:yel-induction}
\end{align}
as $X_i\setminus X_i'\subset X_i^{\ref{conc:yel-deg}}\cup
X_i^{\ref{conc:yel-avoid}}\cup Y$, for $i=2,\ldots,r$. Using~\eqref{eq:yel-induction}, we deduce inductively that
\begin{equation}\label{eq:Indk1}
|X_{r-j}^{\ref{conc:yel-avoid}}|\le \left(\frac{8\Omega}\gamma\right)^j\delta
n\;,
\end{equation}
for $j=0,\ldots,r-1$.
(The left-hand side is zero for $j=0$.)
Therefore,
 \begin{align*}
e_1(X_0',X_1')&\geq
 e_1(X_0,X_1)-(|Y|+|X_1^{\ref{conc:yel-avoid}}|)\Omega k-\sum_{v\in
 X_1^{\ref{conc:yel-deg}}}g_1(v)-\sum_{v\in X^{\ref{conc:yel-X0X1}}_0}h(v)\\
 \JUSTIFY{by~\ref{hyp:yel-edges},
 \eqref{eq:Indk1},~\eqref{eq:HezkejVodotrysk},~\eqref{eq:X0X1}}&\ge \eta
 kn-\left(\frac{8\Omega}{\gamma}\right)^r\delta kn-2\delta kn\\
 &\ge \frac{\eta}2kn\;,
\end{align*}
establishing Property~\ref{conc:yel-edges}).
\end{proof}

\begin{lemma}\label{lem:clean-Match}
For all $r, \Omega\in \mathbb N$, $r\ge 2$
and all $\gamma,\eta,\delta,\epsilon,\mu,d>0$ with
\begin{equation}\label{eq:condCYmatch}
\text{ $20\epsilon<d$ \ and }\  \left(\frac{8\Omega}\gamma\right)^r\delta\le\frac\eta{30}
\end{equation}
the following holds. Suppose there are
vertex sets $Y, X_0, X_1,\ldots,X_r\subset V$, where $V$ is a set of
$n$ vertices. Let $P^{(1)}_i,\ldots,P^{(p)}_i$ partition $X_i$, for $i=0,1$.
Suppose that edge sets $E_1,E_2,E_3,\ldots,E_r$ are given on $V$.
The expressions $\deg_i$, $\maxdeg_i$, and $\mindeg_i$ below refer to the edge set $E_i$. Suppose that 
\begin{enumerate}
  \item \label{hyp:Match-Ysmall}$|Y|<\delta n$,
  \item \label{hyp:Match-edges}$|X_1|\geq \eta n$,
  \item \label{hyp:Match-deg}for all $i\in [r-1]$ we have
  $\mindeg_{i+1}(X_i\setminus Y,X_{i+1})\ge \gamma k$,
  \item \label{hyp:Match-reg} the family
  $\left\{(P^{(j)}_0,P^{(j)}_1)\right\}_{j\in[p]}$ is an $(\epsilon,d,\mu k)$-semiregular matching with respect to the edge set $E_1$, and
  \item \label{hyp:Match-bounded}for all $i\in\{0,\ldots,r-1\}$,
  $\maxdeg_{i+1}(X_{i+1})\le \Omega k$, and (when $i\not=r$) $\maxdeg_{i+1}(X_i)\le \Omega k$.
\end{enumerate}
Then there is a non-empty family $\{(Q_0^{(j)},Q_1^{(j)})\}_{j\in\mathcal Y}$
of vertex-disjoint $(4\epsilon,\frac d4)$-super-regular pairs with respect to $E_1$, with
\begin{enumerate}[a)]
 \item \label{conc:Match-superreg}  $|Q_0^{(j)}|,|Q_1^{(j)}|\ge \frac{\mu k}2$ for each  $j\in\mathcal Y$,
\end{enumerate}
 and sets $X_0':=\bigcup Q_0^{(j)}\subset X_0\setminus Y$, $X_1':=\bigcup Q_1^{(j)}\subset X_1\setminus Y$, $X_i'\subseteq X_i\setminus Y$ ($i=2,\ldots,r$) such that
\begin{enumerate}[a)]
  \setcounter{enumi}{1}
  \item \label{conc:Match-deg}for all $i\in [r-1]$ we have 
  $\mindeg_{i+1}(X_{i+1}',X_{i}')\ge \delta k$, and
  \item \label{conc:Match-avoid}for all $i\in [r-1]$, we have
  $\maxdeg_{i+1}(X_i',X_{i+1}\setminus X_{i+1}')<\gamma k/2$.
\end{enumerate}
\end{lemma}
\begin{proof}
Initially, set $\mathcal{J}:=\emptyset$ and $X_i':=X_i\setminus Y$ for each $i=0,\ldots,r$. Discard sequentially from
$X_i'$ any vertex that violates any of the Properties~\ref{conc:Match-deg}) or~\ref{conc:Match-avoid}).
We would like to keep track of these vertices and therefore we call $X^b_i, X^c_i\subset X_i$ the sets
of vertices removed from $X'_i$ because of
Property~\ref{conc:Match-deg}), and~\ref{conc:Match-avoid}),
respectively.
Further, for $i=0,1$ and for $j\in[p]$
remove any vertex $v\in X'_i\cap P^{(j)}_i$ from $X'_i$ if
\begin{equation}\label{eq:pocitac}
\deg_1(v,X'_{1-i}\cap P^{(j)}_{1-i})\le \frac{d|P_{1-i}^{(j)}|}4\;. 
\end{equation}
For $i=0,1$, let $X_i^a$ be the set of those vertices of $X_i$ that were removed because of~\eqref{eq:pocitac}.

Last,  if for some $j\in[p]$ we have $|P^{(j)}_0\cap Y|>\frac{|P^{(j)}_0|}4$ or $|P^{(j)}_1\cap (Y\cup X^c_1)|>\frac{|P^{(j)}_1|}4$ we remove simultaneously the sets $P^{(j)}_0$ and $P^{(j)}_1$ entirely from $X_0'$ and $X_1'$, i.e., we set $X'_0:=X'_0\setminus P^{(j)}_0$ and $X'_1:=X'_1\setminus P^{(j)}_1$. We also add the index $j$ to the set $\mathcal{J}$ in this case.

When the procedure terminates define $\mathcal Y:=[p]\setminus \mathcal J$, and for $j\in\mathcal Y$ set $(Q_0^{(j)},Q_1^{(j)}):=(P_0^{(j)}\cap X_0',P_1^{(j)}\cap X_1')$. The sets $X_i'$ obviously satisfy Properties \ref{conc:Match-deg})--\ref{conc:Match-avoid}). We now turn to verifying Property~\ref{conc:Match-superreg}). This relies on the following claim.
\begin{claim}\label{cl:useReg}
If $j\in[p]\sm \mathcal J$ then
$|P^{(j)}_0\cap X^a_0|\le\frac{|P^{(j)}_0|}4$ and $|P^{(j)}_1\cap X^a_1|\le\frac{|P^{(j)}_1|}4$.
\end{claim}
\begin{proof}[Proof of Claim~\ref{cl:useReg}]
Recall that $E_1$ is the relevant underlying edge set when working with the
pairs $(P_0^{(j)},P_1^{(j)})$.  Also, recall that only vertices from $Y\cup
X^a_0$ were removed from $P_0^{(j)}$ and only vertices from $Y\cup X^a_1\cup
X^c_1$ were removed from $P_1^{(j)}$.

Since $j\notin \mathcal J$, the pair $(P_0^{(j)}\setminus Y,P_1^{(j)}\setminus (Y\cup X_1^c))$ is $2\epsilon$-regular of density at least $0.9d$ by Fact~\ref{fact:BigSubpairsInRegularPairs}. Let 
\begin{align*}
K_0&:=\big\{v\in P_0^{(j)}\setminus Y\::\:\deg_1(v,P_1^{(j)}\setminus (Y\cup X_1^c))<0.8d|P_1^{(j)}\setminus (Y\cup X_1^b)|\big\}\;\mbox{, and}\\
K_1&:=\big\{v\in P_1^{(j)}\setminus (Y\cup X_1^c)\::\:\deg_1(v,P_0^{(j)}\setminus Y )<0.8d|P_0^{(j)}\setminus Y|\big\}\;.
\end{align*}
By Fact~\ref{fact:manyTypicalVertices}, we have $|K_0|\le 2\epsilon
|P_0^{(j)}\setminus Y|\le 0.1d|P_0^{(j)}|$ and $|K_1|\le 0.1 d|P_1^{(j)}|$. In particular, we have
\begin{align}
\begin{split}\label{eq:WO1}
\mindeg_1(P^{(j)}_0\setminus (Y\cup K_0),P^{(j)}_1\setminus (Y\cup X^c_1\cup K_1))&\ge 0.8 d|P_1^{(j)}\setminus (Y\cup X_1^c)| -|K_1|\\
&\ge 0.8 d\cdot 0.75 |P_1^{(j)}|-0.1d|P_1^{(j)}|\\
&>0.25d|P_1^{(j)}|\;,\mbox{and}
\end{split}\\
\begin{split}\label{eq:WO2}
\mindeg_1(P^{(j)}_1\setminus (Y\cup X^c_1\cup K_1),P^{(j)}_0\setminus (Y\cup
K_0))&\ge 0.8 d|P_0^{(j)}\setminus Y| -|K_0|\\&\ge 0.8 d\cdot 0.75
|P_0^{(j)}|-0.1d|P_0^{(j)}|\\
&>0.25d|P_0^{(j)}|\;.
\end{split}
\end{align}
Then~\eqref{eq:WO1} and~\eqref{eq:WO2} allow us to prove that $P^{(j)}_i\cap X^a_i\subset K_i$ for $i=0,1$. Indeed, assume inductively that $P^{(j)}_i\cap X^a_i\subset K_i$ for $i=0,1$ throughout the cleaning process until a certain step. Then~\eqref{eq:WO1} and~\eqref{eq:WO2} assert that no vertex outside of $P^{(j)}_0\setminus (Y\cup K_0)$ or of $P^{(j)}_1\setminus (Y\cup X^c_1\cup K_1)$ can be removed because of~\eqref{eq:pocitac}, proving the induction step.
The claim follows.
\end{proof}
Putting together the definition of $\mathcal J$ (through which one controls the size of $P_i^{(j)}\cap (Y\cup X_i^c)$) and Claim~\ref{cl:useReg} (which controls the size of $P_i^{(j)}\cap X_i^a$) we get for each $j\in\mathcal Y$ and $i=0,1$, $$|Q_i^{(j)}|\ge\frac{|P_i^{(j)}|}2\ge \frac{\mu k}2\;.$$
Therefore, these pairs are $4\epsilon$-regular (cf.\ Fact~\ref{fact:BigSubpairsInRegularPairs}). Last, we get the property of $(4\epsilon,\frac d4)$-super-regularity from the definition of $X_i^c$ (cf.~\eqref{eq:pocitac}). Thus, the pairs $(Q_0^{(j)},Q_1^{(j)})$ are as required for Lemma~\ref{lem:clean-Match} and satisfy its Property~\ref{conc:Match-superreg}).

\medskip

The only thing we have to prove is that the set $X_1'$ is nonempty. By the definition, for each $j\in \mathcal J$, we either have 
$|P^{(j)}_1| \le 4(|(Y\cup X_1^c)\cap P^{(j)}_1|)$ or $|P^{(j)}_0| \le 4|Y\cap P^{(j)}_0|$. We use that that $|P^{(j)}_0|=|P^{(j)}_1|$ to see that
\begin{equation}\label{eq:casak}
\left|\bigcup_{\mathcal J} P^{(j)}_1\right| \leq 4(|Y|+|X_1^c|) \;.
\end{equation}

For $i\in\{1,\ldots,r\}$ and for $v\in X_i\setminus X_i'$ write
\begin{align*}
f_{i+1}(v)&:=\deg_{i+1}(v,X_{i+1}\setminus X_{i+1}')\;\mbox{, and}\\
g_{i}(v)&:=\deg_i(v,X_{i-1}')\;. 
\end{align*}
where the sets $X_1', X_{i-1}'$ and $X_{i+1}'$
above refer to the moment\footnote{if $v\in Y$ then this
moment is the zero-th step} when $v$ is removed from $X_i'$ (we do not define
$f_{i+1}(v)$ for $i=r$).

Observe that for each $i\in\{2,\ldots,r\}$, we have
\begin{align}\label{eq:HezkejVodotrysk1}
\sum_{v\in X_i^b}g_i(v)&< \delta kn\;.
\end{align}

We set $X_r^c:=\emptyset$. For a given $i\in [r-1]$ we have
\begin{align}\nonumber
|X_i^{\ref{conc:Match-avoid}}|\cdot \frac{\gamma k}2&\le \sum_{v\in
X_i^{\ref{conc:Match-avoid}}}f_{i+1}(v)\\
& \le \sum_{v\in
X_{i+1}\setminus 
X_{i+1}'}g_{i+1}(v)
\\ \label{eq:Match-induction1}
\JUSTIFY{by~\ref{hyp:Match-Ysmall}.\ ,\ref{hyp:Match-bounded}.\ ,
\eqref{eq:HezkejVodotrysk1}} & <\delta
kn+|X_{i+1}^{\ref{conc:Match-avoid}}| \Omega k,
\end{align}
as $X_i\setminus X_i'\subset X_i^b\cup X_i^{\ref{conc:Match-avoid}}\cup Y$,
for $i=2,\ldots,r$. Using~\eqref{eq:Match-induction1}, we deduce inductively that
$|X_{r-j}^{\ref{conc:Match-avoid}}|\le \left(\frac{8\Omega}\gamma\right)^j\delta
n$ for $j=1,2,\ldots,r-1$, and in particular that
\begin{equation}\label{eq:VX1B}
|X_{1}^{\ref{conc:Match-avoid}}|\le
\left(\frac{8\Omega}\gamma\right)^{r-1}\delta n\;.
\end{equation}

As $X_1^a=\emptyset$, we obtain that
\begin{align*}
|X'_1 | & =\left|X_1\setminus \left(\bigcup_{j\in\mathcal J}P^{(j)}_1\cup
\bigcup_{j\in\mathcal Y}\big(P^{(j)}_1\cap (Y\cup X_1^a\cup
X_1^c)\big)\right)\right|\\ 
\JUSTIFY{by~\eqref{eq:casak}}& \geq |X_1|-4(|Y|+|X_1^c|) -
\left|\bigcup_{j\in\mathcal Y}\left(P^{(j)}_1\cap X_1^a
\right)\right|\\ 
\JUSTIFY{by~\ref{hyp:Match-Ysmall}., \eqref{eq:condCYmatch}, \eqref{eq:VX1B}} &
\ge |X_1|-\frac{\eta n}2 - \left|\bigcup_{j\in\mathcal
Y}(P^{(j)}_1\cap X_1^a )\right|\\ \JUSTIFY{by Cl~\ref{cl:useReg}}  & \ge
|X_1|-\frac{\eta n}2 - \frac{|X_1|}4\\ \JUSTIFY{by \ref{hyp:Match-edges}.}& >0,
\end{align*} 
as desired.
\end{proof}

\subsection{Obtaining a configuration}\label{ssec:obtainingConf}
In this section we prove that the structure in the graph
$G\in\LKSgraphs{n}{k}{\eta}$ guaranteed by Lemma~\ref{prop:LKSstruct}
always leads to one of the configurations
$\mathbf{(\diamond1)}$--$\mathbf{(\diamond10)}$. We distinguish two cases. When the set  $\HugeVertices$ of
 vertices of huge degree (coming from a sparse decomposition
of $G$) sees many edges, then one of the configurations $\mathbf{(\diamond1)}$--$\mathbf{(\diamond5)}$ must
occur (cf.  Lemma~\ref{lem:ConfWhenCXAXB}). Otherwise, when the edges incident with $\HugeVertices$ can be neglected,
we obtain one of the configurations
$\mathbf{(\diamond6)}$--$\mathbf{(\diamond10)}$ 
(cf. Lemmas~\ref{lem:ConfWhenNOTCXAXB} and~\ref{lem:ConfWhenMatching}). How these
configurations help in embedding the tree
$T_\PARAMETERPASSING{T}{thm:main}\in\treeclass{k}$ will be shown in
Section~\ref{sec:embed}.

Lemmas~\ref{lem:ConfWhenCXAXB},~\ref{lem:ConfWhenNOTCXAXB},
and~\ref{lem:ConfWhenMatching} are stated in the next section, and their proofs
occupy
Sections~\ref{sssec:ProofConfWhenCXAXB},~\ref{sssec:ProofConfWhenNOTCXAXB},
and~\ref{sssec:ProofConfWhenMatching}, respectively. These results are put together in Lemma~\ref{outerlemma} of Section~\ref{sssec:StatementsofResultsObtaining}. 

\subsubsection{Statements of the results}\label{sssec:StatementsofResultsObtaining}
We first state the main result of this section, Lemma~\ref{outerlemma}.
Its proof is given in Section~\ref{ssec:proofofouterlemma}.

\begin{lemma}\label{outerlemma}
Suppose we are in Settings~\ref{commonsetting} and~\ref{settingsplitting}. Further suppose
that at least one of the cases {\bf(K1)} or {\bf(K2)} from
Lemma~\ref{prop:LKSstruct} occurs in $G$ (with $\Mgood$ as in Lemma~\ref{prop:LKSstruct}~\eqref{Mgoodisblack}). Then one of the configurations
\begin{itemize}
\item$\mathbf{(\diamond1)}$,
\item$\mathbf{(\diamond2)}\left(\frac{\eta^{27}\Omega^{**}}{4\cdot
10^{66}(\Omega^*)^{11}},\frac{\sqrt[4]{\Omega^{**}}}2,\frac{\eta^9\rho^2}{128\cdot
10^{22}\cdot (\Omega^*)^5}\right)$,
\item$\mathbf{(\diamond3)}\left(\frac{\eta^{27}\Omega^{**}}{4\cdot
10^{66}(\Omega^*)^{11}},\frac{\sqrt[4]{\Omega^{**}}}2,\frac\gamma2,\frac{\eta^9\gamma^2}{128\cdot
10^{22}\cdot(\Omega^*)^5}\right)$,
\item$\mathbf{(\diamond4)}\left(\frac{\eta^{27}\Omega^{**}}{4\cdot
10^{66}(\Omega^*)^{11}},\frac{\sqrt[4]{\Omega^{**}}}2,\frac\gamma2,\frac{\eta^9\gamma^3}{384\cdot
10^{22}(\Omega^*)^5}\right)$, 
\item$\mathbf{(\diamond5)}\left(\frac{\eta^{27}\Omega^{**}}{4\cdot
10^{66}(\Omega^*)^{11}},\frac{
\sqrt[4]{\Omega^{**}}}2,\frac{\eta^9}{128\cdot
10^{22}\cdot
(\Omega^*)^3},\frac{\eta}2,\frac{\eta^9}{128\cdot
10^{22}\cdot (\Omega^*)^4}\right)$,
  \item
  $\mathbf{(\diamond6)}\big(\frac{\eta^3\rho^4}{10^{14}(\Omega^*)^4},4\epsilonD,\frac{\gamma^3\rho}{32\Omega^*},\frac{\eta^2\nu}{2\cdot10^4
 },\frac{3\eta^3}{2000},\proporce{2}(1+\frac\eta{20})k\big)$,
  \item $\mathbf{(\diamond7)}\big(\frac
{\eta^3\gamma^3\rho}{10^{12}(\Omega^*)^4},\frac
{\eta\gamma}{400},4\epsilonD,\frac{\gamma^3\rho}{32\Omega^*},\frac{\eta^2\nu}{2\cdot10^4
}, \frac{3\eta^3}{2\cdot 10^3}, \proporce{2}(1+\frac\eta{20})k\big)$,
  \item
$\mathbf{(\diamond8)}\big(\frac{\eta^4\gamma^4\rho}{10^{15}
(\Omega^*)^5},\frac{\eta\gamma}{400},\frac{400\epsilon}{\eta},4\epsilonD,\frac
d2,\frac{\gamma^3\rho}{32\Omega^*},\frac{\eta\pi\clustersize}{200k},\frac{\eta^2\nu}{2\cdot10^4
}, \proporce{1}(1+\frac\eta{20})k,\proporce{2}(1+\frac\eta{20})k\big)$,
 \item 
 $\mathbf{(\diamond9)}\big(\frac{\rho
\eta^8}{10^{27}(\Omega^*)^3},\frac
{2\eta^3}{10^3}, \proporce{1}(1+\frac{\eta}{40})k,
\proporce{2}(1+\frac{\eta}{20})k, \frac{400\varepsilon}{\eta},
\frac{d}2,
\frac{\eta\pi\clustersize}{200k},4\epsilonD,\frac{\gamma^3\rho}{32\Omega^*},
\frac{\eta^2\nu}{2\cdot10^4 }\big)$,
  \item $\mathbf{(\diamond10)}\big( \epsilon, \frac{\gamma^2
d}2,\pi\sqrt{\epsilon'}\nu k, \frac
{(\Omega^*)^2k}{\gamma^2},\frac\eta{40} \big)$
\end{itemize}
occurs in $G$.
\end{lemma}
Lemma~\ref{outerlemma}  will be proved in Section~\ref{ssec:proofofouterlemma}. The proof relies
on Lemmas~\ref{lem:ConfWhenCXAXB},~\ref{lem:ConfWhenNOTCXAXB} and~\ref{lem:ConfWhenMatching} below. For an input graph $G_\PARAMETERPASSING{L}{outerlemma}$ one of these lemmas is applied depending on the majority type of ``good'' edges in $G_\PARAMETERPASSING{L}{outerlemma}$. Observe that {\bf(K1)} of Lemma~\ref{prop:LKSstruct} guarantees edges between $\HugeVertices$ and $\XA\cup \XB$, or between $\XA$ and $\XA\cup \XB$ either in $E(\Gexp)$ or in $E(\GD)$. Lemma~\ref{lem:ConfWhenCXAXB} is used if we find edges between $\HugeVertices$ and $\XA\cup \XB$. Lemma~\ref{lem:ConfWhenNOTCXAXB} is used  if we find edges of $E(\Gexp)$  between $\XA$ and $\XA\cup \XB$. The remaining case can be reduced to the setting of Lemma~\ref{lem:ConfWhenMatching}.  Lemma~\ref{lem:ConfWhenMatching} is also used to obtain a configuration if we are in case {\bf(K2)} of Lemma~\ref{prop:LKSstruct}.

\begin{lem}\label{lem:ConfWhenCXAXB}
\HAPPY{H}
Suppose we are in Setting~\ref{commonsetting}. Assume that
\begin{align}
\label{libelle7.38b}
e_{\Gcapt}(\HugeVertices,\XA\cup\XB)&\ge \frac{\eta^{13}
kn}{10^{28}(\Omega^*)^3}.
\end{align}
Then $G$ contains at least one of the configurations 
\begin{itemize}
\item$\mathbf{(\diamond1)}$,
\item$\mathbf{(\diamond2)}\left(\frac{\eta^{27}\Omega^{**}}{4\cdot
10^{66}(\Omega^*)^{11}},\frac{\sqrt[4]{\Omega^{**}}}2,\frac{\eta^9\rho^2}{128\cdot
10^{22}\cdot (\Omega^*)^5}\right)$,
\item$\mathbf{(\diamond3)}\left(\frac{\eta^{27}\Omega^{**}}{4\cdot
10^{66}(\Omega^*)^{11}},\frac{\sqrt[4]{\Omega^{**}}}2,\frac\gamma2,\frac{\eta^9\gamma^2}{128\cdot
10^{22}\cdot(\Omega^*)^5}\right)$,
\item$\mathbf{(\diamond4)}\left(\frac{\eta^{27}\Omega^{**}}{4\cdot
10^{66}(\Omega^*)^{11}},\frac{\sqrt[4]{\Omega^{**}}}2,\frac\gamma2,\frac{\eta^9\gamma^3}{384\cdot
10^{22}(\Omega^*)^5}\right)$, or 
\item$\mathbf{(\diamond5)}\left(\frac{\eta^{27}\Omega^{**}}{4\cdot
10^{66}(\Omega^*)^{11}},
\frac{\sqrt[4]{\Omega^{**}}}2,\frac{\eta^9}{128\cdot
10^{22}\cdot
(\Omega^*)^3},\frac{\eta}2,\frac{\eta^9}{128\cdot
10^{22}\cdot (\Omega^*)^4}\right)$.
\end{itemize}
\end{lem}

\smallskip

\begin{lem}\label{lem:ConfWhenNOTCXAXB}
Suppose that we are in Setting~\ref{commonsetting} and
Setting~\ref{settingsplitting}. If there exist two disjoint sets $\YA_1,\YA_2\subset V(G)$ such that
\begin{align}
\label{eq:manyXAXAXBobt}
e_{\Gexp}(\YA_1,\YA_2)&\ge 2\rho kn\;,
\end{align}
and either
\begin{align}
\label{2ndcondiObt2}
\YA_1\cup \YA_2\subset \XA\colouringpI{0}\setminus (\gP\cup \exceptVertSplit\cup
\shadowsplit)&\mbox{, or}\\
\label{3rdcondiObt2}
\YA_1\subset \XA\colouringpI{0}\setminus (\gP\cup \exceptVertSplit\cup
\shadowsplit\cup \gP_2\cup \gP_3)&\mbox{, and }\YA_2\subset
\XB\colouringpI{0}\setminus (\gP\cup \exceptVertSplit\cup \shadowsplit)
\end{align}
then $G$ has
configuration~$\mathbf{(\diamond6)}(\frac{\eta^3\rho^4}{10^{14}(\Omega^*)^3},0,1,1,
\frac {3\eta^3}{2\cdot 10^3},\proporce{2}(1+\frac{\eta}{20})k)$.
\end{lem} 

\smallskip

\begin{lem}\label{lem:ConfWhenMatching}
Suppose that we are in Setting~\ref{commonsetting} and
Setting~\ref{settingsplitting}. Let $\DenseSpots_\class$ be as in Lemma~\ref{lem:clean-spots}. Suppose that there exists an
$(\bar\epsilon,\bar d,\beta k)$-semiregular matching $\M$, $V(\M)\subset \colouringp{0}$, $|V(\M)|\ge \frac {\rho n}{\Omega^*}$, with one of the following two sets of properties.
\begin{itemize}
 \item[{\bf(M1)}] $\M$ is absorbed by $\Mgood$, $\bar\epsilon:=\frac{10^5\epsilon'}{\eta^2}$, $\bar d:=\frac{\gamma^2}4$, and $\beta:=\frac
{\eta^2\clustersize}{8\cdot 10^3 k}$.
 \item[{\bf(M2)}] $E(\M)\subset E(\DenseSpots_\class)$, $\M$ is absorbed by $\DenseSpots_\class$, $\bar\epsilon:=\epsilonD$, $\bar d:=\frac{\gamma^3\rho}{32\Omega^*}$, and $\beta:=\frac{\alphaD\rho}{\Omega^*}$.
\end{itemize}
Suppose further that one of the following occurs.
\begin{itemize}
  \item [$\mathbf{(cA)}$] $V(\M)\subset \XA\colouringpI{0}\setminus
  (\gP\cup\exceptVertSplit\cup\shadowsplit)$, and we have for the set
$$R:=\shadow_{\Gcapt}\left((\largeintoatoms\cap\largevertices{\eta}{k}{G})\setminus V(\M_A\cup\M_B),\frac{2\eta^2 k}{10^5}\right)$$
one of the following
\begin{itemize}
 \item[{\bf(t1)}] $V_1(\M)\subset \shadow_{\Gcapt}(V(\Gexp),\rho k)$,
 \item[{\bf(t2)}]$V_1(\M)\subset \largeintoatoms$,
 \item[{\bf(t3)}]$V_1(\M)\subset R\setminus (\shadow_{\Gcapt}(V(\Gexp),\rho k)\cup \largeintoatoms)$, or
\item[{\bf(t5)}]$V(\M)\subset V(\Gblack)\setminus\left(\shadow_{\Gcapt}(V(\Gexp),\rho k)\cup \largeintoatoms\cup R\right)$.
\end{itemize}
  \item [$\mathbf{(cB)}$] $V_1(\M)\subset \XA\colouringpI{0}\setminus
  (\gP\cup \gP_2\cup\gP_3\cup\exceptVertSplit\cup\shadowsplit)$ and
  $V_2(\M)\subset \XB\colouringpI{0}\setminus
  (\gP\cup\exceptVertSplit\cup\shadowsplit)$, and we have
\begin{itemize}
 \item[{\bf(t1)}]$V_1(\M)\subset \shadow_{\Gcapt}(V(\Gexp),\rho k)$,
 \item[{\bf(t2)}]$V_1(\M)\subset \largeintoatoms$, or
 \item[{\bf(t3--5)}]$V_1(\M)\cap \left(\shadow_{\Gcapt}(V(\Gexp),\rho k)\cup \largeintoatoms\right)=\emptyset$.
\end{itemize}
\end{itemize}
then at least one of the
following configurations occurs:
\begin{itemize}
  \item $\mathbf{(\diamond6)}\big(\frac{\eta^3\rho^4}{10^{12}(\Omega^*)^4},4\epsilonD,
  \frac{\gamma^3\rho}{32\Omega^*},\frac{\eta^2\nu}{2\cdot10^4
  },\frac{3\eta^3}{2000},\proporce{2}(1+\frac\eta{20})k\big)$,
  \item $\mathbf{(\diamond7)}\big(\frac
{\eta^3\gamma^3\rho}{10^{12}(\Omega^*)^4},\frac
{\eta\gamma}{400},4\epsilonD,\frac{\gamma^3\rho}{32\Omega^*},\frac{\eta^2\nu}{2\cdot10^4
}, \frac{3\eta^3}{2000}, \proporce{2}(1+\frac\eta{20})k\big)$,
  \item
$\mathbf{(\diamond8)}\big(\frac{\eta^4\gamma^4\rho}{10^{15}
(\Omega^*)^5},\frac{\eta\gamma}{400},\frac{400\epsilon}{\eta},4\epsilonD,\frac
d2,\frac{\gamma^3\rho}{32\Omega^*},\frac{\eta\pi\clustersize}{200k},\frac{\eta^2\nu}{2\cdot10^4
}, \proporce{1}(1+\frac\eta{20})k,\proporce{2}(1+\frac\eta{20})k\big)$,
 \item 
 $\mathbf{(\diamond9)}\big(\frac{\rho
\eta^8}{10^{27}(\Omega^*)^3},\frac
{2\eta^3}{10^3}, \proporce{1}(1+\frac{\eta}{40})k,
\proporce{2}(1+\frac{\eta}{20})k, \frac{400\varepsilon}{\eta}, \frac{d}2,
\frac{\eta\pi\clustersize}{200k},4\epsilonD,\frac{\gamma^3\rho}{32\Omega^*},
\frac{\eta^2\nu}{2\cdot10^4}\big)$,
  \item $\mathbf{(\diamond10)}\big( \epsilon, \frac{\gamma^2
d}2,\pi\sqrt{\epsilon'}\nu k,\frac
{(\Omega^*)^2k}{\gamma^2},\frac\eta{40} \big)$.
\end{itemize}
\end{lem}

\subsubsection{Proof of Lemma~\ref{outerlemma}}\label{ssec:proofofouterlemma}
In the proof, we distinguish different types of edges captured in cases
{\bf(K1)} and {\bf(K2)}. If in case {\bf(K1)} many of the captured edges from
$\XA$ to $\XA\cup \XB$ are incident with $\HugeVertices$, we will get one of the
configurations $\mathbf{(\diamond1)}$--$\mathbf{(\diamond5)}$ by employing
Lemma~\ref{lem:ConfWhenCXAXB}. Otherwise, there must be many edges from $\XA$ to
$\XA\cup \XB$ in the graph $\Gexp$, or in $\GD$.
Lemma~\ref{lem:ConfWhenNOTCXAXB} shows that the former case leads to
configuration $\mathbf{(\diamond6)}$. We will reduce the latter case to the
situation in Lemma~\ref{lem:ConfWhenMatching} which gives one of the configurations
$\mathbf{(\diamond6)}$--$\mathbf{(\diamond10)}$.

We use Lemma~\ref{lem:ConfWhenMatching} to give one of the
configurations
$\mathbf{(\diamond6)}$--$\mathbf{(\diamond10)}$ also in case
{\bf(K2)}. \footnote{Actually, our proof of
Lemma~\ref{lem:ConfWhenMatching} implies that one does not get configuration $\mathbf{(\diamond9)}$ in case {\bf(K2)}; but this fact is never needed.} 

\bigskip

Let us now turn to the details of the proof. If
$e_G(\HugeVertices,\XA\cup\XB)\ge \frac{\eta^{13}kn}{10^{28}(\Omega^*)^3}$ then
we use Lemma~\ref{lem:ConfWhenCXAXB} to obtain one of the configurations
$\mathbf{(\diamond1)}$--$\mathbf{(\diamond5)}$, with the parameters as in 
the statement of Lemma~\ref{outerlemma}.

Thus, in the remainder of the proof we assume that
\begin{equation}\label{eq:notmanyfromhuge}
e_G(\HugeVertices,\XA\cup\XB)< \frac{\eta^{13}kn}{10^{28}(\Omega^*)^3}\;.
\end{equation}

We now bound the size of the set $\gP$. By
Setting~\ref{commonsetting}\eqref{commonsetting:numbercaptured} we have that 
\begin{equation*}
 |E(G)\sm E(\Gcapt)|\leq 2\rho kn .
\end{equation*}
 Plugging this into Lemma~\ref{lem:YAYB} we get
$|L_\sharp|\le \frac{40\rho n}{\eta}$, $|\XA\setminus\YA|\le \frac{1200\rho
n}{\eta^2}$, and $|(\XA\cup \XB)\setminus\YB|\le \frac{1200\rho
n}{\eta^2}$. Further, using~\eqref{eq:notmanyfromhuge},
Lemma~\ref{lem:YAYB} also gives that $|\WantiC|\le
\frac{\eta^{12}n}{10^{26}(\Omega^*)^3}$. It follows from Setting~\ref{commonsetting}\eqref{commonsettingNicDoNAtom} that $|\gPatoms|\le \gamma n$. Last, by Setting~\ref{commonsetting}\eqref{commonsetting4} we have $|\gP_1|\le 2\gamma n$.
Thus,
\begin{align}\nonumber
|\gP|&\le |\XA\setminus\YA|+|(\XA\cup
\XB)\setminus\YB|+|\WantiC| +|L_\sharp|+|\gP_1|\\
\nonumber
&~~~+\left|\shadow_{\GD\cup\Gcapt}(\WantiC\cup
L_\sharp\cup \gPatoms\cup \gP_1,\frac{\eta^2 k}{10^5})\right|\\
\label{eq:sizeofP}
&\overset{\eqref{eq:KONST}}\le \frac{2\eta^{10}n}{10^{21}(\Omega^*)^2}\;,
\end{align}
where we used Fact~\ref{fact:shadowbound} to bound the size of the shadows.

\bigskip

Let us first turn our attention to case {\bf(K1)}.
By Definition~\ref{def:proportionalsplitting} we have $\HugeVertices\cap
\colouringp{0}=\emptyset$. Therefore,
\begin{align}
\nonumber
e_{\Gcapt}\big(\XA\colouringpI{0}\setminus \gP, &(\XA\cup
\XB)\colouringpI{0}\setminus \gP\big)=e_{\Gcapt}\big((\XA\setminus
(\HugeVertices\cup \gP))\colouringpI{0},(\XA\setminus
(\HugeVertices\cup \gP))\colouringpI{0}\cup (\XB\setminus
\gP)\colouringpI{0}\big)\\
\nonumber
\JUSTIFY{by Def~\ref{def:proportionalsplitting}~\eqref{It:H6}}&\ge
\proporce{0}^2\cdot e_{\Gcapt}\big(\XA\setminus (\HugeVertices\cup \gP), (\XA\cup
\XB)\setminus (\HugeVertices\cup \gP)\big)-k^{0.6}n^{0.6}\\
\nonumber
\JUSTIFY{by~\eqref{eq:proporcevelke}}&\ge \frac
{\eta^2}{10^4}\big(e_{\Gcapt}(\XA,
\XA\cup \XB)-2e_{\Gcapt}(\HugeVertices,
\XA\cup \XB)-2|\gP|\Omega^*k\big)-k^{0.6}n^{0.6}\\
\nonumber
\JUSTIFY{by~{\bf(K1)},~\eqref{eq:notmanyfromhuge},~\eqref{eq:sizeofP}}&\ge
\frac {\eta^2}{10^4}\Big(\frac {\eta kn}{4}-\frac{2\eta^{13}
kn}{10^{28}(\Omega^*)^3}-\frac{4\eta^{10}
kn}{10^{21}\Omega^*}\Big)-k^{0.6}n^{0.6}\\ 
&>\frac {\eta^3 kn}{10^{5}}\;.
\label{eq:Diana9}
\end{align}

We consider the
following two complementary cases:
\begin{itemize}
  \item [$\mathbf{(cA)}$] {$e_{\Gcapt}((\XA\setminus \gP)\colouringpI{0})\ge
 40\rho kn$.}
  \item [$\mathbf{(cB)}$] {$e_{\Gcapt}((\XA\setminus
  \gP)\colouringpI{0})<40\rho kn$.}
\end{itemize}

Note that  $\XA\setminus
\gP\subseteq \YA$, and $(\XA\cup \XB)\setminus
\gP\subseteq \YB$.
  We shall now define in each of the cases~$\mathbf{(cA)}$ and~$\mathbf{(cB)}$ certain sets $\YA_1, \YA_2$ which will have a minimum number of edges between them. Although the definition of these sets is different for the cases~$\mathbf{(cA)}$ and~$\mathbf{(cB)}$, for ease of notation they receive the same names.  
  
In case~$\mathbf{(cA)}$ a standard argument (take a maximal cut) gives disjoint sets $\YA_1, \YA_2\subseteq 
(\XA\setminus (\gP\cup \exceptVertSplit\cup
\shadowsplit))\colouringpI{0}\subseteq \YA$ with
\begin{align}
\nonumber
e_{\Gcapt}(\YA_1,\YA_2)\ge &\frac 12(e_{\Gcapt}(\XA\setminus
\gP)\colouringpI{0}-|\exceptVertSplit\cup \shadowsplit|\cdot \Omega^*k)\\
\label{eq:PocitaniCaseA}
\JUSTIFY{by
 Def~\ref{def:proportionalsplitting}\eqref{It:H1} and by~\eqref{eq:boundShadowsplit}}\ge & \frac
 12(40\rho kn-2\epsilon \Omega^*kn)\notag \\  > &19\rho kn\;.
\end{align}

Let us now define $\YA_1, \YA_2$ for case~$\mathbf{(cB)}$.
Property~\ref{commonsettingXAS0} of Setting~\ref{commonsetting} implies that
\begin{equation}
\label{eq:Dsmallvecer}
|\gP_2|\le \sqrt\gamma n\;\mbox{.}
\end{equation}

Also, by Definition~\ref{def:proportionalsplitting}\eqref{It:H6} we have
\begin{align*}
e_{\Gcapt}(\XA)&\le \frac{1}{\proporce{0}^2}\big(e_{\Gcapt}((\XA\setminus
\gP)\colouringpI{0})+k^{0.6}n^{0.6}\big)+e_{\Gcapt}(\HugeVertices,\XA)+|\gP|\Omega^*k\\
\JUSTIFY{by~\eqref{eq:proporcevelke},~$\mathbf{(cB)}$,~\eqref{eq:notmanyfromhuge},
and~\eqref{eq:sizeofP}}&\le\frac{10^4}{\eta^2}\cdot \big(40\rho kn+
k^{0.6}n^{0.6}\big)+
\frac{\eta^{13}}{10^{28}(\Omega^*)^3}kn+\frac{\eta^{10}}{10^{20}\Omega^*}kn\\
\JUSTIFY{by~\eqref{eq:KONST}}&<\frac {\eta^{8}}{10^{15}\Omega^*}kn\;.
\end{align*}

Consequently, 
$$|\gP_3|\cdot\frac{\eta^3k}{10^3}\le e_{\Gcapt}(\gP_3,\XA)\leq 2\cdot \frac{\eta^8}{10^{15}\Omega^*}kn,$$
and thus,
\begin{align}
\label{eq:D2mala}
|\gP_3|&\le 2\cdot \frac{\eta^5}{10^{12}\Omega^*}n\;.
\end{align}
Set $\YA_1:=(\XA\setminus (\gP\cup \gP_2\cup \gP_3\cup
\exceptVertSplit\cup \shadowsplit))\colouringpI{0}\subseteq \YA$ and
$\YA_2:=(\XB\setminus (\gP\cup \exceptVertSplit\cup
\shadowsplit))\colouringpI{0}\subseteq \YB$. Then the sets $\YA_1$ and $\YA_2$ are
disjoint and we have
\begin{align}
\nonumber
e_{\Gcapt}(\YA_1,\YA_2)&\ge
e_{\Gcapt}\left((\XA\setminus
\gP)\colouringpI{0},((\XA\cup\XB)\setminus
\gP)\colouringpI{0}\right)-2e_{\Gcapt}((\XA\setminus
\gP)\colouringpI{0})\\
\nonumber &~~~-(|\gP_2|+|\gP_3|+2|\exceptVertSplit|+2|\shadowsplit|)\cdot
\Omega^*k
\\
\nonumber
\JUSTIFY{by~\eqref{eq:Diana9}, $\mathbf{(cB)}$, \eqref{eq:Dsmallvecer},
\eqref{eq:D2mala}, D\ref{def:proportionalsplitting}\eqref{It:H1}, \eqref{eq:boundShadowsplit}} &\ge\frac{\eta^3kn}{
10^{5}}-80\rho kn-\sqrt{\gamma}\Omega^*kn-\frac{2\eta^5}{10^{12}}kn-4\epsilon
\Omega^* kn\\
&\geByRef{eq:KONST}
19\rho kn\;.\label{eq:PocitaniCaseB}
\end{align}
We have thus defined $\YA_1, \YA_2$ for both cases~$\mathbf{(cA)}$ and~$\mathbf{(cB)}$.

Observe first that if
$e_{\Gexp}(\YA_1,\YA_2)\ge 2\rho kn$
then we may apply
Lemma~\ref{lem:ConfWhenNOTCXAXB} to obtain
Configuration~$\mathbf{(\diamond6)}(\frac{\eta^3\rho^4}{10^{14}(\Omega^*)^3},0,1,1,
\frac {3\eta^3}{2\cdot 10^3},\proporce{2}(1+\frac{\eta}{20})k)$.
Hence, from now on, let us assume that $e_{\Gexp}(\YA_1,\YA_2)> 2\rho kn$.
Then
by~\eqref{eq:PocitaniCaseA} and~\eqref{eq:PocitaniCaseB} we have that $$e_{\GD}(\YA_1,\YA_2)\ge 17\rho kn.$$

 We fix a family $\DenseSpots_\class$ as in Lemma~\ref{lem:clean-spots}. In particular, we have 
 \begin{equation}\label{flordelino}
  e_{\DenseSpots_\class}(\YA_1,\YA_2)\ge 16\rho kn. 
 \end{equation}

Let $R:=\shadow_{\Gcapt}\left((\largeintoatoms\cap\largevertices{\eta}{k}{G})\setminus V(\M_A\cup\M_B),\frac{2\eta^2 k}{10^5}\right)$. For $i=1,2$ define 
\begin{align}
\begin{split}\label{eq:defY1Y5}
\mathbb{Y}^{(1)}_i& :=\shadow_{G}(V(\Gexp),\rho k)\cap \YA_i\;,\\
\mathbb{Y}^{(2)}_i& :=(\largeintoatoms\cap \YA_i)\setminus\mathbb{Y}^{(1)}_i\;,\\
\mathbb{Y}^{(3)}_i& :=(R\cap \YA_i)\setminus(\mathbb{Y}^{(1)}_i\cup \mathbb{Y}^{(2)}_i)\;,\\
\mathbb{Y}^{(4)}_i& :=(\smallatoms\cap \YA_i)\setminus(\mathbb{Y}^{(1)}_i\cup \mathbb{Y}^{(2)}_i\cup \mathbb{Y}^{(3)}_i)\;,\\
\mathbb{Y}^{(5)}_i& := \YA_i\setminus(\mathbb{Y}^{(1)}_i\cup\ldots \cup 
\mathbb{Y}^{(4)}_i)\;.
\end{split}
\end{align}
Clearly, the sets $\mathbb{Y}^{(j)}_i$ partition $\YA_i$ for $i=1,2$.

We now present two lemmas (one for
case {\bf(cA)} and one for case {\bf(cB)}) which help to distinguish several subcases based on  the majority type of edges we find between $\YA_1$ and $\YA_2$. The first of the two lemmas follows by
simple counting from~\eqref{flordelino}.

\begin{lemma}\label{lem:MajorityTypeCB}
In case {\bf(cB)}, we have one of the following.
\begin{itemize}
 \item[{\bf(t1)}] $e_{\DenseSpots_{\class}}\left(\mathbb{Y}^{(1)}_1,\YA_2\right)\ge 2\rho kn$,
 \item[{\bf(t2)}] $e_{\DenseSpots_{\class}}\left(\mathbb{Y}^{(2)}_1,\YA_2\right)\ge 2\rho kn$,
\item[{\bf(t3)}] $e_{\DenseSpots_{\class}}\left(\mathbb{Y}^{(3)}_1,\YA_2\right)\ge 2\rho kn$,
\item[{\bf(t4)}] $e_{\DenseSpots_{\class}}\left(\mathbb{Y}^{(4)}_1,\YA_2\right)\ge 2\rho kn$, or
\item[{\bf(t5)}] $e_{\DenseSpots_{\class}}\left(\mathbb{Y}^{(5)}_1,\YA_2\right)\ge 2\rho kn$.
\end{itemize}
\end{lemma}

Our second lemma is a bit more involved.

\begin{lemma}\label{lem:MajorityTypeCA}
In case {\bf(cA)}, we have one of the following.
\begin{itemize}
 \item[{\bf(t1)}] $e_{\DenseSpots_{\class}}(\mathbb{Y}^{(1)}_1,\YA_2)+e_{\DenseSpots_{\class}}(\YA_1,\mathbb{Y}^{(1)}_2)\ge 4\rho kn$,
 \item[{\bf(t2)}] $e_{\DenseSpots_{\class}}\left(\mathbb{Y}^{(2)}_1,\YA_2\setminus \mathbb{Y}^{(1)}_2\right)+e_{\DenseSpots_{\class}}\left(\YA_1\setminus \mathbb{Y}^{(1)}_1,\mathbb{Y}^{(2)}_2\right)\ge 4\rho kn$,
\item[{\bf(t3)}] $e_{\DenseSpots_{\class}}\left(\mathbb{Y}^{(3)}_1,\YA_2\setminus (\mathbb{Y}^{(1)}_2\cup \mathbb{Y}^{(2)}_2)\right)
+
e_{\DenseSpots_{\class}}\left(\YA_1\setminus (\mathbb{Y}^{(1)}_1\cup \mathbb{Y}^{(2)}_1),\mathbb{Y}^{(3)}_2\right)\ge 4\rho kn$, or
\item[{\bf(t5)}] $e_{\DenseSpots_{\class}}\left(\mathbb{Y}^{(5)}_1,\mathbb{Y}^{(5)}_2\right)\ge 2\rho kn$.
\end{itemize}
\end{lemma}
\begin{proof}
By~\eqref{flordelino}, we only need to establish that 
$$e_{\DenseSpots_{\class}}\left(\mathbb{Y}^{(4)}_1,\YA_2\setminus (\mathbb{Y}^{(1)}_2\cup \mathbb{Y}^{(2)}_2\cup \mathbb{Y}^{(3)}_2)\right)
+
e_{\DenseSpots_{\class}}\left(\YA_1\setminus (\mathbb{Y}^{(1)}_1\cup \mathbb{Y}^{(2)}_1\cup \mathbb{Y}^{(3)}_1),\mathbb{Y}^{(4)}_2\right)< \rho kn\;.$$
For this, note that $\mathbb{Y}^{(4)}_1\subset\smallatoms$ and that $\YA_2\setminus (\mathbb{Y}^{(1)}_2\cup \mathbb{Y}^{(2)}_2\cup \mathbb{Y}^{(3)}_2)$ is disjoint from $\largeintoatoms$. Thus we have $e_{\DenseSpots_{\class}}\left(\mathbb{Y}^{(4)}_1,\YA_2\setminus (\mathbb{Y}^{(1)}_2\cup \mathbb{Y}^{(2)}_2\cup \mathbb{Y}^{(3)}_2)\right)<\frac{\rho k n}{100\Omega^*}$. We can bound the other summand using a symmetric argument.
\end{proof}

Next, we prove a lemma that will provide the crucial step for finishing case  {\bf(K1)}.
\begin{lemma}\label{lem:Isabelle}
Let $G^*$ be the spanning subgraph of $\GD$ formed by the edges of
$\DenseSpots_\class$. If there are two disjoint sets $Z_1$ and $Z_2$ with
$e_{G^*}(Z_1,Z_2)\ge 2\rho kn$ then there exists an
$(\epsilonD,\frac{\gamma^3\rho}{32\Omega^*},\frac{\alphaD\rho
k}{\Omega^*})$-semiregular matching $\mathcal N$ in $G^*$ with $V_i(\mathcal
N)\subset Z_i$ ($i=1,2$), and $|V(\mathcal N)|\ge\frac{\rho n}{\Omega^*}$.
\end{lemma}
\begin{proof}
As the maximum degree $G^*$ is bounded by $\Omega^* k$,
we have $|Z_1|\ge \frac{2\rho n}{\Omega^*}\ge \frac{2\rho
k}{\Omega^*}$. Thus, $$(G^*,\DenseSpots_\class,G^*[Z_1,Z_2],\{Z_1\})\in\mathcal
G\left(v(\GD), k, \Omega^*, \frac{\gamma^3}4, \frac{\rho}{\Omega^*},2\rho\right)\;.$$
Lemma~\ref{lem:edgesEmanatingFromDensePairsIII} (which applies with these parameters by the choice of $\alphaD$ and $k_0$ by~\eqref{eq:KONST}, also cf.~page~\pageref{pageref:PAR} for the precise choice)
immediately gives the desired output.
\end{proof}
We use Lemma~\ref{lem:Isabelle} with $Z_1,Z_2$ being the pair of sets containing
many edges as in the cases {\bf(t1)}--{\bf(t3)} and {\bf(t5)} of
Lemma~\ref{lem:MajorityTypeCA}\footnote{The quantities in
Lemma~\ref{lem:MajorityTypeCA} have two summands. We take the sets $Z_1$,$Z_2$
as those appearing in the majority summand.} and {\bf(t1)}--{\bf(t5)} of
Lemma~\ref{lem:MajorityTypeCB}. The lemma outputs a semiregular matching
$\M_\PARAMETERPASSING{L}{lem:ConfWhenMatching}:=\mathcal N_\PARAMETERPASSING{L}{lem:Isabelle}$. This matching is a basis of the input for Lemma~\ref{lem:ConfWhenMatching}{\bf(M2)} (subcase {\bf(t1)}--{\bf(t3)}, {\bf(t5)}, or {\bf(t3--5)}). Thus, we get one of the configurations $\mathbf{(\diamond6)}$--$\mathbf{(\diamond10)}$ as in the statement of the lemma. This finishes the proof for case {\bf(K1)}.

\bigskip

Let us now turn our attention to case {\bf(K2)}.
For every pair $(X,Y)\in \Mgood$, let $X'\subseteq X\colouringpI{0}\setminus
(\gP\cup \exceptVertSplit\cup \shadowsplit)$ and $Y'\subseteq Y\colouringpI{0}\setminus
(\gP\cup \exceptVertSplit\cup \shadowsplit)$ be maximal with $|X'|=|Y'|$. Define
$\mathcal N:=\{(X',Y')\::\: (X,Y)\in \Mgood\;,\; |X'|\ge \frac
{\eta^2\clustersize}{2\cdot 10^3}\}$. By Lemma~\ref{lem:RestrictionSemiregularMatching}, and using~\eqref{eq:KONST} and~\eqref{eq:proporcevelke},
we know that 
$$|V(\Mgood\colouringpI{0})|\ge  \frac {\eta^2n}{400}.$$ Therefore, we have
\begin{align}\label{campoafuera}
|V(\mathcal N)|&\ge |V(\Mgood\colouringpI{0})|-2|\gP\cup \exceptVertSplit\cup
\shadowsplit|-2\frac {\eta^2n}{2\cdot 10^3}\notag \\
\JUSTIFY{by {\bf(K2)}, \eqref{eq:sizeofP}, Def\ref{def:proportionalsplitting}\eqref{It:H1}, \eqref{eq:boundShadowsplit}}
&\ge \frac {\eta^2n}{400}-\frac {4\cdot
\eta^{10} n}{10^{21}(\Omega^*)^2}-4\varepsilon n-\frac {\eta^2n}{10^3}\notag \\&>
\frac {\eta^2n}{1000}\;.
\end{align}
By Fact~\ref{fact:BigSubpairsInRegularPairs}, $\mathcal N$ is a $(\frac {4\cdot
10^3\epsilon'}{\eta^2}, \frac{\gamma^2}2,\frac
{\eta^2\clustersize}{2\cdot 10^3})$-semiregular matching.

We use the definitions of the sets $\mathbb{Y}^{(1)}_i,\ldots,\mathbb{Y}^{(5)}_i$ as given in~\eqref{eq:defY1Y5} with $\YA_i:=V_i(\mathcal N)$ ($i=1,2$). As $V(\mathcal N)\subset V(\Gblack)$, we have that $\mathbb{Y}^{(4)}_i=\emptyset$ ($i=1,2$).
A set $X\in \V_i(\mathcal N)$ is said to be of \emph{Type~1} if $\left|X\cap
\mathbb{Y}^{(1)}_i\right|\ge\frac14|X|$. 
Analogously, we define elements of $\V(\mathcal N)$ of \emph{Type~2}, \emph{Type~3}, and \emph{Type~5}.  

By~\eqref{campoafuera} and as $V(\Mgood)\subset \XA$, we are in subcase $\mathbf{(cA)}$. 
For each  $(X_1,X_2)\in\mathcal N$ with at least one $X_i\in \{X_1,X_2\}$
being of Type~1, set $X_i':=X_i \cap \mathbb{Y}^{(1)}_i$ 
and take an arbitrary set $X_{3-i}'\subset X_{3-i}$ of size $|X_i'|$. Note that
by Fact~\ref{fact:BigSubpairsInRegularPairs} $(X'_i,X_{3-i}')$ forms a
$\frac{10^5\epsilon'}{\eta^2}$-regular pair of density at least $\gamma^2/4$. We let $\mathcal N_1$ be the semiregular matching consisting of all pairs
$(X'_i,X_{3-i}')$ obtained in this way.\footnote{Note that we are thus changing
the orientation of some subpairs.}

Likewise, we construct $\mathcal N_2,\mathcal N_3$ and $\mathcal N_5$ using the features of
Type~2, 3, and 5. Observe that the matchings $\mathcal N_i$ may intersect.

Because of~\eqref{campoafuera} and since we included at least one quarter of each $\mathcal N$-edge into one of
$\mathcal N_1,\mathcal N_2,\mathcal N_3$ and $\mathcal N_5$, one of the
semiregular matchings $\mathcal N_i$ satisfies $|V(\mathcal N_i)|\geq \frac{\eta^2n}{16\cdot 1000} \geq \frac{\rho}{\Omega^*}n$. So, $\mathcal N_i$ serves as a matching $\M_\PARAMETERPASSING{L}{lem:ConfWhenMatching}$ for Lemma~\ref{lem:ConfWhenMatching}{\bf(M1)}. Thus, we get one of the configurations $\mathbf{(\diamond6)}$--$\mathbf{(\diamond10)}$ as in the statement of the lemma. This finishes case {\bf(K2)}.

\subsubsection{Proof of Lemma~\ref{lem:ConfWhenCXAXB}}\label{sssec:ProofConfWhenCXAXB}
Set $\tilde\eta:=\frac{\eta^{13}}{10^{28}(\Omega^*)^3}$. 
Define $\NUP:=\{v\in V(G)\::\: \deg_{\Gcapt}(v,\HugeVertices)\ge
k\}$, and $\NDOWN:=\neighbor_{\Gcapt}(\HugeVertices)\setminus \NUP$. 
Recall that by the definition of the class $\LKSsmallgraphs{n}{k}{\eta}$, the
set $\HugeVertices$ is independent, and thus the sets $\NUP$ and $\NDOWN$ are
disjoint from $\HugeVertices$. Also, using the same definition, we have
\begin{align} 
\label{eq:NCpodL} \neighbor_{\Gcapt}(\HugeVertices)&\subset
\largevertices{\eta}{k}{G}\setminus \HugeVertices\quad\mbox{, and thus }\\
\label{eq:MuzuProniknoutsL}
e_{\Gcapt}(\HugeVertices,
B)&=e_{\Gcapt}(\HugeVertices, B\cap\largevertices{\eta}{k}{G})\;\mbox{for any $B\subset V(G)$.}
\end{align}

We shall distinguish two cases.

\noindent\underline{\bf Case A:} $e_{\Gcapt}(\HugeVertices,\NUP)\ge
e_{\Gcapt}(\HugeVertices,\XA\cup\XB)/8$.\\ Let us focus on the bipartite
subgraph $H'$ of $\Gcapt$ induced by the sets $\HugeVertices$ and $\NUP$.
Obviously, the average degree of the vertices of $\NUP$ in $H'$ is at least $k$.

First, suppose that $|\HugeVertices|\le |\NUP|$. Then,
the average degree of $\HugeVertices$ in $H'$ is at least $k$, and hence, the
average degree of $H'$ is at least $k$. Thus, there exists a bipartite subgraph $H\subset H'$ with $\mindeg(H)\ge k/2$. Furthermore, $\mindeg_{\Gcapt}(V(H))\ge k$. We conclude that we are in Configuration~$\mathbf{(\diamond1)}$.

Now, suppose $|\HugeVertices|>|\NUP|$. Using the bounds given by Case~A, and using
\eqref{libelle7.38b}, we get $$|\NUP|\ge
\frac{e_{\Gcapt}(\HugeVertices,\NUP)}{\Omega^*k}\ge\frac{\tilde{\eta}kn}{8\Omega^*k}=\frac{\tilde{\eta}n}{8\Omega^*}\;.$$
Therefore, we have
 $$e(G)\ge\sum_{v\in\HugeVertices}\deg_{\Gcapt}(v)\ge
 |\HugeVertices|\Omega^{**}k>
 |\NUP|\Omega^{**}k\ge\frac{\tilde{\eta}n}{8\Omega^*}\Omega^{**}k\geByRef{eq:KONST}
 kn\;,$$ a contradiction to Property~\ref{def:LKSsmallC} of Definition~\ref{def:LKSsmall}.

\bigskip
\noindent\underline{\bf Case B:} $e_{\Gcapt}(\HugeVertices,\NUP)<
e_{\Gcapt}(\HugeVertices,\XA\cup\XB)/8$.\\ Consequently, we get
\begin{equation}\label{eq:eGCNUB}
e_{\Gcapt}(\HugeVertices,(\XA\cup\XB)\setminus \NUP)\ge
\frac78e_{\Gcapt}(\HugeVertices,\XA\cup\XB)\geByRef{libelle7.38b} \frac
78\tilde\eta kn\;.
\end{equation}

\def\Lenve{\PARAMETERPASSING{L}{lem:envelope}}

We now apply Lemma~\ref{lem:envelope} to $\Gcapt$ with input sets $P_\Lenve:=\HugeVertices$, $Q_\Lenve:=\largevertices{\eta}{k}{G}\setminus \HugeVertices$,
$Y_\Lenve:=\largevertices{\eta}{k}{G}\setminus\largevertices{\frac9{10}\eta}{k}{\Gcapt}$,
and parameters $\psi_\Lenve:=\tilde \eta/100$, $\Gamma_\Lenve:=\Omega^*$, and
$\Omega_\Lenve:=\Omega^{**}$. Assumption~\eqref{eq:envelopeAss1} of
the lemma follows from~\eqref{eq:NCpodL}. The lemma yields three sets
$L'':=Q''_\Lenve$, $L':=Q'_\Lenve$ , $\HugeVertices':=P'_\Lenve$, and it is easy to check that these witness
Preconfiguration~$\mathbf{(\clubsuit)}(\frac{\tilde\eta^3\Omega^{**}}{4\cdot10^6(\Omega^*)^2})$.

Recall that $e(G)\le kn$. Since by the definition of $Y_\Lenve$, we have
$|Y_\Lenve|\le\frac{40\rho}{\eta}n$,
we obtain from Lemma~\ref{lem:envelope}(d)  that
\begin{align}
e_{\Gcapt}(\HugeVertices,\largevertices{\eta}{k}{G})-e_{\Gcapt}(\HugeVertices',L'')&\le
\frac{\tilde \eta}{100}
e_{\Gcapt}(\HugeVertices,\largevertices{\eta}{k}{G})+\frac{|Y_\Lenve|200(\Omega^*)^2k}{\tilde{\eta}}\nonumber\\
\nonumber &\le \frac{\tilde
\eta}{100}kn+\frac{40\rho
n}{\eta}\cdot\frac{200(\Omega^*)^2k}{\tilde{\eta}}\\
&\leByRef{eq:KONST}
\frac{\tilde\eta}2 kn.\label{libelle7.36b}
\end{align}
 So, 
\begin{align}\nonumber
e_{\Gcapt}\big(\HugeVertices',(L''\cap(\XA\cup\XB))\setminus \NUP\big)
&\ge
e_{\Gcapt}\big(\HugeVertices,(\largevertices{\eta}{k}{G}\cap(\XA\cup\XB))\setminus
\NUP\big)
\nonumber\\
&~~-\big(e_{\Gcapt}(\HugeVertices,\largevertices{\eta}{k}{G}\big)-e_{\Gcapt}(\HugeVertices',L'')\big)
\nonumber\\
\nonumber 
&=e_{\Gcapt}(\HugeVertices,(\XA\cup\XB)\setminus
\NUP)
\\ \nonumber
&~~-\big(e_{\Gcapt}(\HugeVertices,\largevertices{\eta}{k}{G}\big)-e_{\Gcapt}(\HugeVertices',L'')\big)\\
&\nonumber\geByRef{libelle7.36b}e_{\Gcapt}(\HugeVertices,(\XA\cup\XB)\setminus \NUP)
-\frac{\tilde\eta}2 kn\\ &\geByRef{eq:eGCNUB} \frac38\tilde\eta kn\;.\label{eq:eGCNUP}
\end{align}

We define $$\HugeVertices^*:=\left\{v\in\HugeVertices'\::\: \deg_{\Gcapt}(v,L''\cap (\XA\cup\XB)\cap
\NDOWN)\ge\sqrt{\Omega^{**}}k\right\}\;.$$ 

Using that $e(G)\le kn$, we shall show the following.
\begin{lemma}\label{lem:C*XAXBNDOWN}
We have 
$e_{\Gcapt}(\HugeVertices^*,L''\cap (\XA\cup\XB)\cap \NDOWN)\ge \frac18\tilde\eta kn$.
\end{lemma}
\begin{proof}
Suppose otherwise. Then by~\eqref{eq:eGCNUP}, we obtain that
 $$e_{\Gcapt}(\HugeVertices'\setminus \HugeVertices^*,L''\cap (\XA\cup\XB)\cap\NDOWN)
 \ge \frac14\tilde\eta kn\;.$$ On the other hand, by the definition of $\HugeVertices^*$,
 $$|\HugeVertices'\setminus \HugeVertices^*|\sqrt{\Omega^{**}}k\ge e_{\Gcapt}(\HugeVertices'\setminus \HugeVertices^*,L''\cap (\XA\cup\XB)\cap\NDOWN)\;.$$ 
Consequently, we have $$|\HugeVertices'\setminus \HugeVertices^*|\ge\frac{\tilde\eta kn}{4\sqrt{\Omega^{**}}k}=\frac{\tilde\eta n}{4\sqrt{\Omega^{**}}}\;.$$ 
Thus, as $\HugeVertices$ is independent, $$e(G)\ge \sum_{v\in\HugeVertices}\deg_{\Gcapt}(v)\ge |\HugeVertices|\Omega^{**}k\ge |\HugeVertices'\setminus \HugeVertices^*|\Omega^{**}k\ge \frac{\tilde\eta}4 \sqrt{\Omega^{**}}kn\gByRef{eq:KONST} kn\;,$$ a contradiction.
\end{proof}

Let us define $O:=\shadow_{\Gcapt}(\smallatoms,\gamma k)$. 
Next, we define \begin{align*}
N_1&:=V(\Gexp)\cap L''\cap (\XA\cup\XB)\cap \NDOWN\;,\\
N_2&:=\smallatoms\cap L''\cap 
(\XA\cup\XB)\cap \NDOWN \;,\\
N_3&:=O\cap L''\cap  (\XA\cup\XB)\cap \NDOWN\;\mbox{, and}\\
N_4&:=(L''\cap (\XA\cup\XB)\cap
\NDOWN)\setminus(N_1\cup N_2\cup N_3)\;.
                \end{align*}
Observe that
\begin{equation}\label{eq:ON4}
O\cap N_4=\emptyset\;.
\end{equation}
Further, for $i=1,\ldots,4$ define
$$C_i:=\left\{v\in\HugeVertices^*\::\: \deg_{\Gcapt}(v, N_i)\ge \deg_{\Gcapt}(v, L''\cap (\XA\cup\XB)\cap
\NDOWN)/4\right\}\;.$$ Easy counting gives that there exists an index $i\in[4]$ such
that
\begin{equation}\label{eq:OneInFour}
e_{\Gcapt}(C_i,N_i)\ge \frac1{16}e_{\Gcapt}(\HugeVertices^*,L''\cap (\XA\cup\XB)\cap
\NDOWN)\geBy{L\ref{lem:C*XAXBNDOWN}}\frac1{128}\tilde\eta kn\;.
\end{equation}

\def\LSP{\PARAMETERPASSING{L}{lem:clean-C+black}}
\def\L74{\PARAMETERPASSING{L}{lem:clean-C+yellow}}
Set $Y:=(\XA\cup\XB)\setminus (\YB\cup\HugeVertices)=(\XA\cup\XB)\setminus \YB$, and
$\eta_\L74=\eta_\LSP:=\frac{1}{128}\tilde\eta$. By Lemma~\ref{lem:YAYB} we have
\begin{equation}\label{eq:Ysmall}
|Y|< \frac{\eta_\L74 n}{4\Omega^*}\;.
\end{equation}

We split the rest of the proof into four subcases according to the value of $i$.

\noindent\underline{\bf Subcase B, $i=1$.}\\
We shall apply Lemma~\ref{lem:clean-C+yellow} with
 $r_\L74:=2$,
$\Omega^*_\L74:=\Omega^*,\Omega^{**}_\L74:=\sqrt{\Omega^{**}}/4$,
$\delta_\L74:=\frac{\eta_\L74\rho^2}{100(\Omega^*)^2}$, $\gamma_\L74:=\rho$,
$\eta_\L74$, $X_0:=C_1$, $X_1:=N_1$, and $X_2:=V(\Gexp)$, and
$Y$, and the graph $G_\L74$, which is formed by the vertices of $G$, with all edges from $E(\Gcapt)$ that are
in $E(\Gexp)$ or that  are incident with $\HugeVertices$. We briefly
verify the assumptions of Lemma~\ref{lem:clean-C+yellow}. First of all the choice of  $\delta_\L74$ guarantees
that $\left(\frac{3\Omega_\L74^*}{\gamma_\L74}\right)^2\delta_\L74<\frac{\eta_\L74}{10}$. Assumption~\ref{hyp:Ysmall} is
given by~\eqref{eq:Ysmall}. Assumption~\ref{hyp:C+y-edges} holds since we assume 
that~\eqref{eq:OneInFour} is satisfied for $i=1$ and by definition of
$\eta_\L74$. Assumption~\ref{hyp:C+y-large} follows from the definitions of $C_1$ and of $\HugeVertices^*$. Assumption~\ref{hyp:C+y-deg} follows from the fact that $X_1\subset V(\Gexp)=X_2$, and since $\mindeg(\Gexp)>\rho k$ which is guaranteed by the definition of
a $(k,\Omega^{**},\Omega^*,\Lambda,\gamma,\epsilon',\nu,\rho)$-sparse decomposition. This definition also guarantees 
Assumption~\ref{hyp:C+y-bounded}, as $Y\cup X_1\cup X_2\subset V(G)\setminus \HugeVertices$.

Lemma~\ref{lem:clean-C+yellow} outputs sets $\HugeVertices'':=X_0'$, $V_1:=X_1'$,
$V_2:=X_2'$ with $\mindeg_{\Gcapt}(\HugeVertices'', V_1)\ge \sqrt[4]{\Omega^{**}}k/2$ (by~\eqref{conc:C+y-large}),
$\maxdeg_{\Gexp}(V_1,X_2\setminus V_2)< \rho k/2$ (by~\eqref{conc:C+y-avoid}),
$\mindeg_{\Gcapt}(V_1,\HugeVertices'')\ge \delta_\L74 k$ (by~\eqref{conc:C+y-deg}), and
$\mindeg_{\Gexp}(V_2,V_1)\ge \delta_\L74 k$ (by~\eqref{conc:C+y-deg}).
By~\eqref{conc:X1Ydisj},  we have that $V_1\subset \YB\cap
L''$. As
$\mindeg_{\Gexp}(V_1,X_2)\ge \mindeg(\Gexp)\ge \rho k$, we have $\mindeg_{\Gexp}(V_1,V_2)\ge \mindeg_{\Gexp}(V_1,X_2)-\maxdeg_{\Gexp}(V_1,X_2\setminus
V_2)\ge \delta_\L74 k$. 

Since $L'$, $L''$ and $\HugeVertices'$ witness
Preconfiguration~$\mathbf{(\clubsuit)}(\frac{\tilde \eta^{3}\Omega^{**}}{4\cdot
10^{66}(\Omega^*)^{11}})$, this verifies that we have Configuration $\mathbf{(\diamond2)}\left(\frac{\tilde \eta^{3}\Omega^{**}}{4\cdot
10^{66}(\Omega^*)^{11}},\sqrt[4]{\Omega^{**}}/2,\frac{\tilde\eta\rho^2}{12800(\Omega^*)^2}\right)$.

\noindent\underline{\bf Subcase B, $i=2$.}\\
We apply Lemma~\ref{lem:clean-C+yellow} with numerical parameters $r_\L74:=2$,
$\Omega^*_\L74:=\Omega^*$, $\Omega^{**}_\L74:=\sqrt{\Omega^{**}}/4$,
$\delta_\L74:=\frac{\eta_\L74
\gamma^2}{100(\Omega^*)^2}$, $\gamma_\L74:=\gamma$, and $\eta_\L74$. Further input
to the lemma are sets $X_0:=C_2$, $X_1:=N_2$, and $X_2:=V(G)\setminus
\HugeVertices$, and the set $Y$. The underlying graph $G_\L74$ is the graph
$\GD$ with all egdes incident with $\HugeVertices$ added. Verifying assumptions
of Lemma~\ref{lem:clean-C+yellow} is analogous to Subcase~B, $i=1$ with the exception of Assumption 4. Let us therefore turn to verify it.
To this end, it suffices to observe that each each vertex in $X_1$ is contained
in at least one $(\gamma k,\gamma)$-dense spot from $\DenseSpots$
(cf.~Definition~\ref{def:avoiding}), and thus has degree at least $\gamma k$ in
$X_2$.
 
The output of Lemma~\ref{lem:clean-C+yellow} are sets $X_0',X_1'$, and $X_2'$ which witness
Configuration~$\mathbf{(\diamond3)}(\frac{\tilde \eta^{3}\Omega^{**}}{4\cdot
10^{66}(\Omega^*)^{11}},$ $\sqrt[4]{\Omega^{**}}/2,\gamma/2,\frac{\tilde\eta\gamma^2}{12800(\Omega^*)^2})$. In fact, the only thing not analogous to the preceding subcase is that we have to check~\eqref{eq:WHtc}, in other words, we have to verify that $$\maxdeg_{\GD}\big(X_1', V(G)\setminus (X_2'\cup\HugeVertices)\big)\leq \frac{\gamma k}2\;.$$ As $V(G)\setminus (X_2'\cup\HugeVertices)=X_2\setminus X'_2$, this follows from~\eqref{conc:C+y-avoid} of Lemma~\ref{lem:clean-C+yellow}.

\noindent\underline{\bf Subcase B, $i=3$.}\\
We apply Lemma~\ref{lem:clean-C+yellow} with numerical parameters $r_\L74:=3$,
$\Omega^*_\L74:=\Omega^*$, $\Omega^{**}_\L74:=\sqrt{\Omega^{**}}/4$, 
$\delta_\L74:=\frac{\eta_\L74\gamma^3}{300(\Omega^*)^3}$, $\gamma_\L74:=\gamma$,
and $\eta_\L74$. Further inputs are the sets $X_0:=C_3$, $X_1:=N_3$, $X_2:=\smallatoms$, and $X_3:=V(G)\setminus \HugeVertices$, and the set $Y$. The underlying graph is
$G_\L74:=\Gcapt\cup\GD$.
Verifying assumptions Lemma~\ref{lem:clean-C+yellow} is analogous to Subcase~B,
$i=1$, only for Assumption 4 we observe that $\mindeg_{\Gcapt\cup\GD}
(X_1,X_2)\ge\mindeg_{\Gcapt}
(X_1,X_2)\geq\gamma k$ by definition of $X_1=N_3\subset O$, and
$\mindeg_{\Gcapt\cup\GD}(X_2,X_3)\ge \mindeg_{\GD}(X_2,X_3)\ge\gamma
k$ for the same reason as in Subcase~B, $i=2$.
  
Lemma~\ref{lem:clean-C+yellow} outputs
Configuration~$\mathbf{(\diamond4)}\left(\frac{\tilde \eta^{3}\Omega^{**}}{4\cdot
10^{66}(\Omega^*)^{11}},\sqrt[4]{\Omega^{**}}/2,\gamma/2,\frac{\tilde\eta\gamma^3}{38400(\Omega^*)^3}\right)$, with $\HugeVertices'':=X_0'$, $V_1:=X_1'$, $\smallatoms':=X'_2$ and $V_2:=X'_3$. Indeed, all calculations are similar to the ones in the preceding two subcases, we only need to note additionally that $\mindeg_{\Gcapt\cup\GD}(V_1,\smallatoms')\geq \frac{\gamma k}{2} \frac{\tilde\eta\gamma^3 k}{38400(\Omega^*)^3}$, which follows from the definition of $N_3$ and of $O$.

\noindent\underline{\bf Subcase B, $i=4$.}\\
We have $\clusters\neq\emptyset$ and $\clustersize$ is the size of an arbitrary cluster in $\clusters$. We are going apply
Lemma~\ref{lem:clean-C+black} with $\delta_\LSP:=\eta_\LSP/100$, $\eta_\LSP$,
$h_\LSP:=\eta_\LSP \clustersize/(100\Omega^*)$, $\Omega^*_\LSP:=\Omega^*$,
$\Omega^{**}_\LSP:=\sqrt{\Omega^{**}}/4$ and sets $X_0:=C_4$, $X_1:=N_4$, and
$Y$. The underlying graph is $G_\LSP:=\Gcapt$, and $\mathcal C_\LSP$ is the set
of clusters $\clusters$.

The fact $e(G)\leq kn$ together with~\eqref{eq:OneInFour} and the choice of $\eta_\LSP$ gives
Assumption~\ref{hyp:C+b-edges} of  Lemma~\ref{lem:clean-C+black}. The choice of
$C_4$ and $\HugeVertices^*$ gives Assumption~\ref{hyp:C+b-large}. The fact that
$X_1\cap \HugeVertices=\emptyset$ yields Assumption~\ref{hyp:C+b-bounded}. With the
help of~\eqref{eq:KONST} it is easy to check Assumption~\ref{hypfburg}.
Inequality~\eqref{eq:Ysmall} implies Assumption~\ref{hypfburg5}.
To
verify Assumption~\ref{hypfburg6}, it
is enough to use that $|\mathcal C_\LSP|\le \frac{n}\clustersize$. We have thus verified all the
assumptions of Lemma~\ref{lem:clean-C+black}.

We claim that Lemma~\ref{lem:clean-C+black} outputs Configuration
$\mathbf{(\diamond5)}\left(
\frac{\tilde \eta^{3}\Omega^{**}}{4\cdot
10^{66}(\Omega^*)^{11}},
\sqrt[4]{\Omega^{**}}/2,\frac{\tilde\eta}{12800},\frac\eta2,\frac{\tilde\eta}{12800\Omega^*}\right)$, with $\HugeVertices'':=X_0'$ and $V_1:=X_1'$.
In fact,  all conditions of the configuration, except 
condition~\eqref{confi5last}, which we check below, are easy to verify. (Note that $V_1\subseteq \YB$ since $V_1\subseteq X_1=N_4\subseteq\XA\cup\XB$. Also,  $V_1\subseteq L''$, and thus disjoint from $\HugeVertices$. Moreover, by the conditions of Lemma~\ref{lem:clean-C+black}, $V_1$ is disjoint from $Y$. So, $V_1\subset\YB$.) 
For~\eqref{confi5last}, observe that~\eqref{eq:ON4}
implies that $\maxdeg_{\Gcapt}(N_4,\smallatoms)\leq\gamma k$. Further, we have
$X_1'\subset N_4\setminus Y$. So for all $x\in X_1'\subseteq
\NDOWN\setminus Y$, we have that $\deg_{\Gcapt}(x,V(G)\setminus \HugeVertices)\ge \frac{9\eta k}{10}$. As $N_4\subseteq \bigcup\clusters\setminus
V(\Gexp)$, we obtain $\deg_{\Gblack}(x)\ge \frac{9\eta
k}{10}-\gamma k\ge\frac{\eta k}2$, fulfilling~\eqref{confi5last}.

\subsubsection{Proof of Lemma~\ref{lem:ConfWhenNOTCXAXB}\HAPPY{M}}\label{sssec:ProofConfWhenNOTCXAXB}
Set $\YA_1':=\{v\in\YA_1\::\:\deg_{\Gexp}(v,\YA_2)\ge\rho k\}$. By~\eqref{eq:manyXAXAXBobt} we have
\begin{align}
\label{eq:manyXAXAXBobtLP}
e_{\Gexp}(\YA_1',\YA_2)&\ge \rho kn\;.
\end{align}

Set $r_\PARAMETERPASSING{L}{lem:clean-yellow}:=3$,
$\Omega_\PARAMETERPASSING{L}{lem:clean-yellow}:=\Omega^*$,
$\gamma_\PARAMETERPASSING{L}{lem:clean-yellow}:=\frac{\rho\eta}{10^3}$, 
$\delta_\PARAMETERPASSING{L}{lem:clean-yellow}:=\frac
{\eta^3\rho^4}{10^{14}(\Omega^*)^3}$,
$\eta_\PARAMETERPASSING{L}{lem:clean-yellow}:=\rho$.
Observe that~\eqref{eq:condCY} is satisfied for these parameters.
Set $Y_\PARAMETERPASSING{L}{lem:clean-yellow}:=\exceptVertSplit$, $X_0:=\YA_2$,
$X_1:=\YA_1'$, $X_2=X_3:=V(\Gexp)\colouringpI{1}$, and $V:=V(G)$.
Let $E_2:=E(\Gcapt)$, and $E_1=E_3:=E(\Gexp)$. We now briefly verify
conditions~\ref{hyp:yel-Ysmall}--\ref{hyp:yel-bounded} of Lemma~\ref{lem:clean-yellow}.
Condition~\ref{hyp:yel-Ysmall} follows from 
Definition~\ref{def:proportionalsplitting}\eqref{It:H1} and~\eqref{eq:KONST}.
Condition~\ref{hyp:yel-edges} follows from~\eqref{eq:manyXAXAXBobtLP}. 
Using Definition~\ref{def:proportionalsplitting}\eqref{It:H5},~\eqref{eq:proporcevelke} and~\eqref{eq:KONST}, we see that
Condition~\ref{hyp:yel-deg} for $i=1$ follows from the definition of $\YA_1'$, and for $i=2$ from the fact that $\mindeg(\Gexp)\ge\rho k$.
Last, Condition~\ref{hyp:yel-bounded} follows from the fact that
$\bigcup_{i=0}^3 X_i$ is disjoint from $\HugeVertices$. 

Lemma~\ref{lem:clean-yellow} yields four non-empty
sets $X_0',\ldots,X_3'$. By 
assertions~\eqref{conc:yel-deg},~\eqref{conc:yel-avoid},~\eqref{conc:yel-X0X1}, and
hypothesis~\ref{hyp:yel-deg} of Lemma~\ref{lem:clean-yellow},
for all $i\in\{0,1,2,3\}$, $j\in\{i-1, i+1\}\sm \{-1,4\}$ we have
\begin{equation}
\label{eq:towardsD62D63Exp}
\mindeg_{H_{i,j}}(X_i',X_j')\ge
\delta_\PARAMETERPASSING{L}{lem:clean-yellow} k,
\end{equation}
where $H_{i,j}=\Gexp$, except for $\{i,j\}=\{1,2\}$, where $H_{i,j}=G_\class$.

Thus, the sets $X'_0$ and $X'_1$ witness
Preconfiguration~$\mathbf{(exp)}(\delta_\PARAMETERPASSING{L}{lem:clean-yellow})$.
By Lemma~\ref{lem:propertyYA12}, and by~\eqref{2ndcondiObt2} and~\eqref{3rdcondiObt2}, the pair
$X_0',X_1'$ together with the cover $\mathcal F$ from~\eqref{def:Fcover} witnesses either 
Preconfiguration~$\mathbf{(\heartsuit1)}(\frac
{3\eta^3}{2\cdot 10^3},\proporce{2}(1+\frac{\eta}{20})k)$ (with respect to $\mathcal F$)
or Preconfiguration~$\mathbf{(\heartsuit2)}(\proporce{2}(1+\frac{\eta}{20})k)$.

Notice that~\eqref{eq:towardsD62D63Exp}
establishes the properties~\eqref{COND:D6:1}--\eqref{COND:D6:4}. Thus the sets $X_0',\ldots,X_3'$ witness
Configuration~$\mathbf{(\diamond6)}(\delta_\PARAMETERPASSING{L}{lem:clean-yellow},0,1,1,
\frac{3\eta^3}{2\cdot 10^3},\proporce{2}(1+\frac{\eta}{20})k)$.

\HIDDENTEXT{The old version (something might be wanted to be copied from it for OBT in HIDDENTEXT.txt, ``PALOALTO''}

\subsubsection{Proof of Lemma~\ref{lem:ConfWhenMatching}}\label{sssec:ProofConfWhenMatching}
In Lemmas~\ref{whatwegetfrom(t1)}, \ref{whatwegetfrom(t2)}, \ref{whatwegetfrom(t3)}, \ref{ObtConf9}, \ref{whatwegetfrom(t5)} below, we show that cases $\mathbf{(t1)}$, $\mathbf{(t2)}$, $\mathbf{(t3)}$, $\textbf{(t3--t5)}$, and $\mathbf{(t5)}$ of Lemma~\ref{lem:ConfWhenMatching}
lead to configuration $\mathbf{(\diamond6)}$, $\mathbf{(\diamond7)}$, $\mathbf{(\diamond8)}$, $\mathbf{(\diamond9)}$, and $\mathbf{(\diamond10)}$, respectively. While the first three of these cases are resolved by a fairly straightforward application of the Cleaning Lemma (Lemma~\ref{lem:clean-Match}), the latter two cases require some further non-trivial computations.

\def\MgoodR{\Mgood\colouringpI{0}} 

\begin{lemma}\label{whatwegetfrom(t1)} 
In case $\mathbf{(t1)}$ (of either subcase $\mathbf{(cA)}$ or subcase $\mathbf{(cB)}$) we obtain Configuration
$\mathbf{(\diamond6)}\big(\frac{\eta^3\rho^4}{10^{12}(\Omega^*)^4},4\epsilonD, \frac{\gamma^3\rho}{32\Omega^*}, 
 \frac{\eta^2\nu}{2\cdot10^4}, \frac{3\eta^3}{2000},
\proporce{2}(1+\frac\eta{20})k\big)$.
\end{lemma}
\begin{proof}

 We use Lemma~\ref{lem:clean-Match}
with the following input parameters:
$r_\PARAMETERPASSING{L}{lem:clean-Match}:=3$,
$\Omega_\PARAMETERPASSING{L}{lem:clean-Match}:=\Omega^*$,
$\gamma_\PARAMETERPASSING{L}{lem:clean-Match}:=\eta\rho/200$,
$\eta_\PARAMETERPASSING{L}{lem:clean-Match}:=\rho/(2\Omega^*)$,
$\delta_\PARAMETERPASSING{L}{lem:clean-Match}:=\eta^3\rho^4/(10^{12}(\Omega^*)^4)$,
$\epsilon_\PARAMETERPASSING{L}{lem:clean-Match}:=\bar{\epsilon}$,
$\mu_\PARAMETERPASSING{L}{lem:clean-Match}:=\beta$ and 
$d_\PARAMETERPASSING{L}{lem:clean-Match}:=\bar{d}$. Note these parameters
satisfy the numerical conditions of Lemma~\ref{lem:clean-Match}. We use the vertex sets
$Y_\PARAMETERPASSING{L}{lem:clean-Match}:=\exceptVertSplit\cup \shadowsplit$, $X_0:=V_2(\M)$,
$X_1:=V_1(\M)$, $X_2=X_3:=V(\Gexp)\colouringpI{1}$, and $V:=V(G)$. 
The partitions of $X_0$
and $X_1$ in Lemma~\ref{lem:clean-Match} are the ones induced by $\V(\M)$, and
the set $E_1$ consists of all edges from $E(\DenseSpots_\class)$ between pairs from $\M$. 
Further, set $E_2:=E(\Gcapt)$ and $E_3:=E(\Gexp)$. 
 
Let us verify the conditions of Lemma~\ref{lem:clean-Match}.
Condition~\ref{hyp:Match-Ysmall} follows from 
Definition~\ref{def:proportionalsplitting}\eqref{It:H1} and~\eqref{eq:boundShadowsplit}.
Condition~\ref{hyp:Match-edges} holds by the assumption on $\M$.
Condition~\ref{hyp:Match-deg}  follows from Definition~\ref{def:proportionalsplitting}\eqref{It:H5}
by~\eqref{eq:proporcevelke}, and for $i=1$ also from the definition of $\M$.
Conditions~\ref{hyp:Match-reg} hold by the definition of $\M$. Finally,
Condition~\ref{hyp:Match-bounded} follows from the properties of the sparse
decomposition~$\class$.

The output of Lemma~\ref{lem:clean-Match} are four sets $X'_0,\ldots,X'_3$. By
Lemma~\ref{lem:propertyYA12}, the sets $X_0'$ and $X_1'$ witness
Preconfiguration~$\mathbf{(\heartsuit
1)}(3\eta^3/(2\cdot 10^3),\proporce{2}(1+\frac{\eta}{20})k)$, or
$\mathbf{(\heartsuit 2)}(\proporce{2}(1+\frac{\eta}{20})k)$.
Further, Lemma~\ref{lem:clean-Match}\eqref{conc:Match-superreg} gives that $(X_0',X_1')$ witnesses Preconfiguration
$\mathbf{(reg)}(4\bar{\epsilon},\bar{d}/4,\beta/2)$.
 It is now easy to verify that we have Configuration
$\mathbf{(\diamond6)}\big(\frac{\eta^3\rho^4}{10^{12}(\Omega^*)^4},4\bar{\epsilon},\frac
{\bar{d}}{4},\frac{\beta}{2},\frac{3\eta^3}{2\cdot 10^3},
\proporce{2}(1+\frac\eta{20})k\big)$.

This leads to
Configuration~$\mathbf{(\diamond6)}$ with parameters as claimed. Indeed, no matter whether we have {\bf(M1)} or
{\bf(M2)}, we have $4\epsilonD\ge 4\cdot \frac{10^5\epsilon'}{\eta^2}$, and
$\gamma^3\rho/(32\Omega^*)\le\gamma^2/4$, and $\eta^2\nu/(2\cdot10^4)\le\eta^2\clustersize/(8\cdot10^3
 k)\le\eta^2\epsilon'/(8\cdot10^3)\le \alphaD\rho/\Omega^*$ (for the latter recall that $\clustersize\leq \eps ' k$ by Definition~\ref{bclassdef}~\eqref{Csize}).
\end{proof}

\begin{lemma}\label{whatwegetfrom(t2)}
In case $\mathbf{(t2)}$ (of either subcase $\mathbf{(cA)}$ or subcase $\mathbf{(cB)}$) we obtain Configuration
$\mathbf{(\diamond7)}\big(\frac{\eta^3\gamma^3\rho}{10^{12}(\Omega^*)^4},\frac{\eta\gamma}{400},4\epsilonD, \frac{\gamma^3\rho}{32\Omega^*},
 \frac{\eta^2\nu}{2\cdot10^4},
\frac{3\eta^3}{2\cdot 10^3}, \proporce{2}(1+\frac\eta{20})k\big)$.
\end{lemma}
\begin{proof}
 We use
Lemma~\ref{lem:clean-Match} with the following input parameters:
$r_\PARAMETERPASSING{L}{lem:clean-Match}:=3$,
$\Omega_\PARAMETERPASSING{L}{lem:clean-Match}:=\Omega^*$,
$\gamma_\PARAMETERPASSING{L}{lem:clean-Match}:=\eta\gamma/200$,
$\eta_\PARAMETERPASSING{L}{lem:clean-Match}:=\rho/\Omega^*$,
$\delta_\PARAMETERPASSING{L}{lem:clean-Match}:=\eta^3\gamma^3\rho/(10^{12}(\Omega^*)^4)$,
$\epsilon_\PARAMETERPASSING{L}{lem:clean-Match}:=\bar{\epsilon}$,
$\mu_\PARAMETERPASSING{L}{lem:clean-Match}:=\beta$ and 
$d_\PARAMETERPASSING{L}{lem:clean-Match}:=\bar d$. We use the
 vertex sets
$Y_\PARAMETERPASSING{L}{lem:clean-Match}:=\exceptVertSplit\cup \shadowsplit$, $X_0:=V_2(\M)$,
$X_1:=V_1(\M)$, $X_2:=\smallatoms\colouringpI{1}$,
$X_3:=\colouringp{1}$, and $V:=V(G)$. The partitions of $X_0$
and $X_1$ in Lemma~\ref{lem:clean-Match} are the ones induced by $\V(\M)$, and
the set $E_1$ consists of all edges from $E(\DenseSpots_{\class})$ between pairs from $\M$. Further, set $E_2:=E(\Gcapt)$ and $E_3:=E(\GD)$. 

The conditions of Lemma~\ref{lem:clean-Match} are verified as before, let us just note that
Condition~\ref{hyp:Match-deg} follows from Definition~\ref{def:proportionalsplitting}\eqref{It:H5}
and by~\eqref{eq:proporcevelke}, and for $i=1$ from the definition of $\M$,  
while for $i=2$ it holds since $\smallatoms$ is covered by the set $\DenseSpots$
of $(\gamma k, \gamma)$-dense spots (cf.~Definition~\ref{def:avoiding}).

It is now easy to check that the output of Lemma~\ref{lem:clean-Match} are sets that witness Configuration 
$\mathbf{(\diamond7)}\big(\frac{\eta^3\gamma^3\rho}{10^{12}(\Omega^*)^4},\frac{\eta\gamma}{400},4\bar{\epsilon},\frac{\bar{d}}{4},\frac{\beta}{2},
\frac{3\eta^3}{2\cdot 10^3}, \proporce{2}(1+\frac\eta{20})k\big)$.
\end{proof}

Before proceeding with dealing with cases $\mathbf{(t3)}$, $\mathbf{(t5)}$ and
{\bf (t3--5)} we state some properties of the matching
$\bar\M:=\big(\M_A\cup\M_B\big)\colouringpI{1}$.

\begin{lemma}\label{lem:MRes}
For $V_{\mathrm{leftover}}:=V(\M_A\cup\M_B)\colouringpI{1}\setminus V(\bar\M)$ and  $Y_{\bar\M}:=\exceptVertSplit\cup\shadowsplit\cup\shadow_{\GD}(V_{\mathrm{leftover}},\frac{\eta^2 k}{1000})$, we have
\begin{enumerate}[(a)]
\item \label{eq:M-semiregN}
$\bar\M$ is a
$(\frac{400\varepsilon}{\eta},\frac
d2,\frac{\eta\pi\clustersize}{200})$-semiregular matching absorbed by $\M_A\cup
\M_B$ and
$V(\bar\M)\subseteq\colouringp{1}$, and
\item  \label{eq:sizeYM}
$|Y_{\bar\M}|\le\frac{3000\epsilon\Omega^* n}{\eta^2}$.
\end{enumerate}
\end{lemma}
\begin{proof}
Lemma~\ref{lem:MRes}~\eqref{eq:M-semiregN} follows from Lemma~\ref{lem:RestrictionSemiregularMatching}.

Observe that from properties \eqref{It:H1} and \eqref{It:H3} of Definition~\ref{def:proportionalsplitting} we can calculate that
\begin{equation}\label{sizeofleftoverN}
|V_\text{leftover}|\leq 3\cdot k^{0.9}\cdot |\M_A\cup \M_B|+\left|\bigcup
\exceptSemSplit\cup \exceptSemSplit^*\right|\leq 3\cdot k^{0.9}\cdot \frac{n}{2\pi
\clustersize}+2\exp(-k^{0.1})\leByRef{eq:KONST}2\epsilon n.
\end{equation}
Then
\begin{align*}
|Y_{\bar\M}|&\le |\bar V|+|\shadowsplit|+
\left|\shadow_{\GD}\left(V_\text{leftover}, \frac{\eta^2 k}{1000}\right)\right|\\ \JUSTIFY{by Fact~\ref{fact:shadowbound}}&\le |\bar V| + |\shadowsplit|+|V_\text{leftover}| \frac{1000\Omega^*}{\eta^2}\\
\JUSTIFY{by~\eqref{sizeofleftoverN}, \textrm{D}\ref{def:proportionalsplitting}\eqref{It:H1}, \eqref{eq:KONST} \eqref{eq:boundShadowsplit}}&<
\frac{3000\epsilon\Omega^* n}{\eta^2}\;,
\end{align*}
as desired for Lemma~\ref{lem:MRes}\eqref{eq:sizeYM}.
\end{proof}

\begin{lemma}\label{whatwegetfrom(t3)}
In Case $\mathbf{(t3)} \mathbf{(cA)}$  we obtain Configuration
$\mathbf{(\diamond8)}\big(\frac{\eta^4\gamma^4\rho }{10^{15}
(\Omega^*)^5},\frac{\eta\gamma}{400},\frac{400\epsilon}{\eta},4\bar{\epsilon},\frac
d2,\frac{\bar{d}}{4},\frac{\eta\pi\clustersize}{200k},\frac{\beta}{2},$ $
\proporce{1}(1+\frac\eta{20})k,\proporce{2}(1+\frac\eta{20})k\big)$.
\end{lemma}

\begin{proof}
We use Lemma~\ref{lem:clean-Match} with the following input parameters:
$r_\PARAMETERPASSING{L}{lem:clean-Match}:=4$,
$\Omega_\PARAMETERPASSING{L}{lem:clean-Match}:=\Omega^*$,
$\gamma_\PARAMETERPASSING{L}{lem:clean-Match}:=\eta\gamma/200$,
$\eta_\PARAMETERPASSING{L}{lem:clean-Match}:=\rho/\Omega^*$,
$\delta_\PARAMETERPASSING{L}{lem:clean-Match}:=\eta^4\gamma^4\rho/(10^{15}(\Omega^*)^5)$,
$\epsilon_\PARAMETERPASSING{L}{lem:clean-Match}:=\bar{\epsilon}$,
$\mu_\PARAMETERPASSING{L}{lem:clean-Match}:=\beta$ and
$d_\PARAMETERPASSING{L}{lem:clean-Match}:=\bar d$. 
We use the following vertex sets
$Y_\PARAMETERPASSING{L}{lem:clean-Match}:=Y_{\bar\M}$, $X_0:=V_2(\M)$,
$X_1:=V_1(\M)$, $$X_2:=(\largevertices{\eta}{k}{G}\cap\largeintoatoms)\colouringpI{0}\setminus\big(V(\Gexp)\cup\smallatoms\cup V(\M_A\cup\M_B)\cup \WantiC\cup L_\sharp\cup \gPatoms\cup\gP_1\big)\;,$$
$X_3:=\smallatoms\colouringpI{1}$,  $X_4:=\colouringp{1}$, and $V:=V(G)$. The partitions $P^{(j)}_i$ of $X_0$ and $X_1$ in Lemma~\ref{lem:clean-Match} are the ones induced by $\V(\M)$, and
the set $E_1$ consists of all edges from $E(\DenseSpots_\class)$ between pairs from $\M$.  Further, set $E_2=E_3:=E(\Gcapt)$ and $E_4:=E(\GD)$.

Most of the conditions of Lemma~\ref{lem:clean-Match} are verified as before, let us only note the few differences. Condition~\ref{hyp:Match-Ysmall} follows from Lemma~\ref{lem:MRes}\eqref{eq:sizeYM}. Using Definition~\ref{def:proportionalsplitting}\eqref{It:H5}
and~\eqref{eq:proporcevelke}, we find that
Condition~\ref{hyp:Match-deg} for $i=2$ follows from the definition of $\largeintoatoms$, and Condition~\ref{hyp:Match-deg}
for $i=3$ holds as it is the same as Condition~\ref{hyp:Match-deg}
for $i=2$ in Lemma~\ref{whatwegetfrom(t2)}. In order to prove Condition~\ref{hyp:Match-deg} for $i=1$ we first observe that since we are in case  $\mathbf{(t3)} $, we have 
\begin{equation}\label{eq:usememe}
V_1(\M)\subset 
\shadow_{\Gcapt}\left((\largeintoatoms\cap\largevertices{\eta}{k}{G})\setminus
V(\M_A\cup\M_B),\frac{2\eta^2 k}{10^5}\right)\setminus(
\shadow_{\Gcapt}(V(\Gexp),\rho k)\cup \largeintoatoms)\;.
\end{equation}
Also, since we in case $\mathbf{(cA)}$, we have
\begin{equation}\label{eq:usememezwei}
V_1(\M)\cap\gP=\emptyset\;.
\end{equation}

Thus, for each $v\in V_1(\M)$ we have, using Definition~\ref{def:proportionalsplitting}\eqref{It:H5},
\begin{align*}
\deg_{\Gcapt}(v,X_2)&\ge\proporce{0} \Big(\deg_{\Gcapt}(v,(\largevertices{\eta}{k}{G}\cap\largeintoatoms)\setminus V(\M_A\cup\M_B))\\
&~~~~~~~~~-\deg_{\Gcapt}(v,V(\Gexp)\cup\smallatoms\cup\WantiC\cup L_\sharp\cup \gPatoms\cup\gP_1) \Big)-k^{0.9}\\
\JUSTIFY{by~\eqref{eq:usememe} \& \eqref{eq:usememezwei} \& \eqref{eq:proporcevelke}}&\ge \frac{\eta}{100}\left(\frac{2\eta^2 k}{10^5}-\rho k- \frac{\rho k}{100\Omega^*}-\frac{\eta^2 k}{10^5}\right)-k^{0.9}\\
\JUSTIFY{by~\eqref{eq:KONST}}&\ge \frac{\eta \gamma k}{200}\;,
\end{align*}
which indeed verifies Condition~\ref{hyp:Match-deg} for $i=1$.

Define $\mathcal N:=\bar\M\setminus\{(X,Y)\in\bar\M\::\:X\cup Y\subset
V(\NAtom)\}$. By Lemma~\ref{lem:MRes}~\eqref{eq:M-semiregN} we have that
$\mathcal N\subseteq \bar{\mathcal M}$ is a
$(\frac{400\epsilon}{\eta},\frac{d}{2},\frac{\eta\pi
\clustersize}{200})$-semiregular matching absorbed by $\M_A\cup\M_B$, and that
$V(\mathcal N)\subset\colouringp{1}$.

To see that the output of
Lemma~\ref{lem:clean-Match} together with the matching $\mathcal N$ leads to Configuration 
$\mathbf{(\diamond8)}\big(\frac{\eta^4\gamma^4\rho }{10^{15}
(\Omega^*)^5},\frac{\eta\gamma}{400},\frac{400\epsilon}{\eta},4\bar{\epsilon},\frac
d2,\frac{\bar{d}}{4},\frac{\eta\pi\clustersize}{200k},\frac{\beta}{2},
\proporce{1}(1+\frac\eta{20})k,\proporce{2}(1+\frac\eta{20})k\big)$ let us show that~\eqref{COND:D8:7} is satisfied (the other conditions are more easily seen to hold). 

For this, let $v\in X_2'$. We have to show that

\begin{equation}\label{todososolhos}
\deg_{\GD}(v,X'_3)+\deg_{\Gblack}(v,V(\mathcal N))
\ge  \proporce{1}(1+\frac{\eta}{20})k.
\end{equation}

 Note that $v\not\in V(\Gexp)$, and thus $\deg_{\Gexp}(v)=0$. This allows us to
calculate as follows:

\begin{align}
\begin{split}\label{eq:OPL}
\deg_{\GD}(v,X'_3)+\deg_{\Gblack}(v,V(\mathcal N))&\ge
\deg_{\Gcapt}(v,\colouringp{1})-\deg_{\GD}(v,X_3\setminus X'_3)\\
&~~~-\deg_{\Gblack}(v,V(\NAtom))-\deg_{\Gblack}(v,V_\mathrm{leftover})\\
&~~~-\deg_{\Gblack}(v,V(G)\setminus V(\M_A\cup\M_B))\;.
\end{split}
\end{align}

We now bound the terms of the right-hand side of~\eqref{eq:OPL}.
From Definition~\ref{def:proportionalsplitting}\eqref{It:H5} we obtain that
$\deg_{\Gcapt}(v,\colouringp{1})\ge
\proporce{1}\left(\deg_{\Gcapt}(v)-\deg_G(v,\HugeVertices)\right)-k^{0.9}$.
Lemma~\ref{lem:clean-Match}\eqref{conc:Match-avoid} gives that
$\deg_{\GD}(v,X_3\setminus X'_3)\le \frac{\eta\gamma k}{400}$. As $v\not\in
\gPatoms\cup V(\M_A\cup \M_B)$, we have $\deg_{\Gblack}(v,V(\NAtom))<\gamma k$. As $v\not\in Y_{\bar\M}$ and thus $v\not\in
\shadow_{\GD}\left(V_\text{leftover}, \frac{\eta^2 k}{1000}\right)$  we have
$\deg_{\GD}(v,V_\text{leftover})\le \frac{\eta^2 k}{1000}$. Last, recall that
$v\not\in \gP_1\cup V(\M_A\cup \M_B)$, and consequently
$\deg_{\Gblack}(v,V(G)\setminus V(\M_A\cup\M_B))<\gamma k$. Putting these bounds together, we find that
\begin{align*}
\deg_{\GD}(v,X'_3)+\deg_{\Gblack}(v,V(\mathcal N))&\ge
\proporce{1}\left(\deg_{\Gcapt}(v)-\deg_G(v,\HugeVertices)\right)-\frac{2\eta^2 k}{1000}\\
\JUSTIFY{as $v\in \largevertices{\eta}{k}{G}\setminus
(L_\sharp\cup\WantiC)$}&\ge \proporce{1}\left((1+\frac{9\eta}{10})k-\frac{\eta
k}{100}\right)-\frac{ \eta^2 k}{500}\\
\JUSTIFY{by~\eqref{eq:proporcevelke} \& \eqref{eq:KONST}}&\ge  \proporce{1}(1+\frac{\eta}{20})k\;.
\end{align*}
This shows~\eqref{todososolhos}.
\end{proof}

\begin{lemma}\label{ObtConf9}
In case {\bf(t3--5)}$\mathbf{(cB)}$ we get Configuration
$\mathbf{(\diamond9)}\big(\frac{\rho
\eta^8}{10^{27}(\Omega^*)^3},\frac
{2\eta^3}{10^3}, \proporce{1}(1+\frac{\eta}{40})k,
\proporce{2}(1+\frac{\eta}{20})k, \frac{400\varepsilon}{\eta},
\frac{d}2, \frac{\eta\pi\clustersize}{200k},4\epsilonD,\frac{\gamma^3\rho}{32\Omega^*}, \frac{\eta^2 \nu}{2\cdot 10^4}\big)$.
\end{lemma}
\begin{proof}
Recall that by Lemma~\ref{lem:propertyYA12}
we know that $\mathcal F$, as defined in~\eqref{def:Fcover}, is an $(\M_A\cup \M_B)$-cover. We introduce another $(\M_A\cup \M_B)$-cover, 
$$\mathcal F':=\mathcal F\cup\{X\in \V(\M_B)\::\: X\subset \smallatoms\}\;.$$
By~\eqref{eq:propertyYA12cB3} and as we are in case $\mathbf{(cB)}$, we have $\maxdeg_{\Gcapt}\left(V_1(\M),\bigcup
\mathcal F\right)\le \frac{2\eta^3}{3\cdot 10^3} k$.
Furthermore, as we are in case {\bf(t3--5)}, we have
$V_1(\M)\cap \largeintoatoms=\emptyset$. Thus,

\begin{equation}\label{cl:laptopN}
\maxdeg_{\Gcapt}\left(V_1(\M),\bigcup
\mathcal F'\right)\le \frac{2\eta^3}{10^3} k.
\end{equation}

 We use
Lemma~\ref{lem:clean-Match} with the following input parameters:
$r_\PARAMETERPASSING{L}{lem:clean-Match}:=2$,
$\Omega_\PARAMETERPASSING{L}{lem:clean-Match}:=\Omega^*$,
$\gamma_\PARAMETERPASSING{L}{lem:clean-Match}:=\eta^4/10^{11}$,
$\eta_\PARAMETERPASSING{L}{lem:clean-Match}:=\rho/2\Omega^*$,
$\delta_\PARAMETERPASSING{L}{lem:clean-Match}:=\rho
\eta^8/(10^{27}(\Omega^*)^3)$,
$\epsilon_\PARAMETERPASSING{L}{lem:clean-Match}:=\bar{\epsilon}$, $\mu_\PARAMETERPASSING{L}{lem:clean-Match}:=\beta$ and
$d_\PARAMETERPASSING{L}{lem:clean-Match}:=\bar d$. 
We use the following vertex sets
$Y_\PARAMETERPASSING{L}{lem:clean-Match}:=Y_{\bar\M}$, $X_0:=V_2(\M)$,
$X_1:=V_1(\M)$, and $X_2:=V(\bar\M)\setminus \bigcup\mathcal F'\subset
\bigcup\clusters\colouringpI{1}$.
The partitions of $X_0$ and $X_1$ in Lemma~\ref{lem:clean-Match} are the ones
induced by $\V(\M)$, and the set $E_1$ consists of all edges from $E(\DenseSpots_{\class})$
between pairs from $\M$.  Further, set $E_2:=E(\GD)$.

Condition~\ref{hyp:Match-Ysmall} of Lemma~\ref{lem:clean-Match} follows from Lemma~\ref{lem:MRes}\eqref{eq:sizeYM}.
Condition~\ref{hyp:Match-edges} follows by the assumption of Lemma~\ref{ObtConf9} on the size of $V(\M)$. 
Condition~\ref{hyp:Match-reg} follows from the definition of $\mathcal M$. 
Condition~\ref{hyp:Match-bounded} holds since $V(\M)$ does not meet $\HugeVertices$.

It remains to see Condition~\ref{hyp:yel-deg}, for $i=1$.
For this, first note that from Lemma~\ref{lem:propertyYA12} we get
that
\begin{equation}\label{batman}
\mindeg_{\Gcapt}\left(V_1(\M),\Vgood\colouringpI{1}\right)\overset{\mathbf{(cB)}}\ge
\mindeg_{\Gcapt}\left(\XA\setminus(\gP\cup
\exceptVertSplit),\Vgood\colouringpI{1}\right)\ge\proporce{1}(1+\frac{\eta}{20})k\;.
\end{equation}

From this, we calculate that
\begin{align}\label{tikitikitiki}
\notag
\mindeg_{\GD}\big(V_1(\M),V(\M_A\cup \M_B)\colouringpI{1}\big) & 
\ge
\mindeg_{\Gcapt}\big(V_1(\M),V(\M_A\cup
 \M_B)\colouringpI{1}\big)\\ 
 \nonumber & ~~~-\maxdeg_{\Gexp}\big(V_1(\M),
 V(\M_A\cup\M_B)\big)\\
 \JUSTIFY{by~\eqref{eq:defVgood} \& ~\eqref{eq:defV+}}&\ge \notag
 \mindeg_{\Gcapt}\big(V_1(\M), \Vgood\colouringpI{1}\big)\\ 
\notag  &~~~-\maxdeg_{\Gcapt}\big(V_1(\M),
 \smallatoms\big)\\ 
  &~~~\notag-\maxdeg_{\Gcapt}\big(V_1(\M),
  \largevertices{\eta}{k}{G}\setminus V(\M_A\cup\M_B)\big) \\  
  \notag &~~~-\maxdeg_{\Gcapt}\big(V_1(\M),
 V(\Gexp)\setminus V(\M_A\cup\M_B)\big)\\
 & \notag ~~~-\maxdeg_{\Gcapt}\big(V_1(\M), V(\Gexp)\cap
 V(\M_A\cup\M_B)\big)\\
  \JUSTIFY{by~\eqref{batman}, as $V_1(\M)\cap
\largeintoatoms=\emptyset$ \& $\mathbf{(cB)}$} &\ge \proporce{1}(1+\frac {\eta}{20}k)-\frac{\rho k}{100\Omega^*}\notag \\ 
 \nonumber &~~~-\maxdeg_{\Gcapt}\big(\XA\setminus \gP_3,\XA\big)\notag \\ 
 \nonumber &~~~-\maxdeg_{\Gcapt}(V_1(\M),
V(\Gexp))\\
  \JUSTIFY{by def of $\gP_3$ \& as $V_1(\M)\cap \shadow_{G}(V(\Gexp),\rho
 k)=\emptyset$ by \bf{(t3--5)}} &\ge  \proporce{1}(1+\frac{\eta}{20})k
 -\frac{\rho k}{100\Omega^*}-\frac{\eta^3 k}{10^3}-\rho k.
\end{align}

We obtain
\begin{align}
\nonumber
\mindeg_{\GD}(V_1(\M)\setminus
Y_\PARAMETERPASSING{L}{lem:clean-Match},X_2)&
\ge  \mindeg_{\GD}\big(V_1(\M)\setminus
Y_{\bar M},V(\bar \M)\big)-\maxdeg_{\GD}(V_1(\M),
\bigcup \mathcal F')\\ 
\JUSTIFY{by def of $\bar \M$, ~\eqref{cl:laptopN}} & \ge  \nonumber
\mindeg_{\GD}\big(V_1(\M),V(\M_A\cup \M_B)\colouringpI{1}\big)\\
\nonumber &~~~-\maxdeg_{\GD}(V_1(\M)\setminus 
 Y_{\bar M}, V_{\mathrm{leftover}})
 -\frac {2\eta^3k}{10^3}\\
 \nonumber
  \JUSTIFY{by~\eqref{tikitikitiki} and by def of $Y_\PARAMETERPASSING{L}{lem:clean-Match}$} &\ge  \proporce{1}(1+\frac{\eta}{20})k -\frac{\rho k}{100\Omega^*}-\frac{\eta^3 k}{10^3}-\rho k-\frac{\eta^2 k}{1000}-\frac {2\eta^3k}{10^3}\\
  \label{eq:densecase-degToV_1N}
&\ge
 \proporce{1}(1+\frac{\eta}{30})k\;.
\end{align}

Since the last term is greater than
$\gamma_\PARAMETERPASSING{L}{lem:clean-Match} k = \frac {\eta^4}{10^{11}}k$ by~\eqref{eq:proporcevelke}, we
see that Condition~\ref{hyp:yel-deg} of Lemma~\ref{lem:clean-Match} is satisfied.

The output of Lemma~\ref{lem:clean-Match} are three non-empty sets
$X_0',X_1',X_2'$ disjoint from $Y_\PARAMETERPASSING{L}{lem:clean-Match}$,
together with $(4\bar\epsilon,\frac{\bar d}4)$-super-regular pairs
$\{Q_0^{(j)},Q_1^{(j)}\}_{j\in\mathcal Y}$ which cover $(X'_0,X'_1)$ with the
following properties.
\begin{align} 
\label{eq:MaSi}
\JUSTIFY{by
Lemma~\ref{lem:clean-Match}~\eqref{conc:Match-superreg}}&\min\left\{|Q_0^{(j)}|,|Q_1^{(j)}|\right\}\ge
\frac{\beta k}{2}\;\mbox{for each $j\in\mathcal Y$}\;,\\
\label{eq:towardsD92N}\JUSTIFY{by
Lemma~\ref{lem:clean-Match}~\eqref{conc:Match-deg}}&\mindeg_{\GD}(X_2',X_1')\ge
\delta_\PARAMETERPASSING{L}{lem:clean-Match} k\;,\\ 
\begin{split}
\label{eq:towardsD93N}\JUSTIFY{by
Lemma~\ref{lem:clean-Match}~\eqref{conc:Match-avoid} and \eqref{eq:densecase-degToV_1N}}&\mindeg_{\GD}(X_1',X_2')
\ge\proporce{1}(1+\frac{\eta}{30})k-\frac
{\eta^4k}{2\cdot 10^{11}}\\& ~~~~~~~~~~~~~~~~~~~~~~~~~\ge \proporce{1}(1+\frac{\eta}{40})k\;.
\end{split}
\end{align}

We now verify that the sets $X_0',X_1',X_2'$, 
the semiregular matching $\mathcal
N_\PARAMETERPASSING{D}{def:CONF9}:=\bar\M$ together with the
$(\M_A\cup\M_B)$-cover $\mathcal F'$, and the family $\{(Q_0^{(j)},Q_0^{(j)})\}_{j\in\mathcal Y}$ satisfy all the conditions of
Configuration~$\mathbf{(\diamond9)}(\delta_\PARAMETERPASSING{L}{lem:clean-Match},\frac
{2\eta^3}{10^3}, \proporce{1}(1+\frac{\eta}{40})k,
\proporce{2}(1+\frac{\eta}{20})k, \frac{400\varepsilon}{\eta},
\frac{d}2, \frac{\eta\pi\clustersize}{200k},4\epsilonD,\gamma^3\rho/32\Omega^*, \eta^2 \nu/2\cdot 10^4)$.

By Lemma~\ref{lem:propertyYA12}, since we are in case $\mathbf{(cB)}$ and by~\eqref{cl:laptopN}, the pair
$X_0',X_1'$ together with the $(\M_A\cup \M_B)$-cover $\mathcal F'$ witnesses
Preconfiguration~$\mathbf{(\heartsuit1)}(\frac
{2\eta^3}{10^3},\proporce{2}(1+\frac{\eta}{20})k)$. By Lemma~\ref{lem:MRes}~\eqref{eq:M-semiregN},
$\bar\M$ is as required for Configuration~$\mathbf{(\diamond9)}$.

To see that $G$ is in Preconfiguration~$\mathbf{(reg)}(4\epsilonD,\gamma^3\rho/32\Omega^*, \eta^2 \nu/2\cdot 10^4)$, note that $4\bar\eps\le 4\epsilonD$ and $\bar d/4\ge  \gamma^3\rho/32\Omega^*$ (in both cases {\bf (M1)} and  {\bf (M1)}).
Further,
 Property~\eqref{COND:reg:0} follows from~\eqref{eq:MaSi} since $\beta/2\ge \eta^2 \nu/2\cdot 10^4$.
 
  Finally, by definition of
$X_2$, the set $X_2'$ is as required, with
 Property~\eqref{conf:D9-XtoV} following from~\eqref{eq:towardsD93N}, and Property~\eqref{conf:D9-VtoX} following from~\eqref{eq:towardsD92N}.
\end{proof}

We are now reaching the last lemma of this section, dealing with the last remaining case.

\begin{lemma}\label{whatwegetfrom(t5)}
In Case $\mathbf{(t5)}\mathbf{(cA)}$ we get Configuration
$\mathbf{(\diamond10)}\big(\epsilon,\frac{\gamma^2
d}{2},\pi\sqrt{\epsilon'}\nu k, \frac
{(\Omega^*)^2k}{\gamma^2},\frac{\eta}{40}\big)$.
\end{lemma}
\begin{proof}
Since we are in case  $\mathbf{(t5)}$, we have  $V(\M)\subset V(\Gblack)$. Therefore,
\begin{align}
\mindeg_{\Gblack}(V(\M),\Vgood)&\ge \notag
\mindeg_{\Gcapt}(V(\M),V_+\setminus L_\sharp)
-\maxdeg_{\Gcapt}(V(\M),\HugeVertices)\\
\notag
&~~~-\maxdeg_{\Gcapt}(V(\M),\smallatoms)
-\maxdeg_{\Gcapt}(V(\M),V(\Gexp))\\
\label{eq:3veci2N}
&\ge (1+\frac\eta{20})k,
\end{align}
where the last line follows as $V(\M)\subseteq \XA\sm\gP\subseteq\YA\setminus\WantiC$ by $\mathbf{(cA)}$ and furthermore, $V(\M)\cap (\shadow_{G}(V(\Gexp),\rho k)\cup \largeintoatoms)=\emptyset$ by  $\mathbf{(t5)}$.

Define
\begin{align}
\nonumber
\gC&:=\left\{C\setminus \big(L_\#\cup V(\M_A\cup \M_B)\cup
\WantiC\cup\gP_1\big)\::\: C\in\clusters\right\}\;,
\index{mathsymbols}{*C@$\gC$}
\\ \nonumber
\gC^-&:=\left\{C\in\gC\::\:
|C|<\sqrt{\epsilon'}\clustersize\right\}\;,
\index{mathsymbols}{*C@$\gC^-$}
\end{align}

 We have
 \begin{equation}
 \left|\bigcup \gC^-\right|\leq \sum_{C\in\gC}\sqrt{\epsilon'}|C|\le
\sqrt{\epsilon'} n\;.\label{eq:CminusMala}
 \end{equation}

Set $\V^\circ:= \V(\M_A\cup \M_B)\cup (\gC\setminus \gC^-)$ and let $G^\circ$ be
the subgraph of $G$ with vertex set $\bigcup \V^\circ$ and all edges from $E(\Gblack)$
induced by $\bigcup \V^\circ$ plus all edges of $E(\Gcapt)\setminus E(\Gexp)$
between $X$ and $Y$ for all $(X,Y)\in\M_A\cup \M_B$.
Apply Fact~\ref{fact:BigSubpairsInRegularPairs} (and recall Definition~\ref{bclassdef}~\eqref{defBC:RL}) to see that each pair of sets $X,Y\in\V^\circ$ forms an $\eps$-regular pair of density either $0$ or at least $\gamma^2d/2$ (whose edges either lie in $\Gblack$ or touch $\smallatoms$).

Next, observe that from Setting~\ref{commonsetting}~\eqref{commonsetting2},
Fact~\ref{fact:sizedensespot} and Fact~\ref{fact:boundedlymanyspots}, and using Definition~\ref{bclassdef}\eqref{defBC:prepartition}, we find that for all $X\in \V^\circ$ which lie in some cluster of $\clusters$, we have
$|\bigcup \neighbor_{G^\circ}(X)|\le |\bigcup\neighbor_{\GD}(X)|\le 
\frac {\Omega^*}{\gamma}\cdot \frac {\Omega^* k}{\gamma}$.
Also, observe that for all $X\in \V^\circ$ which do not lie in some cluster of
$\clusters$, we know from Setting~\ref{commonsetting}~\eqref{commonsetting3}
that $X$ does not see any edges from $E(\Gblack)$. This means that  $\bigcup
\neighbor_{G^\circ}(X)$ is contained in the partner of $X$ in $\M_A\cup M_B$
(which has size at most $\clustersize\le \epsilon'k$ by
Setting~\ref{commonsetting}~\eqref{commonsetting3} and
Definition~\ref{bclassdef}~\eqref{Csize}).

Thus we obtain that
\begin{equation}\label{lem:regularizedobt3}
\text{$(G^\circ,{\V}^\circ)$ is
an $(\epsilon, \frac{\gamma^2
d}2,\pi\sqrt{\epsilon'}\clustersize, \frac
{(\Omega^*)^2k}{\gamma^2})$-regularized graph.}
\end{equation}

 Define $$\mathcal
L^\circ:=\left\{X\in \V^\circ\setminus  \V(\M_A\cup \M_B)\::\: \mindeg_{G^\circ}(X)\ge
(1+\frac{\eta}2)k\right\}.$$

We claim that the following holds.
\begin{claim}\label{lem:thedesiredAandBdoexist}
There are  distinct $X_A,X_B\in \V^\circ$, with $E(G^\circ[X_A,X_B])\neq \emptyset$, such that we have $\deg_{\Gblack}(v,V(\M_A\cup\M_B)\cup\bigcup\mathcal L^\circ)\ge
(1+\frac{\eta}{40})k$ for all but at most $2\epsilon'\clustersize$ vertices $v\in X_A$, and all but at most $2\epsilon'\clustersize$ vertices $v\in X_B$. 
\end{claim}

Then, setting $\tilde G_\PARAMETERPASSING{D}{def:CONF10}:=G^\circ$, $\V_\PARAMETERPASSING{D}{def:CONF10}:=\V^\circ$, $\M_\PARAMETERPASSING{D}{def:CONF10}:=\M_A\cup \M_B$, $\mathcal L^*_\PARAMETERPASSING{D}{def:CONF10}:=\mathcal L^\circ$, $A_\PARAMETERPASSING{D}{def:CONF10}:= X_A$, and  $B_\PARAMETERPASSING{D}{def:CONF10}:= X_B$,
we have obtained
Configuration~$\mathbf{(\diamond10)}\big( \epsilon, \frac{\gamma^2
d}2,\pi\sqrt{\epsilon'}\nu k,\frac {(\Omega^*)^2k}{\gamma^2},\eta/40 \big)$.
Indeed,
 using~\eqref{lem:regularizedobt3}, and the definition of $\mathcal L^\circ$ we see that 
$( \tilde
G_\PARAMETERPASSING{D}{def:CONF10},\V_\PARAMETERPASSING{D}{def:CONF10})$,
$\M_\PARAMETERPASSING{D}{def:CONF10}$ and $\mathcal
L^*_\PARAMETERPASSING{D}{def:CONF10}$ are as desired and
fulfil~\eqref{diamond10cond3}.
  Claim~\ref{lem:thedesiredAandBdoexist} together with the fact that
  $\deg_{G^\circ}(v, V(\M_A\cup \M_B)\cup \bigcup \mathcal L^\circ)\ge
  \deg_{\Gblack}(v, V(\M_A\cup \M_B)\cup \bigcup \mathcal L^\circ)$ for all
  $v\in V(G^\circ)$ ensure that also \eqref{diamond10cond1}
  and~\eqref{diamond10cond2} hold.

\medskip

It only remains to prove Claim~\ref{lem:thedesiredAandBdoexist}.

\begin{proof}[Proof of Claim~\ref{lem:thedesiredAandBdoexist}]
 In order  to find $X_A$ and $X_B$ as in the statement of the lemma, we shall exploit the matching $\M$; the relation between $\M$ and
$(G^\circ,{\V}^\circ)$, $\M_A\cup\M_B$, and $\mathcal L^\circ$ is not direct. We
proceed as follows. In  Subclaim~\ref{lem:ABobt3} we find a suitable $\M$-edge.
In case $\mathbf{(M1)}$ this $\M$-edge gives readily a suitable pair
$(A_\PARAMETERPASSING{D}{def:CONF10},B_\PARAMETERPASSING{D}{def:CONF10})$. In
case $\mathbf{(M2)}$ we have to work on the $\M$-edge to get a suitable
$\BGblack$-edge, this will be done  in Subclaim~\ref{cl:CACB}. Only then do we
find $(A_\PARAMETERPASSING{D}{def:CONF10},B_\PARAMETERPASSING{D}{def:CONF10})$.

\begin{subclaim}\label{lem:ABobt3}
There is an $\M$-edge $(A,B)$ such that
$\deg_{\Gblack}(v,V(\M_A\cup\M_B)\cup\bigcup\mathcal L^\circ)\ge
(1+\frac{\eta}{40})k+\frac{\eta k}{200}$ for at least $|A|/2$ vertices $v\in A$, and at least $|B|/2$
vertices $v\in B$.
\end{subclaim}
\begin{proof}[Proof of Subclaim~\ref{lem:ABobt3}]
Set
$S:=\shadow_{\Gblack}(\bigcup \gC^-,\frac{\eta
 k}{200})$, and note that by Fact~\ref{fact:shadowbound} we have $|S|\le |\bigcup\mathcal C^-|\cdot \frac {200\Omega^*}{\eta}$. So, setting
 $\M_S:=\{(X,Y)\in\M\::\:|(X\cup Y)\cap S|\ge |X\cup Y|/4\}$ we find that $$|V(\M_S)|\ \le \  4|S| \ \overset{\eqref{eq:CminusMala}}\le \ \frac
{800\sqrt{\epsilon'}\Omega^*n}{\eta} \ < \ \frac{\rho n}{ \Omega^*} \ \leq \
|V(\M)|,$$ where the last inequality holds by assumption of
Lemma~\ref{whatwegetfrom(t5)}. Consequently, $\M\neq\M_S$.

Let $(A,B)\in\M\setminus \M_S$. We will show that $(A,B)$ satisfies the
requirements of the subclaim. 
 We start by proving that
 \begin{equation}\label{fact:VplusCapG0}
V_{+}\cap V(G^\circ)\sm (V(\M_A\cup \M_B)\cup \bigcup \mathcal L^\circ)\subseteq  V(\Gexp)\cup
(\largeintoatoms\cap \largevertices{\eta}{k}{G}).
 \end{equation}
Indeed, observe that by~\eqref{defV+eq},
 \begin{align*}
 V_{+}\cap V(G^\circ) &\subseteq V(\M_A\cup \M_B)\cup V(\Gexp) \cup \big(
 \largevertices{\eta}{k}{G}\sm (L_\sharp \cup\WantiC\cup \gP_1)\big)\\
 &\subseteq V(\M_A\cup \M_B)\cup V(\Gexp) \cup 
 \big(\largevertices{\frac
 {9\eta}{10}}{k}{\Gcapt}
 \setminus (\WantiC\cup \gP_1)\big)\;.
 \end{align*}
 
 So, in order to show~\eqref{fact:VplusCapG0}, it suffices to see that 
  for each $X\in \V^\circ\setminus \V(\M_A\cup \M_B)$ with $X\subseteq
 \largevertices{\frac
 {9\eta}{10}}{k}{\Gcapt}
 \setminus (\WantiC\cup \gP_1\cup V(\Gexp)\cup
 \largeintoatoms)$ we have $X\in \mathcal L^\circ$. So assume $X$ is as above. Let $v\in X$. We calculate
 
 \begin{align*}
 \deg_{\Gblack}(v, V(G^\circ))&\ge \deg_{\Gblack}(v, V(\M_A\cup \M_B))\\
 \JUSTIFY{$v\notin V(\Gexp)$}&\ge (1+\frac
 {9\eta}{10})k-\deg_G(v,\HugeVertices)-\deg_{\GD}(v,
 \smallatoms)\\ 
 & \ -\deg_{\Gblack}(v, \bigcup \clusters\setminus V(\M_A\cup \M_B))\\
 \JUSTIFY{$v\notin \WantiC\cup\largeintoatoms\cup \gP_1\cup V(\M_A\cup
 \M_B)$}&\ge (1+\frac {9\eta}{10})k -\frac {\eta k}{100}-\frac{\rho
 k}{100\Omega^*} -\gamma k\\
&\ge (1+\frac {\eta}{2})k\;.
 \end{align*}
 We deduce that $X\in\mathcal L^\circ$, which finishes the proof of~\eqref{fact:VplusCapG0}.
 
 Next, observe that by the definition of $\mathcal C$, we have 
 \begin{align}
  V_+\cap V(G^\circ) & \notag \supseteq \Vgood\cap V(G^\circ)  \\ \notag & \supseteq \Vgood\sm \big( \Vgood\sm V(G^\circ)\big)\\  & \supseteq \Vgood\sm (\WantiC\cup \gP_1\cup \bigcup
\gC^- \cup \smallatoms\cup V(\Gexp)).\label{vorabrechnung}
 \end{align}

We are now ready to prove Subclaim~\ref{lem:ABobt3}.
For each vertex $v\in A\setminus S$, we have
\begin{align*}
\deg_{\Gblack}\left(v,V(\M_A\cup \M_B)\cup \bigcup \mathcal L^\circ\right)
&\ge  \deg_{\Gblack}(v, V_+\cap V(G^\circ))\\
 &~~~-\deg_{\Gblack}\left(v,(V_+\cap V(G^\circ))\setminus (V(\M_A\cup \M_B)\cup
 \mathcal L^\circ)\right)\\
\JUSTIFY{by~\eqref{vorabrechnung}, ~\eqref{fact:VplusCapG0}} 
&\ge
\deg_{\Gblack}(v, \Vgood)
-\deg_{\Gblack}(v, \WantiC\cup \gP_1\cup \bigcup
\gC^-)\\
&~~~-\deg_{\Gblack}(v,\smallatoms)-2\deg_{\Gblack}(v,V(\Gexp))\\
&~~~-\deg_{\Gblack}\Big(v,(\largeintoatoms\cap \largevertices{\eta}{k}{G})\setminus V(\M_A\cup\M_B)\Big)\\
 \JUSTIFY{by ~\eqref{eq:3veci2N}, as $v\not\in  S\cup\gP$, by $\mathbf{(t5)}$} 
 &\ge  (1+\frac{\eta}{20})k-
\frac{\eta^2 k}{10^5}-\frac{\eta k}{200}-\frac{\rho k}{100\Omega^*}-2\rho k
-\frac{2\eta^2 k}{10^5}
\\
&>
  (1+\frac{\eta}{40})k+\frac {\eta k}{200}\;,
\end{align*}
where for the second to last inequality we used the abreviation  `by
$\mathbf{(t5)}$' to indicate that this case implies that  $v\notin
\shadow_{\Gcapt}(V(\Gexp), \rho k)\cup \shadow_{\Gcapt}((\largeintoatoms\cap
\largevertices{\eta}{k}{G})\setminus V(\M_A\cup M_B), \frac{2\eta^2k}{10^5})$.
As $|A\setminus S|\ge |A|/2$, we note that the set $A$ fulfils the requirements of the claim.

The same calculations hold for $B$. This finishes the proof of
Subclaim~\ref{lem:ABobt3}.
\end{proof}

The next auxiliary subclaim is needed in our proof of
Claim~\ref{lem:thedesiredAandBdoexist} in case {\bf(M2)}.
\begin{subclaim}\label{cl:CACB}
Suppose that case {\bf(M2)} occurs. Then there exists an edge $C_AC_B\in E(\BGblack)$ such that $\deg_{\Gblack}(v,V(\M_A\cup\M_B)\cup\bigcup\mathcal L^\circ)\ge
(1+\frac{\eta}{40})k+\frac{\eta k}{400}$ for all but at most
$2\epsilon'\clustersize$ vertices $v\in C_A$, and all but at most
$2\epsilon'\clustersize$ vertices $v\in C_B$. Moreover, there exist $A,B\in
\V(\M)$ such that $|C_A\cap A|>\sqrt{\epsilon'}\clustersize$ and $|C_B\cap
B|>\sqrt{\epsilon'}\clustersize$.
\end{subclaim}
\begin{proof}[Proof of Claim~\ref{cl:CACB}]
Let $(A,B)\in\M$ be given as in Subclaim~\ref{lem:ABobt3}.  Let $P_A\subset A$,
and $P_B\subset B$ be the vertices which fail the assertion of
Subclaim~\ref{lem:ABobt3}. Note that with this notation,
Subclaim~\ref{lem:ABobt3} states that
\begin{equation}\label{apalache}
|A\sm P_A|\geq |A|/2.
\end{equation}

Call a cluster $C\in \clusters$ \emph{$A$-negligible} if $|C\cap (A\setminus
P_A)|\le \frac {\gamma^3\clustersize}{16\Omega^*k}|A|$. Let $R_A$ be the union
of all $A$-negligible clusters. 

Recall that $(A,B)$ is entirely contained in one dense spot from
$(U,W;F)\in \DenseSpots_\class$ (cf.\ {\bf(M2)}). So by
Fact~\ref{fact:sizedensespot}, and since the spots in $\DenseSpots_\class$ are $(\frac{\gamma^3
k}4,\frac{\gamma^3 k}4)$-dense, we know that $\max\{|U|,|W|\}\leq
\frac{4\Omega^* k}{\gamma^3}$. In particular, there are at most $\frac{4\Omega^* k}{\gamma^3\clustersize}$ $A$-negligible clusters which intersect to $A\cap R_A$.

As these clusters are all disjoint, we find that
 $$|(A\cap R_A)\sm P_A|\le \frac{4\Omega^* k}{\gamma^3\clustersize}\cdot  |C\cap (A\setminus
P_A)|\le \frac{|A|}{4}.$$ 
This gives 
$$|A\setminus (P_A\cup R_A)|\geBy{\eqref{apalache}} \frac{|A|}2-|(A\cap R_A)\sm P_A|\ge \frac{|A|}4\;.$$ 

Similarly, we
can introduce the notion $B$-negligible clusters, and the set $R_B$, and get
$|(B\cap R_B)\sm P_B|\le \frac{|B|}{4}$ and $|B\setminus (P_B\cup R_B)|\ge |B|/4$. 

By the regularity of the pair $(A,B)$ there exists at least one edge $ab\in E\big(G^*[A\setminus (P_A\cup R_A),B\setminus (P_B\cup R_B)]\big)$, where $a\in A, b\in B$, and $G^*$ is the graph formed by edges of $\DenseSpots_\class$. As $V(\M)\subset V(\Gblack)$ by the assumption of case {\bf(t5)}, we have that $ab\in E(\Gblack)$. Let $C_A,C_B\in\clusters$ be the clusters containing $a$ and $b$, respectively. Note that $C_AC_B\in E(\BGblack)$.

Now as $a\notin R_A$, also $C_A$ is disjoint from $R_A$, and thus $$|C_A\cap (A\setminus
 P_A)|>\frac {\gamma^3\clustersize}{16\Omega^*k}\cdot \frac {\alphaD\rho
 k}{\Omega^*} > \sqrt{\epsilon'}\clustersize\;.$$
 This proves the ``moreover'' part of the claim for $C_A$.
 So there are at least $2\epsilon'\clustersize$ vertices $v$ in $C_A$ with $\deg_{\Gblack}(v,V(\M_A\cup\M_B)\cup\bigcup\mathcal L^\circ)\ge (1+\frac\eta{40})k+\frac{\eta k}{200}$ (by the definition of $P_A$). By Lemma~\ref{lem:degreeIntoManyPairs}, and using Facts~\ref{fact:sizedensespot} and~\ref{fact:boundedlymanyspots}, 
 we thus have that $\deg_{\Gblack}(v,V(\M_A\cup\M_B)\cup\bigcup\mathcal L^\circ)\ge (1+\frac\eta{40})k+\frac{\eta k}{400}$ for all but at most $2\epsilon'\clustersize$ vertices $v$ of $C_A$.
The same calculations hold for $C_B$. 
\end{proof}

In the remainder of the proof of Claim~\ref{lem:thedesiredAandBdoexist} we have
to distiguish between cases {\bf(M1)} and {\bf(M2)}.

Let us first consider the case {\bf(M2)}.
Let $C_A,C_B\in\clusters$ and $A,B\in\V(\M)$ be given by Subclaim~\ref{cl:CACB}.   
We have $|C_A\setminus (\WantiC
\cup
L_\sharp \cup \gP_1)|> \sqrt{\epsilon'}|C_A|$ by Subclaim~\ref{cl:CACB} and by
the definition of $\M$ and the definition of $\gP$. 
Thus, $C_A\cap V(G^\circ)$ is non-empty. Let $X_A\in\V^\circ$ be an arbitrary
set in $C_A$. Similarly, we obtain a set $X_B\in\V^\circ$, $X_B\subset C_B$. The
claimed properties of the pair $(X_A,X_B)$ follow directly from
Subclaim~\ref{cl:CACB}.

It remains to treat the case {\bf(M1)}.
Let $(A,B)$ be from Subclaim~\ref{lem:ABobt3}. Let $(X_A,X_B)\in\Mgood$ be such
that $X_A\supset A$ and $X_B\supset B$. Claim~\ref{lem:ABobt3} asserts that at least $$\frac{|A|}2\geBy{\bf{(M1)}}\frac{\eta^2\clustersize}{2\cdot 10^4} > 2\epsilon'\clustersize$$ vertices of $A$ have large degree (in $\Gblack$) into the set $V(\M_A\cup\M_B)\cup\bigcup\mathcal L^\circ$. Therefore, by Lemma~\ref{lem:degreeIntoManyPairs}, 
 $X_A$ and $X_B$ satisfy the assertion of the
Claim.

This proves Claim~\ref{lem:thedesiredAandBdoexist}, and thus
 finishes the proof of Lemma~\ref{whatwegetfrom(t5)}. 
 \end{proof}
\end{proof}

The proof of  Lemma~\ref{lem:ConfWhenMatching} follows by putting together Lemmas~\ref{whatwegetfrom(t1)}, \ref{whatwegetfrom(t2)}, \ref{whatwegetfrom(t3)}, \ref{ObtConf9}, and \ref{whatwegetfrom(t5)}.

\section{Embedding trees}\label{sec:embed}
In this section we provide an embedding of a tree $T_\PARAMETERPASSING{T}{thm:main}\in\treeclass{k}$ in
the setting of the configurations introduced in
Section~\ref{sec:configurations}.  In Section~\ref{ssec:embeddingOverview} we
first give a fairly detailed overview of the embedding techniques used. In
Section~\ref{ssec:Duplicate} we introduce a class of stochastic processes which
will be used for some embeddings. Section~\ref{ssec:EmbeddingShrubs} contains a
number of lemmas about embedding small trees, and use them  for embedding knags and shrubs of a given fine partition of $T_\PARAMETERPASSING{T}{thm:main}$. Embedding the entire tree $T_\PARAMETERPASSING{T}{thm:main}$ is then handled in the final Section~\ref{sec:MainEmbedding}. There we have to distinguish between particular configurations. The configurations are grouped into three categories (Section~\ref{sssec:EmbedDiamon0Diamond1}, Section~\ref{sssec:EmbedMoreComplex}, and Section~\ref{sssec:OrderedSkeleton}) corresponding to the similarities between the configurations.

\subsection{Overview of the embedding procedures}\label{ssec:embeddingOverview}
Recall that we are working under Setting~\ref{commonsetting}. 
Given a host graph $G_\PARAMETERPASSING{T}{thm:main}$ with one of the Configurations $\mathbf{(\diamond2)}$--$\mathbf{(\diamond10)}$, we have to embed in it a given tree $T=T_\PARAMETERPASSING{T}{thm:main}\in\treeclass{k}$, which comes with its  
$(\tau k)$-fine partition $(W_A,W_B,\shrubA,\shrubB)$. The $\tau k$-fine partition of $T$ will make it possible to combine embeddings of smaller parts of $T$ into one embedding of the whole tree. This means that we will first develop tools for embedding singular shrubs and
knags of the $(\tau k)$-fine partition into various basic building bricks of the
configurations: the avoiding set $\smallatoms$, the expander $\Gexp$,  regular pairs, and  vertices of huge degree
$\HugeVertices$. Second, we will combine these basic techniques to embed the
entire tree $T$. Here, the order in which different parts of $T$ are embedded
is important. Also, it will be crucial at some points to reserve places for parts of the tree
which will be embedded only later. 

In the following subsections, we draft our embedding techniques. We group them into five categories comprising of related configurations\footnote{Configuration $\mathbf{(\diamond1)}$ is trivial (see
Section~\ref{sssec:EmbedDiamon0Diamond1}) and needs no draft.}: Configurations
$\mathbf{(\diamond2)}$--$\mathbf{(\diamond5)}$, Configurations $\mathbf{(\diamond6)}$--$\mathbf{(\diamond7)}$,  Configuration~$\mathbf{(\diamond8)}$, Configuration~$\mathbf{(\diamond9)}$, and Configuration~$\mathbf{(\diamond10)}$, treated in
Sections~\ref{ssec:EmbedOverview25},~\ref{ssec:EmbedOverview67}, \ref{ssec:EmbedOverview8}, \ref{ssec:EmbedOverview9}, \ref{ssec:EmbedOverview10},
respectively.

\subsubsection{Embedding overview for Configurations
$\mathbf{(\diamond2)}$--$\mathbf{(\diamond5)}$}\label{ssec:EmbedOverview25}
In each of the Configurations
$\mathbf{(\diamond2)}$--$\mathbf{(\diamond5)}$ we have sets $\HugeVertices',\HugeVertices'',L'', L'$ and $V_1$. Further, we have some additional sets ($V_2$ and/or $\smallatoms'$) depending on the particular configuration.

A common embedding scheme for Configurations $\mathbf{(\diamond2)}$--$\mathbf{(\diamond5)}$ is illustrated in Figure~\ref{fig:DIAMOND25overview}. 
\begin{figure}[ht]
\centering 
\includegraphics{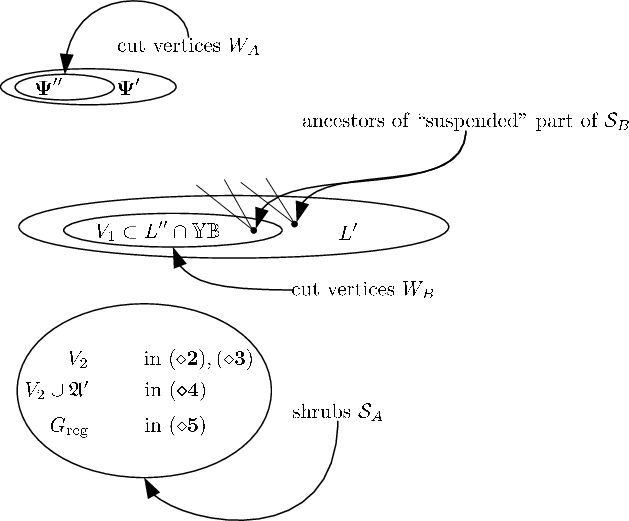}
\caption[Embedding overview for Configurations~$\mathbf{(\diamond2)}$--$\mathbf{(\diamond5)}$]{An overview of embedding 
of a tree $T\in\treeclass{k}$ given with its fine partition $(W_A,W_B,\shrubA,\shrubB)$ using
Configurations~$\mathbf{(\diamond2)}$--$\mathbf{(\diamond5)}$. The knags are
embedded between $\HugeVertices''$ and $V_1$, all the shrubs $\shrubA$ are embedded into sets specific to particular configurations so that the vertices neighbouring $W_A$ are embedded in $V_1$. Parts of the shrubs $\shrubB$ are embedded directly (using various embedding techniques), while the rest is ``suspended'', i.e., the ancestors of the unembedded remainders are embedded on vertices which have large degrees in $\HugeVertices'$. The embedding of $\shrubB$ is then finalized in the last stage.}
\label{fig:DIAMOND25overview}
\end{figure}
There are two stages of the embedding procedure: the knags, the shrubs $\shrubA$ and some parts of the shrubs $\shrubB$ are embedded in Stage~1, and then in Stage~2 the remainders of $\shrubB$ are embedded. Recall that $\shrubA$ contains both internal and end shrubs while $\shrubB$ contains exclusively end shrubs. We note that here the shrubs $\shrubB$ are further subdivided and some parts of them are embedded in the Stage~1 and some in Stage~2.

\begin{itemize}
\item In Stage~1, the knags of $T$ are embedded in $\HugeVertices''$ and $V_1$ so that $W_A$ is mapped to $\HugeVertices''$ and $W_B$ is mapped to $V_1$.
\item In Stage~1, the internal and end shrubs of $\shrubA$ are embedded using the sets $V_1,V_2$ and $\smallatoms'$ which are specific to the particular Configurations~$\mathbf{(\diamond2)}$--$\mathbf{(\diamond5)}$. The vertices of $\shrubA$ neighbouring $W_A$ are always embedded in $V_1$. Parts of the shrubs $\shrubB$ are embedded while the ancestors of the unembedded remainders are embedded on vertices which have large degrees in $\HugeVertices'$. 
\item In Stage~2, the embedding of $\shrubB$ is finalized. The remainders of $\shrubB$ are embedded starting with embedding their roots in $\HugeVertices'$.
\end{itemize}
A hierarchy of the embedding lemmas used to resolve Configurations
$\mathbf{(\diamond2)}$--$\mathbf{(\diamond5)}$ is given in Table~\ref{tab:Conf25}.
\begin{table}
\centering
\begin{tabular}{ccccc}
\hline
\multicolumn{5}{|c|}{Main embedding lemma: Lemma~\ref{lem:conf2-5}}\\
\hline
\multicolumn{1}{c}{$\Uparrow$}& &\multicolumn{1}{c}{$\Uparrow$}& &\multicolumn{1}{c}{$\Uparrow$}\\
\cline{1-1}\cline{3-3}\cline{5-5}
\multicolumn{1}{|c|}{Shrubs $\shrubA$}& &\multicolumn{1}{|c|}{Shrubs $\shrubB$ (Stage 1): Lemma~\ref{lem:blueShrubSuspend}}& &\multicolumn{1}{|c|}{Shrubs $\shrubB$ (Stage 2): Lemma~\ref{lem:embedC'endshrub}}\\
\cline{3-3}\cline{5-5}
\multicolumn{1}{|l|}{$\mathbf{(\diamond2)}$: Lemma~\ref{lem:embed:greyFOREST}}& & & & \\
\multicolumn{1}{|l|}{$\mathbf{(\diamond3)}$: Lemma~\ref{lem:HE3}}& & & & \\
\multicolumn{1}{|l|}{$\mathbf{(\diamond4)}$: Lemma~\ref{lem:HE4}}& & & & \\
\multicolumn{1}{|l|}{$\mathbf{(\diamond5)}$: regularity}& & & & \\
\cline{1-1}
\end{tabular}
\caption[Embedding lemmas for Configurations $\mathbf{(\diamond2)}$--$\mathbf{(\diamond5)}$]{Embedding lemmas employed for Configurations $\mathbf{(\diamond2)}$--$\mathbf{(\diamond5)}$.}
\label{tab:Conf25}
\end{table}

\subsubsection{Embedding overview for Configurations
$\mathbf{(\diamond6)}$--$\mathbf{(\diamond7)}$}\label{ssec:EmbedOverview67}
Suppose Setting~\ref{commonsetting} and~\ref{settingsplitting} (see Remark~\ref{rem:h1h2} below for a comment on the constants $\proporce{0},\proporce{1},\proporce{2}$). Recall that we have in each of these configurations sets $V_0\cup V_1\subset \colouringp{0}$, sets $V_2\cup V_3\subset \colouringp{1}$ and $\Vgood\colouringpI{2}$.

A common embedding scheme for Configurations $\mathbf{(\diamond6)}$--$\mathbf{(\diamond7)}$ is illustrated in Figure~\ref{fig:DIAMOND67overview}. 
\begin{figure}[ht]
\centering 
\includegraphics{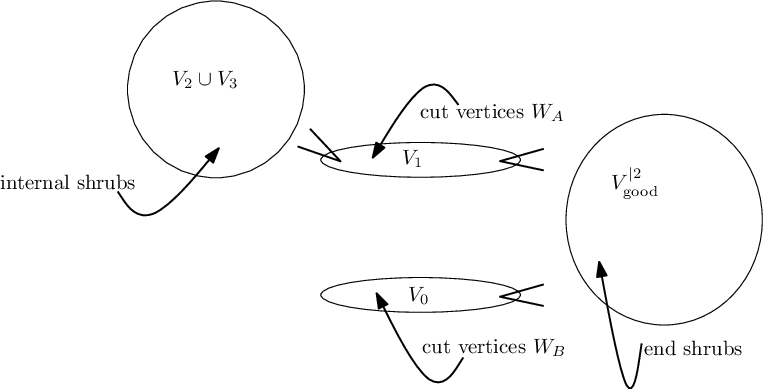}
\caption[Embedding overview for Configurations~$\mathbf{(\diamond6)}$--$\mathbf{(\diamond7)}$]{An overview of embedding a fine partition $(W_A,W_B,\shrubA,\shrubB)$
of a tree $T\in\treeclass{k}$ using
Configurations~$\mathbf{(\diamond6)}$--$\mathbf{(\diamond7)}$. The knags are
embedded between $V_0$ and $V_1$, the internal shrubs are embedded in
$V_2\cup V_3$, and the end shrubs are embedded
using $\Vgood\colouringpI{2}$.}
\label{fig:DIAMOND67overview}
\end{figure}
The embedding has three parts.
\begin{itemize}
\item The knags of $T$ are embedded between $V_0$ and $V_1$ so that $W_A$ is
mapped to $V_1$ and $W_B$ is mapped to $V_0$ using either the
Preconfiguration~$\mathbf{(exp)}$ or $\mathbf{(reg)}$. Thus $W_A\cup W_B$ ar mapped to $\subset \colouringp{0}$.
\item The internal shrubs  of $T$ are embedded in $V_2\cup V_3$, always putting 
neighbours of $W_A$ into $V_2$. Note that the internal shrubs are therefore
embedded in $\colouringp{1}$, and thus there is no interference with embedding
the knags. We need to understand why a mere degree of $\delta k$ (from $V_1$ to $V_2$, ensured by~\eqref{COND:D6:1}
and~\eqref{COND:D7:1}, with $\delta\ll 1$) is sufficient for embedding internal
shrubs of potentially big total order, that is, how to ensure that already
embedded internal trees do not cause a blockage later. Here the
expansion\footnote{This expansion is given by the presence of $\Gexp$ in
Configurations~$\mathbf{(\diamond6)}$
(cf.~\eqref{COND:D6:3}--\eqref{COND:D6:4}), and by the presence of the avoiding
set $\smallatoms$ in Configurations~$\mathbf{(\diamond7)}$ ($V_2\subset
\smallatoms\colouringpI{1}\setminus\exceptVertSplit$).} ruling between the $V_2$
and $V_3$ comes into play. This property (together with other properties of
Preconfigurations~$\mathbf{(exp)}$ and $\mathbf{(reg)}$) will allow that, once
finished embedding an internal tree, the follow-up knag can be embedded in a
place (in $V_1$) which sees very little of the previously embedded internal shrubs.
 
This is the only part of the embedding process which makes use of the specifics of Configurations~$\mathbf{(\diamond6)}$ and $\mathbf{(\diamond7)}$. For this reason we will be able to follow the same embedding scheme as presented here also for Configuration~$\mathbf{(\diamond8)}$, the only difference being the embedding of the internal shrubs (see Section~\ref{ssec:EmbedOverview8}).
\item The end shrubs are embedded in the yet unoccupied part of $G$.
For this we use the properties of Preconfigurations~$\mathbf{(\heartsuit1)}$ or $\mathbf{(\heartsuit2)}$.
The end shrubs are embedded using (but not entirely into) the designated vertex
set $\Vgood\colouringpI{2}$. 
\end{itemize}
The above embedding scheme is divided in two main steps: first the knags and the
internal trees are embedded (see Lemma~\ref{lem:embed:skeleton67}), and this
partial embedding is then extended to end shrubs (see Lemmas~\ref{lem:embed:heart1} and~\ref{lem:embed:heart2}). A more detailed hierarchy of the embedding lemmas which are used is given in Table~\ref{tab:Conf69}.
\begin{table}
\centering
\begin{tabular}{ccccc}
\hline
\multicolumn{5}{|c|}{Main embedding lemma: Lemma~\ref{lem:embed:total68}}\\
\hline
\multicolumn{3}{c}{$\Uparrow$}& &\multicolumn{1}{c}{$\Uparrow$}\\
\cline{1-3}\cline{5-5}
\multicolumn{3}{|c|}{Internal part}&\multicolumn{1}{|c|}{ }&\multicolumn{1}{|c|}{End shrubs}\\
\multicolumn{3}{|c|}{$\mathbf{(\diamond6)}$, $\mathbf{(\diamond7)}$: Lemma~\ref{lem:embed:skeleton67}}
& &\multicolumn{1}{|l|}{$\mathbf{(\heartsuit1)}$: Lemma~\ref{lem:embed:heart1}}\\
\multicolumn{3}{|c|}{$\mathbf{(\diamond8)}$: Lemma~\ref{lem:embed:skeleton8}}
& &\multicolumn{1}{|l|}{$\mathbf{(\heartsuit2)}$: Lemma~\ref{lem:embed:heart2}}\\
\cline{1-3}
\cline{5-5}
\multicolumn{1}{c}{$\Uparrow$}&\multicolumn{1}{c}{ }&\multicolumn{1}{c}{$\Uparrow$}& & \\
\cline{1-1}\cline{3-3}
\multicolumn{1}{|c|}{Knags}&\multicolumn{1}{c}{ }&\multicolumn{1}{|c|}{Internal shrubs}& & \\
\multicolumn{1}{|l|}{$\mathbf{(exp)}$: Lemma~\ref{lem:embed:greyFOREST}}&\multicolumn{1}{c}{ }&\multicolumn{1}{|c|}{$\mathbf{(\diamond6)}$: Lemma~\ref{lem:embedStoch:DIAMOND6}}& &\\
\multicolumn{1}{|l|}{$\mathbf{(reg)}$: Lemma~\ref{lem:embed:superregular}}&\multicolumn{1}{c}{ }&\multicolumn{1}{|c|}{$\mathbf{(\diamond7)}$: Lemma~\ref{lem:embedStoch:DIAMOND7}}& &\\
\cline{1-1}
& & \multicolumn{1}{|c|}{$\mathbf{(\diamond8)}$: Lemmas~\ref{lem:embedStoch:DIAMOND7},~\ref{lem:embed:BALANCED},~\ref{lem:embed:regular}}& & \\
\cline{3-3}
\end{tabular}
\caption[Embedding lemmas for Configurations $\mathbf{(\diamond6)}$--$\mathbf{(\diamond8)}$]{Embedding lemmas employed for Configurations $\mathbf{(\diamond6)}$--$\mathbf{(\diamond8)}$ when embedding a tree $T\in\treeclass{k}$ with a given fine partition.}
\label{tab:Conf69}
\end{table}
\begin{remark}\label{rem:h1h2}
In our application of Lemma~\ref{lem:ConfWhenNOTCXAXB} the number $\proporce{1}$ will be approximately the proportion of the total order of the internal shrubs of a given fine partition $(W_A,W_B,\shrubA,\shrubB)$ of $T$ while $\proporce{2}$ will be approximately the proportion of the total order of the end shrubs. The number $\proporce{0}$ is just a small constant. 

These numbers -- scaled up by $k$ -- determine the parameter $h_1\approx \proporce{1}k$ (in Configurations~$\mathbf{(\diamond8)}$ and $\mathbf{(\diamond9)}$) and $h_2\approx\proporce{2}k$ (in Configurations~$\mathbf{(\diamond6)}$--$\mathbf{(\diamond9)}$). The properties of these configurations will then allow to embed all the internal shrubs and end shrubs. Note that the parameter $h_1$ does not appear in Configurations~$\mathbf{(\diamond6)}$ and $\mathbf{(\diamond7)}$. This suggests that the total order of the internal shrubs is not at all important in 
Configurations~$\mathbf{(\diamond6)}$--$\mathbf{(\diamond7)}$. Indeed, we would succeed even embedding a tree with internal shrubs of total order say $100k$.\footnote{Configuration~$\mathbf{(\diamond8)}$ has this property only in part. We would succeed even embedding a tree with principal subshrubs of total order say $100k$ provided that the  total order of peripheral subshrubs is somewhat smaller than  $h_1$.}

In view of this it might be tempting to think that the end shrubs in $\shrubA$ could also be embedded using the same technique as the internal shrubs into the sets $V_2\cup V_3$ provided by these configurations (cf.\ Figure~\ref{fig:DIAMOND67overview}). This is however not the case. Indeed, the minimum degree conditions~\eqref{COND:D6:1},~\eqref{COND:D7:1}, and~\eqref{COND:D8:1} allow embedding only a small number of shrubs from a single cut-vertex $x\in W_A$ while there may be many end shrubs attached to $x$; cf.\ Remark~\ref{rem:internalVSend}\eqref{it:fewinternaltrees}.
\end{remark}

\subsubsection{Embedding overview for Configuration $\mathbf{(\diamond8)}$}\label{ssec:EmbedOverview8}
Suppose Setting~\ref{commonsetting} and~\ref{settingsplitting}. We are working with sets $V_0$, $V_1$, $\Vgood\colouringpI{2}$, $V_2, V_3$ and $V_4$ and with semiregular matching $\mathcal N$ coming from the configuration.

The embedding scheme follows Table~\ref{tab:Conf69}, and is illustrated in Figure~\ref{fig:DIAMOND8}. 
\begin{figure}[ht]
\centering 
\includegraphics{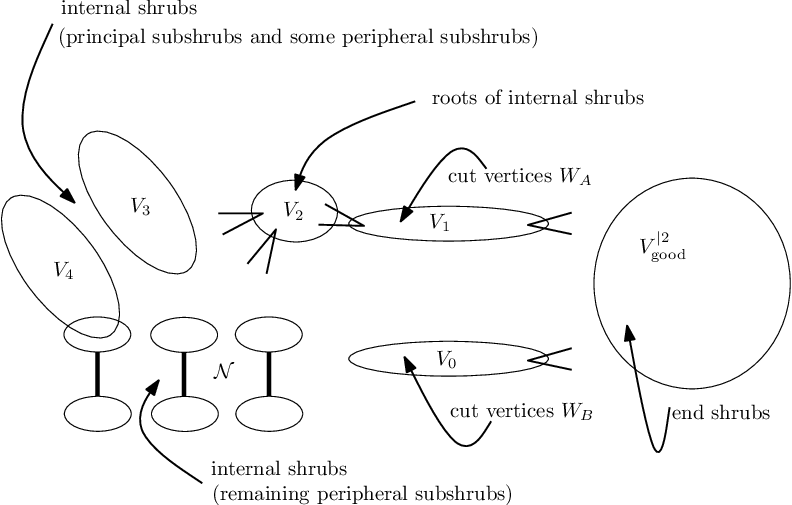}
\caption[Embedding overview for Configuration~$\mathbf{(\diamond8)}$]{An overview of embedding a fine partition $(W_A,W_B,\shrubA,\shrubB)$
of a tree $T\in\treeclass{k}$ using
Configuration~$\mathbf{(\diamond8)}$. The knags are
embedded between $V_0$ and $V_1$. The roots of the internal shrubs are embedded in $V_2$. Some of the subshrubs of the internal shrubs are embedded in $V_3\cup V_4$ and some in $\mathcal N$; principal subshrubs are always embedded in $V_3\cup V_4$. The end shrubs are embedded in using $\Vgood\colouringpI{2}$.}
\label{fig:DIAMOND8}
\end{figure}
Embedding of the knags and of the external shrubs is done in the same way as in Configurations~$\mathbf{(\diamond6)}$--$\mathbf{(\diamond7)}$. We only describe here the way the internal shrubs are embedded. Their roots are embedded in $V_2$. From that point we proceed embedding subshrub by subshrub. Some of the subshrubs get embedded between $V_3$ and $V_4$. This pair of sets has the same expansion property as the pair $V_2,V_3$ in Configuration~$\mathbf{(\diamond7)}$. In particular, it allows to avoid the shadow of the already occupied set so that the follow-up knag can be embedded in location almost isolated from the previous images, similarly as described in Section~\ref{ssec:EmbedOverview67}. For this reason we make sure that principal subshrubs get embedded here. The degree condition from $V_2$ to $V_3$ is too weak to ensure that all remaining subshrubs are embedded between $V_3$ and $V_4$. Therefore we might have to embed some subshrubs in $\mathcal N$.  Condition~\eqref{COND:D8:7} --- where $h_1$ is 
approximately the order of the internal shrubs, as in Remark~\ref{rem:h1h2} --- indicates that it should be possible to accommodate all 
the subshrubs.
For technical reasons, the order in which different types of subshrubs are embedded is very important.

\subsubsection{Embedding overview for Configuration $\mathbf{(\diamond9)}$}\label{ssec:EmbedOverview9}
The embedding process in Configuration~$\mathbf{(\diamond9)}$
follows the same scheme as in Configurations
$\mathbf{(\diamond6)}$--$\mathbf{(\diamond8)}$, but the embedding of the
internal shrubs follows the regularity method. Assuming the simplest situation $\mathcal F=\V_2(\mathcal N)$ and $V_2=V_1(\mathcal N)$, we
would have $\mindeg_{\Gblack}(V_1,V_1(\mathcal N))\ge
h_1$ (cf.~\eqref{conf:D9-XtoV}). See Figure~\ref{fig:DIAMOND9} for an illustration.
\begin{figure}[ht]
\centering 
\includegraphics{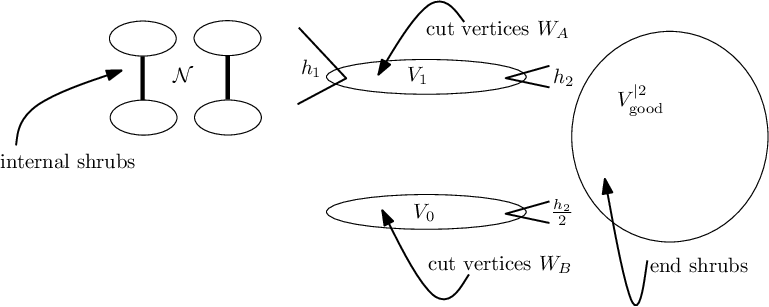}
\caption[Embedding overview for Configuration~$\mathbf{(\diamond9)}$]{An overview of embedding a fine partition $(W_A,W_B,\shrubA,\shrubB)$
of a tree $T\in\treeclass{k}$ using
Configuration~$\mathbf{(\diamond9)}$. The knags are
embedded between $V_0$ and $V_1$, the internal shrubs using the regularity
method in $\mathcal N$ and the end shrubs are embedded
using $\Vgood\colouringpI{2}$.}
\label{fig:DIAMOND9}
\end{figure}
Similarly as above, the knags are embedded between $V_0$ and $V_1$. The internal
shrubs are accommodated using the regularity method in $\mathcal N$, and the end
shrubs are embedded in $\Vgood\colouringpI{2}$ using Preconfiguration~$\mathbf{(\heartsuit1)}$. The embedding lemma for this configuration is given in Lemma~\ref{lem:embed9}.

\subsubsection{Embedding overview for Configuration $\mathbf{(\diamond10)}$}\label{ssec:EmbedOverview10}
Configuration~$\mathbf{(\diamond10)}$ is very closely related to the structure obtained by Piguet and Stein~\cite{PS07+} in their solution of the dense approximate case of Conjecture~\ref{conj:LKS}, Theorem~\ref{thm:PiguetStein}. Let us describe their proof first. Piguet and Stein prove that when $k>qn$ (for some fixed $q>0$ and $k$ sufficiently large) the cluster graph\footnote{ordinary, in the sense of the classic Regularity Lemma} $\BGblack$ of a graph $G\in \LKSgraphs{n}{k}{\eta}$ contains the following structure (cf.~\cite[Lemma~8]{PS07+}). There is a set of clusters $\BL\subset \clusters$ such that each cluster in $\BL$ contains only vertices of captured degrees at least $(1+\frac\eta2)k$. There is a matching $M\subset \BGblack$, and an edge $AB$, with $A,B\in\BL$. One of the following conditions is satisfied
\begin{enumerate}
 \item[\textbf{(H1)}] $M$ covers $\neighbor_{\BGblack}(\{A,B\})$, or
 \item[\textbf{(H2)}] $M$ covers $\neighbor_{\BGblack}(A)$, and the vertices in $B$ have captured degrees at least $(1+\frac\eta2)\frac k2$ into $\bigcup(\BL\cup V(M))$. Further, each edge in $M$ has at most one endvertex in $\neighbor_{\BGblack}(A)$.
\end{enumerate}
Piguet and Stein use structures~\textbf{(H1)} and~\textbf{(H2)} to embed any
given tree $T\in\treeclass{k}$ into $G$ using the regularity method; see
Sections~3.6 and~3.7 in~\cite{PS07+}, respectively. Actually, a slight relaxation of~\textbf{(H1)} and~\textbf{(H2)} would be sufficient for the embedding to work, as can be easily seen from their proof: Again, there is a set of clusters $\BL\subset \clusters$ such that each cluster in $\BL$ contains only vertices of captured degrees at least $(1+\frac\eta2)k$, there is a matching $M\subset \BGblack$, and an edge $AB$, $A,B\in\BL$. One of the following conditions is satisfied
\begin{enumerate}
 \item[\textbf{(H1')}] the vertices in $A\cup B$ have captured degrees at least $(1+\frac\eta2)k$ into the vertices of $\bigcup (\BL\cup V(M))$, or
 \item[\textbf{(H2')}] the vertices in $A$ have captured degrees at least $(1+\frac\eta2)k$ into the vertices of $\bigcup V(M)$, and the vertices in $B$ have captured degrees at least $(1+\frac\eta2)\frac k2$ into $\bigcup(\BL\cup V(M))$. Further, each edge in $M$ has at most one endvertex in $\neighbor_{\BGblack}(A)$.
\end{enumerate}
It can be seen that Configuration~$\mathbf{(\diamond10)}$ is a direct
counterpart to~\textbf{(H1')}.\footnote{Observe that some parts of $\BGblack$
are irrelevant in the embedding process of~\cite{PS07+}. The objects $\BGblack$,
$\BL$, and $M$ in the structural result  of~\cite{PS07+} correspond to $(\tilde G,\V)$, $\mathcal L^*$, and $\M$ in Configuration~$\mathbf{(\diamond10)}$.} (The counterpart of~\textbf{(H2')} is contained in Configuration~$\mathbf{(\diamond9)}$ and the similarity is somewhat weaker.)

\medskip
The embedding lemma for Configuration~$\mathbf{(\diamond10)}$ is stated in Lemma~\ref{lem:embed10}.

\subsection{Stochastic process $\Duplicate(\ell)$}\label{ssec:Duplicate}
Let us introduce a class of stochastic processes, which we call
\index{mathsymbols}{*Duplicate@$\Duplicate(\ell)$}$\Duplicate(\ell)$ ($\ell\in\mathbb N$). These are discrete processes $(X_1,Y_1),(X_2,Y_2),\ldots,(X_q,Y_q)\in\{0,1\}^2$ (where $q\in \mathbb N$ is arbitrary) satisfying the following.
\begin{itemize}
  \item For each $i\in[q]$, we have either
  \begin{enumerate}[(a)]
    \item $X_i=Y_i=0$ (deterministically), or
    \item $X_i=Y_i=1$ (deterministically), or
    \item\label{duplC} exactly one of $X_i$ and $Y_i$ is one, and in that case $\probability[X_i=1]=\frac12$.
  \end{enumerate}
  \item If the distribution of $(X_i,Y_i)$ is according to~\eqref{duplC}, then the random choice is made independently of the values $(X_j,Y_j)$ ($j<i$).
  \item We have $\sum_{i=1}^q(X_i+Y_i)\le \ell$.
\end{itemize}

Needless to say that this definition is not deep and its purpose is only to adopt the language we shall be using later. The following lemma asserts that the first and second component of a process $\Duplicate(\ell)$ are typically balanced.

\begin{lemma}\label{lem:randomduplicate}
Suppose that $(X_1,Y_1),(X_2,Y_2),\ldots,(X_q,Y_q)$ is a process in
$\Duplicate(\ell)$. Then for any $a>0$ we have
$$\probability\left[\sum_{i=1}^q
X_q-\sum_{i=1}^q Y_q\ge a\right]\le \exp\left(-\frac{a^2}{2\ell}\right)\;.$$
\end{lemma}
\begin{proof}
We shall be using the following version of the Chernoff bound for sums of
independent random variables $Z_i$, with distribution
$\probability[Z_i=1]=\probability[Z_i=-1]=\frac12$.
\begin{equation}\label{eq:CHERNOFF}
\probability\left[\sum_{i=1}^n Z_i\ge a\right]\le
\exp\left(-\frac{a^2}{2n}\right)\;.
\end{equation}

Let $J\subset [q]$ be the set of all indices $i$ with $X_i+Y_i=1$. By the definition of $\Duplicate(\ell)$, we have $|J|\le \ell$. 
By~\eqref{eq:CHERNOFF} we have
\begin{align*}
\probability\left[\sum_J
(X_i-Y_i)\ge a\right]\le \exp\left(-\frac{a^2}{2|J|}\right)\le
\exp\left(-\frac{a^2}{2\ell}\right) \;.
\end{align*}
\end{proof}

We shall use the stochastic process $\Duplicate$ to guarantee that certain fixed
vertex sets do not get overfilled during our tree embedding procedure.
$\Duplicate$ is used in Lemmas~\ref{lem:embedStoch:DIAMOND6}
and~\ref{lem:embedStoch:DIAMOND7} through Lemma~\ref{lem:randomshrubembedding}.

\subsection{Embedding small trees}\label{ssec:EmbeddingShrubs}
When embedding the tree $T_\PARAMETERPASSING{T}{thm:main}$ in our proof of Theorem~\ref{thm:main} it will be important to control where different bits of $T_\PARAMETERPASSING{T}{thm:main}$ go. This motivates the following notation. Let $X_1,\ldots,X_\ell\subset V(T)$ be arbitrary vertex sets of a tree $T$, and let $V_1,\ldots,V_\ell\subset V(G)$ be arbitrary vertex sets of a graph $G$. Then an embedding $\phi:V(T)\rightarrow V(G)$ of $T$ in $G$ is an \emph{$(X_1\hookrightarrow V_1,\ldots,X_\ell\hookrightarrow V_\ell)$-embedding} \index{mathsymbols}{**embedding@$(X_1\hookrightarrow V_1,\ldots,X_\ell\hookrightarrow V_\ell)$-embedding}\index{general}{**embedding@$(X_1\hookrightarrow V_1,\ldots,X_\ell\hookrightarrow V_\ell)$-embedding} if $\phi(X_i)\subset V_i$ for each $i\in[\ell]$.

We provide several sufficient conditions for embedding a small tree with
additional constraints. 

\HIDDENTEXT{Lemma about spots hidden under SPOTS}

The first lemma deals with embedding using an avoiding
set.
\begin{lemma}\label{lem:embed:avoidingFOREST}\HAPPY{M}\HAPPY{D} Let $\Lambda,k
\in \mathbb N$ and let $\epsilon,\gamma\in (0,\frac12)$ with $\gamma^2 >\eps$.
Suppose $\smallatoms$ is a $(\Lambda,\epsilon,\gamma,k)$-avoiding set with respect to  a set  $\DenseSpots$ of $(\gamma k,\gamma)$-dense spots in a graph $H$. Suppose that $(T_1,r_1),\ldots,(T_\ell,r_\ell)$ are rooted trees with $|\bigcup_i T_i|\leq \gamma k/2$. Let $U\subset V(H)$ with $|U|\le \Lambda k$, and let $U^*\subseteq  \smallatoms$ with $ |U^*|\ge\eps k+\ell$.
 Then there are mutually disjoint $(r_i\hookrightarrow U^*, V(T_i)\setminus\{r_i\}\hookrightarrow V(H)\setminus U)$-embeddings of the trees $(T_i,r_i)$ in $H$.
\end{lemma}

\begin{proof}
Since $\smallatoms$ is $(\Lambda,\epsilon,\gamma,k)$-avoiding, there
exists a set $Y\subset \smallatoms$ with $|Y|\le\epsilon k$, such that each vertex $v$ in
$\smallatoms\setminus Y$ has degree at least $\gamma k$ into some $(\gamma
k,\gamma)$-dense spot $D\in \DenseSpots$ with $|U\cap V(D)|\le\gamma^2k$. In particular, $U^* \setminus Y$ is large enough so that we can embed all vertices $r_i$ there. We extend this embedding successively to an embedding of $\bigcup_i T_i$, in each step finding a suitable image in $V(D)\setminus U$ for one neighbour of an already embedded vertex $v\in \bigcup_i V(T_i)$. This is possible since the image of $v$ has degree at least $\gamma k - |U\cap V(D)|> \gamma k/2\geq \sum_i v(T_i) $ into $V(D)\setminus U$.
\end{proof}

The next lemma deals with embedding a tree into a nowhere-dense graph, a
primal example of which is the graph $\Gexp$.

\begin{lemma}\label{lem:embed:greyFOREST}
\HAPPY{M}\HAPPY{D}
Let $k\in\mathbb N$, let $Q\ge 1$ and let $\gamma,\zeta\in (0,1)$ be such that
 $128Q\gamma\le\zeta^2$. Let $H$ be a
$(\gamma k,\gamma)$-nowhere-dense graph. Let $(T_1,r_1),\ldots,(T_\ell,r_\ell)$ be rooted trees of total order less than $\zeta k/4$. Let $V_1,V_2,U,U^*\subset V(H)$  be
four sets with $U^*\subset V_1$, $|U|<Q k$,
$|U^*|>\frac{32Q^2\gamma}{\zeta} k +\ell$, and
$\mindeg_{H}(V_j,V_{3-j})\ge\zeta k$ for $j=1,2$.   Then there are mutually
disjoint $(r_i\hookrightarrow U^*, \Veven(T_i) \hookrightarrow V_1\setminus U,\Vodd(T_i)\hookrightarrow V_2\setminus U)$-embeddings of the trees
$(T_i,r_i)$ in $H$.
\end{lemma}
\begin{proof}
Set $B:=\shadow_H(U,\zeta k/2)$. By Fact~\ref{fact:shadowboundEXPANDER}, we have
$|B|\le\frac{32Q^2\gamma}{\zeta} k\le\frac{\zeta}{4}k$. In particular, $U^*\setminus B$ is large enough to accommodate the images $\phi (r_i)$ of all vertices $r_i$.

Successively, extend $\phi$, in each step mapping a neighbour
$u$ of some already embedded vertex $v\in \bigcup_i V(T_i)$ to a yet
unused neighbour of $\phi(v)$ in $V_j\setminus (B\cup
U)$, where $j$ is either 1 or 2, depending on the parity of $\dist_T(r,v)$.
 This is possible as
 $\phi(v)$, lying outside $B$,  has at least $\zeta k/2$ neighbours in $V_i\setminus U$. Thus $\phi (v)$ has  at least $\zeta k/4$ neighbours in $V_i\setminus (U\cup B)$, which is more than $\sum_iv(T_i)$. 
\end{proof}

The next three standard lemmas deal with embedding trees in a regular or a super-regular pair. We omit their proofs.
\begin{lemma}\label{lem:embed:regular}
 Let $\epsilon>0$
and $\beta>2\epsilon$. Let $(C,D)$ be an $\epsilon$-regular pair in a graph $H$, with $|C|=|D|=:\ell$, and with density $\density(C,D)\geq
3\beta$. Suppose that there are sets $X\subseteq C$, $Y\subseteq D$, and
$X^*\subset X$ satisfying $\min\{|X|,|Y|\}\ge 4\frac{\epsilon}{\beta}\ell$ and
$|X^*|> \frac\beta2\ell$.  Let $(T,r)$ be a rooted tree of order $v(T)\le
\epsilon \ell$.  Then there exists an $(r\hookrightarrow
X^*,\Veven(T)\hookrightarrow X,\Vodd(T)\hookrightarrow Y)$-embedding of $T$ in $H$.
\end{lemma}
\HIDDENTEXT{proof hidden under FILLINGREGULAR}

\begin{lemma}\label{lem:fillingCD}
 Let
$\beta,\epsilon>0$ and $\ell \in\mathbb N$ be such that $\beta>2\epsilon$. Let
$(C,D)$ be an $\epsilon$-regular pair with $|C|=|D|=\ell$ of density
$\density(C,D)\geq 3\beta$ in a graph $H$. 
Let $(T_1,r_1),(T_2,r_2),\ldots ,(T_s,r_s)$ be rooted trees with
$v(T_i)\leq\epsilon\ell$ for all $i\in [s]$. Let $U\subset V(H)$ fulfill $|C\cap U|=|D\cap U|$, and let $X^*\subseteq (C\cup D)\setminus U$ be such that
\begin{equation}\label{eq:conFill}
|X^*|\geq \sum_{i=1}^sv(T_i)
+ 50\beta\ell\;.
\end{equation}
 Then there are mutually
disjoint $(r_i\hookrightarrow X^*, V(T_i) \hookrightarrow (C\cup D)\setminus U)$-embeddings of the trees
$(T_i,r_i)$ in~$H$.
\end{lemma}
\HIDDENTEXT{the proof is hidden under FILLINGCDPROOF}

\begin{lemma}\label{lem:embed:superregular}
Let $d>10\epsilon>0$. Suppose that $(A,B)$ forms an $(\epsilon,d)$-super-regular pair with $|A|,|B|\ge \ell$. Let $U_A\subset A$, $U_B\subset B$ be such that $|U_A|\le |A|/2$ and $|U_B|\le d|B|/4$. Let $(T,r)$ be a rooted tree of order at most $d\ell/4$, and let $v\in A\setminus U_A$ be arbitrary. Then there exists an $(r\hookrightarrow v,\Veven(T,r)\hookrightarrow A\setminus U_A, \Vodd(T,r)\hookrightarrow B\setminus U_B)$-embedding of $T$.
\end{lemma}

Suppose that we we have a rooted tree $(T,r)$ to be embedded, and its root was
already on a vertex $\phi(r)$. Suppose that $r$ has degree $\ell_X+\ell_Y$ in a
regular pair $(X,Y)$, where $\ell_X:=\deg(\phi(r), X), \ell_Y:=\deg(\phi(r),
Y)$, with $\ell_X\ge \ell_Y$, say.
The hope is that we can embed $T$ in $(X,Y)$ as long as $v(T)$ is a bit smaller
than $\ell_X+\ell_Y$. For this, the greedy strategy does not work (see
Figure~\ref{fig:embedBalancedUnbalance}) and we need to be somewhat more careful.
\begin{figure}[ht]
\centering 
\includegraphics{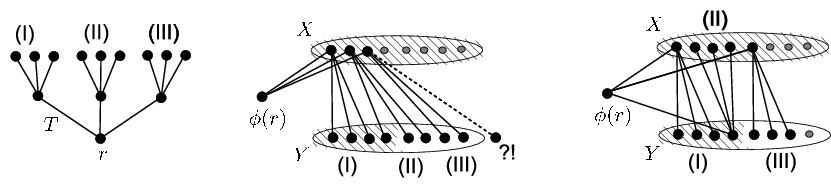}
\caption[Balanced and unbalanced embedding in a regular pair]{An example of a
rooted tree $(T,r)$, depicted on the left. The forest $T-r$ has three components (I), (II), (III) of total
order 12. Say the vertex $r$ is embedded so that for the regular pair $(X,Y)$ we
have $\deg(\phi(r),X)=8$, $\deg(\phi(r),Y)=4$ (neighbourhoods of $\phi(r)$
hatched).

While the greedy strategy does not work (middle), splitting the process into
a balanced and an unbalanced stage (right) does --- here the components~(I) and~(II)
are embedded in the balanced stage and the component~(III) in the unbalanced
stage.}
\label{fig:embedBalancedUnbalance}
\end{figure}
 We split the embedding process into two stages.
In the first stage we choose a subset of the components of $T-r$ of total order 
approximately $2\min\big(\ell_X,\ell_Y\big)=2\ell_Y$. When embedding these, we choose 
orientations of each component in such a way that the image is approximately
balanced with respect to $X$ and $Y$. In the second stage we 
embed the remaining components so that their roots are embedded in $X$.
We refer to the first stage as embedding in an
\index{general}{balanced way of embedding}\index{general}{unbalanced way of embedding}\emph{balanced way}, and as embedding in an \emph{unbalanced way}
to the second stage.

The next lemma says that each regular pair can be filled-up in a balanced way by
trees.
\begin{lemma}\label{lem:embed:BALANCED}
Let $G$ be a graph, $v\in V(G)$ be a vertex, $\M$ be an
$(\epsilon,d,\nu k)$-semiregular matching in $G$, and $\{f_{CD}\}_{(C,D)\in\M}$ a
family of integers between $-\tau k$ and $\tau k$. Suppose $(T,r)$ is a rooted tree,
$$v(T)\le \left(1-\frac{4(\epsilon+\frac{\tau}\nu)}{d-2\epsilon}\right)|V(\M)|\;,$$ with the property that each component of $T-r$ has order at most $\tau k$. If $V(\M)\subset \neighbor_G(v)$ then there
exists an $(r\hookrightarrow v,V(T-r)\hookrightarrow V(\M))$-embedding $\phi$ of $T$ such that for each $(C,D)\in\M$ we have $|C\cap
\phi(T)|+f_{CD}=|D\cap \phi(T)|\pm\tau k$.
\end{lemma}

The proof of Lemma~\ref{lem:embed:BALANCED} is standard, and is given for
example in~\cite[Lemma~6.6]{HlaPig:LKSdenseExact}. 

Lemma~\ref{lem:embed:BALANCED} suggests the following definitions. 
The \emph{discrepancy}\index{general}{discrepancy} of a set $X$ with respect to a pair of sets $(C,D)$ is the number $|C\cap X|-|D\cap X|$.
$X$ is \emph{$s$-balanced}\index{general}{balanced set} with respect to a semiregular matching $\M$ if the discrepancy of $X$ with respect to each $(C,D)\in \M$ is at most $s$ in absolute value.
\begin{lemma}\label{lem:embed:ONESIDE}\HAPPY{D,M}
Let $G$ be a graph, $v\in V(G)$ be a vertex, $\M$ be an
$(\epsilon,d,\nu k)$-semiregular matching in $G$ with an $\M$-cover $\mathcal F$, and $U\subset
V(G)$. Suppose $(T,r)$ is a rooted tree with 
$$
v(T)+|U|\le
\deg_G\left(v,V(\M)\setminus \bigcup\mathcal F\right)-\frac{4(\epsilon+\frac{\tau}\nu)}{d-2\epsilon}|V(\M)|\;,$$
such that each component
of $T-r$ has order at most $\tau k$. Then there
exists an $(r\hookrightarrow v,V(T-r)\hookrightarrow V(\M)\setminus U)$-embedding $\phi$ of $T$.
\end{lemma}
The proof of Lemma~\ref{lem:embed:ONESIDE} is again standard and we again omit it.

The following lemma uses a probabilistic technique to embed a shrub while reserving a  set of vertices in the host graph for later use. We wish the reserved set to use about as much space inside certain given sets $P_i$ as the image of our shrub does. (In later applications the sets $P_i$ correspond to neighbourhoods of vertices which are still `active'.) 

Lemma~\ref{lem:randomshrubembedding} will find an immediate application in all the remaining lemmas of this subsection. However it is really necessary only for Lemmas~\ref{lem:embedStoch:DIAMOND6}--\ref{lem:embedStoch:DIAMOND7}, which deal with
embedding shrubs in the presence of one of the Configurations~$\mathbf{(\diamond
6)}$--$\mathbf{(\diamond8)}$. For Lemmas~\ref{lem:HE3} and~\ref{lem:HE4}, which are for Configurations~$\mathbf{(\diamond
3)}$ and $\mathbf{(\diamond4)}$ a simpler auxiliary lemma (without reservations) would suffice.

\begin{lemma}\label{lem:randomshrubembedding}\HAPPY{M,D}
Let $H$ be a graph, let $X^*,X_1,X_2,
P_1,P_2,\ldots,P_L\subset V(H)$, and let $(T_1,r_1)$, \ldots, 
$(T_\ell,r_\ell)$ be rooted trees, such that $L\le k$,
$|P_j|\le k$ for each $j\in[L]$, and $|X^*|\geq 2\ell$.
Suppose that $\mindeg (X_1\cup X^*,X_{2})\geq 2\sum v(T_i)$ and $\mindeg
(X_2,X_{1})\geq 2\sum v(T_i)$.

Then there exist pairwise disjoint $\left(r_i\hookrightarrow
X^*,\Veven(T_i,r_i)\setminus\{r_i\}\hookrightarrow
X_1,\Vodd(T_i,r_i)\hookrightarrow X_2\right)$-em\-beddings $\phi_i$ of $T_i$ in
$G$ and a set $C\subset (X_1\cup X_2)\sm\bigcup\phi_i(T_i)$ of size $\sum v(T_i)$ such that for each $j\in[L]$ we have
\begin{equation}
  \label{Pjdoesitright}  |P_j\cap \bigcup\phi_i(T_i)|\le |P_j\cap C|+k^{3/4}.
\end{equation}
\end{lemma}

\begin{proof}
Let $m:=\sum v(T_i)$.

We construct pairwise disjoint random $\left(r_i\hookrightarrow
X^*,\Veven(T_i,r_i)\setminus\{r_i\}\hookrightarrow
X_1,\Vodd(T_i,r_i)\hookrightarrow X_2\right)$-embeddings $\phi_i$ and a set
$C\subseteq V(H)\sm\bigcup\phi_i(T_i)$ which satisfies~\eqref{Pjdoesitright}
with positive probability. Then the statement follows.

Enumerate the vertices of $\bigcup T_i$ as $\bigcup
V(T_i)=\{v_1,\ldots,v_{m}\}$ such that $v_i=r_i$ for
$i=1,\ldots,\ell$, and such that for each $j>\ell$ we have that the parent of
$v_j$ lies is the set $\{v_1,\ldots,v_{j-1}\}$.
Pick pairwise disjoint sets $A_1,\ldots,A_\ell\subset X^*$ of size two.
Uniformly at random denote one element of $A_j$ as $x_j$ and the other as $y_j$. 

Now, successively for $i= \ell+1,\ldots, m$, we shall  define vertices
$x_i$ and $y_i$. Let $r$ denote the root of the tree in which $v_i$ lies, and
let $v_s=\parent(v_i)$.
We shall choose $x_{i},y_j\in X_{j_i}$ where $j_i=\dist (r,v_i) \mod 2 +1$.
In step $i$, proceed as follows. Since $x_s\in X_{j_{s}}$ (or since $x_s\in X^*$), we have 
$$\deg (x_{s},X_{j_i}\sm  \bigcup_{h<i}\{x_h,y_h\})\geq 2.$$ Hence, 
we may take an arbitrary subset $A_i\subseteq (\neighbor(x_{s})\cap
X_{{j_{i}}})\setminus  \bigcup_{h<i}\{x_h,y_h\}$ of size exactly two. As above,
randomly label its elements as $x_i$ and $y_i$ independently of all other
choices.

The choices of the maps $(v_j\mapsto x_j)_{j=1}^{m}$ determine
$\phi_1,\ldots,\phi_\ell$.
Then $C:=\{y_1,\ldots,y_m\}$
has  size exactly $m$ and avoids $\bigcup \phi_i (T_i)$. 

 For each
$j\in[L]$ we set up a stochastic process $\mathfrak
S^{(j)}=\left((X_i^{(j)},Y_i^{(j)})\right)_{i=1}^{m}$, defined by
$X^{(j)}_i=\mathbf{1}_{\{x_i\in P_j\}}$ and $Y^{(j)}_i=\mathbf{1}_{\{y_i\in P_j\}}$. Note that $\mathfrak S^{(j)}\in
\Duplicate(|P_j|)\subset\Duplicate(k)$. Thus, for a fixed $j\in[L]$, by
Lemma~\ref{lem:randomduplicate}, the probability that
$|P_j\cap(\bigcup\phi_i(T_i))|>|P_j\cap C|+k^{3/4}$ is at most
$\exp(-\sqrt{k}/2)$.
Using the union bound over all $j\in [L]$ we get that Property~\ref{stochEmbedDIAMOND6P2} holds with probability at least $$1-L\cdot \exp\left(-\frac{\sqrt{k}}{2}\right)>0\;.$$ This finishes the proof.
\end{proof}

We now get to the first application of Lemma~\ref{lem:randomshrubembedding}.

\begin{lemma}\label{lem:embedStoch:DIAMOND6}
Assume we are in Setting~\ref{commonsetting}. Suppose
that the sets $V_2,V_3$ are such that for $j=2,3$ we have
\begin{equation}\label{mucuruvi}
\mindeg_H(V_j, V_{5-j})\geq \delta k,
\end{equation}
where  $\delta>300/k$, and $H$ is a $(\gamma k, \gamma)$-nowhere dense graph. Suppose that $U,U^*,
P_1,P_2,\ldots,P_L\subset V(G)$,
and $L\le k$, are such that $|U|\le \frac{ \delta}{24\sqrt\gamma} k$, $U^*\subset V_2$, $|U^*|\ge\frac{\delta}{4}k$, and $|P_j|\le k$ for each
$j\in[L]$.
Let $(T,r)$ be a rooted tree of order at most $\delta k/8$.

Then there exists a $\left(r\hookrightarrow U^*,\Veven(T,r)\setminus\{r\}\hookrightarrow
V_2\setminus U,\Vodd(T,r)\hookrightarrow V_3\setminus U\right)$-embedding $\phi$ of $T$ in
$G$ and a set
 $C\subset (V_2\cup V_3) \sm (U\cup \phi(T))$ of size $v(T)$ such that for each $j\in[L]$ we have
\begin{equation}
  \label{stochEmbedDIAMOND6P2} |P_j\cap \phi(T)|\le |P_j\cap C|+k^{3/4}.
\end{equation}
\end{lemma}

\begin{proof}
Set $B:=\shadow_{\Gexp}(U,\delta k/4)$.
By Fact~\ref{fact:shadowboundEXPANDER}, we have that
$|B|\le 64\frac{\gamma}{\delta}(\frac{ \delta}{24\sqrt\gamma})^2 k\le\frac{\delta}{4}k-2$. In
particular, $X^*:=U^*\setminus B$ has size at least $2$. Set $X_1:=V_2\sm (U\cup B)$ and set $X_2:=V_3\sm (U\cup B)$. Using~\eqref{mucuruvi},
we find that
$$\mindeg_{\Gexp}(X_j,X_{3-j})\ge \delta k - \maxdeg_{\Gexp}(X_j,U) - |B|\ge 
\delta k - \frac{\delta}{4}k - \frac{\delta}{4}k
\ge 2v(T)$$
 for $j=1,2$.
We may thus apply Lemma~\ref{lem:randomshrubembedding} to obtain the desired embedding $\phi$ and the set $C$.
\end{proof}

\begin{lemma}\label{lem:embedStoch:DIAMOND7}
Assume Setting~\ref{commonsetting} and Setting~\ref{settingsplitting}. Suppose
that we are given sets $Y_1,Y_2\subset\colouringp{1}\setminus \exceptVertSplit$
with $Y_1\subset\smallatoms$, and
\begin{enumerate}[(i)]
\item $\maxdeg_{\GD}(Y_1,\colouringp{1}\setminus
Y_2)\le \frac{\eta \gamma}{400}$, and \label{luckyluke}
\item $\mindeg_{\GD}(Y_2,Y_1)\ge \delta k$.\label{jollyjumper}
\end{enumerate}

Suppose that $U,U^*, P_1,P_2,\ldots,P_L\subset V(G)$ are sets such that $|U|\le
\frac{\Lambda\delta}{2\Omega^*} k$, $U^*\subset Y_1$,
with $|U^*|\ge\frac{\delta}{4}k$, $|P_j|\le k$ for each $j\in[L]$, and $L\le
k$.
Suppose $(T_1,r_1),\ldots,(T_\ell,r_\ell)$ are rooted trees of total order at
most $\delta k/1000$.
Suppose further that $\delta<\eta\gamma/100$, $\epsilon'<\delta/1000$,  and $k>1000/\delta$.

Then there exist pairwise disjoint $\left(r_i\hookrightarrow
U^*,\Veven(T_i,r_i)\hookrightarrow Y_1\setminus U,\Vodd(T_i,r_i)\hookrightarrow
Y_2\setminus U\right)$-em\-beddings $\phi_i$ of $T_i$ in $G$ and a set $C\subset
V(G-\bigcup \phi_i(T_i))$ of size $\sum v(T_i)$ such that for each $j\in[L]$ we
have that
\begin{equation}
 \label{stochEmbedDIAMOND7P2} |P_j\cap \bigcup\phi_i(T_i)|\le |P_j\cap
 C|+k^{3/4}.
\end{equation}
\end{lemma}
\begin{proof} Set $U':=\shadow_{\GD}(U,\delta k/2)\cup U$. By
Fact~\ref{fact:shadowbound}, we have $|U'|\le \Lambda k$. As $Y_1$ is a
$(\Lambda,\epsilon',\gamma,k)$-avoiding set, by Definition~\ref{def:avoiding}
there exists a set $B\subset Y_1$, $|B|\le \epsilon' k$ such that for all $v\in
Y_1\setminus B$ there exists a dense spot $D_v\in\DenseSpots$ with $v\in V(D_v)$
and $|V(D_v)\cap U'|\le \gamma^2k$. As $Y_1$ is disjoint from
$\exceptVertSplit$, by Definition~\ref{def:proportionalsplitting}\eqref{It:H4}
and by~\eqref{eq:proporcevelke}, we have that $\deg_{D_v}(v,V(D_v)\colouringpI{1})\ge\frac{\eta\gamma}{200} k$. 
By~\eqref{luckyluke}, we have that $\deg_{\GD}(v,V(D_v)\colouringpI{1}\setminus Y_2)<\frac{\eta\gamma
}{400}k$, and hence, $$\deg_{\GD}\big(v,(V(D_v)\colouringpI{1}\cap Y_2)\setminus
U'\big)\ge \frac{\eta\gamma k}{400}-\gamma^2k\ge \frac{\eta\gamma k}{800}\;.$$ Thus,
\begin{equation}\label{gr1}
\mindeg_{\GD}(Y_1\setminus B, Y_2\setminus (U'\cup B))\ge
\frac{\eta\gamma k}{800} -\eps'k\ge 2\sum v(T_i)\;.
\end{equation}
Further, by the definition of $U'$ and by~\eqref{jollyjumper}, we have
\begin{equation}\label{gr2}
\mindeg_{\GD}(Y_2\setminus U',Y_1\setminus (U\cup B))\ge \frac{\delta k}2 -\eps'k \ge
2\sum v(T_i)\;.
\end{equation}

Set $X^*:=U^*\setminus B$, and note that $|X^*|\ge \delta
k/4-\epsilon' k\ge 2\ell$. Set $X_1:=Y_1\sm (U\cup B)$ and $X_2:=Y_2\sm (U'\cup
B)$.
Inequalities~\eqref{gr1} and~\eqref{gr2} guarantee that we may apply
Lemma~\ref{lem:randomshrubembedding} to obtain the desired embeddings $\phi_i$.
\end{proof}

\begin{lemma}\label{lem:HE3}
Assume Setting~\ref{commonsetting}. Suppose that the sets $L',L'',\HugeVertices',\HugeVertices'',V_1,V_2$ witness Configuration~$\mathbf{(\diamond3)}(0,0,\gamma/4,\delta)$.
Suppose that $U,U^*\subset V(G)$ are sets such that $|U|\le  k$,
$U^*\subset V_1$, $|U^*|\ge\frac{\delta}{4}k$.
Suppose $(T,r)$ is a rooted tree of order at most $\delta k/1000$. Suppose further that $\delta\leq\gamma/100$, $\epsilon'<\delta/1000$, and $4\Omega^*/\delta\le \Lambda$.

Then there is an $\left(r\hookrightarrow U^*,\Veven(T,r)\setminus\{r\}\hookrightarrow V_1\setminus
U,\Vodd(T,r)\hookrightarrow V_2\setminus U\right)$-embedding of~$T$ in~$G$.
\end{lemma}

\begin{proof} The proof of this lemma is very similar to the one of Lemma~\ref{lem:embedStoch:DIAMOND7} (in fact, even easier). 
Set $U':=\shadow_{\GD}(U,\delta k/2)\cup U$ and note that $|U'|\le \Lambda k$ by
Fact~\ref{fact:shadowbound}. As $V_1$ is $(\Lambda,\epsilon',\gamma,k)$-avoiding, by Definition~\ref{def:avoiding} there is a set $B\subset V_1$, $|B|\le \epsilon' k$ such that for all $v\in V_1\setminus B$
there exists a dense spot $D_v\in\DenseSpots$ with 
$\deg_{D_v}(v,V(D_v)\sm U')\ge \gamma k/2$. 
By~\eqref{eq:WHtc}, we know that $\deg_{\GD}(v,V(D_v)\setminus
V_2)\leq \gamma  k/4$, and hence, 
$\deg_{\GD}\big(v,(V(D_v)\cap V_2)\setminus U'\big)\ge
\gamma k/{4}$.
Thus,
\begin{equation}\label{gr1grgr}
\mindeg_{\GD}(V_1\setminus B, V_2\setminus U')\ge
\frac{\gamma k}{4} \ge 2v(T)\;.
\end{equation}
Further, by the definition of $U'$ and by~\eqref{confi3theothercondi}, we have
\begin{equation}\label{gr2grgr}
\mindeg_{\GD}(V_2\setminus U',V_1\setminus U)\ge \frac{\delta k}2 \ge 2(T)\;.
\end{equation}

Set $X^*:=U^*\setminus B$, and note that $|X^*|\ge \delta k/4-\epsilon' k\ge 2$. Set $X_1:=V_1\sm (U\cup B)$ and $X_2:=V_2\sm (U'\cup B)$. Inequalities~\eqref{gr1grgr} and~\eqref{gr2grgr} guarantee that we may apply Lemma~\ref{lem:randomshrubembedding} (with empty sets $P_i$) to obtain the desired embedding $\phi$.
\end{proof}

\begin{lemma}\label{lem:HE4}\HAPPY{D, M}
Assume Setting~\ref{commonsetting}. Suppose that the sets $L',L'',\HugeVertices',\HugeVertices'',V_1,\smallatoms', V_2$ witness Configuration~$\mathbf{(\diamond4)}(0,0,\gamma/4,\delta)$.
Suppose that $U\subset V(G)$, $U^*\subset V_1$ are sets such
that $|U|\le k$ and $|U^*|\ge\frac{\delta}{4}k$.
Suppose $(T,r)$ is a rooted tree of order at most $\delta k/20$ with a fruit $r'$. Suppose further that $4\epsilon'\le \delta\le \gamma/100$, and $ \Lambda\geq 300(\frac{\Omega^*}{\delta})^3$.

Then there exists an $\left(r\hookrightarrow U^*,r'\hookrightarrow
V_1\setminus U, 
V(T)\setminus\{r,r'\}\hookrightarrow (\smallatoms'\cup V_2)\setminus
 U\right)$-embedding of~$T$ in~$G$.
\end{lemma}
\begin{proof}\
Set 
$$U':=\tilde U\cup \shadow_{\Gcapt-\HugeVertices}(U,\delta k/4)\cup
\shadow^{(2)}_{\Gcapt-\HugeVertices}(\tilde U,\delta k/4)$$ and let $$U'':=\tilde U\cup \shadow_{\GD}(U',\delta
k/2).$$ We use Fact~\ref{fact:shadowbound} to see that $|U'|\le \frac{\delta}{4\Omega^*} \Lambda
 k$ and  $|U''|\le \Lambda
 k$.  We then use Definition~\ref{def:avoiding} and~\eqref{confi4lastcondi} to find a set $B\subset \smallatoms'$ of size at most $\epsilon' k$ such that
\begin{equation}\label{gr81grgr}
\mindeg_{\GD}(\smallatoms'\setminus B, V_2\setminus U'')\ge 2v(T)\;.
\end{equation}

 Using~\eqref{gr81grgr},  and employing~\eqref{confi4:3} and ~\eqref{confi4othercondi}, we see that we may apply Lemma~\ref{lem:randomshrubembedding} with $X^*_\PARAMETERPASSING{L}{lem:randomshrubembedding}:= U^*$,
 $X_{1,\PARAMETERPASSING{L}{lem:randomshrubembedding}}:=\smallatoms'\sm (B\cup U')$ and $X_{2,\PARAMETERPASSING{L}{lem:randomshrubembedding}}:=V_2\sm U''$ (and with empty sets $P_i$) in order to embed the tree $T-T(r,\uparrow r')$ rooted at $r$.
Then embed $T(r,\uparrow r')$, by applying  Lemma~\ref{lem:randomshrubembedding} a second time, using~\eqref{confi4:3} and~\eqref{confi4:4}.
\end{proof}

\subsection{Main embedding lemmas}\label{sec:MainEmbedding}
For this section, we need to introduce the notion of a ghost. Given a semiregular matching $\mathcal N$, we call an involution $\mathfrak{d}:V(\mathcal N)\rightarrow V(\mathcal N)$ with the property that $\mathfrak{d}(S)=T$ for each $(S,T)\in\mathcal N$ a \index{general}{matching involution}\emph{matching involution}.

Assume Setting~\ref{commonsetting} and fix a matching involution $\mathfrak{b}$ for $\mathcal M_A\cup\mathcal M_B$. For any set $U\subset V(G)$ we then define \index{general}{ghost}\index{mathsymbols}{*GHOST@$\ghost$} by
$$\ghost(U):=U\cup \mathfrak{b}\big(U\cap V(\M_A\cup\M_B)\big)\;.$$
Clearly, we have that $|\ghost(U)|\le 2|U|$, and $|\ghost(U)\cap S|=|\ghost(U)\cap T|$ for each $(S,T)\in\M_A\cup\M_B$.

The notion of ghost extends to other semiregular matchings. If $\mathcal N$ is a semiregular matching and $\mathfrak{d}$ a matching involution for $\mathcal N$ then we write $\ghost_{\mathfrak{d}}(U):=U\cup \mathfrak{d}\big(U\cap V(\mathcal N)\big)$.

\subsubsection{Embedding in Configuration~$\mathbf{(\diamond1)}$}\label{sssec:EmbedDiamon0Diamond1}
This subsection contains an easy observation that $\treeclass{k}\subset G$ in case $G$ contains Configuration~$\mathbf{(\diamond1)}$.

\begin{lemma}\label{lem:embed:greedy}
Let   $G$ be a graph, and let $A, B\subseteq V(G)$ be such that
$\mindeg(G[A,B])\ge k/2$, and $\mindeg(A)\ge k$. Then
$\treeclass{k}\subset G$.
\end{lemma}
\begin{proof}
Let $T\in\treeclass{k}$ have colour classes $X$ and $Y$, with $|X|\ge k/2 \ge
|Y|$.  By Fact~\ref{fact:treeshavemanyleaves}, for the set $W$ of those leaves of $T$ that lie in $X$, we have $|X\setminus W|\le k/2$. We embed
$T-W$ greedily in $G$, mapping $Y$ to $A$ and $X\setminus W$ to $B$. We then embed $W$ using the fact that $\mindeg(A)\ge k$.
\end{proof}

\subsubsection{Embedding in Configurations $\mathbf{(\diamond2)}$--$\mathbf{(\diamond5)}$}\label{sssec:EmbedMoreComplex}
In this section we show how to embed $T_\PARAMETERPASSING{T}{thm:main}$ in the presence of configurations $\mathbf{(\diamond2)}$--$\mathbf{(\diamond5)}$. As outlined in Section~\ref{ssec:EmbedOverview25} our main embedding lemma, Lemma~\ref{lem:conf2-5}, builds on Lemma~\ref{lem:blueShrubSuspend} which handles Stage~1 of the embedding, and Lemma~\ref{lem:embedC'endshrub} which handles Stage~2.

\begin{lemma}\label{lem:embedC'endshrub}\HAPPY{D,M}
Assume we are in Setting~\ref{commonsetting}.
Suppose $L'',L'$ and $\HugeVertices'$ witness Preconfiguration
$\mathbf{(\clubsuit)}(\frac{10^5\Omega^*}{\eta})$.
 Let $(T,r)$ be a rooted tree of order at most $\gamma^2\nu k/6$. 
 Let $U\subset V(G)$ with $|U|+v(T)\le k$, and let $v\in \HugeVertices'\setminus U$. 
 Then there exists an $(r\hookrightarrow v,V(T)\hookrightarrow V(G)\setminus U)$-embedding of $(T,r)$.
\end{lemma}
\begin{proof}
We proceed by induction on the order of $T$. The base
$v(T)\leq 2$ obviously holds. Let us assume Lemma~\ref{lem:embedC'endshrub} is true for all trees $T'$ with $v(T')<v(T)$.

Let $U_1:=\shadow_{\Gcapt}(U-\HugeVertices,\eta k/200)$, and $U_2:=\bigcup \{C\in \clusters\::\: |C\cap U|\ge \frac12|C|\}$. We have $|U_1|\le \frac{200\Omega^*}{\eta}k$ by Fact~\ref{fact:shadowbound}, and $|U_2|\le 2|U|$. Set 
\begin{align*}
L_{\smallatoms} &:=L''\cap\shadow_{\Gcapt}(\smallatoms,\frac{\eta k}{50}),\\
L_{\HugeVertices} &:=L''\cap\shadow_{\Gcapt}\left(\HugeVertices, |U\cap\HugeVertices|+\frac{\eta k}{50}\right),\text{ and }\\
L_{\clusters} &:=L''\cap\shadow_{\Gblack}\left(V({\Gblack}),(1+\frac{\eta}{50})k-|U\cap\HugeVertices|\right)\;.
\end{align*}
Observe that $L_{\clusters}\subseteq \bigcup \clusters$ and that since $L''\subseteq 
\largevertices{\frac9{10}\eta}{k}{\Gcapt}\setminus\HugeVertices$, we have
$$L''\subset V(\Gexp)\cup\smallatoms\cup L_{\HugeVertices}\cup L_{\smallatoms}\cup L_{\clusters}\;.$$
As by~\eqref{eq:clubsuitCOND3}, we have $\deg_G(v,L'')\ge \frac{10^5\Omega^*k}{\eta}>5(|U\cup U_1 \cup U_2|+v(T)+\eta k)$,
 one of the following five cases must occur.

\medskip
\noindent{\underline{\sl Case I: $\deg_G(v,V(\Gexp)\setminus U)>v(T)+\eta k$.}}
Lemma~\ref{lem:embed:greyFOREST} gives an embedding of the forest $T-r$ (whose components are rooted at neighbours of $r$). The input sets/parameters of Lemma~\ref{lem:embed:greyFOREST} are $Q_\PARAMETERPASSING{L}{lem:embed:greyFOREST}:=1$, 
$\zeta_\PARAMETERPASSING{L}{lem:embed:greyFOREST}:=12\sqrt\gamma$, $U^*_\PARAMETERPASSING{L}{lem:embed:greyFOREST}:=(\neighbor_G(v)\cap V(\Gexp))\setminus U$, $U_\PARAMETERPASSING{L}{lem:embed:greyFOREST}:=U$, $V_1=V_2:=V(\Gexp)$.

\smallskip
\noindent{\underline{\sl Case II: $\deg_G(v,\smallatoms\setminus U)>v(T)+\eta k$.}} Lemma~\ref{lem:embed:avoidingFOREST} gives an embedding of the forest $T-r$ (whose components are rooted at neighbours of $r$). The input sets/parameters of Lemma~\ref{lem:embed:avoidingFOREST} are $U^*_\PARAMETERPASSING{L}{lem:embed:avoidingFOREST}:=(\neighbor_G(v)\cap \smallatoms)\setminus U$, $U_\PARAMETERPASSING{L}{lem:embed:avoidingFOREST}:=U$ and $\eps_\PARAMETERPASSING{L}{lem:embed:avoidingFOREST}:=\eps'\leq\eta$.
Here, and below, we tacitly implicitly assume parameters of the same name to be the same, i.e.~$\gamma_\PARAMETERPASSING{L}{lem:embed:avoidingFOREST}:=\gamma$.

\smallskip
\noindent{\underline{\sl Case III: $\deg_G(v,L_{\smallatoms}\setminus (U\cup U_1))>v(T)+\eta k$.}} We only outline the strategy. Embed the children of~$r$ in $L_{\smallatoms}\setminus (U\cup U_1)$ using a map $\phi:\children_T(r)\rightarrow L_{\smallatoms}\setminus (U\cup U_1)$. By definition of $L_{\smallatoms}$, and $U_1$, we have $\deg_{\Gcapt}(\phi(w),\smallatoms\setminus U)> \frac{\eta k}{100}$ for each $w\in \children_T(r)$. Now, for every $w\in \children_T(r)$ we can proceed as in Case~II to extend this embedding to the rooted tree $\big(T(r,\uparrow w) ,w\big)$. That is, Case~III is ``Case~II with an extra step in the beginning''.


\smallskip
\noindent{\underline{\sl Case IV: $\deg_G(v,L_{\HugeVertices}\setminus U)>v(T)+\eta k$.}} 
We embed the children $\children_T(r)$ of $r$ in distinct vertices of $L_{\HugeVertices}\setminus U$. This is possible by the assumption of Case~IV. 

Now,~\eqref{eq:clubsuitCOND1} implies that $\mindeg_{\Gcapt}(L_{\HugeVertices},\HugeVertices')\ge |U\cap \HugeVertices|+ \frac{\eta k}{100}$. Consequently, $\mindeg_{\Gcapt}(L_{\HugeVertices},\HugeVertices'\setminus U)\ge \frac{\eta k}{100}$. Therefore, for each $w\in\children_T(r)$ embedded in  $L_{\HugeVertices}\setminus U$ we can find an embedding of $\children_T(w)$ in $\HugeVertices'\setminus U$ such that the images of grandchildren of $r$ are disjoint. We fix such an embedding. We can now apply induction. More specifically, for each grandchild $u$ of $r$ we embed the rooted tree $\big(T(r,\uparrow u), u\big)$ using Lemma~\ref{lem:embedC'endshrub} (employing induction) using the updated set $U$, to which the images of the newly embedded vertices were added.

\smallskip
\noindent{\underline{\sl Case V: $\deg_G(v,L_{\clusters}\setminus (U\cup \cup U_1\cup U_2))\ge v(T)$.}} Let $u_1,\ldots,u_\ell$ be the children of $r$. Let us consider arbitrary distinct neighbours $x_1,\ldots,x_\ell\in L_{\clusters}\setminus (U\cup U_1\cup U_2)$ of $v$. Let $T_i:=T(r,\uparrow u_i)$. We sequentially embed the rooted trees $(T_i,u_i)$, $i=1,\ldots,\ell$, writing $\phi$ for the embedding. In step $i$, consider the set $W_i:=\left(U\cup \bigcup_{j<i} \phi(T_j)\right)\setminus \HugeVertices$. Let $D_i\in\clusters$ be the cluster containing $x_i$. By definition of $L_{\clusters}$ and of $U_1$, $$\deg_{\Gblack}(x_i,V(\Gblack)\setminus W_i)\ge \frac{\eta k}{50}-\frac{\eta k}{200}\ge \frac{\eta k}{100}\;.$$ Fact~\ref{fact:clustersSeenByAvertex} yields a cluster $C_i\in \clusters$ such that $$\deg_{\Gblack}(x_i,C_i\setminus W_i)\ge \frac{\eta}{100}\cdot \frac{\gamma\clustersize}{2(\Omega^*)^2}>\frac{\gamma^2\clustersize}{2}+v(T)>\frac{12\epsilon'\clustersize}{\gamma^2}+v(T)\;.$$ In particular there 
is at least one edge from $E(\Gblack)$ between $C_i$ and $D_i$, and therefore, $(C_i,D_i)$ forms an $\epsilon'$-regular pair of density at least $\gamma^2$ in $\Gblack$. 
Map $u_i$ to $x_i$ and let $F_1,\ldots,F_m$ be the components of the forest $T_i-u_i$.
We now  sequentially embed the trees $F_j$ in the pair $(D_i,C_i)$ using Lemma~\ref{lem:embed:regular}, with  
$X_\PARAMETERPASSING{L}{lem:embed:regular}:=C_i\setminus (W_i\cup \bigcup_{q<j}\phi(F_q))$, $X^*_\PARAMETERPASSING{L}{lem:embed:regular}:=\neighbor_{\Gblack}(x_i,X_\PARAMETERPASSING{L}{lem:embed:regular})$, $Y_\PARAMETERPASSING{L}{lem:embed:regular}:=D_i\setminus (W_i\cup \{x_i\}\cup \bigcup_{q<j}\phi(F_q))$, $\eps_\PARAMETERPASSING{L}{lem:embed:regular}:=\eps'$,
and $\beta_\PARAMETERPASSING{L}{lem:embed:regular}:=\gamma^2/3$.
\end{proof}

We are now ready for the lemma that will handle Stage~1 in configurations $\mathbf{(\diamond2)}$--$\mathbf{(\diamond5)}$.

\begin{lemma}\label{lem:blueShrubSuspend}
\HAPPY{D,M}
Assume we are in Setting~\ref{commonsetting}, with $L'',L',\HugeVertices'$ witnessing
$\mathbf{(\clubsuit)}(\Omega^\star)$ in $G$.
Let $U\subset V(G)\setminus \HugeVertices$
and let $(T,r)$ be  a rooted tree
with $v(T)\le k/2$ and $|U|+v(T)\le k$.
Suppose that each component of $T-r$ has order at most $\tau k$.
Let $x\in (L''\cap \YB)\setminus\bigcup_{i=0}^2\shadow_{\Gcapt}^{(i)}(\ghost(U),
\eta k/1000 )$.

Then there is a subtree $T'$ of $T$ with $r\in V(T')$
which has an $(r\hookrightarrow x, V(T')\setminus\{r\}\hookrightarrow
V(G)\setminus \HugeVertices)$-embedding $\phi$.
Further, the components of $T-T'$ can be partitioned into two (possibly empty) sets $\mathcal C_1$, $\mathcal C_2$, such that
 the
following two assertions
hold.
\begin{enumerate}[(a)]
 \item 
 \label{eq:boty}
If $\mathcal C_1\neq\emptyset$, then $\mindeg_{\Gcapt}(\phi(\parent (V(\bigcup
\mathcal C_1))),\HugeVertices')> k+\frac{\eta k}{100} -v(T')$, 
\item
\label{eq:botyzwei} 
 $\parent (V(\bigcup \mathcal C_2))\subseteq \{r\}$, and
$\deg_{\Gcapt}(x,\HugeVertices')>\frac k2+\frac{\eta k}{100}
-v(T'\cup\bigcup \mathcal C_1)$.
\end{enumerate}
\end{lemma}

\begin{proof}
Let $\mathcal C$ be the set of all components of $T-r$. We start by defining
$\mathcal C_2$. Then, we have to distribute  $T-\bigcup\mathcal C_2$ between
$T'$ and $\mathcal C_1$. First, we find a set $\mathcal C_M\subseteq \mathcal
C\sm\mathcal C_2$ which fits into the matching $\mathcal M_A\cup \mathcal M_B$
(and thus will form part of $T'$). Then, we consider the remaining components of
$\mathcal C\sm\mathcal C_2$: some of these will be embedded entirely,  of others
we only embed the root, and leave the rest for $\mathcal C_1$. Everything embedded
will become a part of $T'$.

Throughout the proof we write $\shadow$ for $\shadow_{\Gcapt}$.

\medskip

Set $\overline\Vgood:=\Vgood\setminus \shadow(\ghost(U),\frac{\eta k}{1000})$, and choose
 $\tilde{\mathcal C}\subseteq\mathcal C$ such that 
\begin{equation}\label{eq:HNY}
\deg_{\Gcapt}(x, \overline\Vgood)-\frac{\eta k}{30}
\ < \ 
\sum_{S\in \tilde{\mathcal C}} v(S)
\ \leq \
\max\left\{0,\deg_{\Gcapt}\left(x, \overline\Vgood\right)-\frac{\eta k}{40}\right\}.
\end{equation}
Set $\mathcal C_2:=\mathcal C\sm \tilde{\mathcal C}$.
Note that this choice clearly satisfies
the first part of~\eqref{eq:botyzwei}. Let us now verify the second part of~\eqref{eq:botyzwei}. For this, we calculate
\begin{align*}
\deg_{\Gcapt}(x, \HugeVertices') & \geq
\deg_{\Gcapt}(x,V_+\setminus L_\#)-\deg_{\Gcapt}(x,\shadow(\ghost(U),\frac{\eta k}{1000}))
\\
&~~~~-\deg_{\Gcapt}(x,V_+\setminus(L_\#\cup \shadow(\ghost(U),\frac{\eta k}{1000})\cup\HugeVertices))\\
&~~~~-\deg_{\Gcapt}(x,\HugeVertices\setminus\HugeVertices')\\
\JUSTIFY{by~\eqref{eq:defYB}, $x\not\in\shadow^{(2)}(\ghost(U),\frac{\eta k}{1000})$,~\eqref{eq:HNY}, \eqref{eq:clubsuitCOND1}}
&\ge \left(\frac k2 +
\frac{\eta k}{20}\right)-\frac{\eta k}{1000}-\left(
\sum_{S\in \tilde{\mathcal C}} v(S)+\frac{\eta
k}{30}\right)-\frac{\eta k}{100}\\
& > \ \frac k2 - \sum_{S\in \tilde{\mathcal C}} v(S) + \frac{\eta k}{20}
\\
&\geq \frac k2
- v(T'\cup\bigcup\mathcal C_1)+ \frac{\eta k}{100},
\end{align*}
as desired for~\eqref{eq:botyzwei}. 

\smallskip

Now, set 
\begin{equation}\label{eq:trickyM}
\mathcal M :=\big\{(X_1,X_2)\in \mathcal
M_A\cup \mathcal M_B\::\: \deg_{\GD}(x,(X_1\cup X_2)\setminus\smallatoms)>
0\big\}\;.
\end{equation}
\begin{claim}\label{cl:Megdes}
We have $|V(\mathcal M)|\le \frac{4(\Omega^*)^2}{\gamma^2}k$.
\end{claim} 
\begin{proof}[Proof of Claim~\ref{cl:Megdes}]
Indeed, let $(X_1,X_2)\in\M$, i.e.~$(X_1,X_2)\in\M_A\cup\M_B$ with $\deg_{\GD}(x,(X_1\cup
X_2)\setminus\smallatoms)> 0$. Then, using Property~\ref{commonsetting3} of
Setting~\ref{commonsetting}, we see that there exists a cluster $C_{(X_1,X_2)}\in\clusters$
such that $\deg_{\GD}(x,C_{(X_1,X_2)})>0$, and either $X_1\subset
C_{(X_1,X_2)}$ or $X_2\subset C_{(X_1,X_2)}$. In particular, there exists
a dense spot $(A_{(X_1,X_2)},B_{(X_1,X_2)};F_{(X_1,X_2)})\in\DenseSpots$ such
that $x\in A_{(X_1,X_2)}$, and $X_1\subset B_{(X_1,X_2)}$ or $X_2\subset
B_{(X_1,X_2)}$. By Fact~\ref{fact:boundedlymanyspots}, there are at most
$\frac{\Omega^*}{\gamma}$ such dense spots, let $Z$ denote the union of all vertices contained in these spots. 
Fact~\ref{fact:sizedensespot} implies that $|Z|\le \frac{2(\Omega^*)^2}{\gamma^2}k$.
Thus $|V(\M)|\le 2 |V(\M)\cap Z|\le 2 |Z|\le
\frac{4(\Omega^*)^2}{\gamma^2}k$. 
\end{proof}
First we shall embed as
many components from $\tilde{\mathcal C}$ as possible in $\mathcal M$. To this
end, consider an inclusion-maximal subset $\mathcal C_M$ of $\tilde{\mathcal
C}$ with
\begin{equation}\label{eq:trickySum}
\sum_{S\in\mathcal C_M}v(S)\leq \deg_{\Gcapt}(x,V(\mathcal
M))-\frac{\eta k}{1000}\;.
\end{equation}

We aim to utilize the degree of $x$ to $V(\M)$
to embed $\mathcal C_M$ in $V(\M)$ using the regularity method.
 
\begin{subremark}
This remark (which may as well be skipped at a first reading) is aimed at those readers that are wondering about a seeming inconsistency of the defining
formulas~\eqref{eq:trickyM} for $\M$, and~\eqref{eq:trickySum} for $\mathcal
C_M$. 
That is,~\eqref{eq:trickyM} involves the degree in $\GD$ and
excludes the set $\smallatoms$, while~\eqref{eq:trickySum} involves the degree in
$\Gcapt$. The setting in~\eqref{eq:trickyM} was chosen so that it allows us to
control the size of $\M$ in Claim~\ref{cl:Megdes}, crucially relying on Property~\ref{commonsetting3} of
Setting~\ref{commonsetting}. Such a
control is necessary to make the regularity method work. Indeed,  in each regular
pair there may be a small number of atypical vertices\footnote{The issue of atypicality itself could be avoided by preprocessing each pair $(S,T)$ of $\M_A\cup\M_B$
and making it super-regular. However this is not possible for atypicality with
respect to a given (but unknown in advance) subpair $(S',T')$.}, and we must
avoid these vertices when embedding the components by the regularity method.
 Thus without the control on $|\M|$ it might happen that the degree of $x$  
 is unusable because $x$ sees very small numbers of atypical vertices in an enormous number 
 of sets corresponding to $\M$-vertices. On the other hand, the edges $x$ sends 
 to $\smallatoms$ can be utilized by other techniques in later stages.
Once we have defined $\M$ we want to use the full degree to $V(\M)$ to ensure we
can embed the shrubs as balanced as possible into the $\M$-edges. This is necessary as otherwise part of the degree of $x$ might be unusable for embedding, e.g.~because it might go to $\M$-vertices whose partners are already full.
\end{subremark}

 For each  $(C,D)\in
\mathcal M$ we choose $\mathcal C_{CD}\subseteq \mathcal C_M$ maximal such that 
\begin{equation}
\sum_{S\in \mathcal C_{CD}}v(S)\le
\deg_{\Gcapt}(x,(C\cup  D)\setminus \ghost
(U))- (\frac\gamma{\Omega^*})^3 |C|\;,\label{eq:hezkamistnost}
\end{equation}
and further, we
require  $\mathcal C_{CD}$ to be disjoint from families
$\mathcal C_{C'D'}$ defined in previous
steps. 
We claim that $\{\mathcal C_{CD}\}_{(C,D)\in\M}$
forms a partition of $\mathcal C_M$, i.e., all the elements of $\mathcal C_M$
are  used. Indeed, otherwise, by the maximality of $\mathcal
C_{CD}$ and since the components of $T-r$ have size at most $\tau k$, we obtain
\begin{align}\label{eq:ForM}
\begin{split}
\sum_{S\in \mathcal C_{CD}}v(S)&\ge \deg_{\Gcapt}(x,(C\cup  D)\setminus \ghost
(U))-(\frac\gamma{\Omega^*})^3 |C|-\tau k\\
&\geByRef{eq:KONST} \deg_{\Gcapt}(x,(C\cup  D)\setminus \ghost
(U))-2(\frac\gamma{\Omega^*})^3 |C|\;,
\end{split}
\end{align}
for each $(C,D)\in\M$.
Then we have
\begin{align*}
\sum_{S\in \mathcal
C_M}v(S)&>\sum_{(C,D)\in \mathcal
M}\sum_{S\in \mathcal C_{CD}} v(S)\\
\JUSTIFY{by~\eqref{eq:ForM}}&\ge \sum_{(C,D)\in
\mathcal M} \big(\deg_{\Gcapt}(x,(C\cup
D)\setminus \ghost(U))-2(\frac\gamma{\Omega^*})^3 |C|\big)\\
\JUSTIFY{by Claim~\ref{cl:Megdes} and
Fact~\ref{fact:boundMatchingClusters}} &\ge \deg_{\Gcapt}(x,V(\mathcal
M)\setminus \ghost(U))-2(\frac\gamma{\Omega^*})^3 \cdot \frac{2(\Omega^*)^2}{\gamma^2}k\\ 
\JUSTIFY{as $x\not\in\shadow(\ghost(U))$}&\ge
\deg_{\Gcapt}(x,V(\mathcal M))- \frac{\eta k}{1000}\\
\JUSTIFY{by~\eqref{eq:trickySum}}&\ge\sum_{S\in \mathcal C_M}v(S)\;,
\end{align*}
a contradiction.

\def\LfCD{\PARAMETERPASSING{L}{lem:fillingCD}}
We use Lemma~\ref{lem:fillingCD} to
embed the components of $\mathcal C_{CD}$ in $(C\cup D)\setminus \ghost(U)$ with the
following setting: $C_\LfCD:=C$, $D_\LfCD:=D$, $U_\LfCD:=\ghost(U)$, $X^*_\LfCD:=(\neighbor_{\Gcapt}(x)\cap (C\cup D))\setminus U_\LfCD$,
 and $(T_i, r_i)$ are the rooted trees from $\mathcal C_{CD}$
with the roots being the neighbours of $r$. The constants in Lemma~\ref{lem:fillingCD} are
$\epsilon_\LfCD:=\epsilon'$, $\beta_\LfCD:=\sqrt{\epsilon '}$, and
$\ell_\LfCD:=|C|\ge \nu\pi k$. The rooted trees in $\mathcal C_{CD}$ are smaller than $\epsilon_\LfCD\ell_\LfCD$ by~\eqref{eq:KONST}. Condition~\eqref{eq:conFill} is satisfied by~\eqref{eq:hezkamistnost}, and since $(\gamma/{\Omega^*})^3\geq 50\sqrt{\eps '}$.

\medskip

It remains to deal with the components $\tilde{\mathcal C}\setminus\mathcal
C_M$.
In the sequel we shall assume that $\tilde{\mathcal C}\setminus\mathcal
C_M\neq \emptyset$ (otherwise skip this step and go directly to the definition of $T'$ and $\mathcal C_1$, with $p=0$). Thus, by our choice of $\mathcal C_M$, we have
\begin{equation}\label{eq:maxT1'}
\sum_{S\in\mathcal C_M}v(S)\ge
\deg_{\Gcapt}(x,V(\mathcal
M))-\frac{\eta k}{900}\;. \end{equation}

Let $T_1,T_2,\ldots,T_p$ be the trees of $\tilde{\mathcal
C}\setminus\mathcal C_M$ rooted at the vertices $r_i\in\children(r)\cap
V(T_i)$. We shall sequentially extend our embedding of $\mathcal C_M$ to subtrees $T'_i\subset T_i$. Let $U_i\subset V(G)$ be the union of the images of $\bigcup\mathcal C_M\cup\{r\}$ and of $T'_1,\ldots,T'_i$ under this embedding.

Suppose that we have embedded the trees $T'_1,\ldots,T'_i$ for some
$i=0,1,\ldots,p-1$. We claim that at least one of the following holds.
\begin{enumerate}
  \item[{\bf (V1)}]
  $\deg_{\Gcapt}(x,V(\Gexp)\setminus
  (U\cup U_i))\ge \frac{\eta k}{1000}$,
  \item[{\bf
  (V2)}]$\deg_{\Gcapt}(x,\smallatoms\setminus
  (U\cup U_i))\ge \frac{\eta k}{1000}$, or
  \item[{\bf
  (V3)}]$\deg_{\Gcapt}(x,L'\setminus
  (V(\Gexp)\cup\smallatoms\cup U\cup U_i\cup \shadow(\ghost(U),\frac{\eta k}{1000})))\ge
  \frac{\eta k}{1000}$.
\end{enumerate}
Indeed, suppose that none of {\bf(V1)}--{\bf (V3)} holds. Then, first note that since $U\subseteq \ghost(U)$ and since $x\notin \shadow(\ghost(U),\eta k/1000)$, we have
\begin{equation}\label{whatxsendstoU}
\deg_{\Gcapt}(x,U)\leq \eta k/1000.
\end{equation}

Also,
\begin{equation}\label{MAMBMatoms}
 \deg_{\GD}(x,V(\M_A\cup \M_B)) \leq\deg_{\GD}(x,V(\M)\cup\smallatoms).
\end{equation}

Thus,
\begin{align*}
\deg_{\Gcapt}&\left(x,\Vgood\setminus\shadow(\ghost(U), \frac{\eta k}{1000})\right)\\
\JUSTIFY{by \eqref{whatxsendstoU} and \eqref{MAMBMatoms}, def of $\Vgood$}&\le 
\deg_{\Gcapt}\left(x,\left(V(\M)\cup V(\Gexp)\cup \smallatoms\cup L'\right)\setminus
(U\cup \shadow(\ghost(U), \frac{\eta k}{1000})\right)\\
&~~~~+\deg_{\Gcapt}\big(x,\largevertices{\frac9{10}\eta}{k}{\Gcapt}\setminus(\HugeVertices\cup
L')\big)+\frac{\eta k}{1000}\\ 
\JUSTIFY{by~\eqref{eq:clubsuitCOND3}}&\le
\deg_{\Gcapt}\left(x,\left( V(\Gexp)\cup
\smallatoms\cup L'\right)\setminus
(V(\M)\cup U\cup \shadow(\ghost(U),\frac{\eta k}{1000}))\right)
\\
&~~+
\deg_{\Gcapt}\left(x,V(\M)\right)+\frac{\eta k}{100}+\frac{\eta k}{1000}
\\
\JUSTIFY{by $\neg{\bf(V1)}$, $\neg{\bf(V2)}$, $\neg{\bf(V3)}$,
by~\eqref{eq:maxT1'}} &\le
3\cdot\frac{\eta
k}{1000}+\sum_{j=1}^iv(T'_j)
+
\sum_{S\in\mathcal C_M}v(S)+\frac{\eta k}{900}
+\frac{\eta k}{100}+\frac{\eta k}{1000}
\\
&< \sum_{S\in\tilde{\mathcal
C}}v(S)+\frac{\eta k}{40}\;,
\end{align*}
a contradiction to~\eqref{eq:HNY}.

In cases {\bf(V1)}--{\bf(V2)} we shall embed the entire tree
$T'_{i+1}:=T_{i+1}$. In case {\bf(V3)} we either embed the entire
tree $T'_{i+1}:=T_{i+1}$, or embed only one vertex $T'_{i+1}:=r_{i+1}$ (that will only happen in case  {\bf (V3c)}). In the latter case,  we keep track of the components of $T_{i+1}-r_{i+1}$ in the set $\mathcal C_{1,i+1}$  (we tacitly assume we set $\mathcal C_{1,i+1}:=\emptyset$ in all cases other than {\bf (V3c)}). The union of the sets $\mathcal C_{1,i}$ will later form the set $\mathcal C_1$. Let us go through
our three cases in detail.

\smallskip

In case {\bf(V1)} we embed $T_{i+1}$ rooted at $r_{i+1}$
using Lemma~\ref{lem:embed:greyFOREST}
\def\Leg{\PARAMETERPASSING{L}{lem:embed:greyFOREST}}
 for one tree (i.e.~$\ell_\Leg:= 1$) with the following sets/parameters:
$H_\Leg:=\Gexp$,
$U_\Leg:=U\cup U_i$, $U^*_\Leg:=\neighbor_{\Gcapt}(x)\cap
(V(\Gexp)\setminus(U\cup U_i))$, $V_1=V_2:=V(\Gexp)$, $Q_\Leg:=1$, $\zeta_\Leg:=\rho$, and 
$\gamma_\Leg:=\gamma$. Note that $|U\cup
U_i|< k$, that $|\neighbor_{\Gcapt}(x)\cap
(V(\Gexp)\setminus (U\cup U_i))|\ge \eta k/1000>32\gamma k/\rho+1$,  
that $v(T_{i+1})\le
\tau k<\rho k/4$ and that $128\gamma<\rho^2$.

\smallskip

In case {\bf(V2)} we embed $T_{i+1}$ rooted at $r_{i+1}$
using Lemma~\ref{lem:embed:avoidingFOREST} 
\def\Lavoid{\PARAMETERPASSING{L}{lem:embed:avoidingFOREST}} for one tree (i.e.~$\ell_\Lavoid:= 1$) with the following setting:
$H_\Lavoid:= G-\HugeVertices$, $\smallatoms_\Lavoid:=\smallatoms$,
$U_\Lavoid:=U\cup U_i$, $U^*_\Lavoid:=\neighbor_{\Gcapt}(x)\cap (\smallatoms\setminus(U\cup U_i))$, $\Lambda_\Lavoid:=\Lambda$, $\gamma_\Lavoid:=\gamma$,
$\epsilon_\Lavoid:=\epsilon'$. Note that $|U\cup U_i|\le k<\Lambda
k$, that $|\neighbor_{\Gcapt}(x)\cap (\smallatoms\setminus (U\cup U_i))|\ge \eta
k/1000 >2\eps' k$, and that $v(T_{i+1})\le \tau k<\gamma k/2$. 

\smallskip

We commence case~{\bf(V3)} with an auxiliary claim.
\begin{claim}\label{cl:TEC}
There exists $C_0\in\clusters$ such that $$\deg_{\GD}\big(x,(C_0\cap L')\setminus (V(\Gexp)\cup U\cup U_i\cup
  \shadow(\ghost(U),\frac{\eta k}{1000}))\big)\ge \frac{\eps'}{\gamma^2} \clustersize\;.$$
\end{claim} 
\begin{proof}[Proof of Claim~\ref{cl:TEC}]
Observe that
$L'\setminus
(V(\Gexp)\cup\smallatoms\cup \HugeVertices\cup U\cup U_i)\subset \bigcup
\clusters$ and that (since $x\in\bigcup\clusters$) \[E_{\Gcapt}[x,L'\setminus (V(\Gexp)\cup\smallatoms\cup
U\cup U_i\cup \shadow(\ghost(U),\frac{\eta k}{1000}))]\subset E(\GD)\;.\] By
Fact~\ref{fact:clustersSeenByAvertex}, there are at most $\frac{2(\Omega^*)^2k}{\gamma^2\clustersize}$ clusters $C\in\clusters$ such that
$\deg_{\GD}(x,C)>0$. Using the assumption~{\bf(V3)}, there exists a
cluster $C_0\in\clusters$ such that 
\begin{align*}
\deg_{\GD}\left(x,(C_0\cap L')\setminus
(V(\Gexp)\cup U\cup U_i\cup \shadow(\ghost(U),\frac{\eta k}{1000}))\right)&\ge
\frac{\eta
k}{1000}\cdot\frac{\gamma^2 \clustersize}{2(\Omega^*)^2k}\\
&\overset{\eqref{eq:KONST}}\ge \frac{\eps'}{\gamma^2}\clustersize\;,
\end{align*} 
as desired.
\end{proof}
Let us take a cluster $C_0$ from Claim~\ref{cl:TEC}. We embed the root $r_{i+1}$ 
of $T_{i+1}$ in an arbitrary neighbour $y$ of $x$ in $(C_0\cap L')\setminus
(V(\Gexp)\cup U\cup U_i\cup \shadow(\ghost(U),\frac{\eta k}{1000}))$.

Let $H\subset G$ be the subgraph of $G$ consisting of all edges in dense spots
$\DenseSpots$, and all edges incident with~$\HugeVertices'$. 
As  by~\eqref{eq:clubsuitCOND1}, $y$ has at most $\eta k/100$ neighbours in $\HugeVertices\setminus\HugeVertices'$, and since $y\in L'\subseteq
 \largevertices{9\eta/10}{k}{\Gcapt}$ and $y\notin\shadow(U,\frac{\eta
k}{100})$, we find that
\begin{align*}
\deg_{H}\left(y,V(G)\setminus
((U\cup U_i)\cup(\HugeVertices\setminus\HugeVertices'))\right)
& \ge\left(1+\frac{9\eta}{10}\right)k-\frac{\eta k}{1000}-|U_i|-\frac{\eta
k}{100}\\
& > k-|U_i|+\frac{\eta k}{2}\;.
\end{align*}
Therefore, one of the three following subcases must occur. (Recall that $y\not\in\smallatoms$ as $y\in C_0\in\clusters$.)
\begin{enumerate}
  \item[{\bf (V3a)}] $\deg_{\Gcapt}(y,\smallatoms\setminus (U\cup U_i))\ge \frac{\eta
  k}{6}$, 
  \item[{\bf (V3b)}] $\deg_{\Gblack}(y,\bigcup\clusters
 \setminus (U\cup U_{i}))\ge\frac{\eta k}{6}$, or
  \item[{\bf (V3c)}] $\deg_{\Gcapt}(y,\HugeVertices')\ge k-|U_i|+\frac{\eta
  k}{6}$.
\end{enumerate}
In case~{\bf (V3a)} we embed the components  of $T_{i+1}-r_{i+1}$ (as trees
rooted at the children of $r_{i+1}$) using the same technique as in
case~{\bf (V2)}, with Lemma~\ref{lem:embed:avoidingFOREST}.

\smallskip

\def\LER{\PARAMETERPASSING{L}{lem:embed:regular}}
In~{\bf (V3b)} we embed the components  of $T_{i+1}-r_{i+1}$ (as trees
rooted at the children of $r_{i+1}$). By
Fact~\ref{fact:clustersSeenByAvertex} there exists a cluster $D\in \clusters$ such that 
\begin{equation}\label{V3bsizeX*}
\deg_{\Gblack}(y,D\setminus (U\cup U_i))\ge
\frac{\eta k}{6}\cdot
\frac{\gamma^2\clustersize}{2(\Omega^*)^2k}>\frac{\gamma^2}2\clustersize.
\end{equation}
We use Lemma~\ref{lem:embed:regular} with input $\epsilon_\LER:=\epsilon'$, 
$\beta_\LER:=\gamma^2$, $C_\LER:=D$, $D_\LER:=C_0$, $X^*_\LER=X_\LER:=D\setminus (U\cup U_i)$ and $Y_\LER:=C_0\setminus (U\cup U_i\cup\{y\})$ 
to embed the tree $T_{i+1}$ into
the pair $(C_0,D)$, by embedding the components of $T_{i+1}-r_{i+1}$ one after
the other. The numerical conditions of Lemma~\ref{lem:embed:regular} hold because of Claim~\eqref{cl:TEC} and because of~\eqref{V3bsizeX*}.

\smallskip

In case~{\bf (V3c)} we set $T'_{i+1}:=r_{i+1}$ and define $\mathcal C_{1,i+1}$
as  set of all components of $T_{i+1}-r_{i+1}$. Then $\phi(\parent(\bigcup
\mathcal C_{1,i+1} )\cap V(T'_{i+1}))=\{y\}$ and
\begin{equation}\label{uiuiuiui}
 \deg_{\Gcapt}(y,\HugeVertices')\geq  k-|U_i|+\frac{\eta k}{6}\;.
\end{equation}

When all the trees $T_1,\ldots,T_p$ are processed, we define $T':=\{r\}\cup \bigcup\mathcal C_M\cup \bigcup_{i=1}^pT'_i$, and set $\mathcal C_1:=\bigcup_{i=1}^p\mathcal C_{1,i}$.
Thus also~\eqref{eq:boty} is satisfied by~\eqref{uiuiuiui} for $i=p$, since $|T'|=|U_p|$.
This finishes the proof of the lemma.
\end{proof}

It turns out that our techniques for embedding a tree $T\in\treeclass{k}$ for 
Configurations~$\mathbf{(\diamond2)}$--$\mathbf{(\diamond5)}$ are very similar.
In Lemma~\ref{lem:conf2-5} below we resolve these tasks at once. The proof of
Lemma~\ref{lem:conf2-5} follows the same basic strategy for each of the
configurations~$\mathbf{(\diamond2)}$--$\mathbf{(\diamond5)}$ and deviates only
in the elementary procedures of embedding shrubs of $T$.

\begin{lemma}\label{lem:conf2-5}
Suppose that we are in Setting~\ref{commonsetting}, and one of the following
configurations can be found in $G$:
\begin{enumerate}[a)]
  \item Configuration~$\mathbf{(\diamond2)}\left((\Omega^*)^2,
  5(\Omega^*)^9, \rho^3 \right)$,
  \item Configuration~$\mathbf{(\diamond3)}\left((\Omega^*)^2,
  5(\Omega^*)^9, \gamma/2, \gamma^3/100\right)$,
  \item Configuration~$\mathbf{(\diamond4)}\left((\Omega^*)^2,
  5(\Omega^*)^9, \gamma/2, \gamma^4/100\right)$, or
  \item Configuration~$\mathbf{(\diamond5)}\left((\Omega^*)^2,
  5(\Omega^*)^9,\epsilon', 2/(\Omega^*)^3,\frac{1}{(\Omega^*)^5}\right)$,
\end{enumerate}
Let $(T,r)$ be a rooted tree of order $k$ with a $(\tau k)$-fine partition
$(W_A,W_B,\shrubA,\shrubB)$. Then $T\subset G$.
\end{lemma}

\begin{proof} First observe that
each of the configurations given by a)--d) contains two sets $\HugeVertices''\subseteq \HugeVertices$ and $V_1\subseteq V(G)\sm\HugeVertices$ with
\begin{align}
\label{eq:sumC1}
\mindeg_{\Gcapt}(\HugeVertices'', V_1)&\ge 5(\Omega^*)^9k\; ,\\
\label{eq:sumC2}
\mindeg_{\Gcapt}(V_1,\HugeVertices'')&\ge \epsilon' k\; .
\end{align}

For any vertex $z\in W_A\cup W_B$ we
define $T(z)$ as the forest consisting of all components of $T-(W_A\cup
W_B)$ that contain children of $z$. 
Throughout the proof, we write $\phi$ for the current partial embedding of $T$ into $G$.

\paragraph{Overview of the embedding procedure.} As outlined in Section~\ref{ssec:EmbedOverview25} the embedding scheme is the same for Configurations~$\mathbf{(\diamond 2)}$--$\mathbf{(\diamond 5)}$. The embedding $\phi$ is
defined in two stages. In Stage~1, we embed $W_A\cup
W_B$, all the internal shrubs, all the end shrubs of $\shrubA$, and a part\footnote{in the sense that individual shrubs $\shrubB$ may be embedded only in part} of the end shrubs of $\shrubB$. In Stage~2 we embed the rest of $\shrubB$. Which part of $\shrubB$ are embedded in Stage~1 and which part in Stage~2 will be determined during Stage~1. We first give a rough outline of  both stages listing some conditions which
we require to be met, and then we describe each of the stages in detail.

Stage~1 is defined in $|W_A\cup\{r\}|$ steps. First we map $r$ to any vertex
in $\HugeVertices''$. Then in each step we pick a vertex $x\in
W_A$ for which the embedding $\phi$ has already been defined but such that $\phi$ is
not yet defined for any of the children of $x$. In this step we embed 
$T(x)$, together with all the children and grandchildren of $x$ in the knag which contains $x$. For each
$y\in W_B\cap \children(x)$, Lemma~\ref{lem:blueShrubSuspend} determines a subforest $T'(y)\subset T(y)$ 
which is embedded in Stage~1, and sets $\mathcal C_1 (y)$ and $\mathcal C_2 (y)$, which will be embedded in Stage~2. 

The
embedding in each step of Stage~1 will be defined so that the following properties hold. 
\begin{enumerate}[(*1)]
\item All vertices from $W_A$ are mapped to $\HugeVertices''$.
\item All vertices except for $W_A$ are mapped to
$V(G)\setminus \HugeVertices$.
\item For each $y\in W_B$, for each $v\in\parent (V(\bigcup \mathcal C_1 (y)))$   it holds that $$\deg_G(\phi(v),\HugeVertices')\ge
k + \tfrac {\eta k}{100} - v(T'(y))\;.$$
  \item For each $y\in W_B$, for each $v\in\parent (V(\bigcup \mathcal C_2 (y)))$  it holds that $$\deg_G(\phi(v),\HugeVertices')\ge
\tfrac k2 + \tfrac {\eta k}{100} - v(T'(y)\cup \bigcup\mathcal C_1(y))\;.$$
\end{enumerate}

In Stage~2, we shall utilize properties~(*3) and~(*4) to embed 
$T_B^*:=\bigcup\shrubB-\bigcup_{y\in W_B} T'(y)$. Stage~2 is substantially simpler than Stage~1; this is due to the fact that $T_B^*$ consists only of end shrubs. 

\paragraph{The embedding step of Stage~1.} The embedding
step is the same for
Configurations~$\mathbf{(\diamond2)}$--$\mathbf{(\diamond5)}$, except for
the embedding of internal shrubs. The order of the embedding steps is illustrated in Figure~\ref{fig:L825}.
\begin{figure}[t]
\centering 
\includegraphics{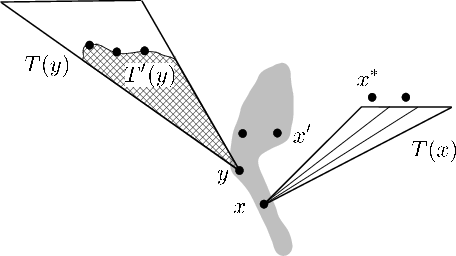}
\caption[Stage~1 of embedding in proof of Lemma~\ref{lem:conf2-5}]{Stage~1 of the embedding in the proof of Lemma~\ref{lem:conf2-5}. Starting from an already embedded vertex $x\in W_A$ we extend the embedding to (in this order)\\
(1) all the children $y\in W_B$ of $x$ in the same knag (in grey),\\
(2) a part $T'(y)$ of the forest $T(y)$,\\
(3) all the grandchildren $x'\in W_A$ of $x$ in the same knag,\\
(4) the forest $T(x)$ together with the bordering cut-vertices $x^*\in W_A$.
}
\label{fig:L825}
\end{figure}

In each step we have picked $x\in W_A$ already embedded in $G$ but such that none of $\children(x)$ are embedded. By (*1), or by the
choice of $\phi(r)$, we have $\phi(x)\in \HugeVertices''$. So by~\eqref{eq:sumC1} we have
\begin{equation}\label{elventarron}
 \deg_{\Gcapt}(\phi(x), V_1\sm U)\ge 5(\Omega^*)^9 k -k.
\end{equation}

First, we embed successively in $|W_B\cap \children (x)|$ steps the vertices
$y\in W_B\cap \children (x)$ together with  components $T'(y)\subset T(y)$ which will be determined on the way.
Suppose that in a certain step we are to embed $y\in W_B\cap \children (x)$
and the (to be determined) tree $T'(y)$. Let
$F:=\bigcup_{i=0}^2\shadow^{(i)}_{\Gcapt-\HugeVertices}(\ghost(U), \frac{\eta
k}{10^5})$, where $U$ is the set of vertices used by the embedding $\phi$ in previous steps, so $|U|\le k$. By Fact~\ref{fact:shadowbound}, $|F|\le \frac {10^{10}(\Omega^*)^2}{\eta^2}k$. We embed $y$
anywhere in $(\neighbor_G(\phi(x))\cap V_1)\setminus F$, cf.~\eqref{eq:sumC1}. Note that then (*2) holds for $y$. We use Lemma~\ref{lem:blueShrubSuspend} in
order to embed $T'(y)\subset T(y)$ (the subtree $T'(y)$ is determined by
Lemma~\ref{lem:blueShrubSuspend}). Lemma~\ref{lem:blueShrubSuspend} ensures
that~(*3) and~(*4) hold and that we have $\phi(V(T'(y)))\subseteq V(G)\setminus
\HugeVertices$. 

Also, we map the vertices $x'\in W_A\cap \children(y)$ to $\HugeVertices''\setminus U$. To justify this step, employing~(*2), it is enough to prove that 
\begin{equation}\label{eq:WNTP}
\deg(\phi(y),\HugeVertices'')\ge |W_A|\;.
\end{equation}
Indeed, on one hand, we have $|W_A|\le 336/\tau$ by Definition~\ref{ellfine}\eqref{few}. On
the other hand, we have that $\phi(y)\in V_1$, and thus~\eqref{eq:sumC2} applies.
 We can thus embed $x'$ as planned, ensuring (*1), and finishing the step for $y$.

Next, we sequentially embed the components $\tilde T$ of $T(x)$. In
the following, we describe such an embedding procedure only for an internal shrub
$\tilde T$, with $x^*$ denoting the other neighbour of
$\tilde T$ in $W_A$ (cf.~(*1)). The case when $\tilde T$ is an end shrub is analoguous: actually it is even easier as we do not 
have to worry about placing $x^*$ well.
The actual embedding of $\tilde T$ together with $x^*$ depends on the configuration we are in. We shall slightly abuse notation by letting $U$ now denote everything embedded before the tree $\tilde T$.

\smallskip

For Configuration~$\mathbf{(\diamond2)}$, we use
Lemma~\ref{lem:embed:greyFOREST} for one tree, namely $\tilde T-x^*$, using the following setting:
$Q_\PARAMETERPASSING{L}{lem:embed:greyFOREST}:= 1,
\gamma_\PARAMETERPASSING{L}{lem:embed:greyFOREST}:= \gamma,
\zeta_\PARAMETERPASSING{L}{lem:embed:greyFOREST}:= \rho^3,
H_\PARAMETERPASSING{L}{lem:embed:greyFOREST}:= \Gexp$,  $U_\PARAMETERPASSING{L}{lem:embed:greyFOREST}:= U$, and 
$U^*_\PARAMETERPASSING{L}{lem:embed:greyFOREST}:=(\neighbor_{\Gcapt}(\phi(x))\cap
V_1)\setminus U$ (this last set is large enough by~\eqref{elventarron}). The child of
$x$ gets embedded in $(\neighbor_{\Gcapt}(\phi(x))\cap V_1)\setminus U$, the vertices at odd distance from $x$ get embedded in $V_1$, and the vertices at even distance from $x$ get embedded in $V_2$. In particular, $\parent_T(x^*)$ gets embedded in $V_1$. After this, we accomodate $x^*$ in a vertex in $\HugeVertices''\setminus U$ which is adjacent to $\phi(\parent_T(x^*))$. This is possible by the same reasoning as in~\eqref{eq:WNTP}.

\smallskip

For Configuration~$\mathbf{(\diamond3)}$, we use
Lemma~\ref{lem:HE3} to embed $\tilde T$ with the setting
$\gamma_\PARAMETERPASSING{L}{lem:HE3}:= \gamma,
\delta_\PARAMETERPASSING{L}{lem:HE3}:= \gamma^3/100,
U_\PARAMETERPASSING{L}{lem:HE3}:= U$ and
$U^*_\PARAMETERPASSING{L}{lem:HE3}:=(\neighbor_{\Gcapt}(\phi(x))\cap V_1)\setminus U$ (this last set is large enough by~\eqref{elventarron}). 
Then the child of $x$ gets embedded in $(\neighbor_{\Gcapt}(\phi(x))\cap V_1)\setminus U$, vertices of $\tilde T$ of odd distance to $x$ (i.e.~of even distance to the root of $\tilde T$) get embedded in $V_1\setminus U$, and vertices of even distance get embedded in $V_2\setminus U$. We extend the embedding by mapping $x^*$ to a suitable vertex in $\HugeVertices''\setminus U$ adjacent to $\phi(\parent_T(x^*))$ in the same way as above.

\smallskip

For Configuration~$\mathbf{(\diamond4)}$, we use
Lemma~\ref{lem:HE4} to embed $\tilde T$ with the setting
$\gamma_\PARAMETERPASSING{L}{lem:HE4}:= \gamma,
\delta_\PARAMETERPASSING{L}{lem:HE4}:= \gamma^4/100,
U_\PARAMETERPASSING{L}{lem:HE4}:= U$ and
$U^*_\PARAMETERPASSING{L}{lem:HE4}:=(\neighbor_{\Gcapt}(\phi(x))\cap V_1)\setminus U$  (this last set is large enough by~\eqref{elventarron}). 
The fruit  $r'_\PARAMETERPASSING{L}{lem:HE4}$ in the lemma is chosen as $\parent_T(x^*)$, note that this is indeed a fruit (in $\tilde T$) because of Definition~\ref{ellfine}~\eqref{short}.
Then the child of $x$ gets embedded in
$(\neighbor_{\Gcapt}(\phi(x))\cap V_1)\setminus U$, the vertex $r'_\PARAMETERPASSING{L}{lem:HE4}=\parent_T(x^*)$ gets embedded in $V_1\setminus U$, and the rest of $\tilde T$ gets embedded in $(\smallatoms'\cup V_2)\setminus U$. This allows us to extend the embedding to $x^*$ as above.

\smallskip

In Configuration~$\mathbf{(\diamond5)}$, let $\mathbf W\subset \clusters$  denote the set of those clusters, which have at least an $\frac{1}{2(\Omega^*)^5}$-fraction of their vertices contained in the set $U':=U\cup \shadow_{\Gblack}(U,k/(\Omega^*)^3)$. We get from Fact~\ref{fact:shadowbound} that $|U'|\le 2(\Omega^*)^4 k$, and consequently $|U'\cup \bigcup \mathbf W|\le 4(\Omega^*)^9 k$. By~\eqref{elventarron} we can find a vertex $v\in (\neighbor_G(\phi(x))\cap V_1)\setminus (U'\cup \bigcup \mathbf W)$. 

We use the fact that $v\not \in \shadow_{\Gblack}(U,k/(\Omega^*)^3)$ together with inequality~\eqref{confi5last}
to see that $\deg_{\Gblack}(v,V(\Gblack)\sm U)\geq k/(\Omega^*)^3$.
Now, since there are only boundedly many clusters seen from $v$ (cf. Fact~\ref{fact:clustersSeenByAvertex}),  there must be a cluster $D\in \clusters$ such that 
\begin{equation}\label{eq:extendfromx}
 \deg_{\Gblack}(v,D\setminus U)\ge \frac{\gamma^2}{2\cdot(\Omega^*)^5}|D|\ge\gamma^3 |D|\;.
\end{equation}
 Let $C$ be the cluster containing $v$. We have $|(C\cap V_1)\setminus U|\ge \frac{1}{2(\Omega^*)^5}|C|\ge\gamma^3 |C|$ because of~\eqref{eq:diamond5P4} and since $C\notin\mathbf W$. Thus, by Fact~\ref{fact:BigSubpairsInRegularPairs}, $\big((C\cap V_1)\setminus U,D\setminus U\big)$ is an $2\epsilon'/\gamma^3$-regular pair of density at least $\gamma^2/2$. We can therefore embed $\tilde T$ in this pair using the regularity method. Moreover, by~\eqref{eq:extendfromx}, we can do so by mapping the child $z$ of $x$ to $v$. Thus the parent of $x^*$ (lying at even distance to $z$) will be embedded in $(C\cap V_1)\setminus U$. We can then extend our embedding to $x^*$ as above.

\smallskip

This finishes our embedding of $T(x)$. Note that
in all cases we have $\phi(x^*)\in \HugeVertices''$ and $\phi(V(\tilde T))\subseteq V(G)\setminus \HugeVertices$, as required by~(*1) and~(*2).

\paragraph{The embedding steps of Stage~2.}
For $i=1,2$, set $Z_i:=\bigcup_{y\in W_B} \children(T'(y))\cap \bigcup\mathcal C_i(y)$.

First, we embed all the vertices $z\in Z_2$ in $\HugeVertices'$. By~(*2), until now, only vertices of $W_A\cup Z_2$ are mapped to $\HugeVertices'$, and using~(*4) and the properties~\eqref{few}, \eqref{Bend} and~\eqref{Bsmall} of Definiton~\ref{ellfine}, we see that
\begin{align*}
\deg_G(\phi(\parent(z)),\HugeVertices')
&\geq
\frac{\eta k}{100} + (\frac k2 - \bigcup_{y\in W_B} (T'(y)\cup \bigcup\mathcal C_1(y))\\
&>|W_A|+ |Z_2|\;.
\end{align*}

So there is space for the vertex $z$ in $\HugeVertices'\cap
\phi(\neighbor_G(\parent(z)))$. 

Next, we embed all the vertices $z\in Z_1$ in $\HugeVertices'$. By~(*2), until now, only vertices of $W_A\cup Z_2\cup Z_1$ are mapped to $\HugeVertices'$, and by~(*3) we have, similarly as above,
 $$\deg_G(\phi(\parent(z)),
\HugeVertices')>|W_A|+|Z_2|+|Z_1|\;.$$ So $z$ can be embedded in $\HugeVertices'\cap
\neighbor_G(\phi(\parent(z)))$ as planned.

Finally,
 for $z\in Z_1\cup Z_2$, denote by $T_z$ the component of $\mathcal C_1\cup\mathcal C_2$
that contains $z$.  We use
Lemma~\ref{lem:embedC'endshrub} to embed the rest of the
rooted tree $(T_z, z)$. (Note that our parameters work because of~\eqref{eq:KONST}.) 
Once all rooted trees $(T_z, z)$, $z\in Z_1\cup Z_2$ have been processed, we have finished Stage~2 and thus the proof of the lemma.
\end{proof} 

\subsubsection{Embedding in Configurations $\mathbf{(\diamond6)}$--$\mathbf{(\diamond10)}$}\label{sssec:OrderedSkeleton}
We follow the schemes outlined in Sections~\ref{ssec:EmbedOverview67}, \ref{ssec:EmbedOverview8}, \ref{ssec:EmbedOverview9}, and~\ref{ssec:EmbedOverview10}.

Embedding a tree $T_\PARAMETERPASSING{T}{thm:main}\in\treeclass{k}$ using Configurations~$\mathbf{(\diamond6)}$, $\mathbf{(\diamond7)}$,
$\mathbf{(\diamond8)}$ has two parts: first the internal part of
$T_\PARAMETERPASSING{T}{thm:main}$ is embedded, and then this partial embedding is extended to end shrubs of $T_\PARAMETERPASSING{T}{thm:main}$ as well. Lemma~\ref{lem:embed:skeleton67} (for configurations $\mathbf{(\diamond6)}$ and $\mathbf{(\diamond7)}$) and Lemma~\ref{lem:embed:skeleton8} (for configuration $\mathbf{(\diamond8)}$) are used for the former part, and Lemmas~\ref{lem:embed:heart1} and~\ref{lem:embed:heart2} (depending on whether we have $\mathbf{(\heartsuit1)}$ or $\mathbf{(\heartsuit2)}$) for the latter. Lemma~\ref{lem:embed:total68} then puts these two pieces together.

Embedding using Configurations~$\mathbf{(\diamond9)}$ and
$\mathbf{(\diamond10)}$ is resolved in
Lemmas~\ref{lem:embed9} and~\ref{lem:embed10}, respectively.

\begin{lemma}\label{lem:embed:skeleton67}
Suppose we are in Setting~\ref{commonsetting} and~\ref{settingsplitting}, and we
have one of the following two configurations:
\begin{itemize}
 \item Configuration~$\mathbf{(\diamond6)}(\delta_6, \tilde\epsilon,d',\mu,1,0)$, or
 \item
 Configuration~$\mathbf{(\diamond7)}(\delta_7,\frac{\eta \gamma}{400},\tilde\epsilon,d',\mu,1,0)$,
\end{itemize}
with $10^5\sqrt\gamma(\Omega^*)^2\le \delta_6^4 \le 1$, 
$10^2\sqrt{\gamma}(\Omega^*)^3/\Lambda\le \delta_7^3<\eta^3\gamma^3/10^6$, $d'>10\tilde \epsilon>0$, and $d'\mu\tau k\ge 4\cdot 10^3$.
Both configurations contain distinguished sets $V_0,V_1\subseteq  \colouringp{0}$ and $V_2,V_3\subseteq  \colouringp{1}$.

Suppose that $(W_A,W_B,\shrubA,\shrubB)$ is a $(\tau k)$-fine partition of a rooted tree
$(T,r)$ of order at most~$k$ such that $|W_A\cup W_B|\leq k^{0.1}$. Let $T'$ be the tree induced by all the cut-vertices
$W_A\cup W_B$ and all the internal shrubs.

Then there exists an embedding $\phi$ of $T'$ such that $\phi(W_A)\subset V_1$,
$\phi(W_B)\subset V_0$, and $\phi(T'-(W_A\cup W_B))\subset \colouringp{1}$.
\end{lemma}
\begin{proof}
For simplicity, let us assume that $r\in W_A$. The case when $r\in W_B$ is similar. 
The $(\tau k)$-fine partition
$(W_A,W_B,\shrubA,\shrubB)$ induces a $(\tau k)$-fine partition in $T'$.
By
Lemma~\ref{lem:orderedskeleton}, the tree $T'$ has an ordered skeleton $(X_0, X_1,\ldots, X_m)$ where the $X_i$ are either shrubs or knags ($X_0$ being a knag).

Our strategy is as follows. We sequentially
embed the knags and the internal shrubs in the order given by the ordered
skeleton.  For embedding the knags we use Lemma~\ref{lem:embed:greyFOREST} in Preconfiguration~$\mathbf{\mathbf{(exp)}}$, and
Lemma~\ref{lem:embed:superregular} in Preconfiguration~$\mathbf{\mathbf{(reg)}}$. For embedding the internal shrubs, we use
Lemmas~\ref{lem:embedStoch:DIAMOND6} and~\ref{lem:embedStoch:DIAMOND7}
if we have Configurations~$\mathbf{\mathbf{(\diamond6)}}$, and~$\mathbf{\mathbf{(\diamond7)}}$, respectively.

Throughout, $\phi$
denotes the current (partial) embedding of
$(X_0,X_1\ldots,X_m)$.
In consecutive steps, we extend $\phi$.  We  define  auxiliary sets $D_i\subset V(G)$  which will serve for reserving space for the roots of the shrubs $X_i$. So the set $Z_{<i}:=\bigcup_{j<i} (\phi(X_{j})\cup D_{j})$ contains what is already used and what should (mainly) be avoided. 

Let $W_{A,i}:=W_A\cap V(X_i)$, and $W_{B,i}:=W_B\cap V(X_i)$. For each $y\in
W_{A,{j}}$ with $j\leq i$ let $$S_y:=(V_2\cap \neighbor_G(\phi(y)))\setminus
Z_{<i},$$ except if the latter set has size $>k$, in that case we choose a
subset of size $k$. This is a target set for the roots of shrubs adjacent to $y$.

Also, in the case $X_i$ is a shrub, we write $r_i$ for its root, and $f_i$ for the only other vertex neighbouring $W_A\cup W_B$.  Note that $f_i$ is a fruit of $(X_i,r_i)$.

The value $h=6$ or $h=7$ indicates whether we have configuration $\mathbf{\mathbf{(\diamond6)}}$ or $\mathbf{\mathbf{(\diamond7)}}$.
Define 
\begin{align}\label{eq:defFF}
 F_i:=\shadow_{G-\HugeVertices}\left(Z_{<i},\frac{\delta_h k}4\right)\cup Z_{<i}\;.
\end{align}

Define $U_i:=F_i$ if we have Preconfiguration $\mathbf{\mathbf{(exp)}}$ (note that in that case we have Configuration~$\mathbf{\mathbf{(\diamond6)}}$). To define $U_i$ in case of Preconfiguration $\mathbf{\mathbf{(reg)}}$ we make use of the super-regular pairs $(Q^{(j)}_0,Q^{(j)}_1)$ ($j\in\mathcal Y$). Set
\begin{equation}\label{eq:defUU}
U_i:=F_i\cup\bigcup\left\{Q^{(j)}_1\::\: j\in\mathcal Y, |Q^{(j)}_1\cap F_i|\ge \frac{|Q^{(j)}_1|}2\right\}\;.
\end{equation}
In either case, we have $|U_i|\le 2 |F_i|$.

Finally, set 
\begin{align}
 \label{eq:defWW}
W_i&:=\shadow_{G-\HugeVertices}\left(U_i,\frac{\delta_h k}2\right)\cup Z_{<i}\;.
\end{align}

We will now show how to embed successively all $X_i$. At each step $i$, our embedding $\phi$ will have the following properties:
\begin{enumerate}[(a)]
\item $\phi(W_{A,i})\subset V_1\sm F_i$ and $\phi(W_{B,i})\subset V_0$,\label{eatmoreseeds}
\item for each $y\in W_{A,{j}}$ with $j\leq i$ we have   $
|S_y\cap
  \phi(X_i)|\le |S_y\cap D_{i}|+k^{3/4}$,  \label{thereservationfortheroots}
\item $|Z_{< i+1}|\le 2k$,\label{sizeCiDi} 
\item\label{Didisjunkt} $D_i\subseteq \colouringp{1}\sm (\phi (X_i)\cup Z_{<i})$,
\item\label{Xifastdisjunkt} $\phi(X_i-r_i)$ is disjoint from $ \bigcup_{j<i}\cup D_j$,
\item $\phi(f_i)\in V_2\sm W_i$ if $X_i$ is a shrub,\label{eatmorefruit}
\item $\phi (X_i)\subset \colouringp{1}$ if $X_i$ is a shrub.\label{putitintocolour1}
\end{enumerate}

(We remark that since $r_i$ is not defined for knags $X_i$, condition~\eqref{Xifastdisjunkt} means that  $\phi(X_i)$ is disjoint from $ \bigcup_{j<i}\cup D_j$ for knags $X_i$.)

It is clear that the first together with the last condition ensures that in step $m$ we have found the desired embedding for $T'$.

Before we show how to embed each $X_i$ fulfilling the properties above, let us quickly calculate a useful bound.
By Fact~\ref{fact:shadowbound} and~\eqref{sizeCiDi}, we have that $|F_i|\le \frac{9\Omega^*}{\delta_h}k$
 for all $i\le m$. Thus,
using  $|U_i|\le 2 |F_i|$ and again Fact~\ref{fact:shadowbound} and~\eqref{sizeCiDi}, this shows
\begin{equation}\label{eq:Usmallish}
|W_i|\le \frac{38(\Omega^*)^2}{\delta_h^2}k\;.
\end{equation}

\smallskip

Now suppose we are at step $i$ with $0\leq i\leq m$. That is, we have already embedded all $X_j$ with $j<i$, and are about to embed $X_i$. 

\medskip

First assume that $X_i$ is a knag. Note that if $i\neq 0$, then there is exactly one fruit $f_\ell$ with $\ell<i$ which neighbours $X_i$. Set $N_i:=\neighbor_G(\phi(f_\ell))$ in this case, and let $N_i:=V(G)$ for $i=0$. 
We distinguish between the two preconfigurations we might be in. 

\smallskip

Suppose first we are in 
Preconfiguration~$\mathbf{\mathbf{(exp)}}$. Recall that then we are in Configuration~$\mathbf{(\diamond6)}$. 

We use Lemma~\ref{lem:embed:greyFOREST} to embed the single tree $X_i$  with the following setting: $\ell_\PARAMETERPASSING{L}{lem:embed:greyFOREST}:=1$,  $V_{1,\PARAMETERPASSING{L}{lem:embed:greyFOREST}}:=V_1$, $V_{2,\PARAMETERPASSING{L}{lem:embed:greyFOREST}}:=V_0$, $U^*_\PARAMETERPASSING{L}{lem:embed:greyFOREST}:=(N_i\cap V_1)\setminus U_i= (N_i\cap V_1)\setminus F_i$, $U_\PARAMETERPASSING{L}{lem:embed:greyFOREST}:=U_i=F_i$, $Q_\PARAMETERPASSING{L}{lem:embed:greyFOREST}:=\frac{18\Omega^*}{\delta_6}$, $\zeta_\PARAMETERPASSING{L}{lem:embed:greyFOREST}:=\delta_6$, and $\gamma_\PARAMETERPASSING{L}{lem:embed:greyFOREST}:=\gamma$. 
Note that $U^*_\PARAMETERPASSING{L}{lem:embed:greyFOREST}$ is large enough by~\eqref{eatmorefruit} for $\ell$ and by~\eqref{COND:D6:2} and~\eqref{COND:D7:2}, respectively.
Lemma~\ref{lem:embed:greyFOREST} gives an embedding of the tree $X_i$ such that $\phi(\Veven(X_i))\subset V_1\setminus F_i$ and $\phi(\Vodd(X_i))\subset V_0\setminus F_i$ , which maps the root of $X_i$ to the neighbourhood of its parent's image. Note that this ensures~\eqref{eatmoreseeds} and~\eqref{Xifastdisjunkt} for step $i$, and setting $D_i:=\emptyset$ we also ensure~\eqref{sizeCiDi} and~\eqref{Didisjunkt}. Property~\eqref{thereservationfortheroots} holds since $V_2\cap \phi (X_i)=\emptyset$. Since $X_i$ is a knag,~\eqref{eatmorefruit} and~\eqref{putitintocolour1} are empty.

\smallskip

Suppose now we are in 
Preconfiguration~$\mathbf{\mathbf{(reg)}}$. Then
let $j\in \mathcal Y$ be such that $(N_i\cap Q_1^{(j)})\setminus U_i\neq \emptyset$. Such an index $j$ exists by~\eqref{eatmorefruit} for $\ell$ and by~\eqref{COND:D6:2} and~\eqref{COND:D7:2}, respectively, if $i\neq 0$, and trivially if $i=0$. We shall use Lemma~\ref{lem:embed:superregular} to embed $X_i$ in $(Q^{(j)}_0,Q^{(j)}_1)$. More precisely, we use Lemma~\ref{lem:embed:superregular}
with $A_\PARAMETERPASSING{L}{lem:embed:superregular}:=Q^{(j)}_1$, $B_\PARAMETERPASSING{L}{lem:embed:superregular}:=Q^{(j)}_0$, $\epsilon_\PARAMETERPASSING{L}{lem:embed:superregular}:=\tilde \epsilon$, $d_\PARAMETERPASSING{L}{lem:embed:superregular}:=d'$, $\ell_\PARAMETERPASSING{L}{lem:embed:superregular}:=\mu k$, $U_A:=U_i\cap A$, $U_B:=\phi(W_{B,<i})\cap B$ (then $|U_A|\leq |A|/2$ by the definition of $U_i$ and the choice of $j$).

Lemma~\ref{lem:embed:superregular} yields a $(\Veven(X_i)\hookrightarrow
V_1\setminus F_i, \Vodd(X_i)\hookrightarrow V_0)$-embedding of $X_i$, which maps the root of $X_i$ to the neighbourhood of its parent's image. Setting $D_i:=\emptyset$, we have~\eqref{eatmoreseeds}--\eqref{putitintocolour1}.

\medskip

So let us now assume that $X_i$ is a shrub. The  parent $y$ of
the root $r_{i}$ of $X_{i}$ lies in $W_{A,\ell}$ for some $\ell<i$.  By~\eqref{eatmoreseeds} for $\ell$, we mapped $y$ to a vertex $\phi(y)\in V_1\sm F_\ell$.
 As $\deg_{G}(\phi(y),V_2)\ge \delta_h k$ (by~\eqref{COND:D6:1} and \eqref{COND:D7:1}, respectively), and since $\phi (y)\notin F_\ell$, we have 
 \begin{equation}\label{cerropochoco}
 |S_y|\ge \frac{3\delta_h k}4.
 \end{equation} 
 
 Using~\eqref{thereservationfortheroots} for all $j$ with $\ell\leq j<i$, and using that the sets $D_j$ are pairwise disjoint by~\eqref{Didisjunkt}, we see that
  $$|S_y\cap \phi(X_0\cup\ldots \cup X_{i-1})|=|S_y\cap \phi(X_\ell \cup\ldots
  \cup X_{i-1})|\le |S_y\cap \bigcup_{\ell\le j<i}D_j|+m\cdot k^{3/4}\le 
  |S_y\cap \bigcup_{0\le j<i}D_j| +m\cdot k^{3/4}. $$ Therefore, and as
  by~\eqref{Didisjunkt} and~\eqref{Xifastdisjunkt}, the sets $\phi(X_0\cup\ldots
  X_{i-1})$ and $\bigcup_{0\le j<i}D_j$ are disjoint except for the at most
  $m\le |W_A\cup W_B|\le k^{0.1}$ roots $r_j$ of shrubs $X_j$, and since
  $k\gg1$, we have $$|S_y|\geq |S_y\cap \phi(X_0\cup\ldots\cup X_{i-1})| + 
  |S_y\cap \bigcup_{0\le j<i}D_j| -m\ge 2|S_y\cap \phi(X_0\cup\ldots\cup
  X_{i-1})|-k^{0.9}. $$ Thus,
\begin{equation*}
|S_y\setminus \phi (X_0\cup\ldots\cup X_{i-1})|\ge \frac{|S_y|- k^{0.9}}{2}\overset{\eqref{cerropochoco}}\geq \frac{3\delta_h k}{8}-\frac{k^{0.9}}{2}>\frac{\delta_h k}{3}\;.
\end{equation*}

So for $U^*:=S_{y}\setminus \phi (X_0\cup\ldots\cup X_{i-1})$ we have that $|U^*|\ge \frac{\delta_h
k}{3}$.
 If we have Configuration~$\mathbf{(\diamond6)}$ or $\mathbf{(\diamond7)}$ we use 
Lemma~\ref{lem:embedStoch:DIAMOND6} or~\ref{lem:embedStoch:DIAMOND7},
respectively, with input 
$U_\PARAMETERPASSINGR{L}{lem:embedStoch:DIAMOND6}{lem:embedStoch:DIAMOND7}:=W_{i}$,
$U^*_\PARAMETERPASSINGR{L}{lem:embedStoch:DIAMOND6}{lem:embedStoch:DIAMOND7}:=U^*$,
$L_\PARAMETERPASSINGR{L}{lem:embedStoch:DIAMOND6}{lem:embedStoch:DIAMOND7}:=|W_{A,{i}}|$,
$\gamma_\PARAMETERPASSINGR{L}{lem:embedStoch:DIAMOND6}{lem:embedStoch:DIAMOND7}:=\gamma$,
the  family $\{P_t\}_\PARAMETERPASSINGR{L}{lem:embedStoch:DIAMOND6}{lem:embedStoch:DIAMOND7}:=\{S_y\}_{y\in
W_{A,{j}}, j<i}$, and the rooted tree $(X_{i},r_{i})$ with fruit $f_{i}$. Further,
for Configuration~$\mathbf{(\diamond6)}$, set $\delta_\PARAMETERPASSING{L}{lem:embedStoch:DIAMOND6}:=\delta_6$,  $V_{2,\PARAMETERPASSING{L}{lem:embedStoch:DIAMOND6}}:= V_2$ and $V_{3,\PARAMETERPASSING{L}{lem:embedStoch:DIAMOND6}}:=
V_3$ and for Configuration~$\mathbf{(\diamond7)}$, set $\delta_\PARAMETERPASSING{L}{lem:embedStoch:DIAMOND7}:=\delta_7$, $\ell_{\PARAMETERPASSING{L}{lem:embedStoch:DIAMOND7}}:=1$, $Y_{1,\PARAMETERPASSING{L}{lem:embedStoch:DIAMOND7}}:= V_2$ and $Y_{2,\PARAMETERPASSING{L}{lem:embedStoch:DIAMOND7}}:=
V_3$.
The  output of Lemma~\ref{lem:embedStoch:DIAMOND6} or~\ref{lem:embedStoch:DIAMOND7},
respectively, is the extension of our embedding $\phi$ to $X_{i}$, and
a set
$D_{i}:=C_\PARAMETERPASSINGR{L}{lem:embedStoch:DIAMOND6}{lem:embedStoch:DIAMOND7}\subseteq (V_2\cup V_3)\sm (W_i \cup \phi(X_i))$
for which properties~\eqref{eatmoreseeds} (which is empty) and  properties~\eqref{thereservationfortheroots}--\eqref{putitintocolour1} hold. 

\end{proof}
 
\begin{lemma}\label{lem:embed:skeleton8}
Suppose we are in Setting~\ref{commonsetting} and~\ref{settingsplitting} and suppose further we
have Configuration~$\mathbf{(\diamond8)}(\delta,
\frac{\eta\gamma}{400},\epsilon_1,\epsilon_2,d_1,d_2,\mu_1,\mu_2,h_1,0)$, with
$2\cdot 10^5(\Omega^*)^6/\Lambda\le\delta^6$,
$\delta<\gamma^2\eta^4/(10^{16}(\Omega^*)^2)$, $d_2>10\epsilon_2>0$,
$d_2\mu_2\tau k\ge 4\cdot 10^3$, and $\max\{\epsilon_1, \tau/\mu_1\}\le \eta^2\gamma^2d_1/(10^{10}(\Omega^*)^3)$. Recall that we have
distinguished sets $V_0, \ldots, V_4$ and a semiregular matching $\mathcal N$.

Let $(W_A,W_B,\shrubA,\shrubB)$ be a $(\tau k)$-fine partition of a rooted tree
$(T,r)$ of order at most~$k$. Let $T'$ be the tree induced by all the cut-vertices
$W_A\cup W_B$ and all the internal shrubs. Suppose that 
\begin{equation}\label{eq:takeaway}
v(T')<h_1-\frac{\eta^2 k}{10^5}\;.
\end{equation}

Then there exists an embedding $\phi$ of $T'$ such that $\phi(W_A)\subset V_1$,
$\phi(W_B)\subset V_0$, and $\phi(T')\subset \colouringp{0}\cup\colouringp{1}$.
\end{lemma}
\begin{proof}
We assume that $r\in W_A$. The case when $r\in W_B$ is similar. 

Let $\mathcal K$ be the set of all  knags of the  $(\tau k)$-fine partition  $(W_A,W_B,\shrubA,\shrubB)$ of $T$. For each such knag $K\in\mathcal K$ set  $Y_K:=K\cup \children_{T'}(K)$.
 We call the subgraphs $Y_K$ {\em extended knags}.
 Set $\mathcal Y:=\{Y_K:K\in\mathcal K\}$ and $W_C:=V(\bigcup\mathcal Y\sm \bigcup \mathcal K)$. Since $W_C\subseteq V(T')$, we clearly have that $|W_C|\leq|W_A\cup W_B|$.

Note that the forest $T'-\bigcup\mathcal Y$ consists of the set $\mathcal P$ of peripheral subshrubs of internal  shrubs of  $(W_A,W_B,\shrubA,\shrubB)$, and the set $\mathcal S$ of principal subshrubs of internal shrubs of  $(W_A,W_B,\shrubA,\shrubB)$.
It is not difficult to observe that there is a sequence $(X_0, X_1,\ldots, X_m)$ such that $X_i=(M_i, Y_i,\mathcal P_i)$ such that $M_i\in\mathcal S$ and $\mathcal P_i\subseteq\mathcal P$  for each $i\leq m$, and such that the following holds. 

\begin{enumerate}[(I)]
\item\label{startstartstart}  $M_0=\emptyset$ and $Y_0$ contains $r$.
\item\label{peugeot505goodbye}
 $\mathcal P_i$ are exactly those peripheral subshrubs whose parents lie in $Y_i$.
\item\label{esmeralda} The parent $f_i$ of $Y_i$ lies in $M_i$ (unless $i=0$).
\item\label{prince} The parent $r_i$ of $M_i$ lies in some $Y_j$ with $j<i$ (unless $i=0$),
\item $\bigcup_{i\leq m}V(M_i\cup Y_i\cup\bigcup \mathcal P_i)=V(T')$.
\end{enumerate}
See Figure~\ref{fig:subshrubs} for an illustration.
\begin{figure}[ht]
\centering 
\includegraphics{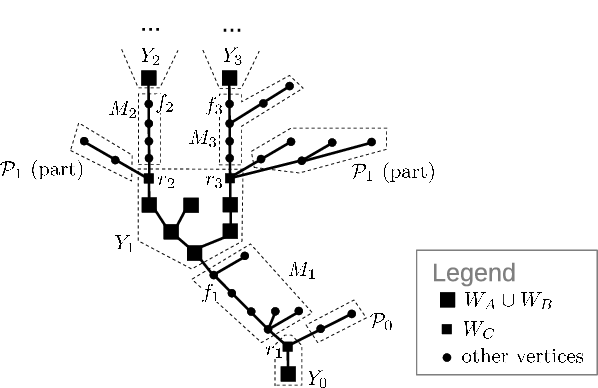}
\caption[Ordering the principal and peripheral subshrubs in
Lemma~\ref{lem:embed:skeleton8}]{An example of a sequence $(X_0,X_1,X_2,X_3,\ldots)$ in Lemma~\ref{lem:embed:skeleton8}.}
\label{fig:subshrubs}
\end{figure}

We now successively embed the elements of $X_i$, except possibly for a part of
the subshrubs in $\mathcal P_i$. The omitted peripheral subshrubs will be
embedded at the very end, after having completed the inductive procedure we are about to describe now.

We shall make use of the following lemmas:
 Lemma~\ref{lem:embed:superregular} (for embedding knags),
 Lemmas~\ref{lem:embed:BALANCED}  and~\ref{lem:embed:regular} (for embedding
 peripheral subshrubs in $\mathcal N$),  Lemma~\ref{lem:embedStoch:DIAMOND7}
 (for embedding principal subshrubs in $V_3\cup V_4$).

Throughout, $\phi$
denotes the current (partial) embedding of
$T'$. In each step $i$ we embed $M_i\cup Y_i$ and a subset of $\mathcal P_i$, and  denote by $\phi (X_i)$ the image of these sets (as far as it is defined). 
We also define an auxiliary set $D_i\subset V(G)$  which will serve to
ensure there is enough space for the roots of the subshrubs $M_\ell$ with
$\ell >i$.
Set $$Z_{<i}:=\bigcup_{j<i} (\phi(X_{j})\cup D_{j}).$$ 
Our plan for embedding the various parts of $X_i$ is depicted in Figure~\ref{fig:Embed8detailed}, which is
a refined version of Figure~\ref{fig:DIAMOND8}.
\begin{figure}[ht]
\centering 
\includegraphics{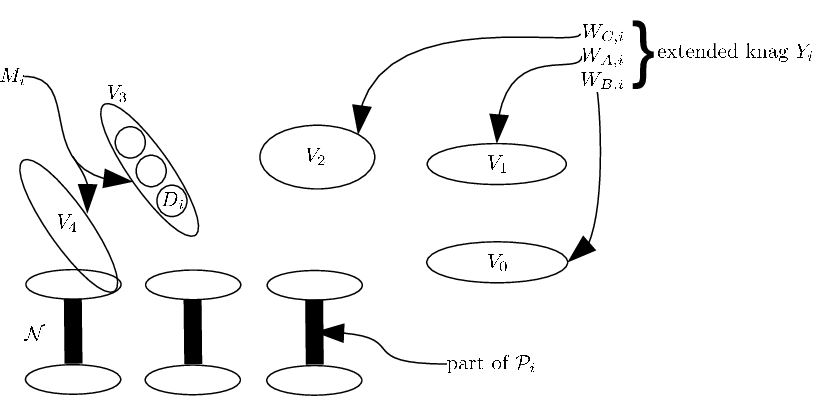}
\caption[Embedding the internal part of the tree in
Configuration~$\mathbf{(\diamond8)}$]{Embedding a part of the internal tree in
 Lemma~\ref{lem:embed:skeleton8}.}
\label{fig:Embed8detailed}
\end{figure} 

Let $W_{O,i}:=W_O\cap V(Y_i)$ for $O=A,B,C$. For each $y\in W_{C,{i}}$  let
$$S_y:=(V_3\cap \neighbor_G(\phi(y)))\setminus Z_{<i},$$ except if this set has
size more than $k$, in which case we choose any subset of size $k$. Similar as
in the preceding lemma, this is a target set for the roots of the
principal subshrub adjacent to $y$.

Fix a matching involution $\mathfrak{d}$ for $\mathcal N$, and 
 for $\ell=1,2$
define
\begin{align}\label{eq:defFF8}
 F^{(\ell)}_i:=Z_{<i}\cup \shadow^{(\ell)}_{G-\HugeVertices}\left(\ghost_{\mathfrak{d}}(Z_{<i}),\frac{\delta k}8\right)\;.
\end{align}

We  use the super-regular pairs $(Q^{(j)}_0,Q^{(j)}_1)$ ($j\in\mathcal Y$) to
define
\begin{equation}\label{eq:defUU8}
U_i:=F^{(2)}_i\cup\bigcup\left\{Q^{(j)}_1\::\: j\in\mathcal Y, |Q^{(j)}_1\cap F^{(2)}_i|\ge \frac{|Q^{(j)}_1|}2\right\}\;.
\end{equation}
We have 
\begin{equation}\label{Uileq2Fi}
|U_i|\le 2 |F^{(2)}_i|.
\end{equation}
Finally, for $\ell=1,2$ set 
\begin{align}
 \label{eq:defWW8}
W^{(\ell)}_i&:=\shadow^{(\ell)}_{G-\HugeVertices}\left(U_i,\frac{\delta k}2\right)\;.
\end{align}

We will now show how to define successively our embedding. At each step $i$, the embedding $\phi$ will be defined for $M_i\cup Y_i$ and a subset of $\mathcal P_i$, and it will have the following properties:
\begin{enumerate}[(a)]
\item $\phi(W_{A,i})\subset V_1\sm F^{(2)}_i$ and $\phi(W_{B,i})\subset V_0$,\label{eatmoreseeds8}
\item $\phi(W_{C,i})\subseteq V_2\sm F^{(1)}_i$,\label{wheretherootsgo}
\item $\phi(f_i)\in V_2\sm (F^{(1)}_i\cup W^{(1)}_i)$,\label{wherethefruitgoes}
\item for each $y\in W_{C,{j}}$ with $j\leq i$ we have   $
|S_y\cap  \phi(X_i)|\le |S_y\cap D_{i}|+ 2k^{3/4}$, 
\label{thereservationfortheroots8}
\item $|Z_{< i+1}|\le 2k$,\label{sizeCiDi8} 
\item\label{Didisjunkt8} $D_i\subseteq V_3\sm (\phi (X_i)\cup Z_{<i})$,
\item\label{Xifastdisjunkt8} $\phi(X_i\sm (V(M_i)\cap \children(W_C)))$ is
disjoint from $
\bigcup_{j<i} D_j $,\footnote{Note that $V(M_i)\cap \children(W_C)$ contains a
single vertex, the root of $M_i$.}
\item $\phi (X_i)\subset \colouringp{1} \cup \phi(Y_i \cup f_i)$,\label{putitintocolour18}
\item if $P\in \mathcal P_i$ is not embedded in step $i$ then for its parent $w\in W_C$ we have that $\deg_{\GD}(\phi(w),V_3)\ge h_1-|\phi(X_i)\cap V(\mathcal N)|-\frac{\eta^2
k}{10^6}$.\label{youknowi'mbad}
\end{enumerate}

Note that for~\eqref{putitintocolour18}, since $f_0$ is not defined, we assume $\phi (f_0)=\emptyset$.

Before we go on let us remark that~\eqref{putitintocolour18} together with~\eqref{Didisjunkt8} implies that at each step $i$ we have
\begin{equation}\label{eq:littlein0}
\left|Z_{<i}\cap\colouringp{0}\right|\le 3\cdot (|W_A|+|W_B|)\leBy{D\ref{ellfine}\eqref{few}} \frac{2016}\tau< \frac{\delta k}8\;.
\end{equation}

Also note that 
by Fact~\ref{fact:shadowbound} and by~\eqref{sizeCiDi8}, we have
\begin{equation}\label{eq:Fsmallish8}
|F^{(2)}_i|\le \frac{65(\Omega^*)^2}{\delta^2}k\;,
\end{equation}
and
\begin{equation}\label{eq:Usmallish8}
|W^{(2)}_i|\le \frac{520(\Omega^*)^4}{\delta^4}k\;.
\end{equation}

By~\eqref{wheretherootsgo} and by~\eqref{COND:D8:3} we have that $|S_y|\ge
\frac{7\delta k}8$. Now, using~\eqref{thereservationfortheroots8},~\eqref{Didisjunkt8} and~\eqref{Xifastdisjunkt8},  we can calculate similarly as in the previous lemma that at each step $i$ we have
\begin{equation}\label{eq:zindD8}
|S_y\setminus \bigcup_{\ell\le i}\phi(X_\ell)|\ge \frac{3\delta k}8\;.
\end{equation}

\medskip

Now assume we are at step $i$ of the inductive procedure, that is, we have already dealt with $X_0,\ldots , X_{i-1}$ and wish to embed (parts of) $X_i$.

\bigskip

 We start with embedding $M_i$, except if $i=0$, in that case we go directly to embedding $Y_0$. We shall embed $M_i$ in
$V_3\cup V_4$, except for the fruit $f_i$, which will be mapped to $V_2$.
The embedding has three stages. First we embed
$M_i-M_i(\uparrow f_i)$, then we embed $f_i$, and finally we embed the
forest $M_i(\uparrow f_i)-f_i$. The embedding of $M_i-M_i(\uparrow f_i)$ is an
application of Lemma~\ref{lem:embedStoch:DIAMOND7} analogous to the case of
Configuration~$\mathbf{(\diamond7)}$ in the previous
Lemma~\ref{lem:embed:skeleton67}. That is, set
$Y_{1,\PARAMETERPASSING{L}{lem:embedStoch:DIAMOND7}}:=V_3$,
$Y_{2,\PARAMETERPASSING{L}{lem:embedStoch:DIAMOND7}}:=V_4$, let
$$U_\PARAMETERPASSING{L}{lem:embedStoch:DIAMOND7}^*:=S_{r_i}\setminus
\bigcup_{\ell< i}\phi(X_i),$$ where $r_i$ 
 lies in $W_C$ by~\eqref{prince}, and
 $$U_\PARAMETERPASSING{L}{lem:embedStoch:DIAMOND7}:=F^{(2)}_i\cup W^{(2)}_i\;.$$ Note that
\[
|U_\PARAMETERPASSING{L}{lem:embedStoch:DIAMOND7}|\le
\frac{10^3(\Omega^*)^4}{\delta^4}k\le \frac{\delta\Lambda}{2\Omega^*}k,
\]
and by~\eqref{eq:zindD8} (which we use for $i-1$), also
\[
|U^*_\PARAMETERPASSING{L}{lem:embedStoch:DIAMOND7}|\ge
\frac{3\delta k}8.
\]
 The family
$\{P_1,\ldots,P_L\}_\PARAMETERPASSING{L}{lem:embedStoch:DIAMOND7}$ is $\{S_y\}_{y\in \bigcup_{j<i}W_{C,j}}$. There is only one tree to be embedded, namely 
$M_i-M_i(\uparrow f_i)$. It is not difficult to check that all the
conditions of Lemma~\ref{lem:embedStoch:DIAMOND7} are fulfilled.
Lemma~\ref{lem:embedStoch:DIAMOND7} gives an embedding of $M_i-M_i(\uparrow
f_i)$ in $V_3\cup V_4\subset \colouringp{1}$ with the property that
$\parent(f_i)$ is mapped to $V_3\setminus ( F^{(2)}_i \cup W^{(2)}_i )$. The lemma further gives a set
$D':=C_\PARAMETERPASSING{L}{lem:embedStoch:DIAMOND7}$ of size
$v(M_i-M_i(\uparrow f_i))$ such that $$|S_y\cap \phi(M_i-M_i(\uparrow f_i))|\le|S_y\cap
D'|+k^{0.75}$$ for each $y\in \bigcup_{j<i}W_{C,j}$.

Using the degree condition~\eqref{COND:D8:4} we can embed $f_i$
 to $$V_2\setminus ( F^{(1)}_i \cup W^{(1)}_i )$$
(recall that~\eqref{eq:littlein0} asserts that only very little space in $V_2$
is occupied). 
This ensures~\eqref{wherethefruitgoes} for $i$.

To embed $M_i(\uparrow f_i)-f_i$ we use again
Lemma~\ref{lem:embedStoch:DIAMOND7}. The parameters are this time
$Y_{1,\PARAMETERPASSING{L}{lem:embedStoch:DIAMOND7}}:=V_3$,
$Y_{2,,\PARAMETERPASSING{L}{lem:embedStoch:DIAMOND7}}:=V_4$,
\begin{align*}
U_\PARAMETERPASSING{L}{lem:embedStoch:DIAMOND7}^*&:=(\neighbor_G(\phi(f_i))\cap
V_3)\setminus (Z_{<i}\cup\phi(M_i-M_i(\uparrow f_i)) )\;\mbox{, and}\\
U_\PARAMETERPASSING{L}{lem:embedStoch:DIAMOND7}&:=Z_{<i}\cup
\phi(M_i-M_i(\uparrow f_i))\cup D'\;.
\end{align*}
Note that $|U_\PARAMETERPASSING{L}{lem:embedStoch:DIAMOND7}^*|\ge \frac{\delta k}4$
by~\eqref{COND:D8:3}, by the fact that $\phi(f_i)\not\in W^{(1)}_i$, and as $v(T_i)+i<\delta k/8$.
The family $\{P_1,\ldots,P_L\}_\PARAMETERPASSING{L}{lem:embedStoch:DIAMOND7}$ is $\{S_y\}_{y\in \bigcup_{j<i}W_{C,j}}$. The trees to be embedded are
the components of $M_i(\uparrow f_i)-f_i$ rooted at the children of
$f_i$.
All the conditions of Lemma~\ref{lem:embedStoch:DIAMOND7} are fulfilled. The
lemma provides an embedding in $V_3\cup V_4\subset \colouringp{1}$. It further
gives a set $D'':=C_\PARAMETERPASSING{L}{lem:embedStoch:DIAMOND7}$ of size
$v(M_i(\uparrow f_i))-1$ such that $$|S_y\cap \phi(M_i(\uparrow
f_i)-f_i)|\le|S_y\cap D''|+k^{0.75}$$ for each $y\in \bigcup_{j<i}W_{C,j}$.
Then $D_i:=V_3\cap(D'\cup D'')$ is such that for each $y\in \bigcup_{j<i}W_{C,j}$,
\begin{equation}\label{eq:greatlyduplicated}
|S_y\cap \phi(M_i)|\le  |S_y\cap D_i|+2k^{0.75}\;,
\end{equation}
as $S_y\subset V_3$ and $\phi(f_i)\notin V_3$. Note that this choice of $D_i$
also ensures~\eqref{sizeCiDi8} for $i$, and we have by the choices of
$U_\PARAMETERPASSING{L}{lem:embedStoch:DIAMOND7}^*$ and
$U_\PARAMETERPASSING{L}{lem:embedStoch:DIAMOND7}$ in both applications of
Lemma~\ref{lem:embedStoch:DIAMOND7} that
\begin{equation}\label{camundongo}
D_i\subseteq V_3\sm (\phi(M_i)\cup Z_{<i})\quad\text{ and }\quad
\phi(X_i\sm (V(M_i)\cap \children(W_C)))\cap
\bigcup_{j<i}D_j=\emptyset.\,
\end{equation}

\bigskip

We now turn to embedding $Y_i$.
 Our plan is to use first Lemma~\ref{lem:embed:superregular} to embed
$Y_i\sm W_C$ in $(Q^{(j)}_0,Q^{(j)}_1)$ for an appropriate index $j$. After
that, we shall show how to embed $W_{C,i}$.

If $i=0$ then take an arbitrary $j\in\mathcal Y$. Otherwise note that by~\eqref{esmeralda}, the parent $f_i$ of the root of $Y_i$ lies in $M_i$. Note that $f_i$ is a fruit in $M_i$. 
Let $j\in \mathcal Y$ be such that $(\neighbor_G(\phi(f_{i}))\cap
Q_1^{(j)})\setminus U_i\neq \emptyset$. 
 Such an index $j$ exists by~\eqref{COND:D8:2}
 and the fact that $
 \phi(f_{i})\not\in W^{(1)}_{i}$ by~\eqref{wherethefruitgoes} for $i$.
 
 We  use
Lemma~\ref{lem:embed:superregular} with
$A_\PARAMETERPASSING{L}{lem:embed:superregular}:=Q^{(j)}_1$,
$B_\PARAMETERPASSING{L}{lem:embed:superregular}:=Q^{(j)}_0$,
$\epsilon_\PARAMETERPASSING{L}{lem:embed:superregular}:= \epsilon_2$,
$d_\PARAMETERPASSING{L}{lem:embed:superregular}:=d_2$,
$\ell_\PARAMETERPASSING{L}{lem:embed:superregular}:=\mu_2 k$, $U_A:=U_i\cap
A_\PARAMETERPASSING{L}{lem:embed:superregular}$, $U_B:=Z_{<i}\cap
B_\PARAMETERPASSING{L}{lem:embed:superregular}$.  By the choice of $j$ and the
definition of $U_i$, we find that $U_A$ is small enough, and
using~\eqref{eq:littlein0} we see that $U_B$ is also small enough.
Lemma~\ref{lem:embed:superregular} yields a $(\Veven(Y_i-
W_C)\hookrightarrow V_1\setminus F^{(2)}_i, \Vodd(Y_i-
W_C)\hookrightarrow V_0)$-embedding of $Y_i-
W_C$. We clearly see
condition~\eqref{eatmoreseeds8} satisfied for $i$.

We now embed successively the vertices of the set
$W_{C,i}=\{w_\ell:\ell=1,\ldots ,|W_{C,i}|\}$.
By the definition of the set $W_C$, we know that the parent $x$ of~$w_\ell$ lies in $W_{A,i}$.
Combining~\eqref{COND:D8:1} with the fact that $\phi(x)\in V_1\sm F^{(2)}_i$ by~\eqref{eatmoreseeds8} for $i$, we have
that 
$$\left|\neighbor_G\left(\phi(x),V_2\setminus (F^{(1)}_i \sm Z_{<i})\right)\right|\ge \frac{7\delta k}8\;.$$ Thus
by~\eqref{eq:littlein0} and since $V_2\subseteq \colouringp{0}$, we can
accommodate $w_\ell$ in $V_2\setminus F^{(1)}_i$. This is as
desired for~\eqref{wheretherootsgo} in step $i$.

\bigskip

We now turn to $\mathcal P_i$. We will embed a subset of these peripheral subshrubs in
 $\mathcal N$. This procedure is  divided into two stages. First we shall aim to embed as many
subshrubs as possible in $\mathcal N$ in a balanced way, with the help of
Lemma~\ref{lem:embed:BALANCED}. When it is no longer
possible to embed any subshrub in a balanced way in $\mathcal N$, we  embed  in $\mathcal N$ as many of the leftover
subshrubs as possible, in an unbalanced way. For this part of the embedding we use
Lemma~\ref{lem:embed:regular}. 

 By~\eqref{peugeot505goodbye} all the parents of the subshrubs in $\mathcal
 P_i$ lie in $W_{C,i}$. For $w_\ell\in W_{C,i}$, let $\mathcal P_{i,\ell}$
 denote the set of all subshrubs in $\mathcal P_i$ adjacent to $w_\ell$. In the
 first stage, we shall embed, successively for $j=1, \ldots ,|W_{C,i}|$, either
 all or none of $\mathcal P_{i,j}$ in a balanced way in $\mathcal N$. Assume inductively that
\begin{equation}\label{eq:INDbalanced}
\phi\Big(\bigcup_{p<j}\mathcal P_{i,p}\Big) \mbox{ is $(\tau k)$-balanced with respect to $\mathcal N$.} 
\end{equation}

Construct a
semiregular matching $\mathcal N_j$ absorbed by $\mathcal N$ as follows. Let 
$\mathcal N_j:=\{(X_1',X_2')\::\: (X_1,X_2)\in\mathcal N\}$, where for
$(X_1,X_2)\in\mathcal N$ we define $(X'_1,X_2')$ as the maximal balanced
unoccupied subpair seen from $\phi(w_{j})$, i.e., for $b=1,2$, we take
$$X'_b\subset \left(X_b\cap \neighbor_{\Gblack}(\phi(w_{j})\right)\setminus
\left(\phi(\bigcup_{p<j}\mathcal P_{i,p})\cup \bigcup_{\ell <i}\phi (X_\ell)\right)$$ maximal subject to $|X'_1|=|X'_2|$.
If $|V(\mathcal N_j)|\ge \frac{\eta^2 k}{10^7\Omega^*}$ then we shall embed $\mathcal P_{i,j}$, otherwise we do not embed  $\mathcal P_{i,j}$ in this step. 
So assume we decided to embed $\mathcal P_{i,j}$. Recall that the total order of the subshrubs in this set
 is at most $\tau k$.
Using the same argument as for Claim~\ref{cl:Megdes} we have 
$$\left|\bigcup\{X\cup Y\::\:(X,Y)\in\mathcal
N,\deg_{\GD}(\phi(w_{j}),X\cup
Y)>0\right|\le\frac{4(\Omega^*)^2}{\gamma^2}k\;.$$ Thus, there exists a subpair
$(X_1',X_2')\in \mathcal N_j$ of some $(X_1,X_2)\in \mathcal N$ with
\begin{equation}\label{eq:SIMcalc}
\frac{|X_1'|}{|X_1|}\ge
\frac{\tfrac{\eta^2}{10^7\Omega^*}k}{\tfrac{4(\Omega^*)^2}{\gamma^2}k}\ge \frac{\gamma^2 \eta^2}{10^8 (\Omega^*)^3}\;.
\end{equation}
In particular, $(X_1',X_2')$ forms a
$\frac{2\cdot 10^8\epsilon_1(\Omega^*)^3}{\gamma^2\eta^2}$-regular pair of
density at least $d_1/2$ by Fact~\ref{fact:BigSubpairsInRegularPairs}.
We use Lemma~\ref{lem:embed:BALANCED} to embed $\mathcal P_{i,j}$ in $\M_\PARAMETERPASSING{L}{lem:embed:BALANCED}:=\{(X_1',X_2')\}$. The family
$\{f_{CD}\}_\PARAMETERPASSING{L}{lem:embed:BALANCED}$ comprises of a single
number $f_{(X_1',X_2')}$ which is the discrepancy of $\bigcup_{p<j}\phi (\mathcal P_{i,p})$ with respect to $(X_1,X_2)$. This guarantees
that~\eqref{eq:INDbalanced} is preserved. This finishes the $j$-th step.
We repeat this step until $j=|W_{C,i}|$, then we go to the next stage.

Denote by $\mathcal Q_i$ the set of all $P\in \mathcal P_{i}$ that have not been embedded in the first stage. Note that for each $Q\in \mathcal Q_i$, with $Q\in \mathcal P_{i,j}$, say, and for each $(X_1,X_2)\in\mathcal N$ there is a $b_{(X_1,X_2)}\in\{1,2\}$ such that for
\[
O_j:=\bigcup_{(X_1,X_2)\in\mathcal N}\left(X_{b_{(X_1,X_2)}}\cap \neighbor_{\Gblack}(\phi(w_{j})\right)\setminus
\left(\phi(\bigcup_{p<j}\mathcal P_{i,p})\cup \bigcup_{\ell <i}\phi (X_\ell)\right)
\]
we have that
\begin{equation}\label{boli}
|O_j|<\frac{\eta^2 k}{10^7\Omega^*}.
\end{equation}
The fact that $O_j$ is small implies that there an $\mathcal N$-cover such that
the $\Gblack$-neighborhood of $w_j$ restricted to this cover is essentially
exhausted by the image of $T'$.

\medskip
In the second stage, we shall embed some of the peripheral subshrubs of
$\mathcal Q_i$. They will be mapped in an unbalanced way to $\mathcal
N$. We will do this in steps $j=1, \ldots ,|W_{C,i}|$, and denote by
$\mathcal R_j$ the set of those $\mathcal P\subseteq \mathcal Q_i$ embedded
until step $j$. At step $j$, we decide to embed $\mathcal P_{i,j}$ if $\mathcal
P_{i,j}\subseteq \mathcal Q_i$ and

\begin{equation}\label{lacumparsita}
\deg_{\Gblack}\Big(\phi(w_{j}),V(\mathcal N)\setminus \phi(\bigcup\mathcal P_{i}\sm \mathcal Q_i)\Big) - |\bigcup\mathcal R_{ j-1}|\ge\frac{\eta^2 k}{10^6}\;.
\end{equation}

Let $$\tilde{\mathcal N}:=\left\{(X,Y)\in\mathcal N\::\: |(X\cup Y)\cap
Z_{<i}|<\frac{\gamma^2\eta^2}{10^9(\Omega^*)^2}|X|\right\}\;.$$ 

As by~\eqref{wheretherootsgo} we know that $w_{j}$
was embedded in $V_2\sm F^{(1)}_i$, we have
\begin{equation}\label{wasverlorengeht}
\deg_{\Gblack}(\phi(w_{j}),V(\mathcal N\sm \tilde{\mathcal N}))
\le
\frac{2\cdot 10^9 (\Omega^*)^2}{\gamma^2\eta^2} \cdot\frac{\delta k}{8}
\le 
\frac{\eta^2}{10^7} k\;.
\end{equation}

Using~\eqref{boli},~\eqref{lacumparsita} and~\eqref{wasverlorengeht}, similar calculations as in~\eqref{eq:SIMcalc} show the existence of a pair
$(X,Y)\in\tilde{\mathcal N}$  with
$$\deg_{\Gblack}(\phi(w_{j}),(X\cup Y)\setminus ( O_j\cup \phi(\bigcup\mathcal P_{i}\sm \mathcal Q_i))) - |(X\cup Y)\cap \phi(\bigcup\mathcal R_{ j-1})|\ge
\frac{\gamma^2\eta^2}{10^8(\Omega^*)^2}|X\cup Y|\;.$$

Then by the definition
of $\tilde{\mathcal N}$, and setting $Z_{<i}^+:=\ghost_{\mathfrak d}(Z_{<i})$ we get that
 \begin{align*}
 \deg_{\Gblack} & (\phi(w_{j}), (X\cup Y)\setminus (Z_{<i}^+\cup O_j \cup \phi(\bigcup\mathcal P_{i}\sm \mathcal Q_i))) - |(X\cup Y)\cap \phi(\bigcup\mathcal R_{ j-1})|\\ &
 \ge \frac{\gamma^2\eta^2}{10^9(\Omega^*)^2}|X\cup Y|. 
\end{align*}
 By the definition of $O_j$, all of the degree counted here goes to one side of the matching edge $(X,Y)$, say to $X$. So
 \begin{align}
 \deg_{\Gblack}(\phi(w_{j}),X\setminus (Z_{<i}^+\cup \phi(\bigcup\mathcal
 P_{i}\sm \mathcal Q_i\cup \bigcup\mathcal R_{j-1}))) - | Y\cap \phi(\bigcup\mathcal R_{ j-1})| & \ge 
 \frac{\gamma^2\eta^2}{10^9(\Omega^*)^2}|X|\label{aha1}
 \\ & \ge 
12\frac{\eps_1}{d_1}|X|+\tau k.\label{aha2}
\end{align}
We claim that furthermore,
\begin{equation}
|Y\setminus (Z_{<i}^+\cup\phi(\bigcup\mathcal P_{i}\sm \mathcal Q_i\cup \bigcup\mathcal R_{j-1}))|
\ge
\frac{\gamma^2\eta^2}{ 10^{10}(\Omega^*)^2}|Y|
\ge
12\frac{\eps_1}{d_1}|Y|+\tau k\;.\label{jajaja}
\end{equation}
Indeed, otherwise we get by~\eqref{aha1} that
\[
|X\setminus (Z^+_{<i}\cup\phi(\bigcup\mathcal P_{i}\sm \mathcal Q_i))|
>
 |Y\setminus (Z^+_{<i}\cup\phi(\bigcup\mathcal P_{i}\sm \mathcal Q_i))| + \frac{\gamma^2\eta^2}{ 10^{10}(\Omega^*)^2}|X|,
\]
which  is impossible by~\eqref{eq:INDbalanced} and since $|X|\geq \mu_1 k$.

Hence, by~\eqref{aha2} and~\eqref{jajaja}, we can  embed $\mathcal P_{i,j}$ into the unoccupied part $(X,Y)$ using
Lemma~\ref{lem:embed:regular} repeatedly.\footnote{Recall that the
total order of $\mathcal P_{i,j}$ is at most $\tau k$.}

\medskip

Note that  if some $\mathcal
P_{i,j}$ has not been embedded in either of the two stages,  then the vertex
$w_{j}$ must have a somewhat insufficient degree in $\mathcal N$. More precisely, employing~\eqref{lacumparsita} we see that
$\deg_{\Gblack}(\phi(w_{j}),V(\mathcal N))-|\phi(X_i)\cap V(\mathcal
N)|<\frac{\eta^2 k}{10^6}$. Combined with~\eqref{COND:D8:7}, we find that $$\deg_{\GD}(\phi(w_j),V_3)\ge h_1-|\phi(X_i)\cap V(\mathcal N)|-\frac{\eta^2
k}{10^6}\;,$$
in other words, \eqref{youknowi'mbad} holds for $i$. 

This finishes step $i$ of the embedding procedure. 
Recall that the sets $V_3$ and $V(\mathcal N)$ are
disjoint. Hence, by~\eqref{eatmoreseeds8} and~\eqref{wheretherootsgo}, the
principal subshrubs are the only parts of $T'$ that were embedded in $V_3$ (and
possibly elsewhere).
Thus, using~\eqref{camundongo}, we see that~\eqref{Didisjunkt8},~\eqref{Xifastdisjunkt8}  and~\eqref{putitintocolour18} are satisfied for $i$. Also, by~\eqref{eq:greatlyduplicated},~\eqref{thereservationfortheroots8} holds for $i$.

\bigskip

After having completed the inductive procedure, we still have to embed some
peripheral subshrubs.
Let us take sequentially those $P\in\mathcal P$  which  were not embedded. Say $w$ is the parent of $P$.
By~\eqref{youknowi'mbad} we have 
$$\deg_{\GD}(\phi(w),V_3\setminus \textrm{Im}(\phi))\ge
h_1-| \textrm{Im}(\phi)\cap V(\mathcal N)|-| \textrm{Im}(\phi)\cap V_3|-\frac{\eta^2
k}{10^6}\geByRef{eq:takeaway} \frac{\eta^2 k}{10^6}\;.$$
An application of Lemma~\ref{lem:embedStoch:DIAMOND7} in which
$Y_{1,\PARAMETERPASSING{L}{lem:embedStoch:DIAMOND7}}:=V_3$,
$Y_{2,\PARAMETERPASSING{L}{lem:embedStoch:DIAMOND7}}:=V_4$,
$U_\PARAMETERPASSING{L}{lem:embedStoch:DIAMOND7}:=\textrm{Im}(\phi)$, $U^*_\PARAMETERPASSING{L}{lem:embedStoch:DIAMOND7}:=\neighbor_{\GD}(\phi(w))\cap
V_3\setminus \textrm{Im}(\phi)$, and
$\{P_1,\ldots,P_L\}_\PARAMETERPASSING{L}{lem:embedStoch:DIAMOND7}:=\emptyset$
gives an embedding of $P$ in $V_3\cup V_4\subset \colouringp{1}$.

By conditions~\eqref{eatmoreseeds8},~\eqref{wheretherootsgo},~\eqref{wherethefruitgoes} and~\eqref{putitintocolour18}  we have thus found the desired embedding for $T'$.
\end{proof}

\begin{lemma}\label{lem:embed:heart1}\HAPPY{D}
Suppose we are in Setting~\ref{commonsetting} and~\ref{settingsplitting}, and
that the sets $V_0$ and $V_1$ witness
Preconfiguration~$\mathbf{(\heartsuit1)}(2\eta^3 k/10^3,h)$. Suppose that
$U\subset \colouringp{0}\cup\colouringp{1}$. Suppose that
$\{x_j\}_{j=1}^\ell\subset V_0$ and $\{y_j\}_{j=1}^{\ell'}\subset V_1$ are sets of 
mutually distinct vertices. Let $\{(T_j,r_j)\}_{j=1}^\ell$ and
$\{(T'_j,r'_j)\}_{j=1}^{\ell'}$ be families of rooted trees such that each
component of $T_j-r_j$ and of $T'_j-r'_j$  has order at most $\tau k$.

If 
\begin{align}\label{eq:MALYSTROM1}
\sum_j v(T_j)&\le \frac h2-\frac{\eta^2 k}{1000}\;,\\
\label{eq:MALYSTROM2}
\sum_j v(T_j)+\sum_j v(T'_j)&\le h-\frac{\eta^2 k}{1000}\;,\mbox{and}\\
\label{eq:MALYSTROM3}
|U|+\sum_j v(T_j)+\sum_j v(T'_j)&\le k\;,
\end{align}
 then there exist  $(r_j\hookrightarrow x_j, V(T_j)\setminus\{r_j\}\hookrightarrow V(G)\setminus U)$-embeddings of $T_j$ and $(r'_j\hookrightarrow y_j, V(T'_j)\setminus\{r'_j\}\hookrightarrow V(G)\setminus U)$-embeddings of $T'_j$ in $G$, all mutually disjoint.
\end{lemma}
\begin{proof}
The embedding has three stages. In Stage~I we embed some components of $T_j-r_j$
(for all $j=1,\ldots, \ell$) in the parts of $(\M_A\cup\M_B)$-edges which are
``seen in a balanced way from $x_j$''. In Stage~II we embed the remaining
components of $T_j-r_j$. Last, in Stage~III we embed all the components
$T'_j-r'_j$ (for all $j=1,\ldots, \ell'$).

Let us first give a bound on the total size of $(\M_A\cup\M_B)$-vertices $C\in\V(\M_A\cup\M_B)$, $C\subset \bigcup \clusters$ seen from a given vertex via edges of $\GD$. This bound will be used repeatedly.
\begin{claim}\label{cl:PROCH}
Let $v\in V(G)$. Then for  $\mathcal U:=\{C\in\V(\M_A\cup \M_B)\::\: C\subset\bigcup \clusters, \deg_{\GD}(x, C)>0\}$ we have 
\begin{align}\label{eq:PROCH1}
 |\bigcup \mathcal U|&\le  \frac{2(\Omega^*)^2 k}{\gamma^2}\;,\mbox{and}\\
\label{eq:PROCH2}
|\mathcal U|&\le  \frac{2(\Omega^*)^2k}{\gamma^2\pi\clustersize}\;.
\end{align}
\end{claim}
\begin{proof}[Proof of Claim~\ref{cl:PROCH}]
Let $\mathbf U\subset \clusters$ be the set of  those clusters which intersect $\neighbor_{\GD}(x_j)$. Using the same argument as in the proof of Claim~\ref{cl:Megdes} we get that $|\bigcup\mathbf U|\le \frac{2(\Omega^*)^2 k}{\gamma^2}$, i.e.~\eqref{eq:PROCH1} holds. Also,~\eqref{eq:PROCH2} follows since $\M_A\cup\M_B$ is $(\epsilon,d,\pi\clustersize)$-semiregular.
\end{proof}

\medskip\noindent\underline{Stage I:} We proceed inductively for $j=1,\ldots, \ell$. Suppose that we embedded some components $\mathcal F_1,\ldots, \mathcal F_{j-1}$ of the forests $T_1-r_1,\ldots,T_{j-1}-r_{j-1}$. We write $F_{j-1}$ for the partial images of this embedding. We inductively assume that
\begin{equation}\label{eq:Fjbal}
\mbox{$F_{j-1}$ is $\tau k$-balanced w.r.t.\ $\M_A\cup\M_B$.} 
\end{equation}

For each $(A,B)\in\M_A\cup \M_B$ with $\deg_{\GD}(x_j,(A\cup B)\setminus \smallatoms)>0$ take a subpair $(A',B')$,
$$A'\subset(A\cap\neighbor_{\GD\cup\Gcapt}(x_j))\colouringpI{2}\setminus F_{j-1} \quad\mbox{and}\quad B'\subset(B\cap\neighbor_{\GD\cup\Gcapt}(x_j))\colouringpI{2}\setminus F_{j-1}\;,$$
such that 
$$|A'|=|B'|=\min\big\{ |(A\cap\neighbor_{\GD\cup\Gcapt}(x_j))\colouringpI{2}\setminus F_{j-1}|, |(B\cap\neighbor_{\GD\cup\Gcapt}(x_j))\colouringpI{2}\setminus F_{j-1}|\big\}\;.$$
These pairs comprise a semiregular matching $\mathcal N_j$. (Pairs
$(A,B)\in\M_A\cup \M_B$ with $\deg_{\GD}(x_j,(A\cup B)\setminus
\smallatoms)=0$ are not considered for the construction of $\mathcal N_j$.) 

Let
$\M_j:=\{(A',B')\in\mathcal N_j\::\: |A'|>\alpha |A|\}$, for 
 $$\alpha:=\frac{\eta^3\gamma^2}{10^{10}(\Omega^*)^2}.$$
  By
Fact~\ref{fact:BigSubpairsInRegularPairs} $\M_j$ is a $(2\epsilon/\alpha, d/2,
\alpha\pi \clustersize)$-semiregular matching.
\begin{claim}\label{cl:MjVSNj}
We have $|V(\mathcal M_j)|\ge |V(\mathcal N_j)|-\frac{\eta^3 k}{10^9}$. 
\end{claim}
\begin{proof}[Proof of Claim~\ref{cl:MjVSNj}]
By~\eqref{eq:PROCH1}, and by Property~\ref{commonsetting3} of
Setting~\ref{commonsetting}, we have $|V(\mathcal M_j)|\ge |V(\mathcal N_j)|-\alpha\cdot 2\cdot \frac{2(\Omega^*)^2 k}{\gamma^2}$.
\end{proof}
Let $\mathcal F_j$ be a maximal set of components of $T_j-r_j$ such that
\begin{equation}\label{eq:maxFj}
v(\mathcal F_j)\le |V(\M_j)|-\frac{\eta^3 k}{10^9}\;.
\end{equation}
Observe that if $\mathcal F_j$ does not contain all the components of $T_j-r_j$ then
\begin{equation}\label{eq:natrave}
v(\mathcal F_j)\ge |V(\M_j)|-\frac{\eta^3 k}{10^9}-\tau k\ge |V(\M_j)|-\frac{2\eta^3 k}{10^9}\;.
\end{equation}

Lemma~\ref{lem:embed:BALANCED} yields an embedding of $\mathcal F_j$ in $\M_j$. Further the lemma together with the induction hypothesis~\eqref{eq:Fjbal} guarantees that the embedding can be chosen so that the new image set $F_{j}$ is $\tau k$-balanced w.r.t.\ $\M_A\cup\M_B$. We fix this embedding, thus ensuring~\eqref{eq:Fjbal} for step $i$. If $\mathcal F_j$ does not contain all the components of $T_j-r_j$ then~\eqref{eq:natrave} gives
\begin{equation}\label{eq:OnLawn}
|V(\M_j)\setminus F_j| \le \frac{2\eta^3 k}{10^9}\;.
\end{equation}

\medskip\noindent\underline{After Stage I:} Let $\mathcal N^*$ be a maximal
semiregular matching contained in $(\M_A\cup\M_B)\colouringpI{2}$ which avoids
$F_\ell$. We need two auxiliary claims.
\begin{claim}\label{cl:ztratimeMaloW}
We have 
$$\maxdeg_{\GD}\left(V_0\cup V_1,  V(\M_A\cup\M_B)\colouringpI{2}\setminus(V(\mathcal N^*)\cup F_{\ell}\cup\smallatoms)\right)<\frac{\eta^3 k}{10^9}\;.$$
\end{claim}
\begin{proof}[Proof of Claim~\ref{cl:ztratimeMaloW}]
Let us consider an arbitrary vertex $x\in V_0\cup V_1$. By~\eqref{eq:PROCH2} the
number of $(\M_A\cup\M_B)$-vertices $C\subset \bigcup\clusters$ such that
$\deg_{\GD}(x,C)>0$ is at most $\frac{2(\Omega^*)^2k}{\gamma^2\pi\clustersize}$.

Due to~\eqref{eq:Fjbal}, we have
 for each $(\M_A\cup\M_B)$-edge $(A,B)$  that 
\begin{equation}\label{eq:sumicek}
\left|(A\cup B)\colouringpI{2}\setminus (V(\mathcal N^*)\cup F_\ell)\right|\le \tau k\;.
\end{equation}
                                                                                                                                                                                     
Thus summing~\eqref{eq:sumicek} over all $(\M_A\cup\M_B)$-edges $(A,B)$ with $\deg_{\GD}(x,(A\cup B)\setminus \smallatoms)>0$ we get
$$\deg_{\GD}\left(x,  V(\M_A\cup\M_B)\colouringpI{2}\setminus(V(\mathcal
N^*)\cup F_{\ell}\cup\smallatoms)\right)\le
\frac{4(\Omega^*)^2k}{\gamma^2\pi\clustersize} \cdot \tau k\;.$$ The claim now
follows by~\eqref{eq:KONST}.
\end{proof}
\begin{claim}\label{cl:OneSideSee}
Let $j\in[\ell]$ be such that $\mathcal F_j$ does not consist of all the components of $T_j-r_j$. Then there exists an $\mathcal N^*$-cover $\mathcal X_j$ such that $\deg_{\GD}\left(x_j,\bigcup \mathcal X_j\right)\le \frac{3 \eta^3k}{10^9}$.
\end{claim}
\begin{proof}[Proof of Claim~\ref{cl:OneSideSee}]
First, we define an $(\M_A\cup\M_B)$-cover $\mathcal R_j$ as follows. For an $(\M_A\cup\M_B)$-edge $(A,B)$ let $\mathcal R_j$ contain $A$ if
$$|(A\cap\neighbor_{\GD\cup\Gcapt}(x_j))\colouringpI{2}\setminus F_{j-1}|\le |(B\cap\neighbor_{\GD\cup\Gcapt}(x_j))\colouringpI{2}\setminus F_{j-1}|\;,$$ and $B$ otherwise. Observe that by the definition of $\mathcal N_j$, we have
\begin{equation}\label{eq:nulanatrave}
\deg_{\GD}\left(x_j,\bigcup \mathcal R_j\setminus V(\mathcal N_j)\right)=0\;.
\end{equation}
Also, we have $V(\mathcal N^*)\cap \bigcup \mathcal R_j\cap V(\M_j)\subset V(\mathcal N^*)\cap V(\M_j)\subset V(\M_j)\setminus F_j$. In particular,~\eqref{eq:OnLawn} gives that
\begin{equation}\label{eq:wac}
\left|V(\mathcal N^*)\cap \bigcup \mathcal R_j\cap V(\M_j)\right| \le \frac{2\eta^3 k}{10^9}\;.
\end{equation}

Let $\mathcal X_j$ be the restriction of $\mathcal R_j$ to $\mathcal N^*$. We then have 
\begin{align*}
\deg_{\GD}\left(x_j,\bigcup \mathcal X_j\right)&=\deg_{\GD}\left(x_j,V(\mathcal N^*)\cap \bigcup \mathcal R_j\right)\\
\JUSTIFY{by~\eqref{eq:nulanatrave}}&\le \deg_{\GD}\left(x_j, V(\mathcal N^*)\cap\bigcup \mathcal R_j\cap V(\mathcal M_j)\right)+\deg_{\GD}\left(x_j, V(\mathcal N_j)\setminus V(\M_j)\right)\\
\JUSTIFY{by~\eqref{eq:wac}, Claim~\ref{cl:MjVSNj}}&\le \frac{3\eta^3 k}{10^9}\;.
\end{align*}
\end{proof}

For every $j\in[\ell]$ we define $\mathcal N^*_j\subset \mathcal N^*$ as those $\mathcal N^*$-edges $(A,B)$ for which we have $$\big((A\cup B)\setminus\bigcup \mathcal X_j\big)\cap\smallatoms=\emptyset\;.$$

\medskip\noindent\underline{Stage II:} We shall inductively for $j=1,\ldots,
\ell$ embed those components of $T_j-r_j$ that are not included in $\mathcal
F_j$; let us denote the set of these components by $\mathcal K_j$.
There is nothing to do when $\mathcal K_j=\emptyset$, so let us assume otherwise.

We write  $\BL:=\{C\in\clusters\::\:C\subset\largevertices{\eta}{k}{G}\}$. Let
$K\in \mathcal K_j$ be a component that has not been embedded yet. We write $U'$
for the total image of what has been embedded (in Stage~I, and Stage~II so far), combined with $U$. We claim that $x_j$ has a substantial degree into one of four specific vertex sets.
\begin{claim}\label{cl:OneInFour}
At least one of the following four cases occurs.
\begin{itemize}
 \item[$\mathbf{(U1)}$] $\deg_{\GD}\left(x_j, V(\mathcal
 N^*_j)\setminus \bigcup \mathcal X_j\right)-|U'\cap
  V(\mathcal N^*_j)|\ge \frac{\eta^2
 k}{10^4}$,
 \item[$\mathbf{(U2)}$] $\deg_{\GD}\left(x_j, \smallatoms\setminus
 U'\right)\ge \frac{\eta^2 k}{10^4}$,
 \item[$\mathbf{(U3)}$] $\deg_{\Gcapt}\left(x_j, V(\Gexp)\setminus
 U'\right)\ge \frac{\eta^2 k}{10^4}$,
 \item[$\mathbf{(U4)}$] $\deg_{\GD}\left(x_j, \bigcup \BL\setminus (L_\#\cup
 V(\Gexp)\cup U')\right)\ge \frac{\eta^2 k}{10^4}$.
\end{itemize}
\end{claim}
\begin{proof}
Write $U'':=(U')\colouringpI{2}=U'\sm U$. By~\eqref{COND:P1:3}, we have 
\begin{align*}
\frac h2&\le  \deg_{\Gcapt}(x_j,\Vgood\colouringpI{2})\\
&\le 
\deg_{\GD}\left(x_j,V(\mathcal N^*_j)\colouringpI{2}  \setminus  \bigcup
\mathcal X_j\right)
+ \deg_{\GD}\left(x_j,\smallatoms\colouringpI{2}\setminus
(V(\mathcal N^*_j)\cup V(\Gexp) \cup \bigcup
\mathcal X_j)\right)\\
&~~~+
\deg_{\Gcapt}\left(x_j,V(\Gexp)\colouringpI{2}\right)+
\deg_{\GD}\left(x_j,\bigcup\BL\colouringpI{2}\setminus (L_\#\cup V(\Gexp)\cup
V(\mathcal N^*_j))\right)\\
&~~~+\deg_{\GD}\left(x_j,V(\M_A\cup \M_B)\colouringpI{2}\setminus (V(\mathcal
N^*)\cup\smallatoms)\right)+ \deg_{\GD}\left(x_j,\bigcup\mathcal X_j\right)\\
\JUSTIFY{by C\ref{cl:ztratimeMaloW}, C\ref{cl:OneSideSee}}
&\le \deg_{\GD}\left(x_j,V(\mathcal N^*_j)\setminus
 \bigcup \mathcal X_j\right)-\left|U'\cap V(\mathcal N^*_j)\right|\\
&~~~+\deg_{\GD}\left(x_j,\smallatoms\colouringpI{2}\setminus (V(\mathcal
N^*_j)\cup \bigcup \mathcal X_j\cup U'')\right)+
\deg_{\Gcapt}\left(x_j,V(\Gexp)\colouringpI{2}\setminus  U''\right)\\
&~~~+\deg_{\GD}\left(x_j,\bigcup\BL\colouringpI{2}\setminus (L_\#\cup
V(\Gexp)\cup  V(\mathcal N^*_j)\cup U'')\right)\\
&~~~+\frac{4\eta^3 k}{10^9}+|U''|\;.
\end{align*}
The claim follows since $|U''|\le \frac h2-\frac{\eta^2 k}{1000}$ by~\eqref{eq:MALYSTROM1}.
\end{proof}

We now now briefly describe how to embed $K$ in each of the cases
$\mathbf{(U1)}$--$\mathbf{(U4)}$.

\begin{itemize}
 \item In case~$\mathbf{(U1)}$ 
recall that each $(\M_A\cup \M_B)$-edge contains at most one $\mathcal
N^*_j$-edge. Thus by~\eqref{eq:PROCH1} we get that there is an $(\M_A\cup \M_B)$-edge $(A,B)$ with
\begin{equation}\label{eq:smethod}
 \deg_{\GD}\left(x_j, (V(\mathcal N^*_j)\cap (A\cup B))\setminus \bigcup
 \mathcal X_j\right)-|V(\mathcal N^*_j)\cap U'\cap (A\cup B)|\ge \frac{\eta^2
 k}{10^4}\cdot \frac{\gamma^2}{2(\Omega^*)^2 k}\cdot |A|. \ \ \ \ 
\end{equation}
Let us fix this edge $(A,B)$, and let $(A',B')$ be the corresponding edge in
$\mathcal N^*_j$. Suppose without loss of generality that $B\in\mathcal X_j$. We
can now embed $K$ in $(A',B')$ using Lemma~\ref{lem:embed:regular}
with the following input: $C_\PARAMETERPASSING{L}{lem:embed:regular}:=A'
,D_\PARAMETERPASSING{L}{lem:embed:regular}:=B',
X_\PARAMETERPASSING{L}{lem:embed:regular}:=A'\setminus U',
X^*_\PARAMETERPASSING{L}{lem:embed:regular}:=\neighbor_{\GD}(x_j, A'\setminus
U'), 
Y_\PARAMETERPASSING{L}{lem:embed:regular}:=B'\setminus U',
\epsilon_\PARAMETERPASSING{L}{lem:embed:regular}:=\frac{8\cdot
10^4(\Omega^*)^2\epsilon}{\gamma^2\eta^2},
\beta_\PARAMETERPASSING{L}{lem:embed:regular}:=d/6$.
 With the help of~\eqref{eq:smethod}, we calculate that $\min\{X_\PARAMETERPASSING{L}{lem:embed:regular},Y_\PARAMETERPASSING{L}{lem:embed:regular}\}\ge
|X^*_\PARAMETERPASSING{L}{lem:embed:regular}|\ge \frac
{\gamma^2\eta^2|A|}{4\cdot 10^4(\Omega^*)^2} \ge 4\frac{\eps_\PARAMETERPASSING{L}{lem:embed:regular}}{\beta_\PARAMETERPASSING{L}{lem:embed:regular}}|A'|$.

 \item In Case~$\mathbf{(U2)}$ we embed $K$ using
 Lemma~\ref{lem:embed:avoidingFOREST} with the following input:
 $\epsilon_\PARAMETERPASSING{L}{lem:embed:avoidingFOREST}:= \epsilon',
 U_\PARAMETERPASSING{L}{lem:embed:avoidingFOREST}:= U',
 U^*_\PARAMETERPASSING{L}{lem:embed:avoidingFOREST}:= \neighbor_{\GD}(x_j,
 \smallatoms\setminus U'), \ell:=1$.

 \item In Case~$\mathbf{(U3)}$ we embed $K$ using Lemma~\ref{lem:embed:greyFOREST}
 with the following input: $H_\PARAMETERPASSING{L}{lem:embed:greyFOREST}:=\Gexp, 
 V_{1,\PARAMETERPASSING{L}{lem:embed:greyFOREST}}:=V_{2,\PARAMETERPASSING{L}{lem:embed:greyFOREST}}:=V(\Gexp),
 U_\PARAMETERPASSING{L}{lem:embed:greyFOREST}:=U',
 U^*_\PARAMETERPASSING{L}{lem:embed:greyFOREST}:= \neighbor_{\Gexp}(x_j,
 V(\Gexp)\setminus U'), Q_\PARAMETERPASSING{L}{lem:embed:greyFOREST}:=1,
 \zeta_\PARAMETERPASSING{L}{lem:embed:greyFOREST}:= \rho,
 \ell_\PARAMETERPASSING{L}{lem:embed:greyFOREST}:=1$.

 \item In Case~$\mathbf{(U4)}$ we proceed as follows. As
 $\deg_{\GD}(x_j,\WantiC)<\frac{\eta^2k}{10^5}$ (cf.
 Definition~\ref{def:heart1}), we have $$\deg_{\GD}\left(x_j, \bigcup
 \BL\setminus (L_\#\cup V(\Gexp)\cup \WantiC\cup U')\right)\ge \frac{2
 \eta^2 k}{10^5}\;.
 $$ 
As for~\eqref{eq:smethod}, we use~\eqref{eq:PROCH1} to find a cluster $A\in\BL$ with
\begin{equation}\label{barrioyungay}
 \deg_{\GD}\left(x_j, A\setminus (L_\#\cup V(\Gexp)\cup \WantiC \cup
 U')\right) \ge \frac{2\eta^2 k}{10^5}\cdot \frac{\gamma^2}{2(\Omega^*)^2
 k}\cdot |A|= \frac{\eta^2 \gamma^2}{10^5(\Omega^*)^2
 }\cdot |A|.
\end{equation}
Recall that by the definition of $L_\#$ and $ \WantiC$, we have that $\mindeg_{\Gcapt}(A\setminus (L_\#\cup \WantiC), V(G)\setminus
\HugeVertices)\ge (1+\frac{4\eta}{5})k$. Thus at least one of the
following subcases must occur for the set $A^*:=(\neighbor_{\GD}(x_j)\cap
A)\setminus (L_\#\cup V(\Gexp)\cup \WantiC \cup U')$:
\begin{enumerate}
 \item[$\mathbf{(U4a)}$] For at least $\frac12|A^*|$ vertices $v\in A^*$ we have
 $\deg_{\Gcapt}(v,\smallatoms\setminus U')\ge \frac{2\eta k}5$.
 \item[$\mathbf{(U4b)}$] For at least $\frac12|A^*|$ vertices $v\in A^*$ we have
 $\deg_{\Gblack}(v,\bigcup\clusters\setminus U')\ge \frac{2\eta k}5$.
\end{enumerate}
In case $\mathbf{(U4a)}$ we embed $K$ using
Lemma~\ref{lem:embed:avoidingFOREST}.  Details are very similar to
$\mathbf{(U2)}$. As for case $\mathbf{(U2b)}$,
let as take an arbitrary vertex $v\in A^*$ with
$\deg_{\Gblack}(v,\bigcup\clusters\setminus U')\ge \frac{2\eta k}5$. In
particular, using~\eqref{eq:PROCH1}, we find a cluster $B\in \clusters$ with 
\begin{equation}\label{stantonmoore}
\deg_{\Gblack}(v,B\setminus
U')\ge \frac{\gamma^2\eta}{10(\Omega^*)^2}|B|\;.
\end{equation}
 Map the root $r_K$ of $K$ to
$v$ and embed $K-r_K$ in $(A,B)$ using
Lemma~\ref{lem:embed:regular}\footnote{Lemma~\ref{lem:embed:regular} deals with
embedding a single tree in a regular pair, whereas $K-r_K$ has several
components. We therefore apply the lemma repeatedly for each component.} with
the following input:
$C_\PARAMETERPASSING{L}{lem:embed:regular}:= B,
D_\PARAMETERPASSING{L}{lem:embed:regular}:= A,
X_\PARAMETERPASSING{L}{lem:embed:regular}:= B\setminus U',
Y_\PARAMETERPASSING{L}{lem:embed:regular}:= A\setminus U',
X^*_\PARAMETERPASSING{L}{lem:embed:regular}:= \neighbor_{\Gblack}(v, B\setminus
U'),
\beta_\PARAMETERPASSING{L}{lem:embed:regular}:= \gamma^2\eta/(10(\Omega^*)^2),
\epsilon_\PARAMETERPASSING{L}{lem:embed:regular}:= \epsilon'$. By~\eqref{barrioyungay} and~\eqref{stantonmoore} we see that $X_\PARAMETERPASSING{L}{lem:embed:regular},
Y_\PARAMETERPASSING{L}{lem:embed:regular}$ and $
X^*_\PARAMETERPASSING{L}{lem:embed:regular}$ are large enough.
\end{itemize}

\medskip\noindent\underline{Stage III:} In this stage we embed the trees
$\{T'_j\}_{j=1}^{\ell'}$. The embedding techniques are as in Stage~II. The cover
$\mathcal F'$ from Definition~\ref{def:heart1} plays the same role as the covers
$\mathcal X_j$ in Stage~II. Observe that $\mathcal F'$ is universal whereas the
covers $\mathcal X_j$ are specific for each vertex $x_j$. A second simplification is that in Stage~III we use
the semiregular matching $(\M_A\cup \M_B)\colouringpI{2}$ for embedding (in a
counterpart of $\mathbf{(U1)}$) instead of $\mathcal N^*_j$.

Again we proceed inductively for $j=1,\ldots,
\ell$ with embedding the components of $T'_j-r'_j$, which we denote by $\mathcal
K'_j$.
 Let
$K\in \mathcal K'_j$ be a component that has not been embedded yet. We write
$U'$ for the total image of what has been embedded (in Stage~I,~II, and
Stage~III so far), combined with $U$ and let $U''=U'\cap \colouringp{2}$. We
claim that $y_j$ has a substantial degree into one of four specific vertex sets.
\begin{claim}\label{cl:OneInFour'}
At least one of the following four cases occurs.
\begin{itemize}
 \item[$\mathbf{(U1')}$] $\deg_{\GD}\left(y_j, V((\M_A\cup
 \M_B)\colouringpI{2})\setminus (\smallatoms\cup \bigcup \mathcal
 F')\right)\\-\left|U''\cap\left(\bigcup \mathcal F'\cup ( V((\M_A\cup
 \M_B)\colouringpI{2})\setminus \smallatoms\right)\right|\ge \frac{\eta^2 k}{10^4}$,
 \item[$\mathbf{(U2')}$] $\deg_{\GD}\left(y_j, \smallatoms\setminus
 U'\right)\ge \frac{\eta^2 k}{10^4}$,
 \item[$\mathbf{(U3')}$] $\deg_{\Gcapt}\left(y_j, V(\Gexp)\setminus
 U'\right)\ge \frac{\eta^2 k}{10^4}$,
 \item[$\mathbf{(U4')}$] $\deg_{\GD}\left(y_j, \bigcup \BL\setminus (L_\#\cup
 V(\Gexp)\cup U')\right)\ge \frac{\eta^2 k}{10^4}$.
\end{itemize}
\end{claim}
\begin{proof}
 As $y_j\in V_1$, we have that
\begin{align*}
h&\le  \deg_{\Gcapt}(y_j,\Vgood\colouringpI{2})\\
&\le 
\deg_{\GD}\left(y_j,V((\M_A\cup \M_B)\colouringpI{2})\setminus (\smallatoms\cup
V(\Gexp)\cup \bigcup \mathcal F')\right)+
\deg_{\GD}\left(y_j,\smallatoms\colouringpI{2}\setminus (V(\Gexp)\cup  \bigcup
\mathcal F'\right)\\ &~~~+\deg_{\GD}(y_j,\bigcup \mathcal F') +
\deg_{\GD}\left(y_j,\bigcup\BL\colouringpI{2}\setminus (L_\#\cup V(\Gexp)\cup 
V(\M_A\cup \M_B)\right)\\
&~~~+\deg_{\Gcapt}\left(y_j, V(\Gexp)\colouringpI{2}\right)+
\deg_{\GD}\left(y_j, V(\M_A\cup\M_B)\colouringpI{2}\setminus
V((\M_A\cup \M_B)\colouringpI{2})\right)\\ 
\JUSTIFY{by L~\ref{lem:RestrictionSemiregularMatching}}&\le
\deg_{\GD}\left(y_j,V((\M_A\cup \M_B)\colouringpI{2})\setminus (\smallatoms\cup V(\Gexp)\cup \bigcup \mathcal
F')\right)\\ &~~~-\left|U''\cap \left(\bigcup \mathcal F'\cup  (V((\M_A\cup
\M_B)\colouringpI{2})\setminus \smallatoms)\right)\setminus V(\Gexp)\right|\\
&~~~+\deg_{\GD}\left(y_j,\smallatoms\colouringpI{2}\setminus (U''\cup
V(\Gexp)\cup \bigcup \mathcal F')\right)
+\deg_{\Gcapt}\left(y_j,V(\Gexp)\colouringpI{2}\setminus U''\right)\\
&~~~+\deg_{\GD}\left(y_j,\bigcup\BL\colouringpI{2}\setminus (L_\#\cup
V(\Gexp)\cup V(\M_A\cup \M_B)\cup U'')\right)+\frac{2\eta^3
k}{10^3}+\frac {\eta^2k}{10^5}+|U''|\;.
\end{align*}
The claim follows since $|U''|\le \sum_{j}T_j+\sum_{j}T_j'\le h-\frac
{\eta^2k}{1000}$.

\end{proof}

Cases~$\mathbf{(U1')}$--$\mathbf{(U4')}$ are treated analogously as
Cases~$\mathbf{(U1)}$--$\mathbf{(U4)}$.

\end{proof}

\begin{lemma}\label{lem:embed:heart2}\HAPPY{D}
Suppose we are in Setting~\ref{commonsetting} and~\ref{settingsplitting}, and that the sets $V_0$ and $V_1$ witness Preconfiguration~$\mathbf{(\heartsuit2)}(h)$. Suppose that $U\subset \colouringp{0}\cup\colouringp{1}$, such that $|U|\le k$. Suppose that $\{x_j\}_{j=1}^\ell\subset V_0\cup V_1$ are distinct vertices. Let $\{(T_j,r_j)\}_{j=1}^\ell$ be a family of rooted trees such that each component of $T_j-r_j$ has order at most $\tau k$.

If $\sum_j v(T_j)\le h-\eta^2 k/1000$ and $|U|+\sum_j v(T_j)\le k$ then there exist disjoint $(r_j\hookrightarrow x_j, V(T_j)\setminus\{r_j\}\hookrightarrow V(G)\setminus U)$-embeddings of $T_j$ in $G$.
\end{lemma}
\begin{proof}
The proof is contained in the proof of Lemma~\ref{lem:embed:heart1}. It suffices to repeat the first two stages of the embedding process in the proof. In that setting, we use $h_\PARAMETERPASSING{L}{lem:embed:heart1}=2h$. Note that the condition $\{x_j\}\subset V_0$ in the setting of Lemma~\ref{lem:embed:heart1} gives us the same possibilities for embedding as the condition $\{x_j\}\subset V_0\cup V_1$ in the current setting (cf.~\eqref{COND:P1:3} and~\eqref{COND:P2:4}).
\end{proof}

\begin{lemma}\label{lem:embed:total68}\HAPPY{D}
Suppose we are in Setting~\ref{commonsetting} and~\ref{settingsplitting}, and at least one of the following configurations occurs:
\begin{itemize}
 \item
 Configuration~$\mathbf{(\diamond6)}(\frac{\eta^3\rho^4}{10^{14}(\Omega^*)^4)},4\epsilonD,
 \frac {\gamma^3\rho}{32\Omega^*}, \frac{\eta^2\nu}{2\cdot
 10^4},\frac {3\eta^3}{2\cdot 10^3}, h)$,
 \item
 Configuration~$\mathbf{(\diamond7)}(\frac{\eta^3\gamma^3\rho}{10^{12}(\Omega^*)^4)},
 \frac {\eta\gamma}{400},4\epsilonD, \frac {\gamma^3\rho}{32\Omega^*},
 \frac{\eta^2\nu}{2\cdot 10^4},\frac {3\eta^3}{2\cdot 10^3}, h)$, or
 \item
 Configuration~$\mathbf{(\diamond8)}(\frac{\eta^4\gamma^4\rho}{10^{15}(\Omega^*)^5)},
 \frac {\eta\gamma}{400},\frac {4\epsilon}{\proporce{1}},4\epsilonD ,\frac
 d2,\frac {\gamma^3\rho}{32\Omega^*}
 ,\frac{\proporce{1}\pi\clustersize}{2k},\frac {\eta^2\nu}{2\cdot
 10^4} ,h_1,h)$.
\end{itemize}Suppose
that $(W_A,W_B,\shrubA,\shrubB)$ is a $(\tau k)$-fine partition of a rooted tree
$(T,r)$ of order $k$.
If the total order of the end shrubs is at most $h-2\frac{\eta^2 k}{10^3}$ and
the total order of the internal shrubs is at most $h_1-2\frac {\eta^2k}{10^5}$,
then $T\subset G$.
\end{lemma}
\begin{proof}
Let $T'$ be the tree induced by all the cut-vertices $W_A\cup W_B$ and all the
internal shrubs. Summing up the order of the internal shrub and the
cut-vertices, we get that $v(T')<h_1-\frac {\eta^2k}{10^5}$.  Fix an embedding
of $T'$ as in Lemma~\ref{lem:embed:skeleton67} (in configurations~$\mathbf{(\diamond6)}$ and $\mathbf{(\diamond7)}$), or as in Lemma~\ref{lem:embed:skeleton8} (in configuration~$\mathbf{(\diamond8)}$). This embedding now extends to external shrubs by Lemma~\ref{lem:embed:heart1} (in Preconfiguration~$\mathbf{(\heartsuit1)}$, which can only occur in Configuration~$\mathbf{(\diamond6)}$ and $\mathbf{(\diamond7)}$), or by Lemma~\ref{lem:embed:heart2} (in Preconfiguration~$\mathbf{(\heartsuit2)}$).  It is important to remember here that by Definition~\ref{ellfine}\eqref{Bsmall}, the total order of end shrubs in $\shrubB$ is at most half the size of the total order of end shrubs.
\end{proof}

The next lemma completely resolves Theorem~\ref{thm:main} in the presence of Configuration~$\mathbf{(\diamond9)}$.

\begin{lemma}\label{lem:embed9}
Suppose we are in Setting~\ref{commonsetting} and~\ref{settingsplitting}, and assume we
have Configuration~$\mathbf{(\diamond9)}(\delta,
\frac{2\eta^3}{10^3},$ $h_1,h_2,\epsilon_1, d_1, \mu_1,\epsilon_2, d_2, \mu_2)$
with $d_2>10\epsilon_2>0$, $4\cdot 10^3\le d_2\mu_2\tau k$,
$\max\{d_1,\tau/\mu_1\}\le \gamma^2\eta^2/(4\cdot 10^7(\Omega^*)^2)$,
$d_1^2/6>\epsilon_1\ge \tau/\mu_1$ and $\delta k>10^3/\tau$.

Suppose that $(W_A,W_B,\shrubA,\shrubB)$ is a $(\tau k)$-fine partition of a rooted tree $(T,r)$ of order $k$.
If the total order of the internal shrubs of $(W_A,W_B,\shrubA,\shrubB)$ is at most $h_1-\frac{\eta^2 k}{10^5}$, and the total order of the end shrubs is at most $h_2-2\frac{\eta^2 k}{10^3}$
then $T\subset G$.
\end{lemma}
\begin{proof}
Let $V_0,V_1, V_2,\mathcal N, \{Q_0^{(j)},Q_1^{(j)}\}_{j\in \mathcal Y}$ and
$\mathcal F'$ witness $\mathbf{(\diamond9)}$.
The embedding process has two stages. In the first stage we embed the knags and
the internal shrubs of $T$. In the second stage we embed the end shrubs. The
knags will be embedded in $V_0\cup V_1$, and the internal shrubs will be
embedded in $V(\mathcal N)$. Lemma~\ref{lem:embed:heart1} will be used to embed
the end shrubs.

The knags of $(W_A,W_B,\shrubA,\shrubB)$ are embedded in such a way that $W_A$
is embedded in $V_1$ and $W_B$ is embedded in $V_0$. Since no other part of $T$
is embedded in $V_0\cup V_1$ in the first stage, each knag can be embedded
greedily using the minimum degree condition arising
from the super-regularity of the pairs
$\{(Q_0^{(j)},Q_1^{(j)})\}_{j\in\mathcal Y}$ using the bound on the total order
of knags coming from Definition~\ref{ellfine}\eqref{few}, and using  Lemma~\ref{lem:embed:superregular} with the following input:
$\epsilon_\PARAMETERPASSING{L}{lem:embed:superregular}:= \epsilon_2$,
$d_\PARAMETERPASSING{L}{lem:embed:superregular}:= d_2$,
$\ell_\PARAMETERPASSING{L}{lem:embed:superregular}:= \mu_2k$, $U_A\cup U_B$ is
the image of $W_A\cup W_B$ embedded so far and 
$\{A_\PARAMETERPASSING{L}{lem:embed:superregular},B_\PARAMETERPASSING{L}{lem:embed:superregular}\}:=
\{Q_0^{(j)},Q_1^{(j)}\}$, where $j\in \mathcal Y$ is arbitrary for the first
knag, and for all other knags $P$ has the property that  
\[
\neighbor_{\GD}(\phi(\parent(P)))\cap Q_1^{(j)}\setminus
U_A\neq \emptyset.
\]
The existence of such an index $j$ follows from the fact that 
\begin{equation}\label{wiesoweshalbwarum}
\phi(\parent(P))\in V_2,
\end{equation}
together with condition~\eqref{conf:D9-VtoX}. We shall
 ensure~\eqref{wiesoweshalbwarum} during our embedding of the internal shrubs, see below.

We now describe how to embed an internal shrub $T^*\in\shrubA$ whose parent
$u\in W_A$ is embedded in a vertex $x\in V_1$. Let $w\in V(T^*)$ be the unique
neighbor of a vertex from $W_A\setminus\{u\}$ (cf.\
Definition~\ref{ellfine}\eqref{2seeds}). Let $U$ be the image of the part of $T$
embedded so far. The next claim will be useful for finding a suitable $\mathcal N$-edge for
accommodating $T^*$.
\begin{claim}\label{claim:DTn}
There exists an $\mathcal N$-edge $(A,B)$, or an $\mathcal N$-edge $(B,A)$ such
that $$\min\big\{|\neighbor_{\GD}(x)\cap V_2\cap (A\setminus U)|, |B\setminus
U|\big\} \ge 100 d_1|A|+\tau k\;.$$
\end{claim}
\begin{proof}[Proof of Claim~\ref{claim:DTn}]
For the purpose of this claim we reorient $\mathcal N$ so that $V_2(\mathcal
N)\subset \bigcup \mathcal F'$.

Suppose the claim fails to be true. Then for each $(A,B)\in \mathcal N$ we have
$|\neighbor_{\GD}(x)\cap V_2\cap (A\setminus U)|<100d_1|A|+\tau k$ or $|B\setminus U|<100d_1|A|+\tau k$. In
either case we get
\begin{equation}\label{eq:HvHv}
|\neighbor_{\GD}(x)\cap V_2\cap A|-|U\cap (A\cup B)| <
100d_1|A|+\tau k\;.
\end{equation}
We write $S:=\bigcup \{V(D):D\in\DenseSpots, x\in V(D)\}$. Combining Fact~\ref{fact:sizedensespot} and Fact~\ref{fact:boundedlymanyspots} we get that 
\begin{equation}\label{eq:girls}
|S|\le \frac{2(\Omega^*)^2k}{\gamma^2}\;.
\end{equation}
Let us look at the number
\begin{equation}\label{eq:dlambda}
\lambda:= \sum_{(A,B)\in\mathcal N} \big(|\neighbor_{\GD}(x)\cap V_2\cap A|-|U\cap (A\cup B)|\big)\;.
\end{equation}
For a lower bound on $\lambda$, we write $\lambda=|\neighbor_{\GD}(x)\cap V_2|-|U\cap V(\mathcal N)|$. 
(Note that $V_2\subseteq V(\mathcal N)$ as we are in Configuration~$\mathbf{(\diamond9)}$.)
The first term is at least $h_1$ by~\eqref{conf:D9-XtoV}, while the second term is at most $h_1-\frac{\eta^2 k}{10^5}$ by the assumptions of the lemma. Thus $\lambda\ge \frac{\eta^2 k}{10^5}$.

For an upper bound on $\lambda$ we only consider those $\mathcal N$-edges
$(A,B)$ for which $\neighbor_{\GD}(x)\cap A\neq\emptyset$. In that case
$A\subset S$ (cf. \ref{commonsetting2} of Setting~\ref{commonsetting}). Thus,
since $\mathcal N$ is
$(\epsilon_1,d_1,\mu_1 k)$-semiregular we get that
\begin{equation}\label{eq:girls2}
\left|\{(A,B)\in\mathcal N\::\: \neighbor_{\GD}\cap A\neq \emptyset\}\right|\le
\frac{|S|}{\mu_1k}\;.
\end{equation}
Thus, 
 \begin{align*}
 \lambda&\le  \sum_{(A,B)\in\mathcal N, \neighbor_{\GD}(x)\cap A\neq\emptyset} \big(|\neighbor_{\GD}(x)\cap V_2\cap A|
 -|U\cap (A\cup B) |\big)\\
\JUSTIFY{by~\eqref{eq:HvHv},~\eqref{eq:girls2}}&\le
100d_1|S|+
\frac{|S|}{\mu_1k}\tau k\\
\JUSTIFY{by~\eqref{eq:girls}}&<\frac{\eta^2 k}{10^5}\;,
\end{align*}
a contradiction. This finishes the proof of the claim.
\end{proof}
By symmetry we suppose that Claim~\ref{claim:DTn} gives an $\mathcal N$-edge
$(A,B)$ such that $\min\big\{|\neighbor_{\GD}(x)\cap V_2\cap (A\setminus U)|,
|B\setminus U|\big\} \ge 100d_1|A|+\tau k$. We apply
Lemma~\ref{lem:embed:regular} with input
$C_\PARAMETERPASSING{L}{lem:embed:regular}:= A$,
$D_\PARAMETERPASSING{L}{lem:embed:regular}:= B$
$X_\PARAMETERPASSING{L}{lem:embed:regular}=X^*_\PARAMETERPASSING{L}{lem:embed:regular}:=\neighbor_{\GD}(x)\cap
V_2\cap (A\setminus U)$, $Y_\PARAMETERPASSING{L}{lem:embed:regular}:=B\setminus
U$ , $\epsilon_\PARAMETERPASSING{L}{lem:embed:regular}:= \epsilon_1$,
$\beta_\PARAMETERPASSING{L}{lem:embed:regular}:= d_1/3$. Then  there exists an
embedding of $T^*$ in $V(\mathcal N)\setminus U$ such that $w$ is embedded in $V_2$. This ensures~\eqref{wiesoweshalbwarum}.

We remark that there may be several internal shrubs extending from $u\in W_A$. However Claim~\ref{claim:DTn} and the subsequent application of Lemma~\ref{lem:embed:regular} allows a sequential embedding of these shrubs.
This finishes the first stage of the embedding process. 

For the second stage, i.e., the embedding of the end shrubs of
$(W_A,W_B,\shrubA,\shrubB)$, we first recall that the total order of end shrubs
in $\shrubA$ is at most $h_2-2\frac{\eta^2 k}{10^3}$, and the total order of end
shrubs in $\shrubB$ is at most $\frac12\big(h_2-2\frac{\eta^2 k}{10^3}\big)$ by
Definition~\ref{ellfine}\eqref{Bsmall}. 
The embedding is a straightforward
application of Lemma~\ref{lem:embed:heart1}.
\end{proof}

The next lemma  resolves Theorem~\ref{thm:main} in the presence of Configuration~$\mathbf{(\diamond10)}$.
\begin{lemma}\label{lem:embed10}
Suppose we are in Setting~\ref{commonsetting}.
For every $\eta',d',\Omega>0$ there exists $\tilde\epsilon>0$ such that for
every $\nu'>0$ with the property that 
\begin{equation}\label{eq:plm}
\frac{\eta'\nu'}{200 \Omega}>\tau
\end{equation} there exists a number $k_0$ such that the following holds for each $k>k_0$.

If $G$ is a graph with Configuration~$\mathbf{(\diamond10)}(\tilde\epsilon,d',\nu'k, \Omega k,\eta')$
then $\treeclass{k}\subset G$.
\end{lemma}
\begin{proof}
We give a sketch of a proof, following~\cite{PS07+}. The main difference is indicated in Section~\ref{ssec:EmbedOverview10}.

Suppose we have Configuration~$\mathbf{(\diamond10)}(\tilde\epsilon,d',\nu'k, \Omega k,\eta')$, and are given a rooted tree $(T,r)$ of order $k$ with a $(\tau k)$-fine partition $(W_A,W_B,\shrubA,\shrubB)$ given by Lemma~\ref{lem:TreePartition}.  
By replacing $\mathcal L^*$ by $\mathcal L^*\setminus \V(\M)$,\footnote{This does not change validity of the conditions in Definition~\ref{def:CONF10}.} we can assume that $\mathcal L^*$ and $\V(\M)$ are disjoint.

For each shrub $F\in\mathcal S_A\cup\mathcal S_B$, let $x_F\in V(F)$ be its root, i.e., its minimal element in the topological order. If $F$ is internal then we also define $y_F$ as its (unique) maximal element that neighbours $W_A$.
We can partition the semiregular matching $\mathcal M$ and the set $\mathcal L^*$ into two parts: $\mathcal M_A\cup \mathcal L^*_A$ and $\mathcal M_B\cup \mathcal L^*_B$ so that the partition satisfies
\begin{align}
\label{eq:splitSA}
\deg_{\tilde G}\big(v,V(\M_A)\cup \bigcup \LargeTen_A\big)&\ge v(\mathcal S_A)+\frac{\eta' k}4\quad\mbox{and}\\
\label{eq:splitSB}
\deg_{\tilde G}\big(w,V(\M_B)\cup \bigcup \LargeTen_B\big)&\ge v(\mathcal S_B)+\frac{\eta' k}4\;,
\end{align}
for all but at most $2\tilde\epsilon |A|$ vertices $v\in A$ and for all but at most $2\tilde\epsilon|B|$ vertices $w\in B$. To see this, observe that the nature of the regularized graph allows us to
treat\footnote{up to a small error} 
conditions~\eqref{eq:splitSA}, \eqref{eq:splitSB}, or that of Definition~\ref{def:CONF10}\eqref{diamond10cond2} in terms of average degrees of vertices in $A$ and $B$, rather than in terms of individual degrees.\footnote{This property is also key in the classical dense setting of the Regularity Lemma.} If $A$ and $B$ were connected to each cluster $X\in \mathcal L^*\cup \mathcal V(\mathcal M)$ by regular pairs of the same density, say $d_X$, it would suffice to split $\mathcal L^*$ and $\mathcal M$ in the ratio $v(\mathcal S_A):v(\mathcal S_B)$. In the general setting, this can also be achieved, as was done in~\cite[Lemma~9]{PS07+}. 
Let $h_{A,\mathcal L^*}$, $h_{B,\mathcal L^*}$, $h_{A,\M}$, $h_{B,\M}$ be the average degrees of vertices of $A$ and $B$ into $\mathcal L^*_A,\mathcal L^*_B,\M_A,\M_B$.

We will now use the regularity to embed the shrubs and the seeds in $\tilde G$. We start with mapping $r$ to $A$ or $B$ (depending whether $r\in W_A$ or $r\in W_B$), and proceed along a topological order on $T$. We denote the partial embedding of $T$ at any particular stage as $\phi$. The vertices of $W_A$ are mapped to $A$, the vertices of $W_B$ are mapped to $B$. As for embedding the shrubs, initially we start with embedding the shrubs of $\mathcal S_A$ to $\M_A$ (we say that \emph{$A$ is in  the $\M$-mode}), and embedding the shrubs of $\mathcal S_B$ to $M_B$ (\emph{$B$ is in  the $\M$-mode}). 
By filling up the $\M$-edges with the shrubs as balanced as possible we can guarantee that we do not run out of space in $\M_A$ before embedding $\mathcal S_A$-shrubs of total order at least $h_{A,\M}-\eta'k/100$. An analogous property holds for embedding $\mathcal S_B$-shrubs. We omit details and instead refer to a very similar procedure in Lemma~\ref{lem:embed9}.\footnote{In Lemma~\ref{lem:embed9} it was shown how to utilize~\eqref{conf:D9-XtoV} for embedding shrubs of order up to $\approx h_1$ in regular pairs.} 

At some moment we may run out of space in $\M_A$, or in $\M_B$. Say that this happens first with the matching $\M_A$. Let $\mathcal S_A^*\subset \mathcal S_A$ be the set of shrubs not embedded so far. We now describe how to proceed when \emph{$A$ is in the $\mathcal L^*$-mode}. In this mode, we will not embed an upcoming shrub $F\in\mathcal S^*_A$, but only reserve a set $U_F$, $|U_F|\le v(F)$ which serves a reminder that we want to accommodate $F$ later on. Suppose that the parent $\parent(F)\in W_A$ of $F$ has been mapped to a typical\footnote{in the sense of Definition~\ref{def:CONF10}\eqref{diamond10cond2}} vertex $z\in A$ already. We have 
\begin{equation*}
\deg_{\tilde G}(z,\bigcup \mathcal L^*_A)\ge v(\mathcal S_A^*)+\frac{\eta'k}{100}\ge \sum_{F'} |U_F'|+\frac{\eta'k}{100}\;,
\end{equation*}
where the sum ranges over the already processed $\mathcal S_A^*$-shrubs $F'$. Consequently, there is a cluster $X\in \V$ such that 
\begin{equation}\label{eq:thresholdFilling}
\deg_{\tilde G}\Big(z,X\setminus \bigcup_{F'}U_{F'}\Big)>\frac{\eta'|X|}{100\Omega}\;.
\end{equation} Let us view $F$ as a bipartite graph, and let $a_F$ be the size of its color class that contains $x_F$. Let $U_F$ be an arbitrary set of $(\neighbor_{\tilde G}(z)\cap X)\setminus \bigcup_{F'}U_{F'}$ of size $a_F$, and also let us fix an image $\phi(x_F)\in U_F$ arbitrarily. If $F$ is an internal shrub, we further define $\phi(y_F)\in U_F\setminus\{\phi(x_B)\}$ arbitrarily. At this stage we consider $F$ as processed. 

Later, of course, also $B$ can switch in the $\mathcal L^*$-mode as well. At that moment, we define $\mathcal S^*_B$, and start to only make reservations $U_K$ in clusters of $\mathcal L^*_B$ instead of embedding shrubs $K\in\mathcal S^*_B$.

After all shrubs of $\mathcal S^*_A\cup\mathcal S^*_B$ have been processed we finalize the embedding. Consider a shrub $F\in\mathcal S^*_A\cup\mathcal S^*_B$. Suppose that $U_F\subset X$ for some $X\in\V$. We use Definition~\ref{def:CONF10}\eqref{diamond10cond3} to find a cluster $Y$ such that $$\density(X,Y)\ge \frac{\left|Y\cap(\text{im}\big(\phi)\cup\bigcup_{F'\text{ yet unembedded}}U_{F'}\big)\right|}{|Y|}+\frac{\eta'}{100\Omega}\;.$$ 
As $\phi(x_F)$ and $\phi(y_F)$ are typical\footnote{in the sense of Definition~\ref{def:CONF10}\eqref{diamond10cond3}}, we can additionally require that 
$$\deg_{\tilde G}(\phi(x_F),Y),\deg_{\tilde G}(\phi(y_F),Y)\ge (\density(X,Y)-\sqrt{\tilde\epsilon})|Y|\;.$$
Therefore, the regularity method allows us to embed $F$ in the pair $(X,Y)$ avoiding the already defined image of $\phi$, and the sets $U_{F'}$ corresponding to yet unembedded shrubs $F'$. The fact that the threshold in~\eqref{eq:thresholdFilling} was taken quite high (compared to the size of the shrubs, see~\eqref{eq:plm}) allows us to avoid atypical vertices. We also need this embedding to be compatible with the existing placements $\phi(x_F)$ and $\phi(y_F)$. In particular, we need to find a path of length $\dist_F(x_F,y_F)$ from $\phi(x_F)$ to $\phi(y_F)$. Here, it is crucial that $\dist_F(x_F,y_F)\ge 4$ (cf.~Definition~\ref{ellfine}\eqref{short}).\footnote{Indeed, it could be that $\neighbor(\phi(x_F))\cap\neighbor(\phi(y_F))=\emptyset$, which would make it impossible to find a path of length~2 from $\phi(x_F)$ to $\phi(y_F)$. If, on the other hand $\dist_F(x_F,y_F)\ge 4$, then we can always find such a path using a look-ahead embedding in the regular pair $(X,Y)$.}
We remark, that in general we cannot guarantee that $X\cap \phi(F)=U_F$. So the set $U_F$ should be regarded merely as a measure of future occupation of $X$, rather than an indication of exact future placement.
\end{proof}

\section{Proof of Theorem~\ref{thm:main}}\label{sec:proof}\HAPPY{D}
Let $\alpha>0$ be given. We set $$\eta:=\min\{\frac1{30},\frac{\alpha}2\}.$$ We wish to fix further constants as in~\eqref{eq:KONST}. A~trouble is that we do not know the right choice of $\Omega^*$ and $\Omega^{**}$ yet. Therefore we take $g:=\lfloor\frac{100}{\eta^2}\rfloor+1$ and fix suitable constants
\begin{align*}
\eta\gg\frac1{\Omega_1}&\gg \frac1{\Omega_2}\gg \ldots\gg \frac1{\Omega_{g+1}}
\gg\rho\gg\gamma\gg d\ge \frac1{\Lambda}\ge \epsilon\ge
\pi\ge  \alphaD \ge
\epsilon'\ge
\nu\gg \tau \gg \frac{1}{k_0}>0\;,
\end{align*}
where the relations between the parameters are more exactly as follows:
\begin{align*}
\frac 1{\Omega_1}&\le \frac {\eta^9}{10^{25}}\;,\\
\frac 1{\Omega_{j+1}}&\le \frac
{\eta^{27}}{10^{67}\Omega_j^{36}}\quad\mbox{for each $j=1,\ldots,g$\;,}\\
\rho&\le \frac {\eta^9}{10^{25}\Omega_{g+1}^5}\;,\\
\gamma&\le \frac {\eta^{18}\rho^{24}}{10^{90}\Omega_{g+1}^{28}}\;,\\
d&\le \min\left\{\frac {\gamma^2\eta^2}{10^8\Omega_{g+1}^2},
\beta_\PARAMETERPASSING{L}{prop:LKSstruct}(\eta_\PARAMETERPASSING{L}{prop:LKSstruct}:=
\eta, \Omega_\PARAMETERPASSING{L}{prop:LKSstruct}:= \Omega_{g+1},
\gamma_\PARAMETERPASSING{L}{prop:LKSstruct}:= \gamma)\right\}\;,\\
\frac 1\Lambda&\le \min\left\{d,\frac
{\eta^{24}\gamma^{24}\rho}{10^{96}\Omega_{g+1}^{36}}\right\}
\;,
\\
\epsilon&\le \min\left\{\frac 1\Lambda, \frac
{\gamma^2\eta^3d\rho}{10^{13}\Omega_{g+1}^4},
\tilde\eps_\PARAMETERPASSING{L}{lem:embed10}(\eta'_\PARAMETERPASSING{L}{lem:embed10}:=\eta/40,
d'_\PARAMETERPASSING{L}{lem:embed10}:=\gamma^2d/2,
\Omega_\PARAMETERPASSING{L}{lem:embed10}:= \frac {(\Omega_{g+1})^2}{\gamma^2})
\right\}\;,
\\
\pi&\le \min\left\{\epsilon,
\pi_\PARAMETERPASSING{L}{prop:LKSstruct}(\eta_\PARAMETERPASSING{L}{prop:LKSstruct}:=
\eta, \Omega_\PARAMETERPASSING{L}{prop:LKSstruct}:= \Omega_{g+1},
\gamma_\PARAMETERPASSING{L}{prop:LKSstruct}:= \gamma,
\epsilon_\PARAMETERPASSING{L}{prop:LKSstruct}:= \epsilon)\right\}\;,\\
\alphaD&\le \min\left\{\epsilonD,
\alpha_\PARAMETERPASSING{L}{lem:edgesEmanatingFromDensePairsIII}\left(
\Omega_\PARAMETERPASSING{L}{lem:edgesEmanatingFromDensePairsIII}:=
\Omega_{g+1},
\rho_\PARAMETERPASSING{L}{lem:edgesEmanatingFromDensePairsIII}:= \frac
{\gamma^2}4,
\epsilon_\PARAMETERPASSING{L}{lem:edgesEmanatingFromDensePairsIII}:=
\epsilonD,
\tau_\PARAMETERPASSING{L}{lem:edgesEmanatingFromDensePairsIII}:=
2\rho\right)\right\}\;,\\
\epsilon'&\le \min\left\{\frac
{\alphaD^2\gamma^6\rho^2}{10^3\Omega_{g+1}^4},
\epsilon'_\PARAMETERPASSING{L}{prop:LKSstruct}(\eta_\PARAMETERPASSING{L}{prop:LKSstruct}:=
\eta, \Omega_\PARAMETERPASSING{L}{prop:LKSstruct}:= \Omega_{g+1}, \gamma_\PARAMETERPASSING{L}{prop:LKSstruct}:= \gamma,
\epsilon_\PARAMETERPASSING{L}{prop:LKSstruct}:= \epsilon)\right\}\;,\\
\nu&\le \min\left\{\frac {\alphaD\rho}{\Omega_{g+1}}, \epsilon',
\nu_\PARAMETERPASSING{L}{lem:LKSsparseClass}(\eta_\PARAMETERPASSING{L}{lem:LKSsparseClass}:=\alpha,
\Lambda_\PARAMETERPASSING{L}{lem:LKSsparseClass}:=\Lambda,
\gamma_\PARAMETERPASSING{L}{lem:LKSsparseClass}:= \gamma,
\epsilon_\PARAMETERPASSING{L}{lem:LKSsparseClass}:= \epsilon',
\rho_\PARAMETERPASSING{L}{lem:LKSsparseClass}:= \rho)\right\}\;,\\
\tau&\le 2\epsilon\pi\nu\;,\\
\frac 1{k_0}&\le \min\left\{\frac
{\gamma^3\rho\eta^8\tau\nu}{10^3\Omega_{g+1}^3}, \frac 1{k_0^*}\right\}\;,
\end{align*}
with
 \label{pageref:PAR}
 \begin{align*}
k_0^*:=\max\Big\{ & k_{0,\PARAMETERPASSING{L}{lem:LKSsparseClass}}(\eta_\PARAMETERPASSING{L}{lem:LKSsparseClass}:=\alpha,
\Lambda_\PARAMETERPASSING{L}{lem:LKSsparseClass}:=\Lambda, \gamma_\PARAMETERPASSING{L}{lem:LKSsparseClass}:= \gamma,
\epsilon_\PARAMETERPASSING{L}{lem:LKSsparseClass}:= \epsilon',
\rho_\PARAMETERPASSING{L}{lem:LKSsparseClass}:= \rho),\\
 & k_{0,\PARAMETERPASSING{L}{lem:edgesEmanatingFromDensePairsIII}}(
\Omega_\PARAMETERPASSING{L}{lem:edgesEmanatingFromDensePairsIII}:=
\Omega_{g+1},
\rho_\PARAMETERPASSING{L}{lem:edgesEmanatingFromDensePairsIII}:= \frac
{\gamma^2}4,
\epsilon_\PARAMETERPASSING{L}{lem:edgesEmanatingFromDensePairsIII}:=
\epsilonD,
\tau_\PARAMETERPASSING{L}{lem:edgesEmanatingFromDensePairsIII}:=
2\rho,
\alpha_\PARAMETERPASSING{L}{lem:edgesEmanatingFromDensePairsIII}:=
\alphaD,
\nu_\PARAMETERPASSING{L}{lem:edgesEmanatingFromDensePairsIII}:=
\frac{2\rho}{\Omega_{g+1}}
),
\\
 & k_{0,\PARAMETERPASSING{L}{prop:LKSstruct}}(\eta_\PARAMETERPASSING{L}{prop:LKSstruct}:=
\eta, \Omega_\PARAMETERPASSING{L}{prop:LKSstruct}:= \Omega_{g+1},
\gamma_\PARAMETERPASSING{L}{prop:LKSstruct}:= \gamma,
\epsilon_\PARAMETERPASSING{L}{prop:LKSstruct}:= \epsilon,
\nu_\PARAMETERPASSING{L}{prop:LKSstruct}:= \nu),
\\
 & k_{0,\PARAMETERPASSING{L}{lem:randomSplit}}(p_\PARAMETERPASSING{L}{lem:randomSplit}:=
10, \alpha_\PARAMETERPASSING{L}{lem:randomSplit}:= \eta/100),
\\
 & k_{0,\PARAMETERPASSING{L}{lem:embed10}}(\eta'_\PARAMETERPASSING{L}{lem:embed10}:=\eta/40,
d'_\PARAMETERPASSING{L}{lem:embed10}:=\gamma^2d/2,\tilde
\epsilon_\PARAMETERPASSING{L}{lem:embed10}:= \epsilon,
\Omega_\PARAMETERPASSING{L}{lem:embed10}:= \frac {(\Omega_{g+1})^2}{\gamma^2},
\nu'_\PARAMETERPASSING{L}{lem:embed10}:=\pi\sqrt{\epsilon'}\nu)\Big\}.
\end{align*}
In particular, this gives us a relation between between $\alpha$ and $k_0$. 
\medskip

Suppose now that $k>k_0$, and $G\in\LKSgraphs{n}{k}{\alpha}$ is a graph, and $T\in\treeclass{k}$ is a tree. It is our goal to show that $T\subset G$.
\medskip

We follow the plan outlined in Figure~\ref{fig:proofstructure}. First, we
process the tree $T$ by considering any $(\tau k)$-fine partition
$(W_A,W_B,\shrubA,\shrubB)$ of $T$ rooted at an arbitrary root $r$. Such a
partition exists by Lemma~\ref{lem:TreePartition}. Let $m_1$ and $m_2$ be the
total order of internal shrubs and the end shrubs, respectively. For $i=1,2$ set
$$\proporce{i}:=\frac{\eta}{100}+\frac{m_i}{(1+\frac{\eta}{30})k},$$ and
$$\proporce{0}:=\frac {\eta}{100}.$$ In particular we have
$\proporce{i}\in[\frac{\eta}{100},1]$ for $i=1,2,3$.

To find a suitable structure in the graph $G$ we proceed as follows. We apply
Lemma~\ref{lem:LKSsparseClass} with input graph
$G_\PARAMETERPASSING{L}{lem:LKSsparseClass}:=G$ and parameters
$\eta_\PARAMETERPASSING{L}{lem:LKSsparseClass}:=\alpha$,
$\Lambda_\PARAMETERPASSING{L}{lem:LKSsparseClass}:=\Lambda$,
$\gamma_\PARAMETERPASSING{L}{lem:LKSsparseClass}:=\gamma$,
$\epsilon_\PARAMETERPASSING{L}{lem:LKSsparseClass}:=\epsilon'$,
$\rho_\PARAMETERPASSING{L}{lem:LKSsparseClass}:=\rho$,  the sequence
$(\Omega_j)_{j=1}^{g+1}$, $k_\PARAMETERPASSING{L}{lem:LKSsparseClass}:=k$ and  $b_\PARAMETERPASSING{L}{lem:LKSsparseClass}:=\frac{\rho k}{100\Omega^*}$. The~lemma gives a graph
$G'_\PARAMETERPASSING{L}{lem:LKSsparseClass}\in\LKSsmallgraphs{n}{k}{\eta}$, and an index $i\in[g]$.  Slightly abusing  notation, we call this graph still~$G$. Set
 $\Omega^*:=\Omega_i$ and $\Omega^{**}:=\Omega_{i+1}$.
Now, item~\eqref{LKSclassif:prepart} of Lemma~\ref{lem:LKSsparseClass}
yields a $(k,\Omega^{**},\Omega^*,\Lambda,\gamma,\epsilon',\nu,\rho)$-sparse
decomposition $\class=(\HugeVertices, \clusters, \DenseSpots, \Gblack,
\Gexp,\smallatoms)$. Let $\clustersize$ be the size of any cluster in
$\clusters$.
 
We now apply Lemma~\ref{prop:LKSstruct} with parameters
$\eta_\PARAMETERPASSING{L}{prop:LKSstruct}:=\eta$,
$\Omega_\PARAMETERPASSING{L}{prop:LKSstruct}:=\Omega_{g+1}$,
$\gamma_\PARAMETERPASSING{L}{prop:LKSstruct}:=\gamma$,
$\epsilon_\PARAMETERPASSING{L}{prop:LKSstruct}:=\epsilon$, 
$k_\PARAMETERPASSING{L}{prop:LKSstruct}:=k$,  and
$\Omega^*_\PARAMETERPASSING{L}{prop:LKSstruct}:=\Omega^*$. Given the graph $G$ with its sparse decomposition $\class$ the lemma gives three
$(\epsilon,d,\pi\clustersize)$-semiregular matchings $\M_A$, $\M_B$, and
$\Mgood\subset \M_A$ which fulfill the assertion either of case {\bf(K1)}, or
of {\bf(K2)}. The matchings $\M_A$ and $\M_B$ also define the sets $\XA$ and $\XB$.

The additional features provided by Lemma~\ref{lem:LKSsparseClass} and Lemma~\ref{prop:LKSstruct} guarantee that we are in the situation described in Setting~\ref{commonsetting}. We apply Lemma~\ref{lem:randomSplit} as described in Definition~\ref{def:proportionalsplitting}; the numbers $\proporce{0},\proporce{1},\proporce{2}$ are as defined above. This puts us in the setting described in Setting~\ref{settingsplitting}. We now use Lemma~\ref{outerlemma} to obtain one of the following configurations.
\begin{itemize}
\item$\mathbf{(\diamond1)}$,
\item$\mathbf{(\diamond2)}\left(\frac{\eta^{27}\Omega^{**}}{4\cdot
10^{66}(\Omega^*)^{11}},\frac{\sqrt[4]{\Omega^{**}}}2,\frac{\eta^9\rho^2}{128\cdot
10^{22}\cdot (\Omega^*)^5}\right)$,
\item$\mathbf{(\diamond3)}\left(\frac{\eta^{27}\Omega^{**}}{4\cdot
10^{66}(\Omega^*)^{11}},\frac{\sqrt[4]{\Omega^{**}}}2,\frac\gamma2,\frac{\eta^9\gamma^2}{128\cdot
10^{22}\cdot(\Omega^*)^5}\right)$,
\item$\mathbf{(\diamond4)}\left(\frac{\eta^{27}\Omega^{**}}{4\cdot
10^{66}(\Omega^*)^{11}},\frac{\sqrt[4]{\Omega^{**}}}2,\frac\gamma2,\frac{\eta^9\gamma^3}{384\cdot
10^{22}(\Omega^*)^5}\right)$, 
\item$\mathbf{(\diamond5)}\left(\frac{\eta^{27}\Omega^{**}}{4\cdot
10^{66}(\Omega^*)^{11}},\frac{
\sqrt[4]{\Omega^{**}}}2,\frac{\eta^9}{128\cdot
10^{22}\cdot
(\Omega^*)^3},\frac{\eta}2,\frac{\eta^9}{128\cdot
10^{22}\cdot (\Omega^*)^4}\right)$,
  \item
  $\mathbf{(\diamond6)}\big(\frac{\eta^3\rho^4}{10^{14}(\Omega^*)^4},4\epsilonD,
  \frac{\gamma^3\rho}{32\Omega^*},\frac{\eta^2\nu}{2\cdot10^4
},\frac{3\eta^3}{2000},\proporce{2}(1+\frac\eta{20})k\big)$,
  \item $\mathbf{(\diamond7)}\big(\frac
{\eta^3\gamma^3\rho}{10^{12}(\Omega^*)^4},\frac
{\eta\gamma}{400},4\epsilonD,\frac{\gamma^3\rho}{32\Omega^*},\frac{\eta^2\nu}{2\cdot10^4
}, \frac{3\eta^3}{2\cdot 10^3}, \proporce{2}(1+\frac\eta{20})k\big)$,
  \item
$\mathbf{(\diamond8)}\big(\frac{\eta^4\gamma^4\rho}{10^{15}
(\Omega^*)^5},\frac{\eta\gamma}{400},\frac{400\epsilon}{\eta},4\epsilonD,\frac
d2,\frac{\gamma^3\rho}{32\Omega^*},\frac{\eta\pi\clustersize}{200k},\frac{\eta^2\nu}{2\cdot10^4}
, \proporce{1}(1+\frac\eta{20})k,\proporce{2}(1+\frac\eta{20})k\big)$,
 \item 
 $\mathbf{(\diamond9)}\big(\frac{\rho
\eta^8}{10^{27}(\Omega^*)^3},\frac
{2\eta^3}{10^3}, \proporce{1}(1+\frac{\eta}{40})k,
\proporce{2}(1+\frac{\eta}{20})k, \frac{400\varepsilon}{\eta},
\frac{d}2,
\frac{\eta\pi\clustersize}{200k},4\epsilonD,\frac{\gamma^3\rho}{32\Omega^*},
\frac{\eta^2\nu}{2\cdot10^4 }\big)$,
  \item $\mathbf{(\diamond10)}\big( \epsilon, \frac{\gamma^2
d}2,\pi\sqrt{\epsilon'}\nu k, \frac
{(\Omega^*)^2k}{\gamma^2},\frac\eta{40} \big)$
\end{itemize}

Depending on the actual configuration Lemma~\ref{lem:embed:greedy}, Lemma~\ref{lem:conf2-5}, Lemma~\ref{lem:embed:total68}, Lemma~\ref{lem:embed9}, or Lemma~\ref{lem:embed10} guarantee that $T\subset G$. This finishes the proof of the theorem.

\section{Theorem~\ref{thm:main} algorithmically}\label{ssec:algorithmic}\label{sec:conclremarks}
We now discuss the algorithmic aspects of our proof. That is, we would
like to find an algorithm which finds a copy of a given tree
$T\in\treeclass{k}$ in any given graph $G\in\LKSgraphs{n}{k}{\alpha}$ in time $O(n^C)$. Here the
degree $C$ of the polynomial is allowed to depend on $\alpha$, but not on $k$. It can be verified that each of
the steps of our proof --- except the extraction of dense spots (cf. Section~\ref{sssec:DecomposeAlgorithmically}) --- can be turned into a polynomial time algorithm. The two randomized steps --- random splitting in Section~\ref{ssec:RandomSplittins} and the use of the stochastic process $\Duplicate$ in Section~\ref{sec:embed} --- can be also efficiently derandomized using a standard technique for derandomizing the Chernoff bound. Let us sketch how to deal with extracting dense spots.

The idea is as follows. Initially, we pretend that $\Gexp$ consists of the entire bounded-degree part $G-\HugeVertices$ (cleaned for minimum degree $\rho k$ as in~\eqref{eq:almostalldashed}). With such a supposed sparse classification $\class_1$ we go through Lemma~\ref{prop:LKSstruct} and Lemma~\ref{outerlemma} (which builds on Lemmas~\ref{lem:ConfWhenCXAXB},~\ref{lem:ConfWhenNOTCXAXB},
and~\ref{lem:ConfWhenMatching}) to obtain a configuration. We now start embedding $T$ as in Section~\ref{sec:embed}. Note that $\Gblack$ and $\smallatoms$ are absent, and so, the only embedding techniques are those involving $\HugeVertices$ and $\Gexp$. Now, either we embed $T$, or we fail. The only possible reason for the failure is that we were unable to perform the one-step look-ahead strategy described in Section~\ref{sssec:whyGexp} because $\Gexp$ was not really nowhere-dense. But then we actually localized a dense spot $D_1$. We get an updated supposed sparse classification $\class_2$ in which $D_1$ is removed from $\Gexp$ and put in $\DenseSpots$ (which of course can give rise to $\Gblack$ or $\smallatoms$). We keep iterating. Since in each step we extract at least $O(k^2)$ edges we iterate the above at most $e(G)/\Theta(k^2)=O(\frac nk)$ times. We are certain to succeed eventually, since after $\Theta(\frac nk)$ iterations we get an honest sparse classification.

\medskip
It seems that this iterative method is generally applicable for problems which employ a sparse classification.

%
%

\section{Acknowledgments} 
The work on this project lasted from the beginning of 2008 until the end of 2012
and we are very grateful to the following institutions and funding bodies for
their support. 

\smallskip

During the work on this paper JH was also affiliated with Zentrum
Mathematik, TU Munich and Department of Computer Science, University of Warwick. JH was funded by a BAYHOST fellowship, a DAAD fellowship, 
 Charles University grant GAUK~202-10/258009, EPSRC award EP/D063191/1, and by an EPSRC Postdoctoral Fellowship during the work on the project. 
JK and ESz acknowledge the support of NSF grant
DMS-0902241.

DP was also affiliated with the Institute of Theoretical Computer Science, Charles University in Prague, Zentrum
Mathematik, TU Munich, and the Department of Computer Science and DIMAP,
University of Warwick. DP acknowledges the support of the Marie Curie fellowship FIST,
 DFG grant TA 309/2-1, a DAAD fellowship,
 Czech Ministry of
Education project 1M0545,  EPSRC award EP/D063191/1, and grant  PIEF-GA-2009-253925 of the European Union's Seventh Framework
Programme (FP7/2007-2013).
DP would also like to acknowledge the exceptional support of the EPSRC
Additional Sponsorship, with a grant reference of EP/J501414/1 which facilitated her to
travel with her young child and so she could continue to collaborate closely
with her coauthors on this project. This grant was also used to host MS in
Birmingham.

MS was affiliated with the Institute of Mathematics and Statistics, University of S\~ao Paulo, and the Centre for Mathematical Modeling, University of Chile. She was
supported by a FAPESP fellowship, and by FAPESP travel grant  PQ-EX 2008/50338-0, also
CMM-Basal, and  FONDECYT grant 11090141. She also received funding by EPSRC Additional Sponsorship EP/J501414/1.

We enjoyed the hospitality of the School of Mathematics of University of Birmingham, Center for Mathematical Modeling, University of Chile, Alfr\'ed R\'enyi Institute of Mathematics of the Hungarian Academy of Sciences and Charles University, Prague, during our long term visits.

We are very grateful to Mikl\'os Ajtai, J\'anos Koml\'os, Mikl\'os Simonovits, and Endre Szemer\'edi. Their yet unpublished work on the Erd\H{o}s-S\'os Conjecture was the starting point for our project, and our solution crucially relies on the methods developed for the Erd\H{o}s-S\'os Conjecture.

JH would like to thank Maxim Sviridenko for discussion on the algorithmic aspects of the problem.

\medskip
A doctoral thesis entitled \emph{Structural graph theory} submitted by Jan Hladk\'y in September 2012 under the supervision of Daniel Kr\'al at~Charles University in~Prague is based on this paper. The texts of the two works are almost identical. We are grateful to PhD committee members Peter Keevash and Michael Krivelevich. Their valuable comments are reflected in this paper.
\newpage
\printindex{mathsymbols}{Symbol index}
\printindex{general}{General index}

\newpage
\addcontentsline{toc}{section}{Bibliography}
\bibliographystyle{alpha}
\bibliography{bibl}
\end{document}